\documentclass{amsart}
\usepackage{amssymb, amsmath, amsthm}

\usepackage[dvipsnames]{xcolor}
\usepackage{hyperref}
\usepackage{stmaryrd}
\usepackage[all]{xy}
\usepackage{tikz-cd}
\usepackage{cleveref}

\newtheorem{prop}{Proposition}[section]

\newtheorem{lem}[prop]{Lemma}
\newtheorem{cor}[prop]{Corollary}
\newtheorem{thm}[prop]{Theorem}
\theoremstyle{definition}
\newtheorem{defi}[prop]{Definition}
\theoremstyle{remark}
\newtheorem{examp}[prop]{Example}

\newtheorem{remar}[prop]{Remark}

\newcommand{\stackcite}[1]{\cite[\href{https://stacks.math.columbia.edu/tag/#1}{Tag #1}]{stacks-project}}

\DeclareMathAlphabet{\mathpzc}{OT1}{pzc}{m}{it}
\DeclareMathOperator{\Aut}{Aut}

\DeclareMathOperator{\Hom}{Hom}

\DeclareMathOperator{\Sym}{Sym}

\DeclareMathOperator{\GL}{GL}
\DeclareMathOperator{\SL}{SL}

\DeclareMathOperator{\Ker}{Ker}

\DeclareMathOperator{\Gal}{Gal}
\DeclareMathOperator{\Image}{Im}

\def\rank{\mathop{\mathrm{ rank}}\nolimits}
\DeclareMathOperator{\tr}{tr}

\DeclareMathOperator{\Spec}{Spec}

\DeclareMathOperator{\id}{id}

\DeclareMathOperator{\Sp}{Sp}
\DeclareMathOperator{\Rep}{Rep}
\DeclareMathOperator{\coker}{coker}

\DeclareMathOperator{\rad}{rad}

\DeclareMathOperator{\Art}{Art}

\DeclareMathOperator{\Set}{Set}

\newcommand{\Qp}{\mathbb {Q}_p}
\newcommand{\Zp}{\mathbb{Z}_p}
\newcommand{\Qpbar}{\overline{\mathbb{Q}}_p}

\newcommand{\Zbar}{\overline{Z}}

\newcommand{\Eins}{\mathbf 1}

\newcommand{\ZZ}{\mathbb Z}

\newcommand{\QQ}{\mathbb Q}
\newcommand{\Aa}{\mathfrak A}
\newcommand{\ev}{\mathrm{ev}}
\newcommand{\ab}{\mathrm{ab}}
\newcommand{\Fp}{\mathbb F_p}

\newcommand{\Vbar}{\overline{{V}}}

\newcommand{\mm}{\mathfrak m}

\newcommand{\OO}{\mathcal O}

\DeclareMathOperator{\ad}{ad}

\DeclareMathOperator{\wtimes}{\widehat{\otimes}}

\newcommand{\md}{\mathrm m}

\newcommand{\pp}{\mathfrak p}

\newcommand{\br}[1]{\llbracket #1\rrbracket}

\newcommand{\qq}{\mathfrak{q}}

\newcommand{\alg}{\mathrm{alg}}
\newcommand{\cont}{\mathrm{cont}}

\newcommand{\semi}{\mathrm{ss}}

\newcommand{\rhobar}{\overline{\rho}}

\newcommand{\Std}{\mathrm{Std}}

\newcommand{\Ga}{\mathbb G_{a}}
\newcommand{\Gm}{\mathbb G_{m}}

\newcommand{\ps}{\mathrm{ps}}
\newcommand{\CH}{\mathrm{CH}}
\newcommand{\gen}{\mathrm{gen}}
\newcommand{\gps}{\mathrm{git}}
\DeclareMathOperator{\PGL}{PGL}

\newcommand{\red}{\mathrm{red}}
\newcommand{\univ}{\mathrm{univ}}
\newcommand{\Xbar}{\overline{X}}
\newcommand{\Dbar}{\overline{D}}

\DeclareMathOperator{\spcl}{spcl}
\DeclareMathOperator{\nspcl}{n-spcl}

\newcommand{\cyc}{\mathrm{cyc}}
\newcommand{\cha}{\mathrm{char}}

\newcommand{\kappabar}{\overline{\kappa}}
\DeclareMathOperator{\Lie}{Lie}
\newcommand{\RpsGLd}{R^{\ps}_{\GL_d}}

\newcommand{\RpsG}{R^{\ps}_G}
\newcommand{\RpsH}{R^{\ps}_H}

\newcommand{\lin}{\mathrm{lin}}
\newcommand{\XgenGLd}{X^{\gen}_{\GL_d}}

\newcommand{\AgenGLd}{A^{\gen}_{\GL_d}}
\newcommand{\XpsGLd}{X^{\ps}_{\GL_d}}
\newcommand{\XbarpsGLd}{\overline{X}^{\ps}_{\GL_d}}
\newcommand{\XgenG}{X^{\gen}_G}
\newcommand{\XgenGtau}{X^{\gen, \tau}_G}
\newcommand{\XgenGrhobarss}{X^{\gen}_{G,\rhobarss}}

\newcommand{\XgenGtaurhobarss}{X^{\gen, \tau}_{G,\rhobarss}}
\newcommand{\XbargenG}{\overline{X}^{\gen}_G}
\newcommand{\XbargenGtau}{\overline{X}^{\gen,\tau}_G}

\newcommand{\XbargenH}{\overline{X}^{\gen}_H}

\newcommand{\XbargpsHtau}{\overline{X}^{\gps,\tau}_H}

\newcommand{\XgpsG}{X^{\gps}_{G}}
\newcommand{\XgpsGtau}{X^{\gps, \tau}_{G}}
\newcommand{\XgpsHtau}{X^{\gps, \tau}_{H}}
\newcommand{\chara}{\mathrm{char}}
\newcommand{\XbargpsG}{\overline{X}^{\gps}_G}
\newcommand{\XbargpsGtau}{\overline{X}^{\gps,\tau}_G}
\newcommand{\etabar}{\bar{\eta}}
\newcommand{\xbar}{\bar{x}}

\newcommand{\pts}{\mathrm{pts}}
\newcommand{\kbar}{\overline{k}}
\newcommand{\rhobarss}{\rhobar^{\mathrm{ss}}}
\newcommand{\PC}{\mathrm{PC}}
\newcommand{\cPC}{\mathrm{cPC}}
\newcommand{\Map}{\mathrm{Map}}
\newcommand{\hyphen}{\text{-}}
\newcommand{\Gbar}{\overline{G}}
\newcommand{\Pbar}{\overline{P}}
\newcommand{\Lbar}{\overline{L}}

\newcommand{\ybar}{\bar{y}}
\renewcommand{\xbar}{\bar{x}}
\newcommand{\XpsG}{X^{\ps}_G}

\newcommand{\disc}{\operatorname{disc}} 
\newcommand{\Ubar}{\overline{U}}
\newcommand{\psibar}{\overline{\psi}}
\newcommand{\Thetabar}{\overline{\Theta}}
\newcommand{\Cont}{\mathcal{C}}
\newcommand{\eqto}{\xrightarrow{\sim}}
\newcommand{\sbar}{\bar{s}}
\newcommand{\lambdabar}{\bar{\lambda}}
\newcommand{\AgenGtau}{A^{\gen, \tau}_G}
\newcommand{\AgenGtaurhobarss}{A^{\gen, \tau}_{G, \rhobarss}}
\newcommand{\zbar}{\bar{z}}
\newcommand{\Xc}{\mathrm{X}}
\newcommand{\git}{\mathrm{git}}
\newcommand{\psibarss}{\psibar^{\mathrm {ss}}}
\newcommand{\Int}{\operatorname{Int}}
\newcommand{\sic}{\operatorname{sc}}

\newcommand{\Rbar}{\overline{R}}

\usepackage[record=only,toc=false]{glossaries-extra}

\makeatletter
\renewcommand\tableofcontents{
  \@starttoc{toc}
}
\makeatother

\newglossaryentry{ad}{
  name={$\ad\rhobar$},
  description={adjoint representation associated to $\rhobar$}
}

\newglossaryentry{ad0varphi}{
  name={$\ad^{0,\varphi}$},
  description={\Cref{present_over_RH}}
}

\newglossaryentry{AgenGtau}{
  name={$A^{\gen(, \tau)}_G$},
  description={coordinate ring representing $X^{\gen,\tau}_G$}
}

\newglossaryentry{AgenGrhobarss}{
  name={$A^{\gen}_{G,\rhobarss}$},
  description={abbreviation of $A^{\gen, \tau}_{G,\rhobarss}$ as an $R^{\ps}_G$-algebra used after \Cref{indep_tau}}
}

\newglossaryentry{AgenGtaurhobarss}{
  name={$A^{\gen, \tau}_{G,\rhobarss}$},
  description={coordinate ring of $X^{\gen, \tau}_{G,\rhobarss}$}
}

\newglossaryentry{Af}{
  name={$\mathfrak A_{\OO}$},
  description={category of local Artinian $\OO$-algebras with residue field $k$}
}

\newglossaryentry{ALambda}{
  name={$\mathfrak A_{\Lambda}$},
  description={category of local Artinian $\Lambda$-algebras with residue field $\kappa$}
}

\newglossaryentry{Achi}{
    name={$A^{\chi}$},
    description={\Cref{sec_gen_tori}}
}

\newglossaryentry{CH}{
  name={${}^C H$},
  description={$C$-group attached to a connected reductive group $H$}
}

\newglossaryentry{CayHam}{
  name={$\CH(D)$},
  description={Cayley--Hamilton ideal of a determinant law $D$}
}

\newglossaryentry{cPCA}{
    name={$\cPC^{\Gamma}_G(A)$},
    description={set of continuous $G$-pseudocharacters of $\Gamma$ with values in $A$}
}

\newglossaryentry{cPL}{
    name={$c_{P,L}$},
    description={semisimplification map $P(\kappa) \to G(\kappa)$, \Cref{sec_G_ss}}
}

\newglossaryentry{disc}{
    name={$(-)_{\disc}$},
    description={discrete condensed set associated to a set}
}

\newglossaryentry{Dbar}{
    name={$\overline D$},
    description={determinant law associated to $\rhobar$}
}

\newglossaryentry{Drho}{
  name={$D_{\rho}$},
  description={determinant law associated to $\rho$}
}

\newglossaryentry{Df}{
  name={$D^{\square}_{(\rho,) G}$},
  description={framed deformation functor of $\rho$ with values in $G$, \Cref{sec_def_rings}}
}

\newglossaryentry{Du}{
    name={$D^u$},
    description={universal determinant law extended to $R^{\ps}_{\GL_d}\br{\Gamma}$}
}

\newglossaryentry{DuA}{
    name={$D^u_{|A}$},
    description={specialisation of $D^u$ along $R^{\ps}_{\GL_d} \to A$}
}

\newglossaryentry{DuEu}{
    name={$D^u_{E^u}$},
    description={determinant law induced by $D^u$ on $E^u$}
}

\newglossaryentry{DThetabar}{
    name={$D_{\Thetabar}$},
    description={deformation functor of $\Thetabar$, \Cref{pseudodeffunctor}}
}

\newglossaryentry{DpsiGrho}{
    name={$D^{\square,\psi_1}_{G,\rho}$},
    description={\Cref{nightmare_cont}}
}

\newglossaryentry{E}{
  name={$E$},
  description={a finite extension of $F$ such that $\Gamma_E = \ker(\Gamma_F \xrightarrow{\rhobar} G(k) \to (G/G^0)(k))$}
}

\newglossaryentry{Eu}{
  name={$E^u$},
  description={universal Cayley--Hamilton quotient $R^{\ps}_{\GL_d}\br{\Gamma}/\CH(D^u)$}
}

\newglossaryentry{FN}{
    name={$\mathcal F_N$},
    description={free abstract group on $N$ generators}
}

\newglossaryentry{G}{
  name={$G$},
  description={a generalised reductive group scheme over $\OO$}
}

\newglossaryentry{Gprime}{
  name={$G'$},
  description={derived group scheme of $G^0$}
}

\newglossaryentry{Gprimesc}{
  name={$G'_{\sic}$},
  description={simply connected central cover of $G'$}
}

\newglossaryentry{G0}{
  name={$G^0$},
  description={neutral component of $G$}
}

\newglossaryentry{Llambda}{
    name = {$L_{\lambda}$},
    description={R-Levi attached to a cocharacter $\lambda$, \Cref{sec_R_parabolic}}
}

\newglossaryentry{LH}{
    name={${}^LH$},
    description={$L$-group of $H$, \Cref{LandC}}
}

\newglossaryentry{mA}{
  name={$\mm_A$},
  description={maximal ideal of a local ring $A$, \Cref{sec_def_rings}}
}

\newglossaryentry{N}{
    name={$N$},
    description={fixed number of topological generators of $Q$, \Cref{def_and_1prop}}
}

\newglossaryentry{Okappa}{
    name={$\OO_{\kappa}$},
    description={\Cref{sec_def_rings}}
}

\newglossaryentry{PC}{
    name={$\PC^{\Gamma}_G$},
    description={scheme of $G$-pseudocharacters of $\Gamma$, \Cref{Laf}}
}

\newglossaryentry{Plambda}{
    name = {$P_{\lambda}$},
    description={R-parabolic attached to a cocharacter $\lambda$, \Cref{sec_R_parabolic}}
}

\newglossaryentry{Q}{
    name={$Q$},
    description={fixed quotient of $\Gamma$ as in \Cref{fg_quotient}}
}

\newglossaryentry{RepGsq}{
  name={$\Rep^{\Gamma(, \square)}_G$},
  description={scheme of group homomorphisms $\Gamma \to G(A)$, Sections \ref{sec_ci}, \ref{def_and_1prop}}
}

\newglossaryentry{RgitGtau}{
    name={$R^{\git,\tau}_G$},
    description={invariant ring $(A^{\gen, \tau}_G)^{G^0}$, coordinate ring of $X^{\git,\tau}_G$, \Cref{sec_points_in_XgenGtau}}
}

\newglossaryentry{RgitGrhobarss}{
  name={$R^{\git}_{G,\rhobarss}$},
  description={abbreviation of $R^{\git, \tau}_{G,\rhobarss}$ as an $R^{\ps}_G$-algebra used after \Cref{indep_tau}}
}

\newglossaryentry{RgitGtaurhobarss}{
    name={$R^{\git,\tau}_{G, \rhobarss}$},
    description={connected component of $R^{\git,\tau}_G$ corresponding to $\rhobarss$, \Cref{sec_points_in_XgenGtau}}
}

\newglossaryentry{RsqrhoG}{
    name={$R^{\square}_{(\rho,)G}$},
    description={universal deformation ring representing $D^{\square}_{\rho,G}$, \Cref{sec_def_rings}}
}

\newglossaryentry{Rsp}{
  name={$R^{\square}_{\psibar}$},
  description={the universal framed deformation ring of $\psibar$}
}

\newglossaryentry{Rspr}{
  name={$R^{\square,\psi_1}_{\rhobar}$},
  description={the universal deformation ring pro-representing $D^{\square,\psi_1}_{\rhobar}$}
}

\newglossaryentry{RpsG}{
  name={$R^{\ps}_G$},
  description={the complete local noetherian $\OO$-algebra pro-representing $D_{\Thetabar}$, \Cref{sec_def_GPC}}
}

\newglossaryentry{RpsGLd}{
  name={$R^{\ps}_{\GL_d}$},
  description={universal deformation ring of $\overline D$, \Cref{def_and_1prop}}
}

\newglossaryentry{RpsiGrho}{
    name={$R^{\square,\psi_1}_{G,\rho}$},
    description={\Cref{nightmare_cont}}
}

\newglossaryentry{S}{
    name={$\underline S$},
    description={condensed set associated to a topological space $S$}
}

\newglossaryentry{Ugit}{
    name={$U^{\git, \tau}_{G,\rhobarss}$},
    description={punctured space $X^{\git, \tau}_{G,\rhobarss} \setminus \{\text{closed point}\}$}
}

\newglossaryentry{ULG}{
    name={$U_{LG}$, $\overline U_{LG}$},
    description={\Cref{sec_abs_irr_loc}}
}

\newglossaryentry{ULGrhobarss}{
    name={$U_{LG, \rhobarss}$, $\overline U_{LG, \rhobarss}$},
    description={component of $U_{LG}$, $\overline U_{LG}$ associated to $\rhobarss$}
}

\newglossaryentry{UpsG}{
  name={$U^{\ps}_G$},
  description={punctured spectrum $(\Spec R^{\ps}_G) \setminus \{\text{closed point}\}$}
}

\newglossaryentry{Ulambda}{
    name = {$U_{\lambda}$},
    description={unipotent radical of $P_{\lambda}$, \Cref{sec_R_parabolic}}
}

\newglossaryentry{VLG}{
    name={$V_{LG}$, $\overline V_{LG}$},
    description={\Cref{sec_abs_irr_loc}}
}

\newglossaryentry{Vnspcl}{
    name={$V^{\nspcl}_{GG, \rhobarss}$, $\overline V^{\nspcl}_{GG, \rhobarss}$},
    description={\Cref{defi_Vnspcl}}
}

\newglossaryentry{Vspcl}{
    name={$\overline V^{\spcl}_{GG, \rhobarss}$},
    description={\Cref{bound_special}}
}

\newglossaryentry{X//H}{
    name={$X \sslash H$},
    description={affine GIT quotient of affine scheme $X$ by $H$, i.e. $\Spec(\OO(X)^H)$}
}

\newglossaryentry{XgenGLd}{
    name={$X^{\gen}_{\GL_d}$},
    description={Wang-Erickson's space of generic matrices, \Cref{defXgenGLd}}
}

\newglossaryentry{XgenTau}{
  name={$X^{\gen(,\tau)}_G$},
  description={closed subscheme of $X^{\gen}_{\GL_d}$ of representations factoring through $\tau$, \Cref{defiXgenG}}
}

\newglossaryentry{XgenTaubar}{
  name={$\overline X^{\gen,\tau}_G$},
  description={special fibre of $X^{\gen,\tau}_G$}
}

\newglossaryentry{XgenGrhobarss}{
  name={$X^{\gen}_{G,\rhobarss}$},
  description={abbreviation of $X^{\gen, \tau}_{G,\rhobarss}$ as an $R^{\ps}_G$-scheme used after \Cref{indep_tau}}
}

\newglossaryentry{XgentauGrhobarss}{
    name={$X^{\gen, \tau}_{G,\rhobarss}$},
    description={component of $X^{\gen, \tau}_G$ corresponding to $\rhobarss$, \Cref{sec_def_XgenGtaurhobarss}}
}

\newglossaryentry{XgenpsiGrhobarss}{
    name={$X^{\gen, \psi_1}_{G,\rhobarss}$, $\overline X^{\gen, \psi_1}_{G,\rhobarss}$},
    description={\Cref{nightmare_cont}}
}

\newglossaryentry{XgenPrho}{
    name={$X^{\gen}_{P,\rho}$},
    description={\Cref{def_XgenPrho}}
}

\newglossaryentry{XgitGtau}{
    name={$X^{\git(, \tau)}_G$},
    description={GIT quotient $X^{\gen, \tau}_G \sslash G^0$, \Cref{sec_GIT}}
}

\newglossaryentry{XgitGtaubar}{
    name={$\overline X^{\git(, \tau)}_G$},
    description={GIT quotient $\overline X^{\gen, \tau}_G \sslash G^0$, \Cref{sec_GIT}}
}

\newglossaryentry{XgittauHG}{
    name={$X^{\git(, \tau)}_{HG}$, $\overline X^{\git(, \tau)}_{HG}$},
    description={image of $X^{\git, \tau}_H \to X^{\git, \tau}_G$ (resp. $\overline X^{\git, \tau}_H \to \overline X^{\git, \tau}_G$)}
}

\newglossaryentry{XgittauHGrhobarss}{
    name={$X^{\git(, \tau)}_{HG,\rhobarss}$},
    description={component of $X^{\git(, \tau)}_{HG}$ corresponding to $\rhobarss$}
}

\newglossaryentry{XgittauHGbar}{
    name={$\overline X^{\git(, \tau)}_{HG}$},
    description={special fiber of $X^{\git(, \tau)}_{HG}$, as a topological space, \Cref{remar_XgittauHG_bar}}
}

\newglossaryentry{XgitGrhobarss}{
  name={$X^{\git}_{G,\rhobarss}$},
  description={abbreviation of $X^{\git, \tau}_{G,\rhobarss}$ as an $R^{\ps}_G$-scheme used after \Cref{indep_tau}}
}

\newglossaryentry{XgitGtaurhobarss}{
    name={$X^{\git,\tau}_{G, \rhobarss}$},
    description={connected component of $X^{\git,\tau}_G$ corresponding to $\rhobarss$, \Cref{sec_points_in_XgenGtau}}
}

\newglossaryentry{XpsGLd}{
    name={$X^{\ps}_{\GL_d}$},
    description={spectrum of $R^{\ps}_{\GL_d}$, the universal deformation ring of $\overline D$, \Cref{def_and_1prop}}
}

\newglossaryentry{XpsG}{
    name={$X^{\ps}_{G(, \rhobarss)}$, $\overline X^{\ps}_{G(, \rhobarss)}$},
    description={spectrum of $R^{\ps}_{G(, \rhobarss)}$, $R^{\ps}_{G(, \rhobarss)}/\varpi$}
}

\newglossaryentry{Xspcl}{
    name={$X^{\spcl}_{G, \rhobarss}$, $\overline X^{\spcl}_{G, \rhobarss}$},
    description={\Cref{bound_special}}
}

\newglossaryentry{Xmu}{
    name={$X(\mu)$},
    description={group of characters $\mu \to \OO^{\times}$}
}

\newglossaryentry{Xchi}{
    name={$X^{\chi}$},
    description={\Cref{sec_gen_tori}}
}

\newglossaryentry{XsqGrhobar}{
    name={$X^{\square}_{G,\rhobar}$, $\overline X^{\square}_{G,\rhobar}$},
    description={spectrum of $R^{\square}_{G,\rhobar}$, $R^{\square}_{G,\rhobar}/\varpi$}
}

\newglossaryentry{Y}{
    name={$Y$},
    description={preimage of closed points of $X^{\git, \tau}_G$ in $X^{\gen, \tau}_G$, \Cref{sec_abs_irr_loc}}
}

\newglossaryentry{Yrhobarss}{
    name={$Y_{\rhobarss}$},
    description={preimage of closed point of $X^{\ps}_G$ in $X^{\gen}_{G,\rhobarss}$}
}

\newglossaryentry{Ypsirhobarss}{
    name={$Y^{\psi_1}_{\rhobarss}$},
    description={\Cref{sec_gen_tori}}
}

\newglossaryentry{ZGH}{
  name={$Z_G(H)$},
  description={centraliser of a subgroup scheme $H \leq G$}
}

\newglossaryentry{ZH}{
  name={$Z(H)$},
  description={centre of a group scheme $H$, \Cref{ZGexists}}
}

\newglossaryentry{uZH}{
  name={$\underline Z(H)$},
  description={centre functor of a group scheme $H$, \Cref{ZGexists}}
}

\newglossaryentry{Zi}{
    name={$Z^i(\Gamma, V)$},
    description={set of continuous $i$-cocycles}
}

\newglossaryentry{Zspcl}{
    name={$\overline Z^{\spcl}_{GG, \rhobarss}$},
    description={\Cref{bound_special}}
}

\newglossaryentry{Delta}{
    name={$\Delta$},
    description={image of $\pi_G \circ \rhobarss : \Gamma \to (G/G^0)(\kbar)$}
}

\newglossaryentry{uDelta}{
  name={$\underline \Delta$},
  description={constant group scheme associated to $\Delta$}
}

\newglossaryentry{Gammaabp}{
    name={$\Gamma_F^{\ab,p}$},
    description={maximal abelian pro-$p$ quotient of $\Gamma_F$}
}

\newglossaryentry{gammai}{
    name={$\gamma_i$},
    description={for $1 \leq i \leq N$ fixed topological generators of $Q$, \Cref{def_and_1prop}}
}

\newglossaryentry{Thetabar}{
  name={$\Thetabar$},
  description={the $G$-pseudocharacter attached to $\rhobar$}
}

\newglossaryentry{Thetan}{
    name={$\Theta_n$},
    description={homomorphism $\Theta_n : \OO[G^n]^{G^0} \to \Map(\Gamma^n, A)$, part of $G$-pseudocharacter $\Theta$}
}

\newglossaryentry{ThetauA}{
    name={$\Theta^u_{|A}$},
    description={specialisation of $\Theta^u$ along $R^{\ps}_G \to A$}
}

\newglossaryentry{ThetauG}{
    name={$\Theta^u_{(G)}$},
    description={universal $G$-pseudocharacter over $R^{\ps}_G$}
}

\newglossaryentry{Thetarho}{
    name={$\Theta_{\rho}$},
    description={$G$-pseudocharacter associated with $\rho$}
}

\newglossaryentry{kappa}{
  name={$\kappa$},
  description={finite or local field in \Cref{sec_def_rings}, otherwise a (topological) field}
}

\newglossaryentry{kappap}{
  name={$\kappa(\pp)$},
  description={residue field at prime ideal $\pp$}
}

\newglossaryentry{kappaeps}{
  name={$\kappa[\varepsilon]$},
  description={ring of dual numbers $\kappa[\varepsilon] := \kappa[x]/(x^2)$ of $\kappa$}
}

\newglossaryentry{Lambda}{
  name={$\Lambda$},
  description={coefficient ring for $\kappa$, \Cref{sec_def_rings}}
}

\newglossaryentry{Lambda0}{
  name={$\Lambda^0$},
  description={subring of coefficient ring $\Lambda$ for $\kappa$, \Cref{sec_def_rings}}
}

\newglossaryentry{mupinf}{
  name={$\mu_{p^{\infty}}(E)$},
  description={the group of all $p$-power roots of unity contained in $E$}
}

\newglossaryentry{pi1Gprime}{
  name={$\pi_1(G')$},
  description={kernel of $G'_{\sic}\to G'$}
}

\newglossaryentry{piG}{
    name={$\pi_G$},
    description={projection $G \to G/G^0$, \Cref{sec_abs_irr}}
}

\newglossaryentry{rhobar}{
  name={$\rhobar$},
  description={a continuous representation $\Gamma_F \to G(k)$}
}

\newglossaryentry{rhoA}{
  name={$\rho_A$},
  description={a continuous representation $\Gamma_F \to G(A)$ lifting a given residual representation}
}

\newglossaryentry{rhoss}{
  name={$\rho^{\semi}$},
  description={$G$-semisimplification of $\rho$, \Cref{defi_G_ss}}
}

\newglossaryentry{tau}{
  name={$\tau$},
  description={fixed faithful representation $\tau : G \to \GL_d$, \Cref{def_and_1prop}}
}

\newglossaryentry{cyclo}{
  name={$\chi_{\cyc}$},
  description={the $p$-adic cyclotomic character}
}

\newglossaryentry{hatotimes}{
    name={$\widehat{\otimes}$},
    description={completed tensor product, \stackcite{0AMU}}
}

\newglossaryentry{tensd}{
    name={$M^{\otimes_R d}$},
    description={abbreviation for the $d$-fold tensor product $M \otimes_R \cdots \otimes_R M$}
}

\newglossaryentry{sharp}{
    name={$(-)^{\sharp}$},
    description={for a map $f : X \to Y$ of schemes $f^{\sharp} : \Gamma(Y,\OO_Y) \to \Gamma(X,\OO_X)$}
}

\title[On local Galois deformation rings]{On local Galois deformation rings:\\ generalised reductive groups}
\author{Vytautas Pa\v{s}k\={u}nas and  Julian Quast }
\date{\today.}
\begin{document} 

\begin{abstract} We study deformation theory of mod $p$ Galois representations 
of $p$-adic fields with values in generalised reductive group schemes, 
such as $L$-groups and $C$-groups. We show that the corresponding deformation
rings are complete intersections of expected dimension. We determine their 
irreducible components  in many cases and show that they and their special fibres are 
normal and complete intersection.  
\end{abstract}

\maketitle

\section*{\contentsname}
\addtocontents{toc}{\protect\setcounter{tocdepth}{-1}} 
\tableofcontents
\addtocontents{toc}{\protect\setcounter{tocdepth}{2}}

\section{Introduction}

Let $p$ denote any prime number, let $F$ be a finite extension of $\Qp$, and let $\Gamma_F$ denote its absolute Galois group. Let $L$ be another finite extension of $\Qp$ with ring of integers $\OO$, uniformiser $\varpi$ and residue field $k=\OO/\varpi$.

Let $G$ be a smooth affine group scheme over $\OO$, such that its neutral 
component $G^0$ is reductive and the component group $G/G^0$ is a finite 
 group scheme over $\OO$. We call such group schemes \emph{generalised
reductive}. We do not make any assumptions on the prime $p$ regarding $G$. 

We fix a continuous representation $\rhobar:\Gamma_F\rightarrow G(k)$ and denote by $D^{\square}_{\rhobar}: \mathfrak A_{\OO}\rightarrow \Set$ the functor from the category $\mathfrak A_{\OO}$ of local artinian $\OO$-algebras with residue field $k$ to the category of sets, such that for $(A,\mm_A)\in \mathfrak A_{\OO}$, $D^{\square}_{\rhobar}(A)$ is the set of continuous representations 
$\rho_A: \Gamma_F\rightarrow G(A)$, such that 
$\rho_A(\gamma) \equiv \rhobar(\gamma) \pmod{\mm_A}$, for all $\gamma\in \Gamma_F.$
The 
functor $D^{\square}_{\rhobar}$ of framed deformations of $\rhobar$ is 
pro-represented by a complete local 
noetherian $\OO$-algebra $R^{\square}_{\rhobar}$ with residue field~$k$. 

B\"ockle, Iyengar and VP have studied in \cite{BIP_new} ring theoretic 
properties of $R^{\square}_{\rhobar}$, when $G=\GL_d$. In this paper 
we extend the results of \cite{BIP_new} to an arbitrary generalised
reductive group $G$.

\begin{thm}[Corollaries \ref{complete_intersection}, \ref{R_norm_red}]\label{thm_intro-1}
The ring $R^{\square}_{\rhobar}$ is a local complete intersection, flat over $\OO$ and of relative dimension $\dim G_k ([F:\Qp]+1)$. In particular, every continuous representation $\rhobar: \Gamma_F\rightarrow G(k)$ has a lift to characteristic zero. Moreover, $R^{\square}_{\rhobar}$ is reduced and $R^{\square}_{\rhobar}[1/p]$ is normal. 
\end{thm}

Obstruction theory provides a presentation $R^\square_{\rhobar} = \OO\br{x_1,\dots,x_r}/(f_1, \ldots, f_s)$ with $r$ equal to the dimension of the tangent space and $s$ equal to $\dim_k H^2(\Gamma_F, \ad\rhobar)$.
The Euler--Poincar\'e characteristic formula from local class field theory gives
    \begin{equation}\label{EP_bound}
    r-s=\dim G_k ([F:\Qp]+1).
    \end{equation}
Our theorem proves that $\dim R^\square_{\rhobar}/\varpi$ is given by this cohomological quantity, the \emph{expected dimension} in the spirit of the \emph{Dimension Conjecture} of Gouv\^{e}a from \cite[Lecture 4]{Gouvea}. Having the expected dimension implies that $\varpi,f_1, \ldots, f_s$ is a regular sequence and that $R^\square_{\rhobar}$ is a local complete intersection. It also implies
(see \cite[Lemma 7.5]{DDR_GV}) that the derived deformation ring of $\rhobar$ as introduced 
by Galatius and Venkatesh in \cite{DDR_GV} 
is homotopy discrete, which means the derived deformation theory 
of $\rhobar$ does not contain more information than
the usual deformation 
theory of $\rhobar$. 

We expect that our results will play an important role in the categorical
$p$-adic local Langlands correspondence. Conjecturally 
\cite[Conjecture 6.1.14]{padic_LL} on the Galois side one should 
consider a derived category of coherent sheaves (or some $\infty$-category 
version of it) on the Emerton--Gee stack. 
The local rings considered in \Cref{thm_intro-1} should arise as 
versal rings at finite type points of the Emerton--Gee stack, when $G$
is an $L$-group or a $C$-group of a connected reductive 
group defined over $F$ (see \cite[Proposition 3.6.3]{EG_stack}, when $G=\GL_d$). In other words, they are expected to describe the 
local properties of the Emerton--Gee stack. An instance of this is 
a paper of Yu Min \cite{yu_min}, where he uses the analogue of \Cref{thm_intro-1} 
proved in \cite{BIP_new} for $\GL_d$ 
to show that if $G=\GL_d$ then the derived version of the Emerton--Gee 
stack is equivalent to the classical Emerton--Gee stack, obtaining 
an analogue of the result explained above, that 
derived deformation rings are homotopy discrete, over the whole of the stack. The
Emerton--Gee stack has been recently defined and studied for $\mathrm{GSp}_4$ in \cite{heejong_lee} and for
tame groups in \cite{lin}.
We infinitesimally guarantee that the results of \cite{yu_min} will extend 
to this context. 

In the paper we introduce and study a scheme $X^{\gen}_{G, \rhobarss}$;
its definition is reviewed in \Cref{subsub_defi} below.
Its $\kbar$-points correspond to continuous representations 
$\rho:\Gamma_F\rightarrow G(\kbar)$ with $G$-semisimplification 
equal to $\rhobarss$. It carries an 
action of $G^0$ and the quotient stack $[X^{\gen}_{G, \rhobarss}/ G^0]$
is an analogue of the stacks studied by Wang-Erickson in \cite{WE_alg}
for $G=\GL_d$. We expect that the disjoint union 
of the stacks $[X^{\gen}_{G, \rhobarss}/ G^0]$ (or rather their formal 
versions) taken over all possible 
$\rhobarss$ will be the largest substack of the Emerton--Gee stack
on which the universal $(\varphi, \Gamma)$-module can be realised
as a representation of $\Gamma_F$, i.e.\,the situation described 
in \cite[Remark 6.7.4]{EG_stack} for $\GL_d$ will continue to hold 
in a more general setting.

\subsection{The conjecture of B\"ockle--Juschka}
Our second main result proves a generalisation of a conjecture 
of B\"ockle--Juschka concerning  the irreducible components of $R^{\square}_{\rhobar}$.
We assume that $G^0$ is split over $\OO$ and $G/G^0$ is a constant group scheme. 
This can be achieved after replacing $L$ by a finite unramified extension. 
Let $\Gamma_E$ be the kernel of the homomorphism 
$\Gamma_F\overset{\rhobar}{\longrightarrow} G(k) \rightarrow (G/G^0)(k)$. 
We assume that $\OO$ contains $\mu_{p^{\infty}}(E)$. 
One may further assume that this map is surjective, as replacing 
$G$ with a subgroup does not change the 
deformation problem, so that $(G/G^0)(k)= \Gal(E/F)$. 

Let $G'$ be the 
derived group scheme of $G^0$, let $G'_{\sic}\rightarrow G'$ 
be the simply connected central cover of $G'$ and let $\pi_1(G')$ be the the kernel of this map. Then $\pi_1(G')$ is a finite diagonalisable subgroup 
scheme of the centre of $G'_{\sic}$ and is isomorphic to $\prod_{i=1}^r \mu_{n_i}$ for some integers $n_i$. The group scheme $\pi_1(G')$
is \'etale if and only if $p$ does not divide $n_i$ for all $i$.

Our assumptions on $G$ imply that 
$G^0/G'$ is a split torus.
The action of $G/G'$ on $G^0/G'$ by conjugation induces an action of 
$\Gal(E/F)$ on its \emph{character lattice} $M:=X^*(G^0/G')$. 

Let $\psibar: \Gamma_F \rightarrow (G/G')(k)$ be the representation 
$\psibar= \varphi\circ \rhobar$, where $\varphi: G\rightarrow G/G'$ is 
the quotient map. Let $R^{\square}_{\psibar}$ be the universal 
deformation ring of $\psibar$.

\begin{thm}[Corollaries \ref{flat_det},
\ref{conj_BJ}]\label{thm_intro-2}
The natural  map $R^{\square}_{\psibar}\to R^{\square}_{\rhobar}$, induced by composing 
a deformation of $\rhobar$ with $\varphi: G\rightarrow G/G'$, is flat and if $\pi_1(G')$ is \'etale then it induces a bijection between the sets of irreducible components. 
\end{thm}

In a companion paper \cite{defT} we deal with the case, when $G^0$ is a torus. In particular, we show that the irreducible components 
of $R^{\square}_{\psibar}$ can be labelled by characters $\chi: 
(\mu_{p^{\infty}}(E)\otimes M)^{\Gal(E/F)}\rightarrow \OO^{\times}$. 

We assume that $\pi_1(G')$ is \'etale until the end of this subsection. 
Let $R^{\square, \chi}_{\rhobar}$ be an irreducible component of $R^{\square}_{\rhobar}$ corresponding to a character $\chi$ under \Cref{thm_intro-2}. 

\begin{thm}[Corollary \ref{reg_codim_x}]\label{thm_intro-3} If $\pi_1(G')$ is \'etale then the rings $R^{\square, \chi}_{\rhobar}$,
$R^{\square, \chi}_{\rhobar}/\varpi$ are complete intersection, regular in codimension 
$[F:\Qp]$, normal domains. 
\end{thm}

Theorem \ref{thm_intro-3} and a theorem of Grothendieck on factoriality of 
complete intersections imply:

\begin{cor}[Corollary \ref{parafactorial}]\label{cor_intro_para} If 
$\pi_1(G')$ is \'etale and $[F:\Qp]\ge 3$ then $R^{\square, \chi}_{\rhobar}$ and $R^{\square, \chi}_{\rhobar}/\varpi$ are factorial. 
\end{cor}

If a further hypothesis on $\rhobar$ and $G$ is satisfied (for example, 
if $G/Z(G^0)$ does not have $\PGL_2$ as a factor) then we show in 
Corollary \ref{reg_codim_x} that $R^{\square, \chi}_{\rhobar}$ and 
$R^{\square, \chi}_{\rhobar}/\varpi$ are regular in codimension $2[F:\Qp]-1$
and hence \Cref{cor_intro_para} holds if $[F:\Qp]=2$. If $F=\Qp$ then \cite[Remark 4.24, Corollary 4.25]{BIP_new} provide 
examples, where \Cref{cor_intro_para} fails.

\subsection{Deformation problems with fixed partial determinant}\label{intro_partial}
In fact, Theorems \ref{thm_intro-1}, \ref{thm_intro-2}, \ref{thm_intro-3} and Corollary \ref{cor_intro_para} hold 
in a more general setting, which might be called deformation problems with \emph{`fixed partial determinant'}. Assume that we have 
a $\Gal(E/F)$-invariant decomposition $X^*(G^0/G')=M_1\oplus M_2$ and let $H_1$ be the quotient of $G/G'$ such that their  component
groups coincide and the character lattice of $H_1^0$ is equal to $M_1$. We fix a continuous representation $\psi_1: \Gamma_F \rightarrow H_1(\OO)$, 
such that $\psi_1 \equiv \varphi_1\circ \rhobar \pmod{\varpi}$, where $\varphi_1: G\rightarrow H_1$ is the composition
of $\varphi$ with the quotient map $G/G'\twoheadrightarrow H_1$. Let $D^{\square, \psi_1}_{\rhobar}$ be a subfunctor 
of $D^{\square}_{\rhobar}$, such that 
$$ D^{\square, \psi_1}_{\rhobar}(A)=\{ \rho_A\in D^{\square}_{\rhobar}(A): \varphi_1\circ \rho_A= \psi_1\otimes_{\OO} A\}.$$
The functor $D^{\square, \psi_1}_{\rhobar}$ is pro-represented by a quotient $R^{\square, \psi_1}_{\rhobar}$ of $R^{\square}_{\rhobar}$.
Then the above results hold with $R^{\square}_{\rhobar}$ replaced with $R^{\square, \psi_1}_{\rhobar}$.

If $\pi_1(G')$ is \'etale then   
the irreducible components of $R^{\square, \psi_1}_{\rhobar}$ can be indexed by characters $\chi: (\mu_{p^{\infty}}(E)\otimes M_2)^{\Gal(E/F)} \rightarrow \OO^{\times}$. There are always two interesting cases: if $M_1=0$ then we impose an empty condition and $R^{\square, \psi_1}_{\rhobar}=R^{\square}_{\rhobar}$; 
if $M_1=X^*(G^0/G')$ then $R^{\square, \psi_1}_{\rhobar}$ parameterises deformations with `fixed determinant' equal to $\psi_1$ and it follows from above 
that $R^{\square, \psi_1}_{\rhobar}$ and $R^{\square, \psi_1}_{\rhobar}/\varpi$ are integral domains. If $G=\GL_d$ then $X^*(G^0/G')=\ZZ$ and these are
the only cases that can occur. However, in general there are further interesting cases naturally appearing in the Langlands program, as we discuss 
in Section \ref{LandC}.

If $H$ is a connected reductive group over $F$ (or more generally over a number field) which splits over a finite Galois extension $E$ of $F$  then it is expected  in \cite{BG} that the Galois representations
attached to $C$-algebraic automorphic forms on $H$ take values in the $C$-group ${}^C H$, which is  a generalised reductive group scheme over $\OO$
with component group the constant group scheme $\Gal(E/F)$. 
Moreover, they should satisfy the condition $d\circ \rho= \chi_{\cyc}$, where $\chi_{\cyc}$ is the $p$-adic cyclotomic character and 
$d$ is the canonical map $d:{}^C H\rightarrow \Gm$. We explain in Theorem \ref{C} that the functor $D^{\square,\psi_1}_{\rhobar}$  with $G={}^C H$, $H_1= \Gal(E/F)\times \Gm$ and 
$\psi_1: \Gamma_F \rightarrow H_1(\OO)$,  $\gamma \mapsto (\gamma|_E, \chi_{\cyc}(\gamma))$ parameterises deformations $\rho_A$ of 
$\rhobar: \Gamma_F \rightarrow {}^CH(k)$ satisfying $d\circ \rho_A = \chi_{\cyc}\otimes_{\OO} A$. We show that in this case, under the assumption 
that $\pi_1(\widehat{H}')$ is \'etale, 
the irreducible components are canonically in bijection with characters of $p$-power torsion subgroup of $Z(H)^0(F)$,
where $Z(H)^0$ is the neutral component of the centre of $H$. 

\subsection{Deformation space of Lafforgue's pseudocharacters}\label{intro_Laf}
We refer the reader to Section \ref{Laf} for the definition of 
V.\,Lafforgue's $G$-pseudocharacters. We single out some properties for the purpose of this introduction. To every representation 
$\rho: \Gamma_F \rightarrow G(A)$ one may attach an $A$-valued $G$-pseudocharacter $\Theta_{\rho}$. Moreover, conjugation of  $\rho$  with elements of $G^0(A)$ 
does not change the associated $G$-pseudocharacter. If $\Theta$ is a $G$-pseudocharacter with values in an algebraically closed field $\kappa$ 
then there exists a representation $\rho: \Gamma_F \rightarrow G(\kappa)$, uniquely determined up to $G^0(\kappa)$-conjugation, such that $\rho$ is  $G$-semisimple 
(Definition \ref{defi_Gcr}) and 
$\Theta_{\rho}=\Theta$. If $A$ and $\kappa$ carry a topology then there is a notion of continuous $G$-pseudocharacters, which interacts well with
continuous representations. 

Let $\Thetabar$ be the $G$-pseudocharacter attached to $\rhobar$. JQ has shown in \cite{quast} as part of his PhD thesis under the direction of Gebhard B\"ockle 
that the functor $D^{\ps}:\Aa_{\OO}\rightarrow \Set$, such that $D^{\ps}(A)$ is the set of continuous $A$-valued $G$-pseudocharacters deforming $\Thetabar$ 
is pro-represented by a complete local noetherian $\OO$-algebra $R^{\ps}_G$ with residue field $k$. The proof that $R^{\ps}_G$ is noetherian uses 
the fact that $\Gamma_F$ is topologically finitely generated. 

If $G=\GL_d$ then Emerson--Morel have shown in \cite{emerson2023comparison} that there is a natural bijection between  $G$-pseudocharacters and
Chenevier's $d$-dimensional determinant laws defined in \cite{che_durham}. Thus in this case the ring $R^{\ps}_{\GL_d}$ coincides with the deformation ring studied 
by B\"ockle--Juschka in \cite{BJ_new}. The results of that paper are an important input in \cite{BIP_new}. In our paper we extend 
the main result \cite[Theorem 5.5.1]{BJ_new} to $G$-pseudocharacters, also obtaining a new proof of it, when $G=\GL_d$.

\begin{thm}[\Cref{BJ1}] The rings $R^{\ps}_G$ and $R^{\ps}_G/\varpi$ are equidimensional of dimension $d+1$ and $d$, respectively, where $d= \dim G_k [F:\Qp] + \dim Z(G)_k$. 
\end{thm} 

Let $U^{\ps}:=(\Spec R^{\ps}_G)\setminus \{\text{the closed point}\}$. The non-special absolutely irreducible locus $U^{\nspcl}$ is 
an open subscheme of $U^{\ps}$ such that a geometric point $x:\Spec \kappa\rightarrow U^{\ps}$ lies in $U^{\nspcl}$ if and only if the representation 
$\rho_x: \Gamma_F \rightarrow G(\kappa)$ associated to the specialisation of the universal $G$-pseudocharacter at $x$ is 
$G$-irreducible (\Cref{defi_G_irr}) and $H^0(\Gamma_F, (\Lie G'_{\sic})_{x}^*(1))=0$, where $\Gamma_F$ acts on $(\Lie G'_{\sic})_x$ via $\rho_x$ composed with the adjoint action. 
If $\pi_1(G')$ is \'etale or $\chara(\kappa)=0$ then 
$(\Lie G'_{\sic})_x=(\Lie G')_x$ and the 
last condition is equivalent to $H^0(\Gamma_F, (\ad^0 \rho_x)^*(1))=0$.

\begin{thm}[\Cref{BJ2}] $U^{\nspcl}$ is Zariski dense in $\Spec R^{\ps}_G$ and its special fibre is Zariski dense in $\Spec R^{\ps}_G/\varpi$.
\end{thm}

We also show in \Cref{ps_normal} that the reduced subscheme $(\Spec R^{\ps}_G[1/p])^{\red}$  is normal. If $\pi_1(G')$ is \'etale then we
 compute the set of irreducible
components of $\Spec R^{\ps}_G$ in \Cref{irr_comp_ps}, which together with \Cref{thm_intro-2} implies that 
$\Spec R^{\square}_{\rhobar}\rightarrow \Spec R^{\ps}_G$ induces a bijection between the sets of irreducible components.

\subsection{Complete intersection}\label{sec_ci} We will now sketch the proof of \Cref{thm_intro-1}.
Let $\Rep^{\Gamma_F}_G: \OO\text{-}\alg \rightarrow \Set$ be the functor, which maps 
an $\OO$-algebra $A$ to the set of group homomorphisms $\Gamma_F\rightarrow G(A)$.
The scheme $G^0$ acts on $\Rep^{\Gamma_F}_G$ by conjugation. 
We construct an affine scheme $X^{\gen}_{G, \rhobarss}=\Spec A^{\gen}_{G, \rhobarss}$ over $\OO$, 
which is a $G^0$-invariant subfunctor of $\Rep^{\Gamma_F}_G$ and the following desiderata hold:
\begin{itemize} 
\item[(D1)] the specialisation of the universal representation $\rho^{\univ}: \Gamma_F \rightarrow G(A^{\gen}_{G, \rhobarss})$
at $x\in X^{\gen}_{G, \rhobarss}(\kbar)$ induces a bijection $x\mapsto \rho_x$ between 
$X^{\gen}_{G, \rhobarss}(\kbar)$ and the set of continuous representations
$\rho: \Gamma_F \rightarrow G(\kbar)$ such that the $G$-semisimplification (\Cref{defi_G_ss}) of $\rho$ is equal 
to $\rhobarss$;
\item[(D2)] the completion of a local ring at $x\in X^{\gen}_{G, \rhobarss}(\kbar)$ is isomorphic 
to $R^{\square}_{\rho_x}$;
\item[(D3)] the GIT quotient $X^{\gen}_{G, \rhobarss} \sslash G^0$ is the spectrum of  a complete local noetherian 
$\OO$-algebra $R^{\git}_{G,\rhobarss}$ with residue field $k$.
\end{itemize}
We then bound the dimension of the special fibre $\Xbar^{\gen}_{G, \rhobarss}$ from above 
\begin{equation}\label{intro_bound}
\dim \Xbar^{\gen}_{G, \rhobarss}\le \dim G_k([F:\Qp]+1).
\end{equation}
Once we can do this the first part of Theorem \ref{thm_intro-1} follows immediately
as (D2) implies that $\dim R^{\square}_{\rhobar}/\varpi\le  \dim G_k([F:\Qp]+1)$ and the result 
follows from \eqref{EP_bound} and a little commutative algebra (\Cref{pre_lci_flat}). 
So the overall strategy is similar to the strategy in \cite{BIP_new}, however its execution 
requires substantially new ideas.

\subsubsection{Definition}\label{subsub_defi} The difficulty in defining $X^{\gen}_{G, \rhobarss}$ stems from the fact that 
we would like to work with continuous representations of $\Gamma_F$, but for 
arbitrary $A$ there seems to be no sensible topology to put on $G(A)$ which allows us to express the desired continuity condition by simply asking the map $\Gamma_F \to G(A)$ to be continuous. The solution to this problem 
is inspired by the analogous definition by Fargues--Scholze \cite{fargues2021geometrization} in the $\ell\neq p$ case, which uses condensed mathematics.

One may define $X^{\gen}_{G, \rhobarss}$ as a functor from $R^{\ps}_G\text{-}\alg\rightarrow \Set$ 
such that $X^{\gen}_{G, \rhobarss}(A)$ is the set of representations $\rho: \Gamma_F \rightarrow G(A)$, 
such that $\Theta_{\rho}$ is the specialisation of the universal $G$-pseudocharacter $\Theta^u$ along 
$R^{\ps}_G\rightarrow A$ and $\rho$ extends to a homomorphism of condensed groups
$$ \tilde{\rho}: \underline{\Gamma_F} \rightarrow G(A_{\disc}\otimes_{(R^{\ps}_G)_{\disc}} \underline{R^{\ps}_G}).$$
We refer the reader to Section \ref{sec_cond} for the notation. In Lemma \ref{ind_cont_equiv_R_cond} we show that 
the condensed condition can be reformulated as follows: for some (equivalently any) closed immersion 
$\tau: G\hookrightarrow \mathbb A^n$ the image $\tau(\rho(\Gamma_F))$ is contained in a finitely generated 
$R^{\ps}_G$-submodule $M \subseteq A^n = \mathbb A^n(A)$ such that the map $\tau\circ \rho: \Gamma_F \rightarrow M$ 
is continuous for the canonical topology on $M$ as an $R^{\ps}_G$-module. We call such representations $R^{\ps}_G$-\emph{condensed}. This definition 
of $X^{\gen}_{G, \rhobarss}$ is aesthetically pleasing and helpful 
to prove that $X^{\gen}_{G, \rhobarss}$ is functorial in $G$. However, it is
not immediately clear (at least to the authors) from the definition that the functor is representable by a finite type $R^{\ps}_G$-algebra. 

If $G=\GL_d$ then one may linearise representations and work instead with Cayley--Hamilton homomorphisms of $R^{\ps}_{\GL_d}$-algebras $E^{u}\rightarrow M_d(A)$, where
$E^u$ is a universal Cayley--Hamilton quotient of the completed 
group algebra $R^{\ps}_{\GL_d}\br{\Gamma_F}$. We refer the reader to Section \ref{def_and_1prop} for more details. Wang-Erickson has shown in \cite{WE_alg} that the algebra $E^u$ is a finitely generated $R^{\ps}_{\GL_d}$-module.
This result is the reason why $R^{\ps}_{\GL_d}$-condensed representations do not explicitly appear in \cite{BIP_new}. In Lemma \ref{factors_through_Eu}
we show that $X^{\gen}_{\GL_d, \rhobarss}$ coincides with the scheme $X^{\gen}$
defined in \cite{BIP_new}. In particular, it is represented by a finite 
type $R^{\ps}_{\GL_d}$-algebra $A^{\gen}_{\GL_d}$. 

In the general case, we fix a closed immersion $\tau: G \hookrightarrow \GL_d$ 
of $\OO$-group schemes and define $X^{\gen, \tau}_G$ as a closed 
subscheme of $X^{\gen}_{\GL_d}:=X^{\gen}_{\GL_d, (\tau\circ \rhobar)^{\mathrm{ss}}}$, such 
that $\rho\in X^{\gen}_{\GL_d}(A)$ lies in $X^{\gen, \tau}_G(A)$ if and only if 
$\rho(\Gamma_F)$ is contained in $\tau(G(A))$. The functor 
$X^{\gen, \tau}_G$ is representable by a finite type $R^{\ps}_{\GL_d}$-algebra 
$A^{\gen, \tau}_G$. We show in Proposition \ref{XgenGtau_rhobarss_A} 
that the connected component of $X^{\gen,\tau}_G$ containing the point 
corresponding to $\rhobar:\Gamma_F \rightarrow G(k)$ is canonically
isomorphic to the functor $X^{\gen}_{G, \rhobarss}$ defined 
via condensed mathematics as above. In particular, the connected 
component is independent of the chosen embedding $\tau$. Using 
the construction via $X^{\gen, \tau}_G$ we show 
that desiderata (D1), (D2), (D3) hold. The condensed description allows 
to show in Proposition \ref{functoriality_2} that 
$X^{\gen}_{G, \rhobarss}$ is functorial in $G$.

A key input in the above argument is the following 
fundamental finiteness theorem proved by Vinberg \cite{vinberg}, when $S$ is a spectrum of a field of characteristic zero, 
by Martin \cite{martin}, when $S$ is a spectrum of a field of characteristic $p$, and by Cotner \cite{cotner}
in general.

\begin{thm}[Cotner, \cite{cotner}]\label{intro_cotner_main} Let $S$ be a locally noetherian scheme and let $f: H\rightarrow G$ be a finite morphism 
of generalised reductive smooth affine $S$-group schemes. Then the induced map 
on the GIT quotients
$$ H^N\sslash H \rightarrow G^N \sslash G$$
is finite, where the action on both sides is given by conjugation. 
\end{thm}

The homomorphism $R^{\ps}_G \rightarrow A^{\gen}_{G, \rhobarss}$ is $G^0$-equivariant for the trivial action on the source. Hence, its image 
is contained in $G^0$-invariants $R^{\git}_{G, \rhobarss}$, which represents the GIT quotient $X^{\gen}_{G, \rhobarss}\sslash G^0$. 

\begin{prop}[\Cref{nu_fin_u}]\label{zwischen-stop} The morphism %
$X^{\gen}_{G, \rhobarss}\sslash G^0\rightarrow \Spec R^{\ps}_G$
is a finite universal homeomorphism. 
\end{prop}
We expect that the morphism is  adequate in the sense of Alper \cite{alper}, i.e.~that moreover $R^{\ps}_G[1/p] \cong R^{\git}_{G, \rhobarss}[1/p]$, 
and hope to return to this question in future work\footnote{We have proved this assertion in \cite{inf_laf}.}. If this were true then 
one could show that $\Spec R^{\ps}_G[1/p]$ is reduced and hence normal, as explained in Section \ref{intro_Laf}.


We define $X^{\gen}_{G, \rhobarss}$ for any profinite group $\Gamma$ satisfying Mazur's $p$-finiteness condition. However, starting with \Cref{sec_dim_fib} we use in an essential way that $\Gamma=\Gamma_F$:
besides Euler--Poincar\'e characteristic formula, it is crucial to our 
argument that using local Tate duality we may transfer obstructions 
to lifting from $H^2$ to $H^0$. 

\subsubsection{Bounding the dimension}\label{sec_intro_bd_dim} We will now sketch how we 
obtain the bound \eqref{intro_bound}. Let $Y_{\rhobarss}$ be the preimage 
of the closed point of $X^{\ps}_G:= \Spec R^{\ps}_G$ in $X^{\gen}_{G, \rhobarss}$. 
We bound the dimension of  $Y_{\rhobarss}$ in  \Cref{bound_Y}. Its complement 
$V$ and the special fibre $\Vbar$ of $V$ are Jacobson schemes and residue fields of closed points $x\in V$ are local fields.
Moreover, in  \Cref{local_ring_def_ring} we  relate the completions of local rings $\OO_{V, x}$ to the deformation rings 
$R^{\square}_{G, \rho_x}$, and completions of local rings $\OO_{\Vbar, x}$ to $R^{\square}_{G, \rho_x}/\varpi$.
The upshot is that if we control the dimension of  $R^{\square}_{G, \rho_x}/\varpi$ then we 
control the dimension of $\OO_{\Vbar,x}$ and if we can do this at every closed point then we can bound the dimension of 
$\Vbar$ and the dimension of its closure in $\Xbar^{\gen}_{G, \rhobarss}$. In Proposition \ref{present_over_RH} we establish a presentation 
\begin{equation}\label{intro_present}
R^{\square}_{G, \rho_x}\cong \frac{R^{\square}_{G/G', \varphi\circ \rho_x}\br{x_1, \ldots, x_r}}{(f_1,\ldots, f_s)},
\end{equation}
where $s=\dim_{\kappa(x)} H^2(\Gamma_F, \ad^0 \rho_x)$ and $r-s= \dim G'_k ([F:\Qp]+1)$. We deduce from the results of 
\cite{defT} that if $x\in \Vbar$ is a closed point then 
$R^{\square}_{G/G', \varphi\circ \rho_x}/\varpi$ 
is formally smooth over the group algebra $\kappa(x)[\mu]$ of dimension 
$\dim (G/G')_k ([F:\Qp]+1)$, where $\mu=(\mu_{p^{\infty}}(E)\otimes M)^{\Gal(E/F)}$.  
This is easy if $G=\GL_d$ or more generally if $G$ is connected, but requires a non-trivial argument in the non-connected 
case, which we carry out in \cite{defT}. This allows us to conclude that the closure in $\Xbar^{\gen}_{G, \rhobarss}$ of the locus in $\Vbar$, where 
$H^2(\Gamma_F, \ad^0\rho_x)$ vanishes, has dimension $\dim G_k([F:\Qp]+1)$, 
if non-empty. This locus will be referred to as the  $(\Lie G'_k)^*$\emph{-non-special locus} later on.

One is left to understand the locus, where $H^2(\Gamma_F, \ad^0\rho_x)\neq 0$. 
By local Tate duality this is equivalent to 
$H^0(\Gamma_F, (\ad^0 \rho_x)^*(1))\neq 0$. Let us suppose that $\rho_x$ 
is absolutely $G$-irreducible.  If $G=\GL_d$ then B\"ockle--Juschka show in \cite{BJ_new} 
that non-vanishing of $H^0(\Gamma_F, (\ad^0 \rho_x)^*(1))$ implies that 
$\rho_x$ is induced from $\Gamma_{F'}$, where $F'$ is a proper
abelian extension of $F$. We know the dimension of 
$\Xbar^{\ps, \Gamma_{F'}}_{\GL_{d'}}$, where $d'=\frac{d}{[F':F]}$, inductively and 
using this it is shown 
in \cite{BIP_new} that the closure of this locus has positive codimension.
Such arguments might work for classical groups (JQ has studied the case
$G=\Sp_{2n}$ under the assumption $p>2$ in \cite[Section 7]{quast}), but it seems impossible 
to get  a characterisation like this for an arbitrary generalised 
reductive group. 

Our key observation is that instead of having such a precise description of $\rho_x$, it is enough to observe that $\rho_x$ has \emph{`small' image}. Namely, 
let $H$ be the Zariski closure of $\rho_x(\Gamma_F)$ in 
$G(\kappa)$, where $\kappa$ is the algebraic closure of $\kappa(x)$. Since $\rho_x$ is absolutely $G$-irreducible, its image is not contained in any parabolic subgroup of $G_{\kappa}$. A theorem of Martin 
\cite{martin} asserts that $H$ is again generalised reductive. 
In Section \ref{bounds_red_sbgp} we prove two results:  
\begin{prop}[\Cref{cor_one}]\label{intro_matsu} If a generalised reductive subgroup $H$ of 
$G_{\kappa}$ does not contain $G'_{\kappa}$ then $\dim G_{\kappa} -\dim H\ge 2$.
\end{prop} 

Our original  proof of the above result used Matsushima's theorem.
The current simpler proof was suggested by both  Sean Cotner and the anonymous referee. 

\begin{lem}[\Cref{sic}]\label{intro_coadjoint} The $G'_{\kappa}$-invariants in $(\Lie G'_{\sic})^*_{\kappa}$ are zero. 
\end{lem}

We note that the above Lemma would be false in small characteristics if we did 
not take the dual. We first prove it for $\SL_2$ and then use this together with 
arguments with roots to prove it in 
the general case. 

Let us go back to the situation at hand, where $H$ is the Zariski closure 
of $\rho_x(\Gamma_F)$ in $G(\kappa)$. If $x$ lies in the special fibre, then the
twisting operation by the cyclotomic character becomes trivial after the restriction to $\Gamma_{F(\zeta_p)}$. Using this we show that non-vanishing 
of $H^0(\Gamma_F, (\ad \rho_x)^*(1))$ implies that the neutral component
$H^0$ of $H$ has non-zero invariants in $(\Lie G')^*_{\kappa}$. 
If $\pi_1(G')$ is \'etale or $\chara(\kappa)=0$ then 
$(\Lie G')^*_{\kappa} =(\Lie G'_{\sic})^*_{\kappa}$ and 
Proposition 
\ref{intro_matsu} and Lemma \ref{intro_coadjoint} imply that $\dim G_{\kappa}-\dim H^0\ge 2$ and hence $\dim G_{\kappa}- \dim H \ge 2$.
The fact that the difference can be bounded below by $2$ (and not by $1$) plays
an important role in verifying Serre's criterion for normality in the proof of 
\Cref{thm_intro-3}. As you can see there is an induction argument shaping up. To explain it we need two further ingredients. 

Let $W$ be an  algebraic representation of $G$ on a finite free $\OO$-module, and let $\rho: \Gamma_F \rightarrow G(A)$ be a
representation, where $A$ is a noetherian $\OO$-algebra. We show in Lemma \ref{V_iks} that there is a closed  subscheme $X_{W,j}$ of 
$X=\Spec A$ such that 
$$x\in X_{W,j} \iff \dim_{\kappa(x)} H^0(\Gamma_F, W\otimes_{\OO} \kappa(x)(1))\ge j+1.$$ 
We refer to $X_{W,j}$ as the \emph{$W$-special locus of level $j$} and just the \emph{$W$-special locus} when $j=0$. 
The complement of the $W$-special locus is called \emph{$W$-non-special}.
The proof amounts to reformulating these conditions in terms of some finitely generated $A$-module and a little of commutative algebra. We apply this 
notion to $W=(\Lie G'_{\sic})^*$ and if the image of $\rho$ is contained in $P(A)$, where $P$ is a parabolic subgroup scheme 
of $G$ with unipotent radical $U$, then we apply it to $W=(\Lie U)^*$. When working with the special fibre we consider representations $W$ over $k$
instead of $\OO$.

The second ingredient, which uses \Cref{intro_cotner_main} as an input, is the following finiteness result proved in \Cref{finite_functoriality}.
If $H\rightarrow G_k$ is a finite morphism of generalised reductive $k$-group schemes which maps an $H$-pseudocharacter $\overline{\Psi}$ to $\Thetabar$ then 
the natural map $\Xbar^{\ps}_{H, \overline{\Psi}} \rightarrow \Xbar^{\ps}_G$
is finite. Moreover, there are only finitely many $H$-pseudocharacters $\overline{\Psi}$ mapping to $\Thetabar$.
In particular, the dimension of the union $\Xbar^{\ps}_{HG}$ of the scheme theoretic images can be bounded by the largest $\dim \Xbar^{\ps}_{H,\overline{\Psi}}$.

We are now in a position to sketch the proof of the bound in \eqref{intro_bound} when $\rhobar$ is absolutely $G$-irreducible. 
From \eqref{intro_present} and the relation between $R^{\square}_{G, \rho_x}/\varpi$ and the completions of $\OO_{\Vbar,x}$ we 
obtain a lower bound 
\begin{equation}\label{lower_bound}
\dim \Xbar^{\gen}_{G, \rhobarss}\ge \dim G_k ([F:\Qp]+1).
\end{equation}
We then construct a pair $(G_1, \rhobar_1)$ such that $\dim G_1\le \dim G$, $\pi_1(G'_1)$ is \'etale and if \eqref{intro_bound} holds for 
$(G_1, \rhobar_1)$ then it also holds for $(G, \rhobar)$. For example, 
if $G=\PGL_d$ then $G_1= \SL_d/\mu_m$, where $d=p^e m$ and $p \nmid m$, and $\rhobar_1=\rhobar$ as $G(k)=G_1(k)$. In general, 
the presence of the component group of $G$ complicates the construction
of $G_1$, see \Cref{sec_extensions}. This reduction step allows us 
to assume that $\pi_1(G')$ is \'etale, which implies that 
$\Lie G'_{\sic}=\Lie G'$. We then show that 
there are finitely many generalised reductive subgroup schemes $H_i$ of $G_k$, defined over some finite extension of $k$, such that 
$\dim G_k - \dim H_i \ge 2$ and the $(\Lie G'_{\sic})^*_k$-special locus in $\Xbar^{\gen}_{G, \rhobarss}$ is contained in the preimage of the union of $\Xbar^{\ps}_{H_i G}$. The assumption that $\rhobar$ is absolutely $G$-irreducible implies that every fibre 
$X^{\gen}_{G,\rhobarss}\rightarrow X^{\ps}_G$ has dimension $\dim G_k - \dim Z(G)_k$. 
Inductively, we may bound the dimension of $\Xbar^{\ps}_{H_iG}$ by $\dim H_i [F:\Qp] +\dim Z(H_i)$. 
Using this we show that the codimension of the  $(\Lie G'_{\sic})^*_k$-special locus in $\Xbar^{\gen}_{G, \rhobarss}$
is at least $2[F:\Qp]$. It follows from \eqref{lower_bound} that the $(\Lie G')^*_k$-non-special locus in $\Xbar^{\gen}_{G, \rhobarss}$ is non-empty 
and as explained above its closure has dimension $\dim G_k ([F:\Qp]+1)$. This lets us conclude that \eqref{lower_bound} is an equality.

Let us briefly indicate how we find the subgroup schemes $H_i$ in the above argument. The $(\Lie G'_{\sic})^*_k$-special locus 
of level $j$ is non-empty only for finitely many $j$. The group schemes $H_i$ correspond to the generic points of these 
loci for $j\ge 0$. Given such a generic point $\eta$ we let $H$ be the Zariski closure of $\rho_{\eta}(\Gamma_F)$ 
in $G(\overline{\kappa(\eta)})$. Since $\rhobarss$ is assumed to be absolutely $G$-irreducible $H$ is not contained in 
any parabolic subgroup of $G_{\overline{\kappa(\eta)}}$ and hence is generalised reductive. 
A conjugate of $H$ by an element of $G(\overline{\kappa(\eta)})$ can be defined 
over $\kbar$ and hence over a finite extension of $k$. Our arguments with generalised reductive groups over algebraically closed fields
use results of Martin \cite{martin} and Bate--Martin--R\"{o}hrle \cite{BMR} in an essential way.

The general case is an elaboration of this inductive argument. If $\rhobarss$ is not absolutely $G$-irreducible 
then the dimensions of the fibres $X^{\gen}_{G, \rhobarss}\rightarrow X^{\ps}_G$ can vary. We bound their dimensions in 
Section \ref{sec_dim_fib}. 
We stratify $\Xbar^{\ps}_G$ by $\Xbar^{\ps}_{LG}$, where $L$ runs over the Levi subgroup schemes of $G$ containing a fixed 
 maximal split torus. If $y\in \Xbar^{\ps}_{LG}$, but does not lie in any $\Xbar^{\ps}_{L'G}$ for a proper Levi $L'\subsetneq L$ then 
for typical $y$ the fibre will have dimension $\dim G_k - \dim Z(L)_k +\dim U_k [F:\Qp]$, where $U$ is the unipotent radical 
of any parabolic subgroup with Levi $L$. By induction we know that $\dim \Xbar^{\ps}_{LG}$ is at most $\dim L_k [F:\Qp]+ \dim Z(L)_k$.
The sum of these numbers does not exceed $\dim G_k ([F:\Qp]+1)$. In fact, the difference is equal to $\dim U_k [F:\Qp]$. We 
deal with the fibres, which have dimension exceeding $\dim G_k - \dim Z(L)_k +\dim U_k [F:\Qp]$ by considering the $(\Lie U_k)^{\ast}$-special 
locus inside $X^{\ps}_L$. The inductive argument  is carried out in Section \ref{sec_complete_intersection}.

If $X^{\ps}_{LG}$ is non-empty for a Levi subgroup $L$ satisfying $\dim G -\dim L=2$ then the bound $\dim U_k [F:\Qp]$ is not 
good enough to apply the Serre's criterion for normality, when $F=\Qp$, as the assumption on $L$ forces $\dim U_k = 1$. We
analyse the $(\Lie G'_{\sic})^{\ast}$-non-special locus in the preimage 
of $X^{\ps}_{LG}$ in $X^{\gen}_{G, \rhobarss}$ in this case in detail in Section \ref{RLevi2}. In Section \ref{sec:bound_special} we show that the 
$(\Lie G'_{\sic})^{\ast}$-special locus in $X^{\gen}_{G, \rhobarss}$ and also in $\Xbar^{\gen}_{G, \rhobarss}$ has codimension of at least $1+[F:\Qp]$
and of at least $2[F:\Qp]$ if $X^{\ps}_{LG}$ is empty for all Levi subgroups $L$ of codimension $2$.

\subsection{Irreducible components} We will sketch the proof of \Cref{thm_intro-3}. Using the results of \cite{defT} we show that 
\begin{equation}\label{defT_prime}
A^{\gen}_{G/G', \psibarss}\cong \OO[\mu]\br{x_1, \ldots, x_m}[t_1^{\pm 1}, \ldots, t_l^{\pm 1}],
\end{equation}
where $l= \dim (G/G')_k - \dim Z(G/G')_k$, $m+l= \dim (G/G')_k ([F:\Qp]+1)$ and $\mu=(\mu_{p^{\infty}}(E)\otimes M)^{\Gal(E/F)}$. 
Let $\mathrm{X}(\mu)$ be the group of characters $\chi: \mu \rightarrow \OO^{\times}$. If $A$ is an $A^{\gen}_{G/G', \psibarss}$-algebra
then we use \eqref{defT_prime} to define $A^{\chi}:= A\otimes_{\OO[\mu], \chi} \OO$ and if $X=\Spec A$ we write 
$X^{\chi}:=\Spec A^{\chi}$. It follows from \eqref{defT_prime} that the irreducible components of 
$X^{\gen}_{G/G', \psibarss}$ are given by $X^{\gen, \chi}_{G/G', \psibarss}$ for $\chi\in \mathrm X(\mu)$ and are
 regular. In particular, the completions of local rings of $X^{\gen, \chi}_{G/G', \psibarss}$ are also regular. Thus the presentation \eqref{intro_present}
 gives us a presentation
 \begin{equation}\label{intro_present_chi} 
R^{\square,\chi}_{G, \rho_x}\cong \frac{R^{\square,\chi}_{G/G', \varphi\circ \rho_x}\br{x_1, \ldots, x_r}}{(f_1,\ldots, f_s)},
\end{equation}
where $R^{\square,\chi}_{G/G', \varphi\circ \rho_x}$ is a regular ring. Further, if $x$ is in the $(\Lie G')^{\ast}$-non-special locus 
then $s$ in the above presentation is $0$ and we can conclude that $R^{\square,\chi}_{G, \rho_x}$ is a regular ring. We deduce that the $(\Lie G')^{\ast}$-non-special locus in $X^{\gen,\chi}_{G, \rhobarss}$ is contained in the 
regular locus. In general, 
the $(\Lie G')^{\ast}$-non-special locus might be empty, for example
this happens if $G=\PGL_p$. On the other hand the complement of $(\Lie G'_{\sic})^{\ast}$-non-special locus has codimension of at least $1+[F:\Qp]$ in $X^{\gen, \chi}_{G, \rhobarss}$. If $\pi_1(G')$ is \'etale then 
these loci coincide and this implies that  $X^{\gen,\chi}_{G, \rhobarss}$
is regular in codimension $[F:\Qp]$. Since $X^{\gen,\chi}_{G, \rhobarss}$ is excellent the same applies to 
the completions of its local rings, and we may deduce that the deformation rings $R^{\square,\chi}_{G, \rho_x}$ are regular 
in codimension $[F:\Qp]$. To show that they are complete intersection it is enough to bound their dimension from above and then 
use \eqref{intro_present_chi}. Since $k[\mu]$ is a local artinian algebra the underlying reduced subschemes of the special fibres of $X^{\gen}_{G, \rhobarss}$ 
and $X^{\gen, \chi}_{G, \rhobarss}$ coincide and hence they have the same dimension. Serre's criterion 
for normality implies that $R^{\square,\chi}_{G, \rho_x}$ is normal. Since the ring is local, it is a domain. 

Since $(\Lie G')_L= (\Lie G'_{\sic})_L$ the $(\Lie G')^{\ast}$-non-special and $(\Lie G'_{\sic})^{\ast}$-non-special loci in the generic
fibre coincide. This allows us to prove that $X^{\gen}_{G, \rhobarss}[1/p]$ is normal without any assumption on $\pi_1(G')$. 

A similar argument works in the `fixed partial determinant' setting described in Section \ref{intro_partial}. However, there 
is one new idea not present in \cite{BIP_new}. To bound the dimension of $X^{\gen, \psi_1}_{G, \rhobarss}$ we show that there
is a finite morphism $X^{\gen, \psi_1}_{G, \rhobarss}\rightarrow X^{\gen}_{G/Z_1, \rhobarss}$, where 
$Z_1$ is a central subtorus of $G^0$ such that the composition $Z_1\rightarrow G \rightarrow H_1$ is an isogeny. We 
can then bound the dimension of various subschemes of $X^{\gen, \psi_1}_{G, \rhobarss}$ by dimensions 
of the corresponding subschemes in $X^{\gen}_{G/Z_1, \rhobarss}$. So for example, to prove the results 
in the fixed determinant case, when $G=\GL_d$, we use the results proved for $G=\PGL_d$. In \cite{BIP_new} 
one twists by characters to unfix the determinant instead. This argument can be made to work, when $G$ is connected, but 
will not work in general. 

\subsection{Overview by section} In \Cref{sec_grp_sch} we review some facts about the generalised reductive $\OO$-group
schemes and their parabolic subgroup schemes, which we call R-parabolic. We also review $G$-semisimplification 
and $G$-irreducibility. In \Cref{sec_def_rings}
we recall deformation theory of representations
of a profinite group satisfying Mazur's $p$-finiteness condition. Proposition \ref{present_over_RH} is a key result of that section. In \Cref{sec_cont} we discuss and compare notions of continuity via algebra and via condensed sets. As explained in 
Section \ref{subsub_defi} this is used to define $X^{\gen}_{G,\rhobarss}$ later on. 
In \Cref{sec_gen_matrix} we construct the scheme $X^{\gen,\tau}_G$ and study the relation 
between the completion of its local rings and deformation rings of Galois representations. 
In \Cref{sec_GIT} we study the GIT quotient $X^{\git, \tau}_G:=X^{\gen,\tau}_G\sslash G^0$ 
and define $X^{\gen,\tau}_{G,\rhobarss}$ and its GIT quotient $X^{\git}_{G,\rhobarss}$.
In \Cref{Laf} we introduce Lafforgue's $G$-pseudocharacters and their deformation spaces $X^{\ps}_G$ and 
show that the natural map $X^{\git}_{G,\rhobarss}\rightarrow X^{\ps}_G$ is a finite universal homeomorphism. 
In \Cref{sec_funct} we give a condensed description of $X^{\gen,\tau}_{G,\rhobarss}$ as explained in Section \ref{subsub_defi}. 
This allows us to conclude that $X^{\gen,\tau}_{G,\rhobarss}$ is independent of $\tau$, which we then omit from notation, 
and is functorial in $G$. 
In \Cref{sec_abs_irr} we show that without loss of generality one may assume that the map $\Gamma_F \overset{\rhobar}{\longrightarrow} G(k) \rightarrow (G/G^0)(\kbar)$
is surjective. In this case we define the absolutely $G$-irreducible locus in $X^{\gen}_{G, \rhobarss}$ and show that 
if it is non-empty then the difference between its dimension and the dimension of its GIT quotient is $\dim G_k  - \dim Z(G)_k$. 
In \Cref{bounds_red_sbgp} we bound the codimension of generalised reductive subgroups of generalised reductive groups 
over algebraically closed fields. The results of this section are used to bound the dimensions of 
$(\Lie G'_{\sic})^{\ast}$-special and $(\Lie U)^{\ast}$-special loci, 
 as explained in Section \ref{sec_intro_bd_dim}. 
In \Cref{RLevi2} we show that if $G$ is a generalised reductive group over an algebraically closed field with an R-Levi $L$ of codimension $2$ 
such that $L/L^0\rightarrow G/G^0$ is an isomorphism then 
$G/Z(G^0)\cong G_1\times \PGL_2$ and $L/Z(G)\cong G_1\times \Gm$. Proposition \ref{vanish_H2_please} 
is used to bound the dimension of the $(\Lie G'_{\sic})^{\ast}$-special locus in Section \ref{sec:bound_special}.
This section can be omitted if one is only interested in the complete intersection property and does not care for normality and irreducible components. In \Cref{sec_dim_fib} we bound the dimensions of fibres of $X^{\gen}_{G, \rhobarss}\rightarrow X^{\ps}_G$.
In \Cref{sec_complete_intersection}, which is the core of the paper,  we carry out the induction argument described in Section \ref{sec_intro_bd_dim}.  
In  \Cref{sec:bound_special} we bound the dimension of the $(\Lie G'_{\sic})^*$-special locus. 
In \Cref{nightmare_cont} we study the irreducible components of $X^{\gen,\psi_1}_{G, \rhobarss}$.
In \Cref{LandC} we explain that our results apply, when $G$ is an $L$-group or a $C$-group, for the convenience of the reader. 
The appendix  recalls some background on condensed mathematics, which is used in \Cref{sec_cont}.

\subsection{Notation} Let $F$ be a finite extension of $\Qp$. 
We fix an algebraic closure $\overline{F}$ over $F$. Let $\Gamma_F:=\Gal(\overline{F}/F)$
be the absolute Galois group of $F$. 
We will 
denote by $\zeta_p$ a primitive 
$p$-th root of unity in $\overline{F}$. 

Let $L$ be another finite extension of $\mathbb Q_p$, 
such that $\Hom_{\Qp\hyphen\mathrm{alg}}(F, L)$ has cardinality $[F:\Qp]$. 
Let $\OO$ be the ring of integers in $L$, $\varpi$ a uniformiser, and $k$ the residue field. All our schemes are defined over $\OO$ (unless stated otherwise). 
We omit $\Spec \OO$ from the notation, when we take fibre products over it, so 
that $X\times Y:= X\times_{\Spec \OO} Y$.

If $A$ is a commutative ring and $\rho: \Gamma \rightarrow \GL_d(A)$ is a representation, we let 
$$\rho^{\lin}: A[\Gamma]\rightarrow M_d(A), ~\sum_{\gamma\in \Gamma} a_{\gamma} \gamma \mapsto \sum_{\gamma\in \Gamma} a_{\gamma} \rho(\gamma)$$ 
be its \emph{linearisation}, where $M_d(A)$ is the set of $d \times d$ matrices with entries in $A$. 

We denote by $A\hyphen\alg$ the category of commutative $A$-algebras. 

\subsection{Acknowledgements} The authors would like to thank 
Jarod Alper and Sean Cotner for the correspondence regarding their work; 
Tasho Kaletha for a very useful discussion regarding \Cref{RLevi2};
Vincent Pilloni for sharing his notes of the talk he gave on \cite{BIP_new}; 
Sophie Morel for making a preliminary version of \cite{emerson2023comparison}
available to them; Toby Gee, James Newton and Sean Cotner for their comments. The authors would 
like to especially thank Brian Conrad for pointing out an error in an
earlier version of the paper and for the subsequent correspondence. We also heartily thank the referee for their careful reading of the paper.

VP would like to thank his colleagues  Ulrich G\"ortz, Daniel Greb and 
Jochen Heinloth for several stimulating discussions regarding various aspects 
of the paper.

This paper builds on the  collaboration \cite{BIP_new} of VP with Gebhard B\"ockle and Ashwin Iyengar, where the case $G=\GL_d$ is treated.
The paper \cite{BHKT} by Gebhard B\"{o}ckle, Michael Harris, Chandrashekhar Khare
and Jack Thorne also played an important role for us, especially at the beginning of the 
project. 

JQ would like to thank Gebhard Böckle for introducing him to Lafforgue’s theory 
of $G$-pseudocharacters and their deformation theory during his PhD studies at the 
University of Heidelberg.

Parts of the paper were written during the research stay of VP at 
the Hausdorff Research Institute for Mathematics in Bonn 
for the Trimester Program \emph{The Arithmetic of the Langlands Program}.
VP would like to thank the organisers 
Frank Calegari, Ana Caraiani, Laurent Fargues and  Peter Scholze
for the invitation and a stimulating research environment. 
The research stay of VP was funded by the Deutsche Forschungsgemeinschaft (DFG, German Research Foundation) under Germany's Excellence Strategy – EXC-2047/1 – 390685813. 

The research of JQ was funded by the Deutsche Forschungsgemeinschaft (DFG, German Research Foundation) – project number 517234220.

\section{\texorpdfstring{$\OO$-group schemes}{O-group schemes}}\label{sec_grp_sch}

The category of $\OO$-group schemes embeds fully faithfully in the category of fppf group sheaves on $(\OO\hyphen\alg)^{\mathrm{op}}$. Whenever an exact sequence is considered it is understood as an exact sequence of fppf group sheaves. If $G$ is a flat $\OO$-group scheme of finite presentation and $N \leq G$ is a flat closed normal subgroup, then the quotient group sheaf $G/N$ is representable by a scheme \cite[Th\'eor\`eme 4.C]{Anantharaman}.
If $G$ and $N$ are generalised reductive then the quotient is affine by \cite[Theorem 9.4.1]{alper}.

\begin{lem}\label{ZGexists}
    Let $G$ be a flat affine group scheme of finite presentation over $\OO$. Then the functor $\underline Z(G) : \OO\hyphen\alg \to \Set$ given by
    $$\underline Z(G)(A) := \{g \in G(A) \mid \forall A \to A' , \forall h \in G(A') , g = hgh^{-1}\}$$
    is representable by a closed subgroup scheme of $G$, which we will denote by $Z(G)$ and call the \emph{centre} of $G$.
\end{lem}

\begin{proof}
    This is a special case of \cite[Expos\'e VIII, Th\'eorème 6.4]{SGA3Tome2} as explained in \cite[Expos\'e VIII, Exemples 6.5 (e)]{SGA3Tome2}.
\end{proof}

\begin{lem}\label{base_change_Z}
    Let $G$ be a flat affine group scheme of finite presentation over $\OO$ and let $A \to B$ be a map of $\OO$-algebras.
    Then $\underline Z(G_A)$ and $\underline Z(G_B)$ as in \Cref{ZGexists} are representable and we have $Z(G_A)_B = Z(G_B)$.
\end{lem}

\begin{proof}
    By \Cref{ZGexists} $\underline Z(G)$ is representable by an $\OO$-scheme $Z(G)$.
    By definition $\underline Z(G_A)$ is the restriction of $\underline Z(G)$ to the category of $A$-algebras and is thus represented by $Z(G)_A$.
    The same argument shows, that $\underline Z(G_B)$ is representable by $Z(G)_B$ and the claim follows.
\end{proof}

\begin{lem}\label{extofflatisflat}
    Let 
    $ 0 \to G_1 \to G_2 \to G_3 \to 0 $
    be a short exact sequence of affine $\OO$-group schemes of finite presentation. If $G_1$ and $G_3$ are flat, then $G_2 \to G_3$ is faithfully flat and in particular $G_2$ is flat.
\end{lem}

\begin{proof}
    This is \cite[Expos\'e VIB, Proposition 9.2 (xi)]{SGA3_new}.
\end{proof}

\begin{lem}\label{dimLisdimk}
    If $G$ is a flat  $\OO$-group scheme of finite presentation
    then $\dim G_k = \dim G_L$.
\end{lem}

\begin{proof}
    See \cite[Expos\'e VIB, Corollaire 4.3]{SGA3_new}.
\end{proof}

\begin{defi}\label{defi_grg}
    A \emph{generalised reductive group scheme} over a scheme $S$ is a smooth affine $S$-group scheme $G$, such that the geometric fibres of $G^0$ are reductive and $G/G^0\rightarrow S$ is finite.
\end{defi}

\begin{remar}\label{rem_fin_etale}
    It follows from \cite[Proposition 3.1.3]{bcnrd}, that for a generalised reductive $\OO$-group scheme $G$ the quotient $G/G^0$ is \'etale. Sean Cotner has shown in \cite{cotner} using the results of Alper that a smooth affine $S$-group scheme $G$ is generalised reductive if and only if it is geometrically 
    reductive in the sense of \cite[Definition 9.1.1]{alper}. In particular, this holds over $S=\Spec \OO$. 
\end{remar}

\begin{lem}\label{gen_red_field} Let $G$ be an affine group scheme of finite type over a perfect field $\kappa$. If $G$ is reduced 
then $G$ is smooth over $\kappa$. Moreover, if we additionally assume that the 
unipotent radical of $G$ is trivial then $G$ is generalised reductive. 
\end{lem}

\begin{proof} The first assertion follows from \cite[\href{https://stacks.math.columbia.edu/tag/047N}{Tag 047N}, \href{https://stacks.math.columbia.edu/tag/047P}{Tag 047P}]{stacks-project}. Since $\kappa$ is perfect, if the unipotent radical of $G$ is trivial then also the unipotent radical of $G_{K}$ is trivial, for any field extension $K$ of $\kappa$ by \cite[Proposition 1.1.9]{CGP}. Thus $G^0$ is reductive and 
$G/G^0$ is finite by \cite[Proposition 3.1.3]{bcnrd}.
\end{proof}

\begin{prop}\label{O_prime}
    Let $G$ be a generalised reductive $\OO$-group scheme.
    Then there exists a finite unramified extension $L'/L$ with ring of integers $\OO'$, such that the following hold:
    \begin{enumerate}
        \item $(G/G^0)_{\OO'} = G_{\OO'}/G_{\OO'}^0$ is a constant group scheme;
        \item $G(\OO')\rightarrow (G/G^0)(\OO')$ is surjective;
        \item $G^0_{\OO'}$ is split.
    \end{enumerate}
\end{prop}

\begin{proof}
    By definition $G/G^0$ is finite \'etale.
    Using \stackcite{0BND} and the fact that the \'etale fundamental group of $\OO$ is the Galois group of the maximal unramified extension of $L$ we find a finite unramified extension $L'$, such that $(G/G^0)_{\OO'}$ is constant. Since the extension $\OO\rightarrow \OO'$ is flat the sequence
    of group schemes $0\rightarrow G^0\rightarrow G \rightarrow G/G^0\rightarrow 0$ remains exact after base
    change to $\OO'$. 

    Let $k'$ be the residue field of $\OO'$. After further enlarging $k'$ we may ensure that 
    $G(k')\rightarrow (G/G^0)(k')$ is surjective. Since $G$ is smooth and  $\OO'$ is complete 
    the map $G(\OO')\rightarrow G(k')$ is surjective. Since $G/G^0$ is a constant group scheme
    $(G/G^0)(k')=(G/G^0)(\OO')$ and hence we obtain part (2). 

Part (3) follows by the same argument as in part (1) from \cite[Lemma 5.1.3]{bcnrd}.
\end{proof}

\begin{prop}\label{dimkisdimL} Let $G$ be a generalised reductive $\OO$-group scheme. Then $Z(G) \cap G^0$ is flat over $\OO$. In particular, $\dim Z(G)_k = \dim Z(G)_L$. 
\end{prop}

\begin{proof}
    By \Cref{ZGexists} we see that $Z(G)$ exists as a closed subgroup scheme of $G$.
    For flatness and the equality of dimensions, by \Cref{O_prime} we may extend $\OO$, such that $G/G^0$ is constant and the map $G \to G/G^0$ has a section $s : G/G^0 \to G$ as $\OO$-schemes.
    We write $\underline\Delta = G/G^0$ for some finite group $\Delta$.
    The $s$ is determined by a map $\Delta \to G(\OO)$ and for every $\OO$-algebra $A$ the map $G(A) \to (G/G^0)(A)$ is surjective.

    By \cite[Theorem 3.3.4]{bcnrd} $Z(G^0)$ is diagonalizable, i.e. $Z(G^0) = \Spec(\OO[M])$ for some finitely generated abelian group $M$. We have an action of $\Delta$ by conjugation on $Z(G^0)$. 
    Let $Z(G^0)^{\Delta}$ be the maximal closed subgroup scheme of $Z(G^0)$ on which the action of $\Delta$ is trivial.
    The kernel of $\OO[Z(G^0)] \to \OO[Z(G^0)^{\Delta}]$ is generated by $x - \delta x$ for $x \in M$ and $\delta \in \Delta$.
    Hence, $\OO[Z(G^0)^{\Delta}] = \OO[M_{\Delta}]$.
    
    The group $Z(G^0)^{\Delta}$ is equal to $Z(G) \cap G^0$, since for every $\OO$-algebra $A$, the group $G(A)$ is generated by $G^0(A)$ and $s((G/G^0)(A))$.
    In particular, $Z(G) \cap G^0$ is flat over $\OO$. We have a short exact sequence
    \begin{equation}\label{no_name_central}
     1 \to Z(G) \cap G^0 \to Z(G) \to H \to 1 
     \end{equation}
    where $H$ is the image of $Z(G)$ in $G/G^0$, in particular $H$ is finite.  It follows from \eqref{no_name_central} that special (resp. generic) 
    fibres of $Z(G) \cap G^0$ and $Z(G)$ will have the same 
    neutral component, and hence the same dimension. Since
     $(Z(G) \cap G^0)_k = \dim (Z(G) \cap G^0)_L$ by \Cref{dimLisdimk},  we deduce that 
    $\dim Z(G)_k = \dim Z(G)_L$.
    \end{proof}

\begin{lem}\label{smooth_projection}
    Let $G$ be a generalised reductive group scheme over $\OO$. Then the natural map $\pi : G \to G/G'$ is smooth.
\end{lem}

\begin{proof}
    By \cite[Theorem 5.3.1]{bcnrd} $G'$ is a semisimple $\OO$-group scheme and in particular smooth.
    Since $G'$ is the fppf sheafification of the commutator subgroup functor $A \mapsto [G^0(A), G^0(A)]$, it is preserved under any automorphism of $G^0$.
    So $G'$ is a closed normal subgroup scheme of $G$.
    By \Cref{extofflatisflat} the map $\pi$ is faithfully flat and of finite presentation.
    By \cite[Expos\'e VIB, Proposition 9.2 (vii)]{SGA3_new} smoothness of $G'$ implies smoothness of $\pi$.
\end{proof}

\begin{lem}\label{Gprimesc_action} Let $G$ be a generalised reductive group scheme over $\OO$. Let $G'_{\sic}\rightarrow G'$ be the simply connected central cover of $G'$. Then there is a natural action of $G$ on $\Lie(G'_{\sic})$ such that the map $\Lie(G'_{\sic})\rightarrow \Lie G'$ is $G$-equivariant.
\end{lem}

\begin{proof}
    Existence and uniqueness of the simply connected central cover is ensured by \cite[Exercise 6.5.2 (i)]{bcnrd}.
    For any $\OO$-algebra $A$, the map $(G'_{\sic})_A \to (G')_A$ is a simply connected central cover of $(G')_A$.
    By \cite[Exercise 6.5.2 (iii)]{bcnrd} an automorphism of $(G')_A$ lifts to a unique automorphism of $(G'_{\sic})_A$, so the conjugation action of $G(A)$ on $(G')_A$ by $A$-group scheme automorphisms extends to an action of $G(A)$ on $(G'_{\sic})_A$. This action is functorial in $A$ and the map $(G'_{\sic})_A \to (G')_A$ is $G(A)$-equivariant. So $G(A)$ acts functorially on $\Lie(G'_{\sic}) \otimes A$, which gives $\Lie(G'_{\sic})$ the structure of a rational $G$-module. This action is clearly compatible with the action of $G$ on $\Lie G'$.
\end{proof}

\subsection{R-parabolics}\label{sec_R_parabolic}

We recall some facts about parabolic subgroups of a generalised reductive $\OO$-group scheme $G$. The main issue is that we do not assume that $G$ is connected. The letter R stands for Richardson. Let $A$ be an $\OO$-algebra
and let $\lambda: \mathbb{G}_{m, A} \rightarrow G_A$ be a cocharacter. We then define $P_{\lambda}$ (resp. $U_{\lambda}$) to be the closed $A$-group scheme of $G_A$ representing the functor
of points $g$ of $G$ such that the limit $\lim_{t\rightarrow 0} \lambda(t)g \lambda(t)^{-1}$ exists (resp. exists and is equal to the identity), see \cite[Lemmas 2.1.4, 2.1.5]{CGP}. Let $L_{\lambda}= P_{\lambda}\cap P_{-\lambda}$, 
where $-\lambda$ is the cocharacter defined by $-\lambda(t):= \lambda(1/t)= \lambda(t)^{-1}$. We call subgroup schemes of $G$ the form $P_{\lambda}$ 
\emph{R-parabolic} subgroup schemes and $L_{\lambda}$ \emph{R-Levi} subgroup 
schemes of $P_{\lambda}$. We will only consider R-parabolic and R-Levi subgroup schemes over a connected base. 

\begin{lem}\label{list} The following hold
\begin{enumerate} 
\item $L_{\lambda}$ is the scheme theoretic centraliser of $\lambda$;
\item $U_{\lambda}$ is a closed normal subgroup scheme of $P_{\lambda}$; 
\item the multiplication map $L_{\lambda} \ltimes U_{\lambda} \rightarrow P_{\lambda}$ is an isomorphism of $R$-groups;
\item $L_{\lambda}$, $U_{\lambda}$, $P_{\lambda}$ are smooth over $\OO$.
\end{enumerate} 
\end{lem}
\begin{proof} The assertions follow from Lemma 2.1.5 and Proposition 2.1.8(2) of \cite{CGP}. The last assertion is given by 
Theorem 4.1.7 (4) in \cite{bcnrd}.
\end{proof}

\begin{lem}\label{para} Let $P$ be an R-parabolic subgroup of $G$ with R-Levi subgroup $L$, both defined over $\OO$. 
Then $P\cap G^0$ (resp. $L\cap G^0)$ is the neutral component of $P$ (resp. $L$) and is a parabolic  (resp. Levi) 
subgroup of $G^0$. In particular, $L$ is a generalised reductive $\OO$-group scheme. Moreover, there are only finitely 
many R-parabolic subgroups $Q$ of $G$ such that $Q^0=P^0$. 
\end{lem} 
\begin{proof} Let $\lambda: \Gm\rightarrow G$ be a cocharacter such that $P=P_{\lambda, G}$ and $L=L_{\lambda, G}$. 
Since $\Gm$ is connected $\lambda$ factors through $\lambda: \Gm\rightarrow G^0$. 
It follows from \cite[Lemma 2.1.5]{CGP} that $P\cap G^0=P_{\lambda, G^0}$ and $L\cap G^0= L_{\lambda, G^0}$. 
Let $\xbar$ be a geometric point of $\Spec \OO$. Then $(P_{\lambda, G^0})_{\xbar}= P_{\lambda_{\xbar}, G^0_{\xbar}}$
and $(L_{\lambda, G^0})_{\xbar}= L_{\lambda_{\xbar}, G^0_{\xbar}}$. It follows from \cite[Proposition 2.2.9]{CGP} 
that $P_{\lambda_{\xbar}, G^0_{\xbar}}$ is a parabolic subgroup of $G^0_{\xbar}$ with Levi $L_{\lambda_{\xbar}, G^0_{\xbar}}$.
Moreover, both groups are connected and $L_{\lambda_{\xbar}, G^0_{\xbar}}$ is reductive by the classical theory of reductive groups
over algebraically closed fields. Thus $P\cap G^0$ is the neutral component of $P$ and $L\cap G^0$ is the neutral component of 
$L$. Moreover, $P^0$ is a parabolic subgroup of $G^0$ in the sense of \cite[Definition 5.2.1]{bcnrd} and it follows from 
Lemma \ref{list} (3) that $L^0$ is a Levi of $P^0$ in the sense of \cite[Definition 5.4.2]{bcnrd}. Further, $L^0$ 
is reductive in the sense of \cite[Definition 3.1.1]{bcnrd}.
The map $L\rightarrow G/G^0$ identifies $L/L^0$ with a closed subgroup of $G/G^0$. 
Hence $L/L^0$ is finite over $\OO$. We deduce that $L$ is generalised reductive and $L/L^0$ is finite \'etale over $\OO$ by Remark \ref{rem_fin_etale}.

 Since $U$ is contained in $P^0$ the map $L\rightarrow P/P^0$ induces an isomorphism 
$L/L^0\cong P/P^0$. Thus $P/P^0$ is a  closed subgroup of $G/G^0$, which is finite \'etale
over $\OO$. 
Since  $G/G^0$ is finite \'etale over $\OO$ by Remark \ref{rem_fin_etale}, there are only finitely many closed subgroup schemes of $G/G^0$, which are finite \'etale over $\OO$. This implies that 
there are only finitely many R-parabolic subgroups $Q$ such that $Q^0=P^0$. 
\end{proof} 

\begin{examp}\label{ex_comp} Let $G= (\Gm \times \Gm)\rtimes \ZZ/2\ZZ$ with $\ZZ/2\ZZ$ acting on $\Gm\times \Gm$ by permuting the coordinates. Let 
$\lambda, \mu : \Gm \rightarrow G$ be the cocharacters  $\lambda(t)= ((t, 1),0)$ and $\mu(t)=((1,1), 0)$. Then 
$P_{\lambda}= G^0$ and $P_{\mu}=G$, but their neutral components are the same. 
\end{examp}

\begin{remar}  If $G\rightarrow S$ is a connected reductive group scheme 
and $S$ is affine then the notions of R-parabolic and parabolic subgroup schemes coincide;
see the discussion at the end of Section 5.4 in \cite{bcnrd}.
\end{remar}

\begin{prop}\label{RRX}
    Let $G$ be a generalised reductive group over a field, let $L$ be an R-Levi factor of an R-parabolic subgroup $P$, let $U$ be the unipotent radical of $P$ and assume that the natural map $L/L^0 \to G/G^0$ is an isomorphism. Then
    \begin{align}
        \dim L + \dim Z(L) \leq \dim G + \dim Z(G) - \dim U \label{targetineq}
    \end{align}
\end{prop}

\begin{proof} 
    We may assume throughout the proof that the base field is algebraically closed.
    Let $T \subseteq L^0$ be a maximal torus of $G^0$.
    Since $G^0$ is generated by $L^0$, the root subgroups of $U$ and their opposites, we have $Z(G^0) = Z(L^0) \cap \bigcap_a \ker(a)$, where $a : T \to \Gm$ varies over all roots, which generate root subgroups $U_a \subseteq U$. The groups $U_{\pm a}$ commute with $\ker(a)$ and $\dim \ker(a) = \dim T - 1$, so
    \begin{align}
        \dim Z(L^0) \leq \dim Z(G^0) + \dim U. \label{Uestimate}
    \end{align}
    
    Since $\Delta := G/G^0$ is constant, we have a scheme theoretic splitting $s : G/G^0 \to G$ and $\Delta$ acts on $Z(G^0)$ by conjugation. 
    Since we are assuming, that $L/L^0 = G/G^0$, we have a compatible action of $\Delta$ on $Z(L^0)$. We have a diagram
    \begin{center}
        \begin{tikzcd}
            1 \ar[r] & Z(G^0) \ar[r] & Z(L^0) \ar[r] & Z(L^0)/Z(G^0) \ar[r] & 1 \\
            1 \ar[r] & Z(G^0)^{\Delta} \ar[u, hookrightarrow] \ar[r] & Z(L^0)^{\Delta} \ar[u, hookrightarrow] \ar[r] & Z(L^0)^{\Delta}/Z(G^0)^{\Delta} \ar[r] \ar[u, hookrightarrow] & 1
        \end{tikzcd}
    \end{center}
    and the right vertical map is injective, since $Z(G^0)^{\Delta} = Z(L^0)^{\Delta} \cap Z(G^0)$.
    As in the proof of \Cref{dimkisdimL}, we have $Z(G^0)^{\Delta} = Z(G) \cap G^0$ and $Z(L^0)^{\Delta} = Z(L) \cap L^0$, so $\dim Z(G^0)^{\Delta} = \dim Z(G)$ and $\dim Z(L^0)^{\Delta} = \dim Z(L)$. Injectivity of the right vertical map gives
    \begin{align}
        \dim Z(L) - \dim Z(G) \leq \dim Z(L^0) - \dim Z(G^0) \label{GG0comp}
    \end{align}
    We conclude from \eqref{GG0comp} and \eqref{Uestimate} that 
    \begin{align*}
        \dim L + \dim Z(L) &\leq \dim L^0 + \dim Z(L^0) - \dim Z(G^0) + \dim Z(G) \\
         &\leq \dim L^0 + \dim U + \dim Z(G) 
         = \dim G + \dim Z(G) - \dim U,
    \end{align*}
    where we use $\dim G^0 = \dim L^0 + 2 \dim U$ in the last step.
\end{proof}

For the rest of the subsection we assume that $G$ is split and $G/G^0$ is constant. We fix a maximal split torus $T$ of $G$ defined over $\OO$. Let $s$ be the closed point of 
$\Spec \OO$ and let $\sbar$ be a geometric point above $s$. 

\begin{lem}\label{base_change_R_par} Base change induces a bijection between the following sets:
\begin{enumerate}
\item R-parabolics of $G$ containing $T$;
\item R-parabolics of $G_s$ containing $T_s$;
\item R-parabolics of $G_{\sbar}$ containing $T_{\sbar}$.
\end{enumerate}

\end{lem}
\begin{proof} We will first establish that $P\mapsto P_{\sbar}$ induces the bijection between sets in parts (1) and  (3). If $\lambda: \Gm \rightarrow G$ is a cocharacter defined over $\OO$ then 
$(P_{\lambda})_{\sbar} = P_{\lambda_{\sbar}}$
by \cite[Lemma 2.1.4]{CGP}. Moreover, if $P$ contains $T$ then $P_{\sbar}$ will contain $T_{\sbar}$. Hence, the map is well defined. 

If $P$ and $Q$ are R-parabolics of $G$ such that $P_{\sbar}=Q_{\sbar}$ then $P^0_{\sbar}=Q^0_{\sbar}$. 
It follows from \cite[Corollary 5.2.7 (2)]{bcnrd} 
that there exists a Zariski open subset $U$ of $\Spec 
\OO$ containing $s$ such that $P^0_U = Q^0_U$. 
Since $\OO$ is a DVR we have $U=\Spec \OO$ and hence $P^0=Q^0$. It follows from Lemma \ref{para} 
that $P\cap G^0= P^0=Q^0= Q\cap G^0$. Thus both $P$ and $Q$ are contained in the $G$-normaliser $N_G(P^0)$ of $P^0$
and it is enough to show that the images of $P$ and $Q$ in $N_G(P^0)/P^0$ coincide. Since $N_G(P^0)\cap G^0 = P^0$ 
it is 
enough to show that the images of $P$ and $Q$ in 
$G/G^0$ coincide. Since $G/G^0$ is a constant group scheme, we may check this after base change to $\sbar$, which holds as $P_{\sbar}=Q_{\sbar}$. 

Let $\lambdabar: \Gm\rightarrow G_{\sbar}$ be a 
cocharacter. Its image is contained in a maximal 
split torus  $T'$ of $G_{\sbar}$.
Moreover, $T'$ is contained in $P_{\lambdabar}$. If $P_{\lambdabar}$ 
contains $T_{\sbar}$ then there is 
$g\in P_{\lambdabar}(\kappa(\sbar))$ 
such that $T_{\sbar}= g T' g^{-1}$ by \cite[Proposition 11.19]{borel}. 
We thus may assume that the image of 
$\lambdabar$ is contained in $T_{\sbar}$. Since 
both $\Gm$ and $T$ are split tori defined over $\OO$ there is a cocharacter $\lambda: \Gm\rightarrow T$, such that $\lambda_{\sbar}=\lambdabar$. Then 
$P_{\lambda}$ contains $T$ and satisfies $(P_{\lambda})_{\sbar} = P_{\lambdabar}$. 

The proof that base change induces a bijection between the sets in parts (2) and (3) is the same. Since
$P_{\sbar}= (P_s)_{\sbar}$ this also implies that base change induces a bijection between 
the sets in parts (1) and (2).
\end{proof}

\begin{lem}\label{Rlevi} There are only finitely many R-parabolic subgroups of $G$ containing $T$. Each of them has a unique R-Levi subgroup 
containing $T$. 
\end{lem} 
\begin{proof} Lemma \ref{base_change_R_par} implies that it is enough to show that 
there are only finitely many R-parabolics of $G_{\sbar}$ containing $T_{\sbar}$. 
Lemma \ref{para} implies that we may assume that $G_{\sbar}$ is connected. 
Proposition 11.19 (b) in \cite{borel} says that there are only finitely many Borel subgroups of $G^0_{\bar{s}}$ containing $T_{\bar{s}}$. 
Proposition 14.18 in \cite{borel} implies that there are only finitely many parabolic subgroups of $G^0_{\bar{s}}$ containing a given Borel subgroup. This proves the first assertion.

Let $M$ be the unique Levi of $P^0$ containing $T$, \cite[Proposition 5.4.5]{bcnrd}. 
Let $\lambda: \Gm\rightarrow G$ be a cocharacter such that $P=P_{\lambda}$. Then $L^0_{\lambda}$ is also a Levi of $P^0$ by Lemma \ref{para}. The scheme of Levi subgroups of $P^0$ is a $U_{\lambda}$-torsor, \cite[Corollary 5.4.6]{bcnrd}. Thus 
there is $g\in U_{\lambda}(\OO)$ such that $M= g L^0_{\lambda} g^{-1}$. Then $g L_{\lambda} g^{-1}$ 
is an R-Levi of $P$ containing $T$. 

If $L$ and $M$ are R-Levi subgroups of $P$ containing $T$ then $M_{\xbar}=L_{\xbar}$ for every geometric point $\xbar$ of 
$\Spec \OO$ by \cite[Corollary 6.5]{BMR}. This implies that $L=M$.
\end{proof} 

\begin{lem}\label{conj_P} Let $\xbar$ be a geometric point of $\Spec \OO$ and let $Q$ be 
an R-parabolic subgroup of $G_{\xbar}$ with R-Levi $M$. Then there is an R-parabolic subgroup $P$ 
of $G$ defined over $\OO$ with R-Levi $L$ containing $T$ and $g\in G^0(\kappa(\xbar))$, such that 
$g Q g^{-1} = P_{\xbar}$ and $g M g^{-1} = L_{\xbar}$. 
\end{lem} 
\begin{proof} Let $\lambda: \mathbb G_{ m,\xbar}\rightarrow G_{\xbar}$ be a cocharacter 
such that $Q=P_{\lambda}$ and $M=L_{\lambda}$. There is a maximal torus $T'$ of $G_{\xbar}$, such that 
the image of $\lambda$ is contained in $T'$. Since tori are connected, both $T'$ and $T_{\xbar}$
are contained in $G^0_{\xbar}$. 
Since $\kappa(\xbar)$ is algebraically closed any two maximal tori in $G^0_{\xbar}$ are conjugate, 
and hence after conjugating $Q$ we may assume that $\lambda: \mathbb G_{ m,\xbar}\rightarrow T_{\xbar}$.
Since $T$ is split over $\OO$, there is a cocharacter $\mu: \Gm\rightarrow T$ defined 
over $\OO$ such that $\lambda= \mu_{\xbar}$. 
We have $P_\mu\times_{\Spec \OO} \xbar = P_{\mu_{\xbar}}$ by \cite[Lemma 2.1.4]{CGP} 
and $L_{\mu}\times_{\Spec \OO} \xbar = L_{\mu_{\xbar}}$ by \cite[Lemma 2.1.5]{CGP}.
\end{proof}

\subsection{\texorpdfstring{$G$-semisimplification}{G-semisimplification}}
\label{sec_G_ss}

In this section let $G$ be a generalised reductive group over an algebraically closed field $\kappa$. 
Let $\Gamma$ be a group and let $\rho: \Gamma \rightarrow G(\kappa)$ be a representation.
Let $P$ be an R-parabolic of $G$ minimal with respect to the property that $P(\kappa)$ contains $\rho(\Gamma)$. 
Let $L$ be an R-Levi of $P$ and let $U$ be the unipotent radical of $P$. The composition $L\rightarrow P\rightarrow P/U$ is an isomorphism, and we define $c_{P,L}: P(\kappa)\rightarrow G(\kappa)$ as the 
composition $P(\kappa)\rightarrow (P/U)(\kappa)\overset{\cong}{\longrightarrow} L(\kappa)\hookrightarrow G(\kappa)$.

\begin{prop}\label{indepofPL} The $G^0(\kappa)$-conjugacy class of $c_{P,L} \circ \rho$ is independent of the choice 
of $P$ and $L$.
\end{prop}
\begin{proof} If $G$ is connected then this is proved in \cite[Proposition 3.3 (b)]{serre_complete_red}.
The same argument is carried out in our more general setting in \cite[Proposition 2.15]{quast}.
\end{proof} 

\begin{defi}\label{defi_G_ss} The \emph{$G$-semisimplification} of $\rho$ is the $G^0(\kappa)$-conjugacy class of the representation $c_{P,L} \circ \rho$. It will be denoted by $\rho^{\semi}$.
\end{defi}

In \cite[Section 3.2.1]{serre_complete_red} Serre introduced the notion of $G$-complete reducibility in the connected case. This definition extends to generalised reductive group schemes verbatim.

 \begin{defi}\label{defi_Gcr} A subgroup $H$ of $G(\kappa)$ is $G$-\emph{completely reducible} if for every R-parabolic $P \subseteq G$ with $H \subseteq P (\kappa)$, there exists an R-Levi subgroup $L \subseteq P$ with $H \subseteq L(\kappa)$. We say, that a homomorphism $\rho:\Gamma \rightarrow G(\kappa)$ is $G$-\emph{completely reducible}, if its image is 
 $G$-completely reducible. We will also use the term $G$-\emph{semisimple} synonymously for $G$-completely reducible. 
 \end{defi}

\begin{prop}\label{crss} Let $\rho : \Gamma \to G(\kappa)$ be a homomorphism. Then $\rho$ is $G$-completely reducible if and only if $\rho^{\semi}$ is the $G^0(\kappa)$-conjugacy class of $\rho$.
\end{prop}

\begin{proof} 
Let $P$ be a minimal R-parabolic, such that $\rho(\Gamma) \subseteq P(\kappa)$. 
If $\rho$ is $G$-completely reducible, there exists some R-Levi $L$ of $P$, 
such that $\rho(\Gamma) \subseteq L(\kappa)$. 
In particular $\rho = c_{P,L} \circ \rho$. We can apply \Cref{indepofPL} to conclude, $\rho^{\semi}$ is in the
$G^0(\kappa)$-conjugacy class of $\rho$. Conversely, suppose that $\rho\in \rho^{\semi}$, let $L$ be any R-Levi
of $P$. The assumption implies that $\rho$ is conjugate to $c_{P,L}\circ \rho$ by an element of $g\in G^0(\kappa)$. Minimality of $P$ and \cite[Corollary 6.4]{BMR} imply that the image of $c_{P,L}\circ \rho$ is $L$-irreducible, which means that it is not  contained in any proper R-parabolic of $L$. Hence, $\rho(\Gamma)$ 
is $g L g^{-1}$-irreducible. It follows from the equivalence of parts (iv) and (v) of
\cite[Corollary 3.5]{BMR}, which as explained in \cite[Section 6.3]{BMR} also holds for non-connected groups, that $\rho(\Gamma)$ is $G$-completely reducible.
\end{proof}

\begin{prop}\label{same_det} Let $\rho : \Gamma \to G(\kappa)$ be a homomorphism. Then the determinant laws attached to $(\tau \circ \rho)^{\lin}$ and $(\tau \circ \rho^{\semi})^{\lin}$ agree.
\end{prop}

\begin{proof} Let $P$ be an R-parabolic of $G$, minimal with respect to the property that 
$P(\kappa)$ contains $\rho(\Gamma)$. Let $U$ be the unipotent radical of $P$ and let 
let $\lambda$ be a cocharacter, such that $P=P_{\lambda}$.
Then $\lim\nolimits_{t \to 0} \lambda(t) U\lambda(t)^{-1}=\{1\}$ and hence
$c_{P, L_{\lambda}}\circ \rho= \lim\nolimits_{t \to 0} \lambda(t) \rho \lambda(t)^{-1}$. 
Let $D_{\tau \circ \rho} : \kappa[\Gamma] \to \kappa$ be the determinant law attached to $(\tau \circ \rho)^{\lin}$ and let $D_{\tau\circ\rho^{\semi}}$ be the determinant law attached to $(\tau \circ c_{P, L_{\lambda}}\circ \rho)^{\lin}$. We note that replacing $c_{P, L_{\lambda}}\circ \rho$ by a conjugate 
with an element of $G^0(\kappa)$ does not change the determinant law, as they are invariant under conjugation.

We also have a family of determinant laws $D : \kappa[\Gamma] \to \kappa[t,t^{-1}]$ over $\Gm$ given by $$D_A : A[\Gamma] \to A[t,t^{-1}], \quad r \mapsto \det(((\tau(\lambda(t)) (\tau \circ \rho) \tau(\lambda(t))^{-1}) \otimes \id_A)(r)),$$ which is actually constant in $t$ and equal to $D_{\tau \circ \rho}$. So this family extends uniquely to a family over $\mathbb A^1$. Since the limit of $\lambda(t) \rho \lambda(t)^{-1}$ as $t \to 0$ exists and formation of the determinant is algebraic, we obtain $D^{t=0} = D_{\tau \circ \rho^{\semi}}$ and hence $D_{\tau \circ \rho} = D_{\tau \circ \rho^{\semi}}$.
\end{proof}

\begin{remar} In general $(\tau \circ \rho)^{\semi}$ is not isomorphic to $\tau \circ \rho^{\semi}$. But it follows from \Cref{same_det}, that $D_{\tau \circ \rho^{\semi}} = D_{\tau \circ \rho}$ and in particular from \cite[Theorem 2.12]{che_durham} that $(\tau \circ \rho^{\semi})^{\semi}$ is isomorphic to $(\tau \circ \rho)^{\semi}$.
\end{remar}

\begin{defi}\label{defi_G_irr} A subgroup $H$ of $G(\kappa)$ is $G$-\emph{irreducible}  if 
$H$ is not contained in any proper R-parabolic subgroup of $G$. 
 We say, that a homomorphism $\rho:\Gamma \rightarrow G(\kappa)$ is $G$-\emph{irreducible}, if its image is 
 $G$-irreducible.
 \end{defi} 

 \begin{prop}\label{irr_conj} Let $L$ and $M$ be R-Levi subgroups of $G$ and let $H$ be a closed subgroup of $L\cap M$ 
 such that $H$ is both $L$-irreducible and $M$-irreducible. Then there is $g\in Z_G(H)^0(\kappa)$ 
 such that $L= g M g^{-1}$. 
 \end{prop} 
\begin{proof} Let $\lambda, \mu: \Gm\rightarrow G$ be cocharacters such that $L=L_{\lambda}$ and
$M=L_{\mu}$. Since $H$ is contained in $L$ and in $M$, Lemma \ref{list} (1) implies that 
the both $\lambda$ and $\mu$ factor through the inclusion $Z_G(H) \subseteq G$. Let $S$ (resp.~$T$) be a maximal torus of $Z_G(H)$ containing the image of $\lambda$ (resp.~$\mu$). It follows from the proof of \cite[Proposition 8.18]{borel} that $Z_G(S)$ is an 
R-Levi of $G$ containing $H$ and contained in $L$. Since $H$ is $L$-irreducible by assumption 
we get that $L=Z_G(S)$. Similarly, we obtain that $M=Z_G(T)$. There is $g\in Z_G(H)^0(\kappa)$ such that 
$S= g T g^{-1}$ by \cite[Proposition 11.19]{borel}. Thus $L=Z_G ( g T g^{-1}) = g Z_G(T)g^{-1} = g M g^{-1}$. 
\end{proof}

\subsection{Extensions}\label{sec_extensions}

Let $N$ and $\Delta$ be abstract groups. A \emph{rigidified extension of $N$ by $\Delta$} 
is a group law $\ast$ on the set $N\times \Delta$, such that the maps $\iota: N \rightarrow N\times \Delta$, $g\mapsto (g, 1)$ and $\pi: N\times \Delta\rightarrow 
\Delta$, $(g, \delta)\mapsto \delta$ are group homomorphisms.  

\begin{defi}\label{defi_gen_two_cocyc} A \emph{generalised $2$-cocycle} of $\Delta$ with coefficients in $N$ is a
pair $(\omega, c)$, where $\omega: \Delta\rightarrow \Aut(N)$ and $c: \Delta\times \Delta\rightarrow N$ 
are maps (of sets) satisfying the following:
\begin{enumerate}
        \item $\omega(\delta_1) \circ \omega(\delta_2) = \Int_{N}(c(\delta_1, \delta_2)) \circ \omega(\delta_1\delta_2)$ in $\Aut(N)$;
        \item $\omega(1) = \id_{N}$;
        \item $c(\delta_1,\delta_2)c(\delta_1\delta_2, \delta_3) = \omega(\delta_1)(c(\delta_2, \delta_3)) c(\delta_1, \delta_2\delta_3)$ in $N$;
        \item $c(\delta_1, 1) = c(1,\delta_2) = 1$;
    \end{enumerate}
where for $g\in N$, $\Int_N(g)$  is the inner automorphism $x \mapsto gxg^{-1}$ of $N$.
\end{defi}
We emphasise that, since $N$ is not assumed to be abelian, $\omega$ is not going to be a group homomorphism 
in general. Let $(N\times \Delta, \ast)$ be a rigidified extension of $N$ by $\Delta$. It is a pleasant 
exercise in Algebra to verify that if we let
$$ \omega^{\ast}(\delta)(x):= s(\delta) x s(\delta)^{-1}, \quad c^{\ast}(\delta_1, \delta_2):= 
s(\delta_1)s(\delta_2)s(\delta_1 \delta_2)^{-1},$$
for all $x\in N$ and all $\delta, \delta_1, \delta_2\in \Delta$, where $s(\delta)= (1, \delta)$,
then the group axioms for $\ast$ imply that  $(\omega^{\ast}, c^{\ast})$ is a generalised $2$-cocycle.

\begin{lem}\label{gen2cocycle} Mapping $(N\times \Delta, \ast)$ to $(\omega^{\ast}, c^{\ast})$ induces 
a bijection from the set of rigidified extensions of $N$ by $\Delta$ to the set of 
 generalised $2$-cocycles of $\Delta$ with values in $N$. 
\end{lem}
\begin{proof} If $N$ is abelian then $\Int_N =\id_N$ and the result is well known. In 
general, the assertion follows from \cite[Satz I]{schreier}, a nice exposition is given 
in \cite[Section 2.2]{zhang_bachelor} and Lemma follows from \cite[Lemma 2.11]{zhang_bachelor}.
One could also deduce the assertion from 
\cite{EilMac}: property (3) implies that the element denoted by $f_3(x,y,z)$ in 
\cite[Section 7]{EilMac} is $1_N$. Theorem 8.1 in \cite{EilMac} then implies that 
every  generalised $2$-cocycle comes from a rigidified extension. 
\end{proof}

\begin{prop}\label{letter_bcnrd}
    Let $G$ be a generalised reductive group over a perfect field $\kappa$ of characteristic $p$, such that $G^0$ is split semisimple, $G/G^0=\underline{\Delta}$ is constant and the map $G(\kappa)\rightarrow (G/G^0)(\kappa)$ is surjective. 
    Then there exists a surjection of generalised reductive groups $H \to G$ over $\kappa$ with kernel $\mu$, such that the following hold:
    \begin{enumerate}
        \item $H \to G$ induces an isomorphism $H/H^0 \cong G/G^0$;
        \item $0 \to \mu \to H^0 \to G^0 \to 0$ is a central extension with $\mu(\overline{\kappa})=1$;
        \item $\pi_1(H^0)$ is \'etale.
    \end{enumerate}
    Moreover, $H(\kappa) \cong G(\kappa)$.
\end{prop}

\begin{proof} We may choose 
a set-theoretic section $s : \Delta \to G(\kappa)$ such that $s(1) = 1$. It induces an isomorphism 
of $\kappa$-schemes
    \begin{align}
        G^0 \times \underline\Delta \eqto G, \quad (g, \delta) \mapsto gs(\delta). \label{G0_deconstr_iso}
    \end{align}
    We define a map $\omega : \Delta \to \Aut(G^0)$ by $\omega(\delta)(g) := s(\delta)gs(\delta)^{-1}$. 
    We further define a map $c : \Delta \times \Delta \to G^0(\kappa)$ by $c(\delta_1,\delta_2) := s(\delta_1)s(\delta_2)s(\delta_1\delta_2)^{-1}$.
    
    For a $\kappa$-algebra $A$, we get a map $\omega_A : \Delta \to \Aut(G^0(A))$ by composing with the homomorphism $\Aut(G^0) \to \Aut((G^0)_A)\to\Aut(G^0(A))$ with $\omega$ and a map $c_A: \Delta\times \Delta \rightarrow G^0(A)$ 
    by composing $c$ with the map $G^0(\kappa) \to G^0(A)$. 
    
    The natural map $\Delta \to \underline\Delta(A)$ 
    is a group homomorphism. The isomorphism \eqref{G0_deconstr_iso} induces a bijection $G^0(A)\times \underline{\Delta}(A)\eqto G(A)$. We use the bijection to transport the group structure on 
    $G^0(A)\times \underline{\Delta}(A)$. When $A \neq 0$, \Cref{gen2cocycle} implies that 
    $(\omega_A, c_A)$ is a  generalised $2$-cocycle of $\Delta$ with values in $G^0(A)$ 
    corresponding to the restriction of this  group law to the subset $G^0(A)\times \Delta$. 
    Thus, for all $\delta_1,\delta_2, \delta_3 \in \Delta$ we obtain the following identities:
    \begin{itemize}
        \item[(C1)] $\omega(\delta_1) \circ \omega(\delta_2) = \Int_{G^0}(c(\delta_1, \delta_2)) \circ \omega(\delta_1\delta_2)$ \hspace{5pt} in $\Aut(G^0)$,
        \item[(C2)] $\omega(1) = \id_{G^0}$ \hspace{5pt} in $\Aut(G^0)$,
        \item[(C3)] $c(\delta_1,\delta_2)c(\delta_1\delta_2, \delta_3) = \omega_{\kappa}(\delta_1)(c(\delta_2, \delta_3)) c(\delta_1, \delta_2\delta_3)$ \hspace{5pt} in $G^0(\kappa)$, and
        \item[(C4)] $c(1,\delta_1) = c(\delta_1, 1) = 1$ \hspace{5pt} in $G^0(\kappa)$,
    \end{itemize}
    where for $g \in G^0(\kappa)$, $\Int_{G^0}(g)$ is the inner automorphism $x \mapsto gxg^{-1}$ of $G^0$.
    Here (C1) and (C2) follow by considering all $A$. Identities (C3) and (C4) follow by considering the case $A=\kappa$.

    The simply connected central cover $\pi_{\sic} : (G^0)_{\sic} \to G^0$ has finite multiplicative kernel $K$.
    Let $K_0$ be the prime to $p$ part of $K$ and define $H_0 := (G^0)_{\sic}/K_0$. The kernel $\mu$ of $\pi : H_0 \to G^0$ is isomorphic to $K/K_0$ and is thus a connected finite multiplicative $\kappa$-group scheme. Since $G^0$ is split, $\mu$ is a product of group schemes of the form $\mu_{p^r}$.
    We have an exact sequence of non-abelian fppf cohomology groups
    $$ 0 \to \mu(\kappa) \to H_0(\kappa) \to G^0(\kappa) \to H^1_{\mathrm{fppf}}(\Spec(\kappa), \mu). $$
    Since $H^1_{\mathrm{fppf}}(\Spec(\kappa), \mu_{p^n}) = \kappa^{\times}/ (\kappa^{\times})^{p^n}$ by \stackcite{040Q} and $\kappa$ is a perfect field 
    of characteristic $p$, we have
    $\mu(\kappa)=1$ and $H^1_{\mathrm{fppf}}(\Spec(\kappa), \mu) = 0$.
    So we can see $c$ as a map to $H_0(\kappa)$, which we will denote by $\tilde c$. If $A$ is 
    a $\kappa$-algebra we let $\tilde{c}_A$ be the composition of $c$ with the map $H_0(\kappa) \to H_0(A)$.

    By \cite[Exercise 6.5.2 (iii)]{bcnrd} any automorphism $\varphi : G^0 \to G^0$ lifts uniquely to an automorphism $\varphi_{\sic} : (G^0)_{\sic} \to (G^0)_{\sic}$ such that 
    $\pi_{\sic} \circ \varphi_{\sic} = \varphi \circ \pi_{\sic}$ 
    and $\varphi_{\sic}$ preserves $K_0$, so $\varphi$ lifts uniquely to an automorphism $L(\varphi) : H_0 \to H_0$ such that $\pi \circ L(\varphi) = \varphi \circ \pi$ and we obtain an injective group homomorphism $L : \Aut(G^0) \hookrightarrow \Aut(H_0)$.

    Let $\tilde \omega := L \circ \omega$, i.e. for all $\delta \in \Delta$, we have $\omega(\delta) \circ \pi = \pi \circ \tilde\omega(\delta)$. We claim that the pair $(\tilde \omega, \tilde c)$ satisfies 
    the identities (C1), (C2), (C3) and (C4) above. 
    Clearly the map $\tilde c$ satisfies  (C4).
    Via the isomorphism $H_0(\kappa) \cong G^0(\kappa)$ the maps $\omega_{\kappa}$ and $\tilde \omega_{\kappa}$ can be identified, so (C3) holds for $(\tilde \omega, \tilde c)$. Property (C2) holds, as $L$ is a homomorphism. By applying $L$ to the identity (C1) for $G^0$, we get
    \begin{equation}\label{yalla_go}
     \tilde \omega(\delta_1) \circ \tilde \omega(\delta_2) = L(\Int_{G^0}(c(\delta_1, \delta_2))) \circ \tilde \omega(\delta_1\delta_2) \quad \text{ in } \Aut(H_0). 
     \end{equation}
    Since under the isomorphism $H_0(\kappa) \eqto G^0(\kappa)$ the element $\tilde c(\delta_1, \delta_2)$ is mapped to $c(\delta_1, \delta_2)$, we have $\pi \circ \Int_{H_0}(\tilde c(\delta_1, \delta_2)) = \Int_{G^0}(c(\delta_1, \delta_2)) \circ \pi$.
    Thus $L(\Int_{G^0}(c(\delta_1, \delta_2))) = \Int_{H_0}(\tilde c(\delta_1, \delta_2))$  by definition of $L$ and we deduce from \eqref{yalla_go} that 
    (C1) holds for $H_0$, which finishes the proof of the claim.
    We deduce that for every $\kappa$-algebra $A$ the pair $(\tilde{\omega}_A, \tilde{c}_A)$ 
    is a  generalised $2$-cocycle of $\Delta$ with values in $H_0(A)$, which by Lemma \ref{gen2cocycle} defines
    a group law $\ast$ on $H_0(A)\times \Delta$, natural in $A$.

    The functor $A\mapsto (H_0(A)\times \Delta, \ast)$ defines a presheaf of groups, which we denote by  $H^{\mathrm{pre}}$. 
    The map  $H^{\mathrm{pre}}(A) \to G(A)$, $(h, \delta)\mapsto \pi(h) s(\delta)$ is a group homomorphism natural in $A$, as $\pi_A \circ \tilde c_A = c_A$ 
    and for any $\delta \in \Delta$, we have $\omega(\delta) \circ \pi = \pi \circ \tilde \omega(\delta)$. 
    We let $H$ be the Zariski sheafification of $H^{\mathrm{pre}}$. Since $\underline{\Delta}$ is the 
    sheafification of the constant presheaf $\Delta$, 
    $H= H_0\times \underline{\Delta}$ and     $H^{\mathrm{pre}}\rightarrow G$ 
    induces a map of $\kappa$-group schemes $H\rightarrow G$. 
    Properties (1) and (2) claimed in the Lemma hold by construction as $H^0=H_0$ and (3) holds as $\pi_1(H^0)$ is isomorphic to $K_0$.
\end{proof}

\section{Deformation rings} \label{sec_def_rings}

 Let $\Gamma$ be a profinite group such  that 
 $H^1(\Gamma', \Fp)$ is finite for all open subgroups $\Gamma'$ of $\Gamma$. This $p$-finiteness 
 condition introduced by Mazur holds, when $\Gamma=\Gamma_F$. 

Let $\rho: \Gamma\rightarrow G(\kappa)$ be a continuous 
representation, where $G$ is a smooth affine $\OO$-group scheme and
$\kappa$ is either a finite extension of $k$, a finite extension of $L$ 
or a local field of characteristic $p$ containing $k$ equipped with natural topology. We will refer to such fields $\kappa$
as \emph{a finite or a local $\OO$-field}. 
In this section we will study  the deformation theory of $\rho$.
We first define a ring of coefficients $\Lambda$ 
over which the deformation problem is defined
following \cite[Section 3.5]{BIP_new}.
\begin{enumerate}
    \item If $\kappa$ is a finite field then pick an unramified extension $L'$ of $L$ with residue field $\kappa$ and let $\Lambda := \OO_{L'}$ denote the ring of integers in $L'$.
    \item If $\kappa$ is a finite extension of $L$ then let $\Lambda:=\kappa$, let $\Lambda^0$ 
be the ring of integers in $\Lambda$ and let $t=\varpi$.
    \item If $\kappa$ is a local field of characteristic $p$ then let $\OO_{\kappa}$ be the ring of integers 
in $\kappa$ and let $k'$ be its residue field. Since $\cha(\kappa)=p$ by choosing a uniformiser we obtain an isomorphism $\OO_{\kappa}\cong k'\br{t}$. Let $L'$ be an unramified extension of $L$ with residue field $k'$,
let $\Lambda^0:= \OO_{L'} \br{t}$ and let $\Lambda$ 
be the $p$-adic completion of $\Lambda^0[1/t]$. Then 
$\Lambda$ is a complete DVR with uniformiser $\varpi$
and residue field $\kappa$. We equip $\Lambda^0$ 
with its $(\varpi, t)$-adic topology, this induces 
a topology on $\Lambda^0[1/t]$ and 
$\Lambda^0[1/t]/ p^n \Lambda^0[1/t]$ for all $n\ge 1$. We equip $\Lambda= \varprojlim_{n} \Lambda^0[1/t]/ p^n \Lambda^0[1/t]$ with the projective limit topology. 
\end{enumerate}

Let $\mathfrak A_{\Lambda}$ be the category of local
artinian $\Lambda$-algebras with residue field $\kappa$. 
If $(A, \mm_A)\in \mathfrak A_{\Lambda}$ then we define the topology on $A$ in the cases (1), (2) and (3) above as follows:
\begin{enumerate}
    \item $A$ is a finite $\OO/\varpi^n$-module for some $n\ge 1$ with  the discrete topology on $A$;
    \item 
$A$ is a finite dimensional $L$-vector space with the $p$-adic topology on $A$; 
    \item 
$A$ is a $\Lambda^0[1/t]/ \varpi^n \Lambda^0[1/t]$-module of finite length for some $n\ge 1$ and we
put the induced topology on $A$.
\end{enumerate}

 Since $G$ is  smooth the map $G(A)\rightarrow G(\kappa)$ is surjective for all $A\in \Aa_{\Lambda}$. Let  $D^{\square}_{\rho, G}: \mathfrak A_{\Lambda}
\rightarrow \Set$ be the functor such that  
$D^{\square}_{\rho, G}(A)$ is the set of 
continuous group homomorphisms $\rho_A: \Gamma\rightarrow 
G(A)$ satisfying $\rho_A \pmod{\mm_A}=\rho$.  

\begin{lem}\label{dimZ1} Let $V$ be a continuous representation of $\Gamma$ on a finite dimensional 
$\kappa$-vector space $V$. Then 
\begin{equation}\label{eq_gamma}
\dim_{\kappa} Z^1(\Gamma, V)= h^1(\Gamma, V) +\dim_{\kappa} V - h^0(\Gamma, V),
\end{equation}
where $h^i(\Gamma, V):= \dim_{\kappa} H^i(\Gamma, V)$. Moreover, if $\Gamma=\Gamma_F$ then 
\begin{equation}
\begin{split}
\dim_{\kappa} Z^1(\Gamma_F, V)&=\dim_{\kappa} V ( [F:\Qp]+1)+ h^0(\Gamma_F, V^*(1))\\
&= \dim_{\kappa} V ([F:\Qp]+1)+ h^2(\Gamma_F, V).
\end{split}
\end{equation}
\end{lem}

\begin{proof} Mapping $v\in V$ to a coboundary $\gamma\mapsto (\gamma-1) v$ induces an exact sequence:
$$0\rightarrow V^{\Gamma} \rightarrow V \rightarrow Z^1(\Gamma, V)\rightarrow 
H^1(\Gamma, V)\rightarrow 0,$$
which implies \eqref{eq_gamma}. The second part follows from 
local Euler--Poincar\'{e} characteristic formula and local Tate duality, see for example \cite[Theorem 3.4.1]{BJ_new}.
\end{proof} 

Let $\kappa[\varepsilon]$ be the ring of dual numbers over $\kappa$. 
The Lie algebra $\Lie G_{\kappa}$ of $G_{\kappa}$ is a finite dimensional $\kappa$-vector space, which sits in the exact sequence
of groups 
$$ 0\rightarrow   \Lie G_{\kappa}\rightarrow G(\kappa[\varepsilon])\rightarrow G(\kappa)\rightarrow 0.$$
The action of $G(\kappa)$ on $G(\kappa[\varepsilon])$ by conjugation induces 
a $\kappa$-linear action of $G(\kappa)$ on $\Lie G_{\kappa}$. 
We let $\ad \rho$ be the representation of $\Gamma$ on $\Lie G_{\kappa}$
defined by $$\ad(\rho)(\gamma)(X):= \rho(\gamma) X \rho(\gamma)^{-1}, \quad \forall \gamma\in \Gamma, \quad \forall X\in \Lie G_{\kappa}.$$ Then $\ad \rho$ is a continuous 
representation of $\Gamma$ for the topology on $\Lie G_{\kappa}$ induced by the topology on $\kappa$.

 \begin{lem}\label{Dnoetherian} Let $\rho : \Gamma \to G(\kappa)$ be a continuous representation. 
Then $D^{\square}_{\rho, G}$ is pro-representable by a complete local 
 noetherian
$\Lambda$-algebra $R^{\square}_{\rho, G}$ with residue field $\kappa$.
\end{lem}

\begin{proof} If $A,B,C \in \mathfrak A_{\Lambda}$, then a homomorphism $\Gamma \to G(A \times_C B)$ is continuous if and only if its projections to $G(A)$ and $G(B)$ are continuous. It follows that the functor $\mathfrak A_{\Lambda} \to \Set, ~ A \mapsto \Hom^{\cont}_{\mathrm{Group}}(\Gamma, G(A))$
preserves finite limits. The deformation functor $D^{\square}_{\rho, G}$ is defined as a pullback of the latter, so $D^{\square}_{\rho, G}$ also preserves finite limits. By Grothendieck's criterion, $D^{\square}_{\rho, G}$ is pro-representable by a projective limit of objects in $\mathfrak A_{\Lambda}$.

The map $Z^1(\Gamma, \ad \rho)\rightarrow D^{\square}_{\rho, G}(\kappa[\varepsilon])$, $\Phi\mapsto [\gamma \mapsto \rho(\gamma) (1 +\varepsilon \Phi(\gamma))]$ 
is an isomorphism of $\kappa$-vector spaces. Since $\Gamma$ satisfies Mazur's $p$-finiteness
condition
\cite[Theorem 3.4.1]{BJ_new} implies that $h^1(\Gamma, \ad \rho)$ is finite. 
Lemma \ref{dimZ1} implies that $Z^1(\Gamma, \ad \rho)$ is finite dimensional, thus 
$R^{\square}_{\rho, G}$ is noetherian.
\end{proof}

\begin{lem}\label{sectionexists} Let $f : G \to H$ be a continuous surjective open homomorphism between locally profinite groups. Then there is a continuous map 
$s: H \to G$ with $f \circ s = \id_H$.
\end{lem}

\begin{proof} If $G$ and $H$ are profinite, this is \cite[Section I.1, Proposition 1]{SerreGalCoh}.
In general, let $N$ be an open profinite subgroup of $G$. Then $N' := f(N)$ is an open profinite subgroup of $H$ and we find a continuous map $s : N' \to N$ with $f|_N \circ s = \id_{N'}$. Let $R \subseteq H$ be a set of $N'$-coset representatives so that $H = \bigcup_{r \in R} rN'$ and choose $g_r \in G$ with $f(g_r) = r$ for each $r \in R$.
We can extend $s$ to a continuous map $\tilde s : H \to G$ by defining $\tilde s(rn') := g_rs(n')$ for all $n' \in N'$ and see, that $f \circ \tilde s = \id_H$, as desired.
\end{proof}

\begin{lem}\label{cptopensubalg} Let $\kappa$ be  a local $\OO$-field, let $\Lambda$ be a coefficient ring for $\kappa$ and let $(A, \mm_A) \in \mathfrak A_{\Lambda}$. Then there is a module-finite $\Lambda^0$-subalgebra $R$ of $A$ with $R \otimes_{\Lambda^0} \Lambda = A$. In particular, $R$ is a profinite open subring of $A$.
\end{lem}

\begin{proof}
Let $\pi_A : A \to \kappa$ be the projection to the residue field.
Pick $\Lambda$-module generators $x_1, \dots, x_n$ of $A$ and we may assume by rescaling with $t$, that $\pi_A(x_i) \in \OO_{\kappa}$ for all $i$.
Since $\Lambda^0$ surjects onto $\OO_{\kappa}$, we can pick $y_i \in \Lambda^0$ with $\pi_A(y_i) = \pi_A(x_i)$ for all $i$.
Since $\mm_A$ is nilpotent, the finitely many nonzero products of the elements $x_i - y_i \in \mm_A$ generate a multiplicatively closed finitely generated $\Lambda^0$-submodule $I$ of $\mm_A$.
In particular $R := \Lambda^0 + I$ is a module-finite
$\Lambda^0$-subalgebra of $A$ with $R[1/t] = A$.
Since $R$ is a finitely generated $\Lambda^0$-module, $\Lambda^0$ is open in $\Lambda$ and $\Lambda$ is Hausdorff, $R$ is profinite and open.
\end{proof}

\begin{lem}\label{splittingexists} Let $\kappa$ be either a finite or a local $\OO$-field and let $\Lambda$ be a coefficient ring for $\kappa$. Let $G$ be a smooth affine $\OO$-group.
\begin{enumerate}
    \item For every $A \in \mathfrak A_{\Lambda}$, the topological group $G(A)$ is locally profinite.
    \item Let $f : A \to B$ be a surjective morphism in $\mathfrak A_{\Lambda}$. Then the map $\varphi : G(A) \to G(B)$ induced by $f$ is surjective and there is a continuous map $s : G(B) \to G(A)$ with $\varphi \circ s = \id_{G(B)}$.
\end{enumerate}
\end{lem}

\begin{proof} When $\kappa$ is finite, both assertions are clear, so we assume that $\kappa$ is local. Let $R$ be a profinite open subalgebra of $A$ with $R \otimes_{\Lambda^0} \Lambda = A$ as constructed in \Cref{cptopensubalg}.

Part (1). We claim, that $G(R) \subseteq G(A)$ is a profinite open subgroup. Choose a closed immersion $G \hookrightarrow \mathbb A^n$ over $\OO$. By definition $G(A)$ carries the subspace topology of $\mathbb A^n(A)$. Since $\mathbb A^n(R)$ is open in $\mathbb A^n(A)$, $G(R)$ is open in $G(A)$. Since $R$ is Hausdorff, $G(R)$ is closed in $\mathbb A^n(R)$ and since $\mathbb A^n(R)$ is profinite, $G(R)$ is profinite.

Part (2). Surjectivity of $\varphi$ follows from smoothness of $G$. Then $R' := f(R)$ is a compact subalgebra of $B$ with $R' \otimes_{\Lambda^0} \Lambda = B$. Since the topology of $B$ is by definition induced by an arbitrary open $\Lambda^0$-lattice, we see that $R'$ is open in $B$. Since $G$ is smooth, we have a surjection $G(R) \to G(R')$ of open profinite subgroups of $G(A)$ and $G(B)$ respectively. Since any surjective continuous homomorphism between profinite groups is open\footnote{Indeed, any continuous surjective homomorphism from a compact group onto a Hausdorff group is open by Remark 1 after Proposition 24 in \cite[Ch.\,III, \S 2, no.\,8]{Bourbaki_top}.}, we see that $\varphi$ is open. The assertion follows from \Cref{sectionexists}.
\end{proof}

\begin{prop}\label{present_over_RH} Let $\varphi : G \to H$ be a morphism of smooth affine $\OO$-group schemes
such that the induced morphism on neutral components $G^0\rightarrow H^0$ is smooth and surjective.
Let $\rho : \Gamma \to G(\kappa)$ be a continuous representation and let 
$\ad^{0, \varphi}= \ker(\ad \rho\rightarrow \ad \varphi\circ\rho)$.
Then there is a natural  isomorphism of $R^{\square}_{\varphi \circ \rho, H}$-algebras
\begin{equation}\label{eq:present_over_RH}
R^{\square}_{\varphi \circ \rho, H}\br{x_1,\ldots, x_r}/(f_1, \ldots, f_t) \cong R^{\square}_{\rho, G}
\end{equation}
where $r = \dim_{\kappa} Z^1(\Gamma, \ad^{0,\varphi} \rho)$, $t = h^2(\Gamma, \ad^{0,\varphi} \rho)$. 
Moreover, if $\Gamma=\Gamma_F$ then 
\begin{equation}\label{r_minus_t}
r-t = (\dim G_{\kappa} - \dim H_{\kappa}) \cdot ([F : \Qp] + 1).
\end{equation}
\end{prop}

\begin{proof} If $\kappa$ is a finite field then the argument of \cite[Proposition 4.3]{BIP_new}, where  $G = \GL_d$, $H = \GL_1$, $\varphi = \det$, goes through verbatim. 
If $\kappa$ is a local field then one needs to additionally verify that obstruction 
to lifting can be realised by a continuous $2$-cocycle, as noted in 
\cite[Remark 4.4]{BIP_new}. The solution to this problem is explained in detail 
in \cite[Lecture 6]{mod_lift}, when $\kappa$ is a local field of characteristic zero
and $H$ is trivial. We sketch how these arguments fit together. 

The exact sequence $0\rightarrow \ad^{0,\varphi}\rho \rightarrow \ad \rho \rightarrow \ad \varphi\circ\rho \rightarrow 0$ of Galois representations induces 
an exact sequence of abelian groups:  
$$0 \rightarrow Z^1(\Gamma, \ad^{0,\varphi}\rho)\rightarrow Z^1(\Gamma, \ad\rho) \rightarrow 
Z^1(\Gamma, \ad \varphi\circ \rho)$$
and hence $r=\dim_{\kappa} \ker (Z^1(\ad\varphi))$. 
The map $$Z^1(\ad \varphi): Z^1(\Gamma, \ad\rho) \to Z^1(\Gamma, \ad \varphi\circ\rho)$$ is the induced map on Zariski tangent spaces of the map of deformation rings $R_{\varphi \circ \rho}^{\square} \to R^{\square}_{\rho}$, and thus lifts to a surjection
    \[ \tilde\phi: \tilde R := R_{\varphi \circ \rho}^{\square}\br{x_1,\ldots,x_r} \twoheadrightarrow R^{\square}_{\rho}. \]

We set $J := \ker \tilde \phi$.
By Nakayama's lemma, we need to show, that $\dim_{\kappa} J/\tilde \mm J \leq t$. The module $J/\tilde \mm J$ appears as the kernel of the sequence
\begin{align}
    0 \to J/\tilde \mm J \to \tilde R / \tilde \mm J \to \tilde R/J \cong R_{\rho, G}^{\square} \to 0 \label{JmJsequence}
\end{align}
In view of the above sequence, it is enough construct a homomorphism
\begin{equation}\label{define_alpha}
 \alpha : \Hom_{\kappa}(J/\tilde\mm J, \kappa) \to \ker (H^2(\ad \varphi) : H^2(\Gamma, \ad \rho) \to H^2(\Gamma, \ad \varphi \circ \rho)) 
\end{equation}
and show, that $\ker(\alpha)$ injects into $\coker(H^1(\ad \varphi))$. This will imply the existence of the presentation in the statement of the Proposition, since then
\begin{align}
    \dim_{\kappa} J/\tilde\mm J \leq \dim_{\kappa} \ker(H^2(\ad \varphi)) + \dim_{\kappa} \coker(H^1(\ad \varphi)) = h^2(\Gamma,\ad^{0,\varphi} \rho) \label{boundJmJ}
\end{align}
where the last equality comes from the long exact cohomology sequence that arises from $0 \to \ad^{0,\varphi} \rho \to \ad\rho \to \ad\varphi \circ \rho \to 0$.

For all $i > 0$, we define $\tilde R_i:=\tilde R/\tilde \mm^i$, $\tilde\mm_i := \tilde\mm/\tilde\mm^i$, $J_i := (J + \tilde\mm^i)/\tilde\mm^i$ and $R_{\rho, G,i}^{\square} := R_{\rho, G}^{\square}/\tilde\mm^i R_{\rho, G}^{\square}$.
The Artin--Rees lemma implies that $J\cap (\tilde\mm J +\tilde \mm^i)= \tilde\mm J$ for sufficiently large $i \gg 0$, 
thus the map $J/\tilde\mm J \to J_i/\tilde\mm_iJ_i$ is an isomorphism. We will fix such $i$ for the rest of the proof. It is therefore enough to work with objects of $\Aa_{\Lambda}$ whose maximal ideal has vanishing $i$-th power, we will denote this category by $\Aa_{\Lambda, i}$.
The inclusion functor $\Aa_{\Lambda, i} \to \Aa_{\Lambda}$ has a left adjoint given by modding out the $i$-th power of the maximal ideal.
The ring $R_{\rho, G,i}^{\square}$ represents the restriction of $D^{\square}_{\rho,G}$ to $\Aa_{\Lambda,i}$ and we will denote the universal lift by $\rho^{\square}_i : \Gamma \to G(R_{\rho, G,i}^{\square})$. We obtain an exact sequence
\begin{align}
    0 \to J_i/\tilde \mm_i J_i \to \tilde R_i / \tilde \mm_i J_i \to \tilde R_i/J_i \cong R_{\rho, G, i}^{\square} \to 0 \label{JmJsequencei}
\end{align}
where $\tilde R_i / \tilde \mm_i J_i$ and $R_{\rho, G, i}^{\square}$ are objects of $\Aa_{\Lambda, i}$. If $u \in \Hom_{\kappa}(J_i/\tilde\mm_i J_i, \kappa)$ then by modding out $\ker(u)$ in \eqref{JmJsequencei} we obtain an exact sequence
\begin{align}
    0 \to I_u \to R_u \xrightarrow{\phi_u} R^{\square}_{\rho, G,i} \to 0, \label{pushseq}
\end{align} 
 The surjection of locally profinite groups $G(R_u)\twoheadrightarrow G(R^{\square}_{\rho, G,i})$ has a continuous section by \Cref{splittingexists}. Once we have this the proof 
is identical to the proof of \cite[Proposition 4.3]{BIP_new}.
Moreover, if $\Gamma=\Gamma_F$ then \eqref{r_minus_t} follows from Lemma \ref{dimZ1}.
\end{proof}

\begin{remar} We will apply the above proposition
in the setting of \Cref{smooth_projection}.
\end{remar}

\begin{cor}\label{represent_kappa}
There is an isomorphism of local $\Lambda$-algebras
\begin{equation}\label{present_square_x}
R^{\square}_{\rho,G}\cong \Lambda\br{x_1,\ldots, x_r}/(f_1, \ldots, f_s) 
\end{equation}
with $r = \dim_{\kappa} Z^1(\Gamma, \ad \rho)$ and $s = \dim_{\kappa} H^2(\Gamma, \ad \rho)$. Further, if $\Gamma=\Gamma_F$ then
$$r-s= \dim G_{\kappa} ([F:\Qp]+1).$$
\end{cor} 

\begin{proof} The assertion follows from Lemma \ref{Dnoetherian} and Proposition
\ref{present_over_RH} applied with  $H$ equal to the trivial group. 
 \end{proof}

\Cref{present_over_RH} will be often used together with the following lemma.

\begin{lem}\label{pre_lci_flat} Let $A$ be a complete local noetherian $\OO$-algebra with residue field $k$ and let 
$B= A\br{x_1, \ldots, x_r}/(f_1, \ldots, f_t)$. If $A$ is complete intersection and 
$\dim B\le \dim A + r - t$ then $B$ is complete intersection and $\dim B=\dim A + r - t$. Moreover, if $\dim k\otimes_A B\le  r-t$ then $B$ is a  flat $A$-algebra and $\dim k\otimes_A B=  r-t$.
\end{lem} 

\begin{proof} Since $A$ is complete intersection we may write $A= S/(g_1, \ldots, g_s)$, 
where $S$ is a complete local regular ring and $g_1, \ldots, g_s$ is a regular sequence in $S$.
We choose $\tilde{f}_1, \ldots, \tilde{f}_t \in S\br{x_1, \ldots, x_r}$ that map to 
$f_1, \ldots, f_t$. Since when we quotient out by a relation the dimension either goes down 
by one or stays then same, 
the assumption on the dimension of $B$ implies that $\dim B= \dim A + r - t$ and $g_1, \ldots, g_s, \tilde{f}_1, \ldots, \tilde{f}_t$ 
can be extended to a system of parameters of $S\br{x_1,\ldots, x_r}$, and hence is a regular sequence in 
the regular ring $S\br{x_1,\ldots, x_r}$. Thus $B$ is complete intersection and also $C:= S\br{x_1, \ldots, x_r}/(\tilde{f}_1, \ldots, \tilde{f}_t)$ is complete intersection.  
Since $k\otimes_S B= k\otimes_A B$, it follows from miracle flatness \cite[Theorem 23.1]{matsumura} that $C$ is flat over $S$, 
and hence $B= C/(g_1, \ldots, g_s)$ is flat over $A= S/(g_1, \ldots, g_s)$. 
\end{proof}

\begin{lem}\label{extend_scalars1} Let $\rho: \Gamma \rightarrow G(\kappa)$ be a continuous representation, where
$\kappa$ is either local or finite $\OO$-field. Let $\kappa'$ be a finite extension of $\kappa$ 
and let $\rho': \Gamma \rightarrow G(\kappa')$ be the representation obtained by 
composing $\rho$ with the inclusion $G(\kappa) \subseteq G(\kappa')$. Then there is a natural 
isomorphism of local $\Lambda'$-algebras: 
$$ \Lambda'\otimes_{\Lambda} R^{\square}_{G, \rho} \cong R^{\square}_{G, \rho'},$$
where $\Lambda$ and $\Lambda'$ are coefficient rings for $\kappa$ and $\kappa'$, respectively.
\end{lem}
\begin{proof} The argument explained on page 457 of \cite{Wiles} also applies in our setting. 
\end{proof}

\begin{lem}\label{extend_scalars2} Let $\kappa$ be a local $\OO$-field with ring of integers $\OO_{\kappa}$ and 
residue field $k$.  Let $\rho: \Gamma \rightarrow G(\kappa)$ be a continuous representation, 
such that its image is contained in $G(\OO_{\kappa})$, and let $\rhobar:\Gamma \rightarrow 
G(k)$ be the composition of $\rho$ with the reduction map $G(\OO_{\kappa})\rightarrow G(k)$. 
Let $\varphi: R^{\square}_{G, \rhobar} \rightarrow \OO_{\kappa}\hookrightarrow \kappa$ be 
the map induced by considering $\rho$ as a deformation of $\rhobar$ to $\OO_{\kappa}$. Let $\qq$ be the kernel of the natural map 
$$\Lambda\otimes_{\OO} R^{\square}_{G, \rhobar} \rightarrow \kappa, \quad \lambda\otimes a\mapsto \bar{\lambda} \varphi(a),$$
where $\Lambda$ is the coefficient ring for $\kappa$. Then there is a natural isomorphism of local 
$\Lambda$-algebras between $R^{\square}_{G, \rho}$ and the completion  
$\Lambda\otimes_{\OO} R^{\square}_{G, \rhobar}$ with respect to $\qq$. 
\end{lem}
\begin{proof} The proof is the same as the proof of \cite[Theorem 3.3.1]{BJ_new} for $\GL_d$, which 
in turn is based on \cite[Proposition 9.5]{kisin_over}, where the case $G=\GL_d$ and $\chara(\kappa)=0$ is treated.  
\end{proof}

\subsection{Central extensions}\label{sec_centr_ext} In this section, we assume that $G$ is generalised reductive and  let $H=G/Z$, where 
$Z$ is a flat closed subgroup scheme of $Z(G^0)$, which is normal in $G$. 
Part (i) of \cite[Exercise 5.5.9]{bcnrd} implies that $H$ is also generalised reductive. We let 
$\varphi: G \rightarrow H$ be the quotient map. 
 We fix a continuous 
representation $\rho: \Gamma\rightarrow G(\kappa)$ and drop it from the notation.  We will write $R^{\square}_{G}$ for $R^{\square}_{\rho, G}$, $R^{\square}_H$ for 
$R^{\square}_{\varphi\circ \rho, H}$, etc. 

\begin{cor}\label{fin_et} If $Z$ is finite \'etale
 then $R^{\square}_{ H}=R^{\square}_{ G}$. 
\end{cor} 
\begin{proof} Since $Z$ is finite \'etale its Lie algebra is zero. We thus get $\ad^{0,\varphi}\rho=0$ 
and the assertion follows from Proposition \ref{present_over_RH}.
\end{proof}

\begin{prop}\label{nice_torus}
If $Z$ is a torus such that $Z \cap G'$ is \'etale then there is a natural isomorphism of local $\Lambda$-algebras 
$R^{\square}_{G/G'}\wtimes_{R^{\square}_{H/H'}} R^{\square}_{H} 
\eqto R^{\square}_{G}$. 
\end{prop}

\begin{proof} Let $D: \Aa_{\Lambda}\rightarrow \Set$ be the functor such that $D(A)$ is 
the set of pairs $(\psi, \xi)\in D^{\square}_{G/G'}(A)\times D^{\square}_{H}(A)$ such that 
$\pi\circ \psi = \theta \circ \xi$, where $\theta: H\rightarrow H/H'$ and $\pi: 
G/G'\rightarrow H/H'$ are the quotient maps. The Proposition amounts to the claim that the map 
\begin{equation}\label{ride}
D^{\square}_{G}(A)\rightarrow D(A),\quad \rho_A \mapsto (\varphi\circ \rho_A, q\circ \rho_A)
\end{equation}
is bijective.
Since $Z\cap G'$ is \'etale by assumption, using \Cref{fin_et} we may assume that 
$Z\cap G'$ is trivial. 
The map $G^0 \to H^0$ is surjective with kernel $Z$ and the map $G' \to H'$ 
is surjective with kernel $Z \cap G'$.
The map of group schemes 
\begin{equation}\label{jackson}
 G \rightarrow G/G'\times_{H/H'} H
 \end{equation}
has kernel $Z\cap G'$, so it is injective. If 
$([g],h) \in G(A)/G'(A) \times_{H(A)/H'(A)} H(A)$
then, since $G\rightarrow H$ is surjective,  there exists a finite flat extension 
$A \hookrightarrow B$ and $\tilde g\in G(B)$, which maps to $h\in H(A)\subseteq H(B)$.
Then  $\tilde g g^{-1} \in H'(B) = G'(B)$, so $[\tilde g] = [g]$ in $G(B)/G'(B)$.
By sheafification this implies, that \eqref{jackson} is bijective.

This implies that \eqref{ride} is bijective, 
since then given $(\psi, \xi)\in D(A)$ for every $\gamma\in \Gamma$ there is a
unique $\rho'(\gamma)\in G(A)$ such that $\varphi(\rho'(\gamma))=\psi(\gamma)$ and 
$\pi(\rho'(\gamma))=\xi(\gamma)$. Uniqueness implies that $\rho'\in D^{\square}_{G}(A)$, 
which maps to $(\psi,\xi)$. 
\end{proof}

\begin{cor}\label{cor_nice_torus} If $\Gamma=\Gamma_F$, $\kappa$ is
finite and we are in the setting of \Cref{nice_torus} then 
$\dim R^{\square}_G/\varpi = \dim R^{\square}_H/\varpi + ([F:\Qp]+1)(\dim G_k - \dim H_k)$.
\end{cor}

\begin{proof} Let $A'= (R^{\square}_{H/H'}/\varpi)^{\red}$, $B'= (R^{\square}_{G/G'}/\varpi)^{\red}$, $A=(R^{\square}_{H}/\varpi)^{\red}$, where the superscript $\red$ indicates
reduced rings. We have shown in \cite[Corollary 9.9]{defT} that 
$A'\rightarrow B'$ is flat and the dimension $d$ of the fibre $\kappa\otimes_{A'} B'$ is equal to
$$(\dim (G/G')_k -\dim (H/H')_k)([F:\Qp]+1)=(\dim G_k-\dim H_k)([F:\Qp]+1).$$
Hence, $A\rightarrow B'\wtimes_{A'} A$ is flat 
and the fibre has dimension $d$. Thus $\dim A+d=\dim B'\wtimes_{A'} A$ by \cite[\href{https://stacks.math.columbia.edu/tag/00ON}{Tag 00ON}]{stacks-project}. \Cref{nice_torus} implies that 
$(B'\wtimes_{A'} A)^{\red}\cong (R^{\square}_G/\varpi)^{\red}$. 
Hence, $\dim R^{\square}_G/\varpi = \dim A +d = \dim R^{\square}_H/\varpi +d$. 
\end{proof}

For the rest of the section we let $Z$ be a finite connected diagonalisable subgroup scheme of $Z(G^0)$. If $\chara(\kappa)=0$ then $Z$ is trivial so we assume that 
$\chara(\kappa)=p$. In this case there is a finite abelian $p$-group $\md$ 
such that $Z(A)= \Hom(\md, A^{\times})$ for $A\in \OO\text{-}\alg$. 

\begin{lem}\label{soggy_finite}  $R^{\square}_H \rightarrow R^{\square}_G$ is finite.
\end{lem} 
\begin{proof} Let $A:= \kappa\otimes_{R^{\square}_H} R^{\square}_G$ and let $\pp$ be a prime ideal of 
$A$. Since $A/\pp$ is a domain of characteristic $p$, we have $Z(A/\pp)=0$ and thus 
$G(A/\pp)\subseteq H(A/\pp)$. Let $\rho_{\pp}$ be the specialisation of the universal deformation
$\Gamma \rightarrow G(R^{\square}_G)$ at $\pp$. Since $\varphi\circ \rho_{\pp}= \varphi\circ \rho$ 
by definition of $A$, we get that  $\rho_{\pp}(\Gamma)$ is contained in $G(\kappa)$.
Hence, $\pp$ is the maximal ideal of $A$. We conclude that $A$ is zero dimensional and, since
$R^{\square}_G$ is noetherian, $A$ is a finite dimensional $\kappa$-vector space. The claim 
follows from Nakayama's Lemma. 
\end{proof}

\begin{prop}\label{fin_conn} $\dim R^{\square}_H/\varpi = \dim R^{\square}_G/\varpi$.
\end{prop}

\begin{proof} For any local $\kappa$-algebra $A$ we have an exact sequence of non-abelian fppf cohomology groups
\begin{equation}\label{fppf_yeah}
0\rightarrow Z(A) \rightarrow G(A) \rightarrow H(A) \rightarrow H^1_{\mathrm{fppf}}(\Spec(A), Z)
\end{equation}
Let $e\ge 0$ be an integer such that $p^e\cdot\md=0$. Let $A=R^{\square}_H/\varpi$ and let 
$\mm$ be the maximal ideal of $A$. We choose elements $a_1, \ldots, a_r\in \mm$ such that 
their images form a basis of $\mm/\mm^2$ as a $\kappa$-vector space. Let 
$B=A[x_1,\ldots, x_n]/I$, where $I$ is the ideal generated by $x_i^{p^e}-a_i$ for $1\le i \le r$.
Then $B$ is finite flat over $A$ and every $x\in 1+\mm$ is a $p^e$-th power in $B^{\times}$. 
Let $A^{\red}$ and $B^{\red}$ be the maximal reduced quotients of $A$ and $B$, respectively. 
Then $A^{\red}$ is a subring of $B^{\red}$ and every $x\in 1+\mm_{A^{\red}}$ is a $p^{e}$-th power
in $B^{\red}$. Let $\rho': \Gamma\rightarrow H(A^{\red})$ be the specialisation of
 the universal deformation $\Gamma \rightarrow H(R^{\square}_H)$ along $R^{\square}_H \rightarrow A^{\red}$.
 
Since $Z$ is a finite product of $\mu_{p^n}$ with $n\le e$ and $B^{\red}$ is reduced and of characteristic 
$p$, we have $Z(B^{\red})=1$. Since $H^1_{\mathrm{fppf}}(\Spec(A), \mu_{p^n}) = A^{\times}/ (A^{\times})^{p^n}$ by \stackcite{040Q}, it follows from \eqref{fppf_yeah} that there exists a continuous representation 
$\tilde{\rho}: \Gamma \rightarrow G(B^{\red})$ deforming $\rho$ such that 
$\varphi\circ \tilde{\rho} = \rho'\otimes_{A^{\red}} B^{\red}$. This induces a map 
of $A$-algebras $R^{\square}_G/\varpi \rightarrow B^{\red}$. This map is  finite as $B$ is 
a finite $A$-algebra. Hence $\dim (R^{\square}_G/\varpi)\ge \dim B^{\red}= \dim B =\dim A$, 
where the last equality follows as $B$ is finite and flat over $A$. 
Lemma \ref{soggy_finite} implies that $ \dim (R^{\square}_G/\varpi)\le \dim A$. 
\end{proof}

\section{Continuity}\label{sec_cont}
In this section let $\OO$ be any ring, let $R\rightarrow A$ be a map of $\OO$-algebras, 
with $R$  noetherian and profinite. Let $\rho: \Gamma \rightarrow G(A)$ be a group 
homomorphism, where $\Gamma$ is a profinite group and $G$ is an affine $\OO$-group 
scheme of finite type.   

In \Cref{sec_cont_alg} we define an algebraic  notion of continuity for such representations, 
such that whenever $A$ is a topological $R$-algebra the algebraic notion of continuity implies
that the representation $\rho:\Gamma \rightarrow G(A)$ is continuous  for  the topology on the target induced by the topology of $A$.
Our construction is based on the fact that if $M$ is a finitely generated $R$-module 
then since $R$ is noetherian and profinite there is a unique topology on $M$ making 
$M$ into a Hausdorff topological $R$-module; 
this topology is profinite. 
For example, if $G=\Ga$ then our algebraic notion of continuity 
is the requirement that $\rho(\Gamma)$ is contained in a finitely generated $R$-submodule $M \subseteq A=\Ga(A)$ and the map $\rho: \Gamma\rightarrow M$ is continuous. We extend this example 
following Conrad's exposition in \cite{ConradTopologies} on how to topologise $X(A)$, when $A$ is a topological
ring. 

In  \Cref{sec_cond} we reformulate this notion in terms of condensed sets. In fact,
one could reprove the results of \Cref{sec_cont_alg} entirely within the framework 
of condensed sets. However, we have chosen not to do so as the arguments in  \Cref{sec_cont_alg}
are elementary and suffice for the paper.

\subsection{Continuity via algebra}\label{sec_cont_alg}
The following is the key lemma. 
\begin{lem}\label{polyhasclosedzeroset} 
Let $f\in A[y_1,\ldots, y_n]$ be a polynomial, which 
we view as a function on $A^n$ and let $V(f):=\{\underline{a}\in A^n: f(\underline{a})=0\}$. Let $M$ be a finitely generated 
$R$-submodule of $A^n$.
Then $f|_M : M \to A$ is a continuous map with values in a finitely generated $R$-submodule of $A$.
In particular, $V(f)\cap M$ is a closed subset of $M$ 
for the natural topology on $M$ as an $R$-module.
\end{lem}

\begin{proof}
    We first focus on the case, when $f$ is a monomial $y_1^{d_1} \cdots y_n^{d_n}$ of degree $d_1 + \dots + d_n= d \geq 1$.
    There exists an $A$-multilinear form $\beta : A^{\otimes_A d} \to A$, such that $f(a_1, \dots, a_n) = \beta(a_1 \otimes \dots \otimes a_1 \otimes \cdots \otimes a_n \otimes \cdots \otimes a_n)$ for all $a_1, \dots, a_n \in A$, where in the argument of $\beta$, $a_i$ is repeated $d_i$ times.
    Write $\iota : M \to A$ for the inclusion map.
    We have a natural $R$-linear map $\iota^{\otimes} : M^{\otimes_R d} \to A^{\otimes_A d}$ and we denote by $N$ the image of $M^{\otimes_R d}$ under 
    $\tilde\beta:= \beta \circ \iota^{\otimes}$ in $A$.
    In particular $f|_M$ takes values in $N$ and arises as the composition
    $$ M \xrightarrow{\Delta} M^d \xrightarrow{\mathrm{can}} M^{\otimes_R d} \xrightarrow{\tilde\beta} N. $$
    The diagonal $\Delta$ and $\tilde\beta$ are $R$-linear, hence continuous.
    We are left to show, that the universal $d$-multilinear form $M^d \to M^{\otimes_R d}$ is continuous.
    Choose an $R$-linear surjection $R^s \twoheadrightarrow M$. We have a commutative square
    \begin{equation}
    \begin{tikzcd}
        (R^s)^d \rar \dar[twoheadrightarrow] & (R^s)^{\otimes_R d} \dar[twoheadrightarrow] \\
        M^d \rar & M^{\otimes_R d}
    \end{tikzcd}
    \end{equation}
    where the top and the bottom maps are the universal $d$-multilinear forms. The vertical maps are quotient maps.
    The top map is a polynomial map over $R$, hence continuous and thus the bottom map is continuous, which finishes the proof in this case.

    Now assume, that $f$ is a general polynomial and let us write $f = \sum_i a_i f_i$ for monomials $f_i$ and $a_i \in A$.
    As in the previous step, we obtain multilinear maps $\beta_i : A^{\otimes_A d} \to A$ and we define $\tilde \beta_i := \beta_i \circ \iota^{\otimes}$, $N_i := \Image(\tilde \beta_i)$ and $N := \sum_i a_iN_i$. The multiplication maps $N_i \to N, ~x \mapsto a_ix$ are $R$-linear, thus $N$ is a finitely generated $R$-module. So $f|_M = \sum_i a_i f_i|_M$ is a sum of continuous maps and hence continuous, which finishes the proof.
    
    Since $R$ is Hausdorff, so is $N$ and  we get that $V(f) \cap M$ is closed.
\end{proof}

Let $X=\Spec B$ be an affine scheme of finite type
over $\OO$. A point $x\in X(A)$ corresponds
to an $\OO$-algebra homomorphism $x: B\rightarrow A$. We thus have a canonical map
$$ X(A)\times \Gamma(X, \OO_X)\rightarrow A, \quad (x, b)\mapsto \ev_b(x):= x(b).$$
This induces a canonical injection 
$$\iota_X:=\prod_{b\in B}\ev_b: X(A) \hookrightarrow A^B:= \prod_{b\in B} A.$$
We will denote the projection map onto the $b$-th component by $p_b: A^B\rightarrow A$. 
As explained above any finitely generated $R$-module $M$ has a unique topology, which makes
it into a  Hausdorff topological $R$-module. Moreover, $M$ is profinite with respect to this topology. 

\begin{lem}\label{B1} Let $S$ be a profinite set and let $f: S\rightarrow X(A)$ be a map of sets. Then the following
are equivalent: 
\begin{enumerate}
\item for each $b\in B$ there exists a finitely generated $R$-submodule $M_b$ of $A$, such that 
$\iota_X(f (S))$ is contained in $\prod_{b\in B} M_b$ and the map 
$\iota_X\circ f: S \rightarrow \prod_{b\in B} M_b$ is continuous for the product topology on the target.
\item for all $n\ge 1$ and all closed immersions $\tau: X\hookrightarrow \mathbb A^n$ of $\OO$-schemes there exists 
a finitely generated $R$-submodule $M \subseteq \mathbb A^n(A)= A^n$ such that 
$\tau(f(S)) \subseteq M$ and $\tau\circ f: S \rightarrow M$ is continuous for the canonical $R$-module topology on the target.
\item there exists a closed immersion $\tau: X\hookrightarrow \mathbb A^n$ of $\OO$-schemes and a finitely generated $R$-submodule $M \subseteq \mathbb A^n(A)=A^n$ such that $\tau(f(S)) \subseteq M$ and $\tau\circ f: S \rightarrow M$ is continuous for the canonical $R$-module topology  on the target.
\end{enumerate}
\end{lem}
\begin{proof} A closed immersion $\tau: X\hookrightarrow \mathbb A^n$ corresponds 
to a surjective homomorphism of $\OO$-algebras $\OO[x_1, \ldots, x_n]\twoheadrightarrow B$. 
Let $b_1, \ldots b_n$ be the images of $x_1, \ldots x_n$ in $B$.  Then the map 
$ X(A) \overset{\tau}{\longrightarrow} \mathbb A^n(A)=A^n$ factors as 
$X(A) \rightarrow A^B \overset{\prod_{i=1}^n p_{b_i}}{\longrightarrow} A^n$. 
The image of $\prod_{b} M_b$ under this 
map is $\prod_{i=1}^n M_{b_i}$, which is a finitely generated $R$-module. Moreover, this projection 
is continuous. Hence, part (1) implies part (2).

Part (2) trivially implies part (3). 

Let us assume that (3) holds. Let $M_i \subseteq A$ be the image of $M \subseteq \prod_{i=1}^n A$ 
under the projection onto the $i$-th component. Then $M \subseteq \prod_{i=1}^n M_i$ and 
$\prod_{i=1}^n M_i$ is a finitely generated $R$-module. Hence, we may replace $M$ with $\prod_{i=1}^n M_i$.
Let $b_i\in B$ be the images of $x_i$ as before. If $b\in B$ then there is $g\in \OO[x_1, \ldots, x_n]$
such that $b= g(b_1, \ldots, b_n)$. The image of $\prod_{i=1}^n M_i$ under the map 
$\varphi_g: A^n \rightarrow A$, $(a_1, \ldots, a_n) \mapsto g(a_1, \ldots, a_n)$ is contained in a finitely generated 
$R$-module $M_b$ by   \Cref{polyhasclosedzeroset}. Since $\tau(f(S)) \subseteq \prod_{i=1}^n M_i$ 
we have $\varphi_g(\tau(f(S))) \subseteq M_b$. Moreover,  
the induced map $S\rightarrow M_b$ is continuous since polynomial maps are continuous.
The composition 
$\varphi_g \circ \tau: X(A) \rightarrow A$ is the map given by evaluation at $b$. Hence, (3) implies (1). 
\end{proof}

\begin{lem}\label{B2} Let $g: X\rightarrow Y$ be a morphism of affine schemes of finite type over $\OO$. Let $S$ be a profinite set and let $f: S\rightarrow X(A)$ be a map of sets. 
If the equivalent conditions of Lemma \ref{B1} hold for $f$ then they also hold for $g\circ f$.
\end{lem}

\begin{proof} Let us assume that $\iota_X(f(S))$ is contained in $\prod_{b\in \Gamma(X, \OO_X)} M_b$. 
We claim that $\iota_Y(g (f(S)))$ is contained in $\prod_{c\in \Gamma(Y, \OO_Y)} M_{g^{\sharp}(c)}$. Indeed, 
if $c\in \Gamma(Y, \OO_Y)$ and $x\in X(A)$ then $\ev_c ( g(x)) = \ev_{g^{\sharp}(c)}(x)$.
\end{proof}

\begin{lem}\label{B3} Let $S$ be a profinite set and let $f_1: S\rightarrow X_1(A)$ and 
$f_2: S\rightarrow X_2(A)$ be maps of sets, where $X_1$ and $X_2$ are affine schemes
of finite type over $\OO$. If the equivalent conditions of Lemma \ref{B1} hold for $f_1$ and $f_2$ then they 
also hold for the product $f_1\times f_2: S\rightarrow X_1(A)\times X_2(A)=(X_1\times X_2)(A)$. 
\end{lem} 
\begin{proof} We pick closed embeddings $X_1\hookrightarrow \mathbb A^{n_1}$, 
$X_2\hookrightarrow \mathbb A^{n_2}$. This induces a closed embedding 
$X_1\times X_2\hookrightarrow \mathbb A^{n_1}\times \mathbb A^{n_2} \cong 
\mathbb A^{n_1+n_2}$. If $f_1(S)\subseteq M_1\subseteq A^{n_1}$ and $f_2(S)\subseteq M_2 \subseteq A^{n_2}$ then 
$(f_1\times f_2)(S)\subseteq M_1\times M_2\subseteq A^{n_1}\times A^{n_2}$. Moreover, if 
$f_1: S\rightarrow M_1$ and $f_2: S\rightarrow M_2$ are continuous then $f_1\times f_2: S\rightarrow M_1\times M_2$ 
is continuous. 
\end{proof} 

\begin{lem}\label{B4} Let $R'$ be a noetherian profinite $R$-algebra together with a map of $R$-algebras $R'\rightarrow A$.
Let $S$ be a profinite set and let $f: S\rightarrow X(A)$ be a map of sets. If the equivalent 
conditions of Lemma \ref{B1} hold for $f$ with respect to $R$ then they also hold with respect to $R'$.
\end{lem}
\begin{proof} It is immediate that part (1) of Lemma \ref{B1} holds with $M'_b$ equal to the $R'$-submodule  of $A$ generated by $M_b$. 
\end{proof}

\subsection{Continuity via condensed sets}\label{sec_cond} We will now express the equivalent conditions
in Lemma \ref{B1} in terms of condensed sets. Recall that these are certain functors
from the category of profinite sets to the category of sets, see \cite{ScholzeCond} and \Cref{app_cond_set}.
We may view profinite 
sets as condensed sets via the Yoneda embedding $S\mapsto \underline{S}$, where $\underline S (T) := \Cont(T,S)$. The underlying set of a 
condensed set is by definition that functor evaluated at a point.

Recall further, that every profinite group (ring) is naturally a condensed group (ring), as the functor from topological spaces to condensed sets commutes with finite limits \cite[Proposition 1.7]{ScholzeCond}. We can also equip a profinite ring $R$ with the discrete topology and write $R_{\disc}$ to denote the associated discrete condensed ring.  We have a natural map of condensed rings $R_{\disc} \to \underline{R}$. An abstract $R$-module $M$ can be seen as a discrete condensed $R_{\disc}$-module $M_{\disc}$, so that 
$M_{\disc}(T)=\Cont(T, M)$, where the target $M$ is given the discrete topology.
The tensor product of condensed $R_{\disc}$-modules $M_{\disc} \otimes_{R_{\disc}} \underline{R}$ is a condensed $R$-module.\footnote{Fargues--Scholze in \cite[Section VIII]{fargues2021geometrization}
call such modules \emph{relatively discrete} $R$-modules.}

\begin{lem}\label{rel_disc_is_qs} Let $N$ be an abstract $R$-module. Then $N_{\disc} \otimes_{R_{\disc}} \underline{R}$ 
is an inductive limit of profinite submodules along quasi-compact injections. In particular, 
$N_{\disc} \otimes_{R_{\disc}} \underline{R}$ is quasi-separated.
\end{lem}
\begin{proof} We write $N=\varinjlim_i M_i$ where $M_i$ are finitely generated $R$-submodules 
of $N$. As the tensor product of condensed $R_{\disc}$-modules is left adjoint to the internal Hom functor of condensed $R_{\disc}$-modules, it commutes with arbitrary colimits. Thus
    \begin{align*}
        N_{\disc} \otimes_{R_{\disc}} \underline{R} = (\varinjlim\nolimits_i (M_{i})_{\disc}) \otimes_{R_{\disc}} \underline{R} = \varinjlim\nolimits_i ((M_i)_{\disc} \otimes_{R_{\disc}} \underline{R}).
    \end{align*}
    Since $M_i$ are finitely generated they are profinite and we have a natural isomorphism 
    $(M_i)_{\disc}\otimes_{R_{\disc}} \underline{R} \cong \underline{M_i}$. 
    Moreover, the transition maps are quasi-compact injections, so $N_{\disc} \otimes_{R_{\disc}} \underline{R}$ is quasi-separated by \Cref{injlimisqs}.
\end{proof}

Let $X$ be an affine scheme of finite type over $\OO$ as before and let $S$ be a profinite set. 
If $C$ is a commutative condensed $\OO$-algebra then by evaluating it at a profinite set $T$ 
we obtain a commutative $\OO$-algebra $C(T)$ and then we may evaluate $X$ at $C(T)$ to obtain a set $X(C(T))$. One may show that $X(C):= X \circ C$ is a
condensed set. If $C$ is quasi-separated, then $X(C)$ is quasi-separated (see \Cref{subsp_qs}).

\begin{defi} We will say that a map of sets $f:S\rightarrow X(A)$ is \emph{$R$-condensed} if 
there is a morphism of condensed sets $\tilde f : \underline{S} \to X(A_{\disc} \otimes_{R_{\disc}} \underline{R})$, whose map on underlying sets is $f$.
\end{defi}

\begin{lem}\label{ind_cont_equiv_R_cond} Let $f: S\rightarrow X(A)$ be a map of sets. Then $f$ is $R$-condensed if and only if 
the equivalent conditions of Lemma \ref{B1} hold. 
\end{lem} 

\begin{proof}
    Suppose $f$ is $R$-condensed. To verify (1) of \Cref{B1}, let $b \in B$ where $X=\Spec(B)$. This determines a map of schemes $\varphi_b : X \to \mathbb A^1$.
    We obtain a map of condensed sets 
    $$ \underline S \xrightarrow{\tilde f} X(A_{\disc} \otimes_{R_{\disc}} \underline R) \xrightarrow{\varphi_b} \mathbb A^1(A_{\disc} \otimes_{R_{\disc}} \underline R) = A_{\disc} \otimes_{R_{\disc}} \underline R.$$ 
    
    By \Cref{rel_disc_is_qs} $A_{\disc} \otimes_{R_{\disc}} \underline R$ is an inductive limit of condensed $R$-submodules of the form $\underline M$ for some finitely generated $R$-submodule $M \subseteq A$. By \Cref{profiniteiscompact} $\underline S$ is a compact object in the category of condensed sets. In particular there exists a finitely generated $R$-submodule $M_b \subseteq A$, such that $\varphi_b \circ \tilde f$ factors over a map $\underline S \to \underline{M_b}$. It follows, that $\varphi_b(f(S))$ is contained in $M_b$. We deduce condition (1) by applying the universal property of product spaces.

    Assume (2) of \Cref{B1}, i.e.\,there is a closed immersion $\tau : X \hookrightarrow \mathbb A^n$ of $\OO$-schemes, a finitely generated $R$-submodule $M \subseteq \mathbb A^n(A)$, such that $\tau(f(S)) \subseteq M$ and $\tau \circ f : S \to M$ is continuous. We obtain a map of condensed sets
    $$ g : \underline S \to \underline M = M_{\disc} \otimes_{R_{\disc}} \underline R \to \mathbb A^n(A_{\disc} \otimes_{R_{\disc}} \underline R).$$
    By \Cref{cl_subscheme_qc_inj} the map $X(A_{\disc} \otimes_{R_{\disc}} \underline R) \to \mathbb A^n(A_{\disc} \otimes_{R_{\disc}} \underline R)$ is a quasi-compact injection.
    It follows from \cite[Proposition 4.13]{ScholzeAnalytic}, that
    $$ X(A_{\disc} \otimes_{R_{\disc}} \underline R)(S) \cong \mathbb A^n(A_{\disc} \otimes_{R_{\disc}} \underline R)(S) \times_{\Map(S, \mathbb A^n(A))} \Map(S, X(A)).$$
    The pair $(g,f)$ is an element in the right hand side.
    So by Yoneda there exists a map of condensed sets $\tilde f : \underline S \to X(A_{\disc} \otimes_{R_{\disc}} \underline R)$ with underlying map $f$.
\end{proof}

\begin{lem}\label{cond_is_cont}
    Let $A$ be a topological $R$-algebra and let $f : S \to X(A)$ be an $R$-condensed map.
    Then $f$ is continuous.
\end{lem}

\begin{proof}
    Let $\tau : X \hookrightarrow \mathbb A^n$ be a closed immersion. Since $X(A)$ has the subspace topology of $A^n$, it is enough to show that $\tau \circ \rho : \Gamma \to A^n$ is continuous for the product topology on $A^n$. Since $\rho$ is $R$-condensed, there is a finitely generated $R$-submodule $N \subseteq A^n$, such that $\tau(\rho(\Gamma)) \subseteq N$ and such that $\tau \circ \rho : \Gamma \to N$ is continuous for the unique Hausdorff topology on $N$. Let us write $N^s$ for $N$ equipped with the subspace topology of $A^n$. We claim, that the identity $\id_N : N \to N^s$ is continuous for the Hausdorff topology on the source. It is enough to show that the projection $N \to N^s/\overline{\{0\}}$ is continuous, as every closed subset of $N^s$ is a union of translates of $\overline{\{0\}}$, where the closure is taken in $N^s$. But the quotient topology on $N^s/\overline{\{0\}}$ is Hausdorff, so $N \to N^s/\overline{\{0\}}$ is continuous.
\end{proof}

\subsection{\texorpdfstring{$R$-condensed representations}{R-condensed representations}} Let $\Gamma$ be a profinite group, and let $G$ be an affine group scheme of finite type over $\OO$.
\begin{defi}\label{def_R_cond_rep} A  representation $\rho: \Gamma \rightarrow G(A)$ is 
\emph{$R$-condensed} if $\rho$ is $R$-condensed as a map of sets. 
\end{defi}
\begin{remar} Lemma \ref{ind_cont_equiv_R_cond} implies that $\rho$ is $R$-condensed if and only if 
there is a closed immersion $\tau: G\hookrightarrow \mathbb A^n$ of $\OO$-schemes such that $\tau(\rho(\Gamma))$ is contained in a finitely generated $R$-submodule of $A^n=\mathbb A^n(A)$. 
For the purposes of the paper we could take this last statement as the definition of $R$-condensed 
representations, and proceed without actually making use of condensed mathematics. The main point 
of Lemma \ref{B1} is that this definition is independent of the chosen embedding $\tau$.
\end{remar}

\begin{lem}
    Let $\rho : \Gamma \to G(A)$ be an $R$-condensed representation.
    Then there is a unique homomorphism of condensed groups $\tilde \rho : \underline{\Gamma} \to G(A_{\disc} \otimes_{R_{\disc}} \underline{R})$ whose map on underlying sets is $\rho$.
\end{lem}

\begin{proof}
    Since $\rho$ is $R$-condensed, there is a map of condensed sets $$\tilde \rho : \underline{\Gamma} \to G(A_{\disc} \otimes_{R_{\disc}} \underline{R})$$ whose map on underlying sets is $\rho$. 
    Moreover, this map is unique by Lemma \ref{cond_set_unique}. We need to show, that the diagram of condensed sets
    \begin{equation}
    \begin{tikzcd}
        \underline \Gamma \times \underline \Gamma \arrow[d] \arrow[r, "\tilde \rho \times \tilde \rho"] & G(A_{\disc} \otimes_{R_{\disc}} \underline{R}) \times G(A_{\disc} \otimes_{R_{\disc}} \underline{R}) \arrow[d] \\
        \underline \Gamma \arrow[r, "\tilde\rho"] & G(A_{\disc} \otimes_{R_{\disc}} \underline{R})
    \end{tikzcd}
    \end{equation}
    commutes, where the vertical maps are the multiplications.
    The diagram commutes on underlying sets, so the claim follows from uniqueness \Cref{cond_set_unique}. 
\end{proof}

\begin{lem}\label{rest_open} Let $\rho: \Gamma \rightarrow G(A)$ be a representation and 
let $\Gamma'$ be an open subgroup of $\Gamma$. Then the following are equivalent:
\begin{enumerate}
\item $\rho$ is $R$-condensed;
\item the restriction of $\rho$ to $\Gamma'$ is $R$-condensed.
\end{enumerate}
\end{lem}
\begin{proof} Since $\rho(\Gamma')\subseteq \rho(\Gamma)$ part (1) trivially implies 
part (2). 

Since $\Gamma'$ is an open subgroup of $\Gamma$  we may write $\Gamma = \bigcup_i c_i \Gamma'$ for $c_1, \dots, c_r \in \Gamma$ a finite set of coset representatives. Let $\tau : G \hookrightarrow \GL_d$ be a homomorphism of $\OO$-group schemes, which is a closed immersion
    and let $\delta: \GL_d\rightarrow M_d\times M_d$ be the closed immersion $g\mapsto (g, g^{-1})$. 
    Then $\delta\circ \tau : G \hookrightarrow M_d \times M_d$ is a closed immersion. 
    If (2) holds then $\delta(\tau(\rho(\Gamma'))$ is contained in a finitely generated 
    $R$-submodule of $M_d(A)\times M_d(A)$. 
    Since the map 
    $$M_d(A)\times M_d(A)\rightarrow M_d(A)\times M_d(A), \quad (X, Y)\mapsto (\tau(\rho(c_i))X, Y\tau(\rho(c_i))^{-1})$$ 
    is $R$-linear, we conclude that $\rho(\Gamma)$ maps into a finitely generated $R$-submodule of $M_d(A) \times M_d(A)$, which implies that $\rho$ is $R$-condensed.
    \end{proof}
    
\begin{lem}\label{cont_vs_cond}
    Let $A$ be a Hausdorff topological $R$-algebra, such that $A$ is a filtered union of closed $R$-subalgebras $B \subseteq A$, where each $B$ contains an open $R$-subalgebra $B^0 \subseteq B$, which is a finitely generated $R$-module. Assume, that $\Gamma$ is a topologically finitely generated profinite group and let $\rho : \Gamma \to G(A)$ be a homomorphism. Then the following are equivalent:
    \begin{enumerate}
        \item $\rho$ is continuous;
        \item $\rho$ is $R$-condensed.
    \end{enumerate}
\end{lem}

\begin{proof} 
    (1) implies (2). Let $\gamma_1, \dots, \gamma_N \in \Gamma$ be generators of a dense subgroup of $\Gamma$.
    By assumption there is an $R$-subalgebra $B \subseteq A$, which contains an $R$-finite open $R$-subalgebra $B^0$, such that $\rho(\gamma_1), \dots, \rho(\gamma_N) \in G(B)$.
    Since $G(B)$ is closed in $G(A)$, we have $\rho(\Gamma) \subseteq G(B)$. 
    Since $G$ is of finite type over $\OO$, $G(B^0)$ is an open subgroup of $G(B)$. Thus  the preimage $\Gamma':=\rho^{-1}(G(B^0))$ is an open subgroup of $\Gamma$. If $\tau: G\rightarrow 
    \mathbb A^n$ is a closed immersion then $\tau(\rho(\Gamma'))$ is contained in $\mathbb A^n(B^0)$, which a finitely generated $R$-module. Thus the restriction of $\rho$ to $\Gamma'$ 
    is $R$-condensed and Lemma \ref{rest_open} implies that $\rho$ is $R$-condensed.

    (2) implies (1) by \Cref{cond_is_cont}.
\end{proof}

\begin{lem}\label{gener_cond} Let $H$ be a closed subgroup scheme of $G$, let $\{\gamma_i\}_{i\in I}$ a 
set of topological generators of $\Gamma$ and let $\rho: \Gamma\rightarrow G(A)$ be an 
$R$-condensed representation. Then the following are equivalent:
\begin{enumerate}
\item $\rho$ factors through $H(A)\subseteq G(A)$;
\item  $\rho(\gamma_i)\in H(A)$ for all $i\in I$. 
\end{enumerate}
\end{lem} 
\begin{proof} Clearly (1) implies (2). For the converse, let $\Delta$ be the abstract subgroup
of $\Gamma$ generated by $\{\gamma_i\}_{i\in I}$. Let $\tau: G\hookrightarrow \mathbb A^n$ 
be a closed immersion. Since $\rho$ is $R$-condensed $\rho(\Gamma)$ is contained in 
a finitely generated $R$-submodule $M \subseteq \mathbb A^n(A)$ and the map $\rho: \Gamma\rightarrow M$ 
is continuous. The assumption in (2) implies that $\rho(\Delta)\subseteq H(A)$. 
Since $H$ is closed in $G$ the restriction of $\tau$ to $H$ is a closed immersion. It follows 
from Lemma \ref{polyhasclosedzeroset} that $M\cap H(A)$ is a closed subset of $M$. Its preimage 
in $\Gamma$ will be a closed subset of $\Gamma$ containing $\Delta$. Since $\Delta$ is dense in 
$\Gamma$ we conclude that $\rho(\Gamma)\subseteq M\cap H(A)$, which implies part (1).
\end{proof}

\begin{lem}\label{cond_rep_funct} Let $\varphi: G\rightarrow H$ be a morphism of $\OO$-group schemes of finite type, 
and let $\rho: \Gamma \rightarrow G(A)$ be a representation. Then the following hold: 
\begin{enumerate}
\item if $\rho$ is $R$-condensed then $\varphi\circ \rho: \Gamma \rightarrow H(A)$ is $R$-condensed;
\item if $\varphi\circ\rho$ is $R$-condensed and $\varphi$ is a closed immersion then $\rho$ is $R$-condensed. 
\end{enumerate}
\end{lem} 
\begin{proof} Lemma \ref{B2} implies part (1);  part (2) holds trivially.
\end{proof}

\section{The space of generic matrices}\label{sec_gen_matrix}
Let $G$ be  a smooth affine $\OO$-group scheme  and let $\rhobar:\Gamma\rightarrow G(k)$ be a continuous representation, where $\Gamma$ is a profinite group satisfying Mazur's $p$-finiteness condition.
We will introduce a space of generic matrices\footnote{The terminology is taken from \cite{BIP_new} and is inspired by Procesi's paper \cite{Pro87}.} for $G$, which will depend on $\rhobar$ and a faithful representation $\tau : G \to \GL_d$.

\subsection{\texorpdfstring{Definition of $\XgenGLd$ and $\XgenGtau$}{Definition of XgenGLd and XgenGtau}}\label{def_and_1prop}

We fix\footnote{Its existence follows from \cite[Exp.\,$\mathrm{VII}_B$, Rem.\,11.11.1]{SGA3_new}.} a closed immersion $\tau : G \to \GL_d$ of $\OO$-group schemes.  By composing it with $\rhobar$  we obtain a continuous
representation $\tau\circ \rhobar: \Gamma\rightarrow \GL_d(k)$.
The linearisation $(\tau \circ \rhobar)^{\lin} : k[\Gamma] \to M_d(k)$ extends to the completed group ring $(\tau \circ \rhobar)^{\lin} : k\br{\Gamma} \to M_d(k)$. Let $\Dbar:k\br{\Gamma} \rightarrow k$ be the continuous determinant law associated to $(\tau \circ \rhobar)^{\lin}$. (See \cite{che_durham} and \cite[Section 3]{WE_alg}.)

Let $\RpsGLd$ be the universal deformation ring of $\Dbar$ introduced in \cite[Section 3.1]{che_durham} and let $D^u: \RpsGLd\br{\Gamma}\rightarrow \RpsGLd$ be the extension 
of the universal $d$-dimensional
determinant law lifting $\Dbar$. In particular, $\RpsGLd$
is a complete local noetherian $\OO$-algebra with residue field
$k$, and we equip it with its $\mm$-adic topology. We will write $D^u_{|A}$ for the 
specialisation of $D^u$ along $\RpsGLd\rightarrow A$.
We let $X^{\ps}_{\GL_d} := \Spec(\RpsGLd)$.
Note, that $\RpsGLd$ corresponds to $R^{\ps}$ of \cite{BIP_new} when we take the residual representation to be $\tau \circ \rhobar$.

Let $\CH(D^u)$ be the two-sided ideal of $\RpsGLd\br{\Gamma}$ defined in \cite[Section 1.17]{che_durham}, so that  $E^u := \RpsGLd\br{\Gamma}/ \CH(D^u)$ is the largest quotient of $\RpsGLd\br{\Gamma}$ on which $D^u$ descends to a Cayley--Hamilton determinant law. Since $\Gamma$ satisfies Mazur's $p$-finiteness condition, 
$E^u$ is a finitely generated $\RpsGLd$-module and $\CH(D^u)$ is closed by \cite[Proposition 3.6]{WE_alg}. We write $D^u_{E^u} : E^u \to \RpsGLd$ for the determinant law induced on $E^u$ by $D^u$ and $\pi_{E^u} : \RpsGLd\br{\Gamma} \to E^u$ for the projection.

If $f: E^u \rightarrow M_d(A)$ is a homomorphism of $\RpsGLd$-algebras for a commutative $\RpsGLd$-algebra $A$ then we say $f$ is a
\emph{homomorphism of Cayley--Hamilton algebras} 
if $\det \circ (f \otimes A): E^u \otimes_{\RpsGLd} A \rightarrow A$ is equal to $D^u_{E^u} \otimes A$.

\begin{lem}\label{surj} The composition
$\RpsGLd[\Gamma]\rightarrow
\RpsGLd\br{\Gamma}\twoheadrightarrow E^u$
is surjective. 
\end{lem}
\begin{proof} Let $S$ be the image of $\RpsGLd[\Gamma]$
 in $E^u$. Since $\RpsGLd[\Gamma]$ is a dense subring of $\RpsGLd\br{\Gamma}$, $S$  is also dense in $E^u$. It is enough to show that $S$ is closed in $E^u$.

It follows from \cite[Proposition 3.6]{WE_alg} that $E^u$ is a finitely generated $\RpsGLd$-module. This implies that the 
quotient topology on $E^u$ coincides with the topology 
inherited from $\RpsGLd$ as a finitely generated $\RpsGLd$-module, see the proof of \cite[Lemma 3.2]{BIP_new}.
Every $\RpsGLd$-submodule of $E^u$ is closed with respect 
to this topology. Hence, $S$ is closed in $E^u$, and thus 
$S=E^u$. 
\end{proof}

\begin{lem}\label{factors_through_Eu} Let $R\rightarrow A$ be a map of $\RpsGLd$-algebras such that $R$ is a noetherian profinite topological $\RpsGLd$-algebra. Let  $\rho : \Gamma \to \GL_d(A)$ be a 
representation satisfying $\det \circ \rho^{\lin} = D^u_{|A}$. Then the following are equivalent:
\begin{enumerate} 
\item $\rho^{\lin}:\RpsGLd[\Gamma]\rightarrow M_d(A)$ extends to a homomorphism of $\RpsGLd$-algebras
$\widehat{\rho}:\RpsGLd\br{\Gamma}\rightarrow M_d(A)$;
\item $\rho^{\lin}: \RpsGLd[\Gamma]\rightarrow M_d(A)$ factors through 
a homomorphism of  $\RpsGLd$-algebras $E^u\rightarrow M_d(A)$;
\item $\rho$ is $\RpsGLd$-condensed;
\item $\rho$ is $R$-condensed. 
\end{enumerate}
Moreover, if the equivalent conditions hold then $E^u\rightarrow M_d(A)$ in $(2)$ is 
a homomorphism of Cayley--Hamilton algebras. 
\end{lem}

\begin{proof}
    (1) implies (2). The homomorphism $\widehat\rho \otimes A : \RpsGLd\br{\Gamma} \otimes_{\RpsGLd} A \to M_d(A)$ induces a $d$-dimensional $A$-valued Cayley--Hamilton determinant law 
    $$ \det \circ (\widehat\rho \otimes A) : \RpsGLd\br{\Gamma} \otimes_{\RpsGLd} A \to A$$
    which satisfies $\det \circ (\widehat\rho \otimes A) = D^u_{|A}$ by assumption. 

    It follows, that $D^u_{|A}$ is Cayley--Hamilton, hence $\widehat\rho \otimes A$ factors 
    through a homomorphism $E^u \otimes_{\RpsGLd} A \to M_d(A)$ and so 
    $\rho^{\lin}|_{\RpsGLd[\Gamma]}$ factors through a homomorphism of Cayley--Hamilton 
    algebras $E^u\rightarrow M_d(A)$.

    (2) implies (3). Since $E^u$ is finitely generated as $\RpsGLd$-module by considering the closed embedding 
    $\delta: \GL_d\hookrightarrow M_d\times M_d$, $g\mapsto (g, g^{-1})$ we obtain that $\rho$
    is $\RpsGLd$-condensed.

    (3) implies (4) by Lemma \ref{B4}. 

    (4) implies (1).  Since $\rho$ is $R$-condensed, 
    $\delta(\rho(\Gamma))$ is contained in a finitely generated $R$-submodule 
    of $M_d(A)\times M_d(A)$. By projecting onto the first factor we deduce that the 
    the image of $\rho^{\lin}: \RpsGLd[\Gamma]\rightarrow M_d(A)$ is contained in 
    a finitely generated $R$-module $M$. Since $M$ is profinite the map extends
    to a continuous map of profinite $\RpsGLd$-modules $\widehat{\rho}:\RpsGLd\br{\Gamma}\rightarrow M$. Since $\RpsGLd[\Gamma]$ is a dense subalgebra of $\RpsGLd\br{\Gamma}$ we deduce
    that $\rho^{\lin}$ extends to a map of $\RpsGLd$-algebras $\RpsGLd\br{\Gamma}\rightarrow 
    M_d(A)$.
\end{proof}

\begin{defi}\label{defXgenGLd} 

    Let $\XgenGLd : \RpsGLd\hyphen\alg \to \Set$ be the functor which sends $A$ to the 
    set of  representations $\rho: \Gamma\rightarrow \GL_d(A)$ such that 
    $\det \circ \rho^{\lin}= D^u \otimes_{\RpsGLd} A$ and $\rho$ satisfies the equivalent conditions
    of Lemma \ref{factors_through_Eu}.
    \end{defi}

We give an alternative description of the points of $\XgenGLd$. By the following Lemma, we see that $\XgenGLd$ identifies canonically with the scheme $X^{\gen}$ of \cite{BIP_new}.

\begin{lem}\label{XgenGLdA}
    Let $A$ be a commutative $\RpsGLd$-algebra.
    \begin{enumerate}
        \item $\XgenGLd(A)$ is in natural bijection with the set of homomorphisms of Cayley--Hamilton $\RpsGLd$-algebras $f : E^u \to M_d(A)$. In particular, $\XgenGLd$ is representable by a finitely generated commutative $\RpsGLd$-algebra $A^{\gen}_{\GL_d}$.
        \item If $A$ is a topological $\RpsGLd$-algebra then every representation $\rho \in \XgenGLd(A)$ is continuous.
\end{enumerate}
\end{lem}

\begin{proof}
    Part (1). 
    Recall, that in \cite[Lemma 3.1]{BIP_new} the $\RpsGLd$-algebra $A^{\gen}$ is proved to represent the functor, that maps an $\RpsGLd$-algebra $A$ to the set of homomorphisms of Cayley--Hamilton algebras $f : E^u \to M_d(A)$. Once we prove this functor to be naturally isomorphic to $\XgenGLd$ as in \Cref{defXgenGLd}, the claim about representability follows and $\XgenGLd$ is naturally isomorphic to the $\RpsGLd$-scheme $X^{\gen} = \Spec A^{\gen}$ from \cite{BIP_new}.
    
    If $\rho\in \XgenGLd(A)$ then by Lemma \ref{factors_through_Eu} (2) we obtain a natural 
    homomorphism of Cayley--Hamilton algebras $f_{\rho}: E^u \to M_d(A)$. 
    Conversely, a map of Cayley--Hamilton algebras $f : E^u \to M_d(A)$ gives rise to a representation by means of the composition
    $ \rho_f : \Gamma \to (E^u)^{\times} \overset{f}{\to} \GL_d(A). $
    By restriction along $\RpsGLd\br{\Gamma} \otimes_{\RpsGLd} A \to E^u \otimes_{\RpsGLd} A$, we see that
    $$ \det \circ \rho_f^{\lin} = \det \circ ((f \circ \pi_{E^u}) \otimes A) = D^u_{|A}. $$
    Lemma \ref{factors_through_Eu} (1) implies that $\rho_f\in \XgenGLd(A)$. 
    The maps $\rho\mapsto f_{\rho}$ and $f\mapsto \rho_f$ are mutually inverse.

    Part (2) follows from \cite[Lemma 3.2]{BIP_new} or Lemma \ref{cond_is_cont}.
\end{proof}

\begin{lem}\label{fg_quotient} There is a topologically finitely generated quotient $\Gamma\twoheadrightarrow Q$ 
such that $R^{\ps, \Gamma}_{\GL_d}=R^{\ps,Q}_{\GL_d}$ and every $\rho\in X^{\gen}_{\GL_d}(A)$, 
where $A\in R^{\ps}_G\text{-}\alg$, factors as $\rho: \Gamma \rightarrow Q \rightarrow \GL_d(A)$.
\end{lem}
\begin{proof} Let $Q= \Gamma/K$, where $K\subseteq \ker(\rhobar)$ is the smallest closed 
normal subgroup of $\Gamma$ such that $\ker(\rhobar)/K$ is pro-$p$. Since $\ker(\rhobar)$ 
is open in $\Gamma$ and $\Gamma$ satisfies Mazur's $p$-finiteness condition, $Q$ is topologically 
finitely generated. Lemma 3.8 in \cite{che_durham} implies that $R^{\ps, \Gamma}_{\GL_d}=R^{\ps,Q}_{\GL_d}$
and the Cayley--Hamilton algebra $E^u$ does not change if we replace $\Gamma$ by $Q$. The assertion 
follows from Lemma \ref{XgenGLdA} (1). 
\end{proof}

As already explained above, \cite[Proposition 3.6]{WE_alg} implies that $E^{u}$ is a finitely generated $\RpsGLd$-module, and hence $k\otimes_{\RpsGLd}E^u$ is a 
finite dimensional $k$-vector space.

\begin{lem}\label{closedimmersionforGLd} Let $\{\gamma_1, \ldots, \gamma_N\}$ be a subset  of $\Gamma$, such that its image via the map in Lemma \ref{surj} spans $k\otimes_{\RpsGLd}E^u$ as a $k$-vector space. 
Then the  map 
$$ \XgenGLd \rightarrow \GL_d^N\times \XpsGLd, \quad \rho\mapsto (\rho(\gamma_1), \ldots, \rho(\gamma_N))$$
is a closed immersion.
\end{lem}
\begin{proof} It follows from Part (1) of Lemma \ref{XgenGLdA} that $\XgenGLd$ coincides with the scheme $X^{\gen}$, that
was contructed in \cite[Lemma 3.1]{BIP_new} as a closed subscheme of $M_d^{n}\times \XpsGLd$ by choosing generators $\delta_1, \ldots, \delta_n$ of $E^u$ as an $\RpsGLd$-module and mapping $\rho$ 
to the $n$-tuple $(\rho(\delta_1), \ldots, \rho(\delta_n))$. 
Since $\rho(\gamma_i)$ are invertible matrices, it is enough to show that 
the images of $\gamma_1, \ldots, \gamma_N$ generate $E^u$ as
an $\RpsGLd$-module. This follows from Nakayama's lemma. 
\end{proof}

Let $Q$ be a quotient of $\Gamma$ as in Lemma \ref{fg_quotient}. We fix an $N$-tuple $(\gamma_1, \ldots, \gamma_N)\in \Gamma^N$, such that the entries 
generate a dense subgroup of $Q$ and its image spans $E^u\otimes_{\RpsGLd} k$ as a $k$-vector space.

\begin{defi}\label{defiXgenG}
We define the $\RpsGLd$-scheme $\XgenGtau$ as the fibre product 
\begin{equation}\label{def_XgenG}
\XgenGtau := \XgenGLd\times_{\GL_d^N} G^N,
\end{equation}
where $G^N \rightarrow \GL_d^N$ is the map $(g_1, \ldots, g_N)\mapsto (\tau(g_1),\ldots, \tau(g_N))$.    
\end{defi}

Thus $\XgenGtau$ is a closed $\RpsGLd$-subscheme of $\XgenGLd$, and hence a 
closed subfunctor.
We note that $\XgenGtau$ is non-empty, as $\XgenGtau(k)$ contains
the representation $\rhobar$ that we have started with.
\begin{remar}\label{FN} By construction $\XgenGLd$ is a closed subscheme of $\GL_d^N \times \XpsGLd$. Thus $\XgenGtau$ is a closed subscheme of $G^N\times \XpsGLd$, which can interpreted as follows. 

Let $\mathcal F_N$ be a free group generated by elements $x_1, \ldots, x_N$. 
Let $\Rep_G^{\mathcal F_N, \square}$ be the affine $\RpsGLd$-scheme, which represents the functor, that maps a commutative $\RpsGLd$-algebra 
$A$ to the set of (abstract) group homomorphisms $\rho : \mathcal F_N \to G(A)$. Sending $\rho$ to the $N$-tuple
$(\rho(x_1), \ldots, \rho(x_N))$ induces an isomorphism of $\XpsGLd$-schemes 
$$\Rep_G^{\mathcal F_N, \square} \cong G^N\times \XpsGLd.$$ Mapping $x_i$ to $\gamma_i$ induces 
a group homomorphism $\varphi:\mathcal F_N\rightarrow Q$ and it follows from above, that mapping $\rho$ to $\rho\circ \varphi$ induces a closed immersion $\XgenGtau \rightarrow \Rep_G^{\mathcal F_N, \square}$.
\end{remar}

\begin{remar}\label{rem_extend_scalars} If $L'$ is a finite extension of $L$ with ring of integers $\OO'$ then it follows 
from the definition that $(\XgenGtau)_{\OO'}\cong X^{\gen, \tau}_{G_{\OO'}}$. 
The main technical result of this paper is the computation of the dimension of $\XgenGtau$. 
Since $\OO'$ is finite and flat over $\OO$ the dimensions of $(\XgenGtau)_{\OO'}$ and $\XgenGtau$
are the same and using Proposition \ref{O_prime} we may assume that $G$ is split, $G/G^0$ is constant 
and $G(\OO)\rightarrow (G/G^0)(\OO)$ is surjective. 
\end{remar}

We are now ready to give a description of the points of $\XgenGtau$.

\begin{lem}\label{XgenGA} Let $A$ be a commutative $\RpsGLd$-algebra.
\begin{enumerate}
    \item $\XgenGtau(A)$ is the set of representations $\rho: \Gamma \rightarrow G(A)$, such that 
    $\tau\circ \rho \in \XgenGLd(A)$.
    \item If $A$ is a topological $\RpsGLd$-algebra then every representation $\rho \in \XgenGtau(A)$ is continuous.
\end{enumerate}
\end{lem}

\begin{proof} Part (1) follows from Lemmas \ref{gener_cond} and \ref{fg_quotient}.

    For part (2), we observe that part (1) and Lemma \ref{XgenGLdA} (2) imply that
    that the composition $\tau \circ \rho : \Gamma \to G(A) \to \GL_d(A)$ is continuous.
    Since $\tau$ is a closed immersion the topology on $G(A)$ is the subspace topology
    induced by $\GL_d(A)$ and thus $\rho$ is continuous.
  \end{proof}

\begin{remar} It follows from Lemma \ref{XgenGA}
that the definition of $\XgenGtau$ does not depend
on the choice of the quotient $Q$ and $\gamma_1, \ldots, \gamma_N$. However, it will depend a priori on the choice of the faithful representation $\tau$. It will follow from Corollaries \ref{connected_comp} and \ref{indep_tau}
that the connected component of $X^{\gen, \tau}_G$ containing the point corresponding to the representation $\rhobar: \Gamma \rightarrow G(k)$ is 
independent of $\tau$. 
\end{remar}

If $\pp$ is the maximal ideal of $\RpsGLd$ then $\kappa(\pp)=k$ and we equip the algebraic closure 
$\overline{\kappa(\pp)}$ with the discrete topology. If $\pp$ is a prime of   $\RpsGLd$ with $\dim \RpsGLd/\pp =1$ then 
$\kappa(\pp)$ is a local field and the valuation topology on $\kappa(\pp)$ makes it 
into a topological $\RpsGLd$-algebra by \cite[Lemma 3.17]{BIP_new}. We equip
the algebraic closure $\overline{\kappa(\pp)}$ with the topology induced by a valuation, which 
extends the valuation on $\kappa(\pp)$. 

\begin{lem}\label{geom_pts} Let $\pp$ be a 
prime of $\RpsGLd$ such that $\dim \RpsGLd/\pp\le 1$.
The set $\XgenGtau(\overline{\kappa(\pp)})$ is naturally in bijection with the set of 
continuous representations $\rho: \Gamma \rightarrow G(\overline{\kappa(\pp)})$, 
with the topology on $\overline{\kappa(\pp)}$ as above,
such that the 
$d$-dimensional determinant law
$\det \circ (\tau\circ \rho)^{\lin}$ is equal to the specialisation of $D^u$ along 
$\RpsGLd\rightarrow \overline{\kappa(\pp)}$. 
\end{lem}
\begin{proof} \Cref{cont_vs_cond} implies that a representation 
$\rho: \Gamma\rightarrow G(\overline{\kappa(\pp)})$ is continuous if and only if it is $\RpsGLd$-condensed. The assertion then follows from  \Cref{XgenGA}.
\end{proof}

Let $\AgenGLd\twoheadrightarrow \AgenGtau$ be the quotient representing $\XgenGtau$. We will denote by $\XbargenGtau$ the 
special fibre of $\XgenGtau$ over $\OO$.

\begin{lem}\label{dense_again} Let $y: \RpsGLd \rightarrow \Omega$ be a ring  homomorphism 
into an algebraically closed field $\Omega$ and let 
$x\in \XgenGtau(\Omega)$ be a point above $y$.
Let $H$ be the smallest Zariski closed subgroup of $G_{\Omega}$ 
containing $\rho_{x}(\gamma_1), \ldots, \rho_x(\gamma_N)$, where $\rho_x$ is the specialisation at $x$
of the universal representation $\rho: \Gamma \rightarrow \GL_d(A^{\gen}_{\GL_d})$. 
Then $H$ is 
equal to the Zariski closure of $\rho_x(\Gamma)$ in $G_{\Omega}$.
\end{lem}

\begin{proof} Lemmas \ref{gener_cond} and \ref{fg_quotient} imply that 
$\rho_x(\Gamma)=\rho_x(Q)\subseteq H(\Omega)$. Hence, its Zariski closure is contained 
in $H$. The other inclusion follows from the definition of $H$. 
\end{proof}

\subsection{Completion at maximal ideals and deformation problems}\label{def_probs}
 Let $Y\subset X^{\gen,\tau}_G$ be the preimage of the closed point of $X^{\ps}_{\GL_d}$
 and let 
$x$ be either a closed point of $Y$ or a closed point of $X^{\gen,\tau}_G\setminus Y$ and let $y$ be its image in $\Spec R^{\ps}_{\GL_d}$.
 It follows from \cite[Lemmas 3.17, 3.18]{BIP_new} that 
 $\kappa(x)$ is a finite extension of $\kappa(y)$ and 
 there are the following possibilities:
 
 \begin{enumerate}
 \item if $x\in Y$ then $\kappa(x)$ is a finite extension of $k$;
 \item if $x\in X^{\gen,\tau}_G[1/p]$ then $\kappa(x)$ is a finite extension of $L$;
 \item if $x\in \Xbar^{\gen,\tau}_G\setminus Y$ then $\kappa(x)$ is a local field of characteristic $p$.
 \end{enumerate}
The universal property of $A^{\gen,\tau}_G$ gives 
us a continuous Galois representation $$\rho_x: \Gamma \rightarrow G(\kappa(x))$$
for the natural topology on the target. In this section we 
want to relate the completion of the local ring $\OO_{X^{\gen,\tau}_G, x}$ 
to a deformation problem for $\rho_x$.

 \begin{lem}\label{deform} 
 Let $\kappa$ be either a finite or a local $\OO$-field
 and let $\rho: \Gamma\rightarrow G(\kappa)$ be a continuous representation. 
 Let $\Lambda$ be a coefficient ring of $\kappa$
and let $(A,\mm_A)\in \mathfrak A_{\Lambda}$. If $\gamma_1, \ldots, \gamma_N$ topologically generate $\Gamma$ then the map 
$$\rho_A\mapsto (\tau\circ \rho_A, (\rho(\gamma_1), \ldots, \rho(\gamma_N)))$$
induces a bijection of sets
$D^{\square}_{\rho,G}(A)\longrightarrow D^{\square}_{\tau\circ\rho,\GL_d}(A)\times_{\GL_d^N(A)}
G^N(A).$
\end{lem}
\begin{proof} Since $\tau: G(A)\rightarrow \GL_d(A)$ 
is injective, the map is an injection. A point 
in $D^{\square}_{\tau\circ\rho,\GL_d}(A)\times_{\GL_d^N(A)}G^N(A)$ corresponds to continuous representation 
$\rho_A: \Gamma \rightarrow \GL_d(A)$ such that 
$\rho_A \equiv \tau\circ \rho\pmod{\mm_A}$ and 
$(g_1, \ldots, g_N)\in G^N(A)$ such that 
$\rho(\gamma_i)=\tau(g_i)$ for $1\le i \le N$. 
Since $\tau(G(A))$ is a closed subgroup of $\GL_d(A)$, its preimage 
$\rho_A^{-1}( \tau(G(A)))$ is a closed subgroup of $\Gamma$. 
Since
it contains the topological generators $\gamma_1, \ldots, \gamma_N$ 
we deduce that the image of $\rho_A$ is contained  in $\tau(G(A))$. This implies that the map is surjective.
\end{proof}

\begin{lem}\label{completion} Let $(A, \mm)$ be a local (possibly non-noetherian) ring and let $I$ be a finitely generated ideal of $A$. If the $\mm$-adic completion $\hat{A}$ of $A$ is noetherian then $\widehat{(A/I)}= \hat{A}/ I\hat{A}$.
\end{lem}
\begin{proof} It follows from \cite[\href{https://stacks.math.columbia.edu/tag/0BNG}{Tag 0BNG}]{stacks-project} that $A/\mm^n$ is a quotient of $\hat{A}$
for all $n\ge 1$. Since $\hat{A}$ is noetherian, $A/\mm^n$ is noetherian, and hence artinian for all $n\ge 1$. By choosing 
generators of $I=(a_1, \ldots, a_s)$ we obtain a surjection 
$A^s\twoheadrightarrow I$. Let $K_n= \Ker( (A/\mm^n)^s \twoheadrightarrow (I+\mm^n)/\mm^n)$. The modules  $K_n$ are artinian, 
since they are finitely generated over the artinian ring $A/\mm^n$, and thus form a Mittag-Leffler system, which implies that the 
map $\hat{A}^s= \varprojlim_n (A/\mm^n)^s\rightarrow \varprojlim_n (I +\mm^n)/\mm^n$ is surjective.  

The map $\hat{A}\rightarrow \widehat{(A/I)}$ is surjective by \cite[\href{https://stacks.math.columbia.edu/tag/0315}{Tag 0315}]{stacks-project} and the kernel is equal to $\varprojlim_n (I +\mm^n)/\mm^n$, \cite[\href{https://stacks.math.columbia.edu/tag/02N1}{Tag 02N1}]{stacks-project}. It follows from above that this projective limit is 
equal to $I \hat{A}$, which implies the assertion.
\end{proof}

\begin{prop}\label{333} 
Let $x\in X^{\gen, \tau}_G(\kappa)$, where  $\kappa$ is either a finite or a local $\OO$-field.
 Let 
$\qq$ be the kernel of the 
map $$\Lambda\otimes_{\OO} A^{\gen,\tau}_G \rightarrow \kappa, \quad \lambda \otimes a \mapsto \bar{\lambda}\bar{a},$$
where $\bar{\lambda}$ and $\bar{a}$ denote the images
of $\lambda$ and $a$ in $\kappa$ and $\Lambda$ is a coefficient ring for $\kappa$. 
The completion of $\Lambda\otimes_{\OO} A^{\gen,\tau}_G$ with respect to $\qq$ is naturally isomorphic to $R^{\square}_{\rho_x,G}$.
\end{prop}
\begin{proof} If $\rho_A\in D^{\square}_{\rho, G}(A)$ then \cite[Proposition 3.33]{BIP_new} implies that $A$ is naturally an 
$R^{\ps}_{\GL_d}$-algebra and $\tau\circ \rho\in X^{\gen}_{\GL_d}(A)$.
Lemma \ref{fg_quotient} implies that there is a topologically finitely generated quotient $Q$ of $\Gamma$ such that any deformation of $\rho$ factor through $Q$.  Hence, we may 
assume without loss of generality that $\gamma_1, \ldots, \gamma_N$ are topological generators of $\Gamma$.

Let $I$ be the kernel of the surjection  $\OO(\GL_d^N)\twoheadrightarrow \OO(G^N)$. We note that $I$ is a 
finitely generated ideal. Let $\mm_G$ and $\mm_{\GL_d}$ be the maximal ideals of $R^{\square}_{\rho_x,G}$ and $R^{\square}_{\tau\circ \rho_x, \GL_d}$,  respectively. The rings $R^{\square}_{\rho_x,G}/\mm_G^n$ and 
$R^{\square}_{\tau\circ \rho_x, \GL_d}/\mm_{\GL_d}^n$ represent 
the restrictions of functors $D^{\square}_{\rho_x, G}$ and $D^{\square}_{\tau\circ \rho_x, \GL_d}$ to $\mathfrak A_{\Lambda, n}$, 
respectively. It follows from Lemma \ref{deform} that 
$$ R^{\square}_{\rho_x,G}/\mm_G^n\cong (R^{\square}_{\tau\circ \rho_x, \GL_d}/\mm_{\GL_d}^n) \otimes_{\OO(\GL_d^N)} \OO(G^N).$$
By passing to the limit and using Lemma \ref{completion} with $A=R^{\square}_{\tau\circ \rho_x, \GL_d}$  we obtain 
$$R^{\square}_{\rho_x, G}\cong 
R^{\square}_{\tau\circ\rho_x, \GL_d}\otimes_{\OO(\GL_d^N)} \OO(G^N).$$
If $G=\GL_d$ and $\tau=\id$ then the proposition is \cite[Proposition 3.33]{BIP_new}, so that $R^{\square}_{\tau\circ\rho_x, \GL_d}$ is naturally isomorphic to  the 
completion of the local ring of $\Lambda\otimes_{\OO} A^{\gen}_{\GL_d}$ at $\qq$. 
Since $A^{\gen,\tau}_G=A^{\gen}_{\GL_d} \otimes_{\OO(\GL_d^N)} \OO(G^N)= 
A^{\gen}_{\GL_d}/ I A^{\gen}_{\GL_d}$ by definition, and $R^{\square}_{\tau\circ\rho_x, \GL_d}$ is noetherian by Proposition \ref{represent_kappa}, the assertion follows from Lemma \ref{completion} applied 
to $A=(\Lambda\otimes_{\OO} A^{\gen}_{\GL_d})_{\qq}$. 
\end{proof}

\begin{lem}\label{local_ring_def_ring} Let $X:=X^{\gen,\tau}_G$ and let $x$ be either a closed point of 
$Y$ or a closed point of $X\setminus Y$. Let $\OO_{X,x}$ be the local ring at $x$ and let
$\hat{\OO}_{X,x}$ be its completion  with respect to 
the maximal ideal. Then we have the following isomorphisms of local rings: 
\begin{enumerate} 
\item if $x\in X[1/p]$ or $x\in Y$ then $\hat{\OO}_{X,x}\cong R^{\square}_{G, \rho_x}$;
\item if $x\in \Xbar \setminus Y$ then $\hat{\OO}_{X,x}\br{T}\cong R^{\square}_{G, \rho_x}$. 
\end{enumerate} 
In particular, $\OO_{X,x}$ is complete intersection if and only if $R^{\square}_{G, \rho_x}$ is. 
\end{lem}
\begin{proof} The assertion follows from Proposition \ref{333} and \cite[Lemmas 3.36, 3.37]{BIP_new}.
\end{proof}

\section{GIT quotients}\label{sec_GIT}

Let $G$ be a generalised reductive $\OO$-group scheme. If $A$ is an $\RpsGLd$-algebra then the group
$G(A)$ acts on the set $\XgenGtau(A)$ by conjugation.  This defines an action of $G$ on $\XgenGtau$. We let 
\begin{equation}
\XgpsGtau := \XgenGtau \sslash G^0, \quad \XbargpsGtau:= \XbargenGtau\sslash G^0. 
\end{equation}

\begin{remar}\label{homeo} The closed immersion $\XbargenGtau \hookrightarrow \XgenGtau$ induces a morphism $\XbargpsGtau \rightarrow (\XgpsGtau)_k$, 
which is finite by \cite[Theorem 2]{seshadri} and induces a bijection on geometric points by 
\cite[Theorem 3]{seshadri}. 
We deduce that this map is a homeomorphism.
\end{remar}

We will show that $\XgpsGtau$ is the spectrum of a complete semi-local $\OO$-algebra. 
We use its connected components to decompose  $\XgenGtau$ and show in Lemma \ref{connected_comp} that the resulting decomposition is the decomposition of $\XgenGtau$ into connected components. The component
containing the point corresponding to a representation $\rhobar: \Gamma_F \rightarrow G(k)$ 
that we have started with 
in \Cref{def_and_1prop}
will be denoted $X^{\gen, \tau}_{G, \rhobarss}$. We then evaluate both functors at finite and
local fields and study the corresponding Galois representations.

In Corollary \ref{finite_HG} we prove an important finiteness result, which is used 
in an essential way in the induction argument in \Cref{sec_complete_intersection}. This in turn rests on the following 
fundamental finiteness theorem proved by Vinberg \cite{vinberg}, when $S$ is a spectrum of a field of characteristic zero, 
by Martin \cite{martin}, when $S$ is a spectrum of a field of characteristic $p$, and by Cotner \cite{cotner}
in general.

\begin{thm}[\cite{cotner}]\label{cotner_main} Let $S$ be a locally noetherian scheme and let $f: H\rightarrow G$ be a finite morphism 
of generalised reductive smooth affine $S$-group schemes. Then the induced map 
$$ H^N\sslash H \rightarrow G^N \sslash G$$
is finite.
\end{thm}

\begin{cor}\label{cor_cotner} Let us assume the setting of Theorem \ref{cotner_main}. Then the induced map
$$ H^N\sslash H^0 \rightarrow G^N \sslash G^0$$
is also finite.
\end{cor}
\begin{proof} Since $G^N\sslash G^0 \rightarrow S$ and $H^N \sslash H^0 \rightarrow S$ are affine and $S$ is locally noetherian, 
we may assume that $S$ is affine and noetherian. Then $G$, $H$, $G^N$, $H^N$ as well 
as their respective GIT quotients are all 
affine schemes. Since $H$ is generalised reductive $H/ H^0$ is a finite \'etale $S$-group scheme. 
Since $$H^N\sslash H = (H^N \sslash H^0)\sslash (H/H^0),$$  
 \cite[Exp.\,V, Thm.\,4.1\,(ii)]{SGA3_new} applied with $X_0=H^N \sslash H^0$ and $X_1= (H/H^0)\times_S X_0$,
implies that the morphism $H^N \sslash H^0\rightarrow H^N\sslash H$ is integral. 
It follows from \cite[Theorem 6.3.3]{alper} applied to the quotient stack 
$[H^N/H^0]$ that $H^N\sslash H^0$ is of finite type over $S$. (If $S=\Spec R$, where
$R$ is of finite type over a universally Japanese integral domain then the assertion 
also follows from \cite[Theorem 2 (i)]{seshadri}. This setting is sufficient for 
our paper as we only need it over $\OO$.) Hence, $H^N\sslash H^0\rightarrow H^N\sslash H$ is of finite type, and \cite[\href{https://stacks.math.columbia.edu/tag/01WJ}{Tag 01WJ}]{stacks-project}
implies that $H^N \sslash H^0\rightarrow H^N\sslash H$ is finite.

Since the composition of finite morphisms is finite, using 
Theorem \ref{cotner_main} we obtain that  $H^N \sslash H^0\rightarrow G^N\sslash G$ is finite. 
Since this map factors as $H^N \sslash H^0\rightarrow G^N\sslash G^0\rightarrow G^N\sslash G$
and all the schemes are affine we obtain the claim. 
\end{proof}

\begin{prop}\label{GGLd_finite} The maps $\XgpsGtau \to \XpsGLd$ and $\XbargpsGtau\rightarrow \XbarpsGLd$ are finite.
\end{prop}

\begin{proof}
    We consider the following diagram:
    \begin{center}
        \begin{tikzcd}
            \XgpsGtau \ar[r] \ar[d] & (\XpsGLd \times G^N) \sslash G^0 \ar[d] \\
            \XpsGLd \ar[r] & (\XpsGLd \times \GL_d^N) \sslash \GL_d
        \end{tikzcd}
    \end{center}
    By \Cref{cor_cotner} the right map is finite.
    Since the map $\XgenGtau \to \XpsGLd \times G^N$ defined in \Cref{defiXgenG} is a closed immersion, it follows from \cite[Theorem 2 (ii)]{seshadri}, that the induced map of GIT quotients, which is the top map in the diagram, is finite. It follows, that the left map in the diagram is finite. The same argument over $k$ proves the last assertion. 
\end{proof}

 \begin{cor}\label{finite_HG} Let $H$ be a closed generalised reductive subgroup scheme of $G$ (resp. $G_k$). Then the 
 map $\XgpsHtau \rightarrow \XgpsGtau$ (resp. $\XbargpsHtau \rightarrow \XbargpsGtau$) is finite. 
\end{cor}
\begin{proof} It follows from Proposition \ref{GGLd_finite} that the map
$\XgpsHtau \rightarrow \XpsGLd$ is finite. Moreover, it factors as 
$\XgpsHtau \rightarrow\XgpsGtau \rightarrow \XpsGLd$.
Since these are affine schemes we deduce that the map $\XgpsHtau \rightarrow \XgpsGtau$ 
is finite. The same argument over $k$ gives finiteness of $\XbargpsHtau \rightarrow \XbargpsGtau$.
\end{proof}

\begin{prop}[{\cite[Lemma 4.3]{cotner}}]\label{building} Let $\kappa$ be a local $\OO$-field and let $\rho: \Gamma_F \rightarrow 
G(\kappa)$ be a continuous representation for the topology on the target induced
by the valuation on $\kappa$. Then there exists a finite extension $\kappa'$ of 
$\kappa$ and $g\in G^0(\kappa')$ such that $g \rho(\gamma) g^{-1} \in G(\OO_{\kappa'})$ for all $\gamma\in \Gamma_F$.
\end{prop}

\subsection{\texorpdfstring{Points in $\XgenGtau$}{Points in XgenG}}\label{sec_points_in_XgenGtau}
We recall some terminology following \cite{martin}. Let $\kappa$ be an algebraically 
closed field, which is an $\OO$-algebra. Let $H$ be a closed reduced subgroup of $G_{\kappa}$. 
We say that $H$ is \emph{$G$-irreducible} in $G_{\kappa}$ if it is not contained in any proper
R-parabolic subgroup of $G_\kappa$. Let $S$ be a maximal torus in the centraliser $Z_G(H)$. 
We say that $H$ is \emph{strongly reductive} in $G$ if $H$ is $Z_G(S)$-irreducible in $Z_G(S)$.

\begin{prop}\label{closed_orbit} Let $x$ be a geometric point of $\XgenGtau$, 
and let $H$ be the Zariski closure of $\rho_x(\Gamma_F)$ in $G(\kappa(x))$. Then the following are equivalent
\begin{enumerate}
\item the orbit $G^0\cdot x$ is closed in its fibre over $X^{\git, \tau}_G$;
\item the orbit $G\cdot x$ is closed in its fibre over $X^{\git, \tau}_G$; 
\item $H$ is strongly reductive in $G_{\kappa(x)}$;
\item $\rho_x: \Gamma_F \rightarrow G(\kappa(x))$ is $G$-completely reducible.
\end{enumerate} 
\end{prop}
\begin{proof} Since $G^0$ is of finite index in $G$ we may pick $g_1, \ldots, g_n \in G(\kappa(x))$ such that 
\begin{equation}\label{orbits}
G\cdot x = \bigcup_{i=1}^n G^0\cdot (g_i x)= \bigcup_{i=1}^n g_i\cdot (G^0\cdot x).
\end{equation}
It follows from \eqref{orbits} that (1) implies (2).
 If $G\cdot x$ is closed, then as it is $G^0$-invariant
it will contain a closed $G^0$-orbit. We deduce from \eqref{orbits} that 
$G^0\cdot x$ is closed, so that (2) implies (1).
We have established in Lemma \ref{dense_again} that $H$ is equal to the 
Zariski closure in $G(\kappa(x))$ of the subgroup generated by the $N$-tuple
$(\rho_x(\gamma_1), \ldots, \rho_x(\gamma_N))$. The equivalence 
of parts (2), (3) has been established in \cite[Theorem 16.4]{Richardson1988ConjugacyCO}.
The equivalence of (3) and (4)  follows from \cite[Theorem 3.1]{BMR}, see the begining of
 \cite[Section 6.3]{BMR} for the non-connected case. 
\end{proof} 

\begin{lem}\label{closed_orbit_fibre} Let $y$ be a geometric point in $\XgpsGtau$. Then the fibre 
$(\XgenGtau)_y$ contains a unique closed $G^0$-orbit. 
\end{lem}
\begin{proof} This follows from \cite[Theorem 3]{seshadri}.
\end{proof}

\begin{prop}\label{orbit_ss} Let $x$ be a geometric point of $\XgenGtau$. Then there 
is $z\in \overline{G^0\cdot x}$ such that the $G^0$-orbit of $\rho_z$ is equal to 
the $G$-semisimplification of $\rho_x$. Moreover, $G^0\cdot z$ is the unique 
closed $G^0$-orbit contained in $\overline{G^0\cdot x}$. Here the closures are taken in the fibre of $x$ over $X^{\git, \tau}_G$.
 \end{prop}
 
\begin{proof}
    Let $z$ be a point in the closure of the $G^0$-orbit of $x$ in $G^N(\kappa(x))$, such that the $G^0$-orbit of $z$ is closed.
    Since $\XgenGtau$ is closed in $X^{\ps}_{\GL_d}\times G^N$ by \Cref{FN}, we deduce that $z\in \XgenGtau(\kappa(x))$. 
    Lemma \ref{dense_again} implies that the $G^0$-orbit of $\rho_{z}$ is the $G$-semisimplification of $\rho_x$.
    Thus $\rho_{z}$ is $G$-completely reducible  by Proposition \ref{crss}. 
    Thus the orbit $G\cdot z$ is closed by Proposition \ref{closed_orbit}.
    Let $y$ be the image of $x$ in $\XgpsGtau$. The fibre $(\XgenGtau)_y$ is a closed $G^0$-invariant subscheme $\XgenGtau$ and thus it will contain $\overline{G^0\cdot x}$. The last assertion follows from Lemma \ref{closed_orbit_fibre}.
\end{proof}

\begin{cor}\label{same_image} Let $y$ be a geometric point of $\XpsGLd$ and let 
$x, x'\in \XgenGtau(\kappa(y))$. Then $x$ and $x'$ map to the same point in 
$\XgpsGtau$ if and only if $\rho_x$ and $\rho_{x'}$ have the same $G$-semisimplification. 
\end{cor}

\begin{proof}
    Let $x$ and $x'$ map to the same point in $\XgpsGtau$.
    By \Cref{closed_orbit_fibre} there is a unique closed $G^0$-orbit in the fibre of $y$, let $z \in \XgpsGtau$ be a point in this orbit.
    It follows from Proposition \ref{orbit_ss} that the $G$-semisimplification of $\rho_x$ and of $\rho_{x'}$ are equal to the $G^0$-orbit of $\rho_z$. The other implication follows from directly from Proposition \ref{orbit_ss}.
\end{proof}

\begin{cor}\label{conjugate_Levi} Let $y$ be a geometric point of $\XgpsGtau$. Assume that 
$G$ is split and let $T$ be  a maximal split torus of $G$ 
defined over $\OO$. Then there is an R-Levi $L$ of $G$ defined over $\OO$, which contains $T$, 
and $x\in \XgenGtau(\kappa(y))$ such that the following hold:
\begin{enumerate}
\item $x$ maps to $y$;
\item $\rho_x(\Gamma_F)\subseteq L(\kappa(y))$;
\item $\rho_x$ is $L_{\kappa(y)}$-irreducible.
\end{enumerate}
Moreover, if $x'\in \XgenGtau(\kappa(y))$ maps to $y$ and $Q$ is a minimal R-parabolic of $G_{\kappa(y)}$ 
containing $\rho_{x'}(\Gamma_F)$ with an R-Levi $M$ then there is $g\in G^0(\kappa(y))$ such that $M$ and $L_{\kappa(y)}= g M g^{-1}$ and $\rho_x(\gamma)= g c_{Q,M}(\rho_{x'}(\gamma))g^{-1}$ for all $\gamma\in \Gamma_F$.
\end{cor}
\begin{proof} Lemma \ref{closed_orbit_fibre} implies that there exists $x\in \XgenGtau(\kappa(y))$ such that $G^0 \cdot x$ is a closed. The representation $\rho_x$ is $G$-completely reducible by Proposition \ref{closed_orbit}. Let $P$ be a minimal R-parabolic of $G_{\kappa(x)}$ containing $\rho_x(\Gamma_F)$. 
Since $\rho_x$ is $G$-completely reducible the image of $\rho_x(\Gamma_F)$ is contained in an R-Levi 
$L$ of $P$. Lemma \ref{conj_P} implies that after replacing $L$ and $P$ by a conjugate by $g\in G^0(\kappa(y))$
we may assume that $L$ and $P$ are defined over $\OO$ and $L$ contains $T$. 

Let $Q$ be a minimal  R-parabolic containing $\rho_x$ and let $M$ be an R-Levi of $Q$. Since
$x$ and $x'$ both map to $y$, Corollary 
\ref{same_image} implies that $\rho_x$ and $\rho_{x'}$ have the same $G$-semisimplification. It follows from Proposition \ref{crss} that $\rho_x$ is conjugate to 
$\rho':=c_{Q, M}\circ \rho_{x'}$ by  $g\in G^0(\kappa(y))$. After replacing $(Q,M)$ by a conjugate, 
we may assume that $\rho_x=\rho'$, so that $\rho_x(\Gamma_F) \subseteq M$, and $Q$ is minimal with respect to the property containing $\rho_x(\Gamma_F)$. Let $H$ be the Zariski closure of $\rho_x(\Gamma_F)$ in $G(\kappa(y))$. 
Minimality of $P$ and $Q$ implies that $H$ is both $L$-irreducible and $M$-irreducible.
Proposition \ref{irr_conj} implies that $L$ and $M$ are conjugate by an 
element $h\in Z_{G}(H)^0(\kappa(y))\subseteq G^0(\kappa(y))$. Since the image of $\rho'$ is contained in $H$, 
we have $h \rho'(\gamma) h^{-1}= \rho'(\gamma) = \rho(\gamma)$ for all $\gamma\in \Gamma_F$.
\end{proof}

\subsection{\texorpdfstring{Definition of $X^{\gen, \tau}_{G, \rhobarss}$}{Definition of XgentauGrhobarss}}
\label{sec_def_XgenGtaurhobarss}

\Cref{GGLd_finite} implies that the map $\RpsGLd\rightarrow (\AgenGtau)^{G^0}$ is finite. Since $\RpsGLd$ is a complete local ring with residue 
field $k$, part (1) of  \cite[\href{https://stacks.math.columbia.edu/tag/04GH}{Tag 04GH}]{stacks-project}  implies that 
\begin{equation}\label{semi-local}
R^{\gps,\tau}_G:=(A^{\gen, \tau}_G)^{G^0}\cong R_1\times \ldots \times R_n,
\end{equation}
where each $R_j$ is a complete local noetherian $\RpsGLd$-algebra with residue field a finite extension of 
$k$. The product can be indexed by the set of $\Gal(\kbar/k)$-orbits in $X^{\gps,\tau}_G(\kbar)$. 
From  \eqref{semi-local} we obtain a $G^0$-equivariant decomposition 
\begin{equation}\label{disjoint-union}
\XgenGtau\cong \coprod_j \XgenGtau\times_{\XgpsGtau} \Spec R_j.
\end{equation}
Moreover, we have  
\begin{equation}\label{RjGIT}
(\XgenGtau\times_{\XgpsGtau} \Spec R_j)\sslash G^0 \cong \Spec R_j.
\end{equation}

\begin{lem}\label{connected_comp} Each $\XgenGtau\times_{\XgpsGtau} \Spec R_j$ is connected. In particular, 
\eqref{disjoint-union} is the decomposition of $\XgenGtau$ into connected components. 
\end{lem}
\begin{proof} Since a connected component is a union of irreducible 
components, \cite[Lemma 2.1]{BIP_new} implies that the connected components of $\XgenGtau\times_{\XgpsGtau} \Spec R_j$ are $G^0$-invariant. It follows from \cite[Theorem 3 (iii)]{seshadri} 
that if there were more than one then their images in $\Spec R_j$ under \eqref{RjGIT}
would disconnect it.  This is not possible as $R_j$ is  a local ring.
\end{proof}

The representation $\rhobar: \Gamma_F\rightarrow G(k)$ that we have fixed in 
\Cref{sec_gen_matrix} gives us a point $x_0\in X^{\gen, \tau}_G(k)$. We denote the unique
connected component of $X^{\gen,\tau}_G$ which contains $x_0$ by $X^{\gen, \tau}_{G, \rhobarss}$, 
where $\rhobarss$ is the $G$-semisimplification of the representation
$\Gamma_F \overset{\rhobar}{\longrightarrow} G(k) \hookrightarrow G(\kbar)$. The justification for this 
notation is given in the lemma below. Let 
$$ R^{\gps, \tau}_{G, \rhobarss}:= R_{j_0}, \quad X^{\gps, \tau}_{G, \rhobarss}:= \Spec R^{\gps, \tau}_{G, \rhobarss},$$
where $j_0$ is the index corresponding to the image of 
$x_0$ in $X^{\gps, \tau}_G(k)\subseteq X^{\gps, \tau}_G(\kbar)$.  Since 
$x_0$ is $k$-rational, the residue field of $R^{\gps, \tau}_{G, \rhobarss}$ is equal to $k$. 

\begin{lem}\label{get_comp} Let $x$ be a geometric point of $\XgenGtau$ above the closed point 
of $\XpsGLd$. Then $x\in X^{\gen, \tau}_{G, \rhobarss}$ if and only if 
the $G$-semisimplification of $\rho_x$ is equal to $\rhobarss$.
\end{lem}
\begin{proof} Let $z\in X^{\gen, \tau}_{G, \rhobarss}(\kappa(x))$ be a point above $x_0$. 
It follows from Lemma \ref{connected_comp} that $x$ and $z$ lie on the same connected 
component of $X^{\gen, \tau}_G$ if and only if their images in $X^{\gps, \tau}_G$ lie on the 
same connected component of $X^{\gps, \tau}_G$. Since $z\in X^{\gen, \tau}_{G, \rhobarss}$ 
and the residue field of $R^{\gps, \tau}_{G, \rhobarss}$ is $k$, $X^{\gps, \tau}_{G, \rhobarss}(\kappa(x))$
is a point, and thus this is equivalent to $z$ and $x$ mapping to the same point  in 
$\XgpsGtau$. Corollary \ref{same_image} implies that this is equivalent 
to $\rho_x$ and $\rho_z$ having  the same $G$-semi\-simpli\-fi\-cation.
\end{proof}

\begin{remar}\label{rem_k_rat_comp}
After replacing $L$ by an unramified extension we may assume that all the
factors in \eqref{semi-local} have residue field $k$. It follows from Lemma \ref{get_comp} 
that in this case the connected components of $X^{\gen, \tau}_G$ are indexed 
by $G^0(\kbar)$-conjugacy classes of continuous $G$-completely reducible representations
$\rhobarss: \Gamma_F \rightarrow G(\kbar)$ such that $\det\circ (\tau\circ \rhobarss)^{\lin}=\Dbar$. Lemma \ref{connected_comp} implies that after replacing $L$ 
by a finite unramified extension we may assume that every connected component of 
$X^{\gen, \tau}_G$ has a $k$-rational point.
\end{remar}

\section{\texorpdfstring{Lafforgue's $G$-pseudocharacters}{Lafforgue's G-pseudocharacters}}\label{Laf}

In his work on the local Langlands correspondence for function fields \cite{Laf} Vincent Lafforgue introduced a notion of $G$-\emph{pseudocharacter} for a reductive group $G$.
If $G=\GL_d$ then Emerson and Morel \cite{emerson2023comparison} have shown that a 
$\GL_d$-pseudocharacter after Lafforgue is equivalent to a determinant law after Chenevier \cite{che_durham}. If $d!$ is invertible in the base ring then the trace 
of representation valued in $\GL_d$ determines its $\GL_d$-pseudocharacter. In the arithmetic 
setting this was first used  by Wiles \cite{Wiles1988} and Taylor \cite{Tay91}.

\subsection{\texorpdfstring{Definition of $G$-pseudocharacter}{Definition of G-pseudocharacter}} The definition of $G$-pseudocharacter we use is a slight modification of Lafforgue’s original
definition \cite[Section 11]{Laf}, in that we work over the base ring $\OO$ and allow arbitrary (disconnected) generalised reductive group schemes for $G$.

\begin{defi}\label{LafPC} Let $\Gamma$ be an abstract group and let $A$ be a commutative $\OO$-algebra. A \emph{$G$-pseudocharacter} $\Theta$ of $\Gamma$ over $A$ is a sequence $(\Theta_n)_{n \geq 1}$ of $\OO$-algebra maps
$$\Theta_n : \OO[G^n]^{G^0} \to \mathrm{Map}(\Gamma^n,A)$$ for $n \geq 1$, satisfying the following conditions\footnote{Here $G$ acts on $G^n$ by $g \cdot (g_1, \dots, g_n) = (gg_1g^{-1}, \dots, gg_ng^{-1})$. This induces a rational action of $G$ on the affine coordinate ring $\OO[G^n]$ of $G^n$. The submodule $\OO[G^n]^{G^0} \subseteq \OO[G^n]$ is defined as the rational invariant module of the $G^0$-representation $\OO[G^n]$. It is an $\OO$-subalgebra, since $G$ acts by $\OO$-linear automorphisms.}:
\begin{enumerate}
    \item For each $n,m \geq 1$, each map $\zeta : \{1, \dots, m\} \to \{1, \dots,n\}$, $f \in \OO[G^m]^{G^0}$ and $\gamma_1, \dots, \gamma_n \in \Gamma$, we have
    $$ \Theta_n(f^{\zeta})(\gamma_1, \dots, \gamma_n) = \Theta_m(f)(\gamma_{\zeta(1)}, \dots, \gamma_{\zeta(m)}) $$
    where $f^{\zeta}(g_1, \dots, g_n) = f(g_{\zeta(1)}, \dots, g_{\zeta(m)})$.
    \item For each $n \geq 1$, for each $\gamma_1, \dots, \gamma_{n+1} \in \Gamma$ and each $f \in \OO[G^n]^{G^0}$, we have
    $$ \Theta_{n+1}(\hat f)(\gamma_1, \dots, \gamma_{n+1}) = \Theta_n(f)(\gamma_1, \dots, \gamma_n\gamma_{n+1}) $$
    where $\hat f(g_1, \dots, g_{n+1}) = f(g_1, \dots, g_ng_{n+1})$.
\end{enumerate}
\end{defi}
We denote the set of $G$-pseudocharacters of $\Gamma$ over $A$ by $\PC_G^{\Gamma}(A)$.
If $f : A \to B$ is a homomorphism of $\OO$-algebras, then there is an induced map $f_* : \mathrm{PC}_{G}^{\Gamma}(A) \to \mathrm{PC}_{G}^{\Gamma}(B)$.
For $\Theta \in \mathrm{PC}_{G}^{\Gamma}(A)$, the image $f_*(\Theta)$ is called the \emph{specialisation} of $\Theta$ along $f$ and is denoted by $\Theta \otimes_A B$. When $\varphi : G \to H$ is a homomorphism of generalised reductive $\OO$-group schemes, the induced maps $\varphi^*_n : \OO[H^n]^{H^0} \to \OO[G^n]^{G^0}$ give rise to an $H$-pseudocharacter $(\Theta_n \circ \varphi^*_n)_{n \geq 1}$. By analogy with the notation for representations we denote this $H$-pseudocharacter by $\varphi \circ \Theta$. So we also have an induced map $\PC_G^{\Gamma}(A) \to \PC_H^{\Gamma}(A)$.
It is easy to verify that specialisation along $f : A \to B$ commutes with composition with $\varphi$, i.e. $(\varphi \circ \Theta) \otimes_A B = \varphi \circ (\Theta \otimes_A B)$.

When $\Gamma$ is a topological group, $A$ is a topological ring and for all $n \geq 1$, the map $\Theta_n$ has image in the set $\Cont(\Gamma^n, A)$ of continuous maps $\Gamma^n \to A$, we say that $\Theta$ is \emph{continuous}. We denote the set of continuous $G$-pseudocharacters of $\Gamma$ with values in $A$ by $\cPC_G^{\Gamma}(A)$.

When $\OO'$ is a commutative $\OO$-algebra and $A$ is a commutative $\OO'$-algebra the natural maps $\OO[G^n]^{G^0} \to \OO'[G^n]^{G^0}$ induce a map $\PC_{G_{\OO'}}^{\Gamma}(A)\to \PC_G^{\Gamma}(A)$. This map is always a bijection, as we will now explain. When $\OO'$ is flat over $\OO$ the map $\OO[G^n]^{G^0} \otimes_{\OO} \OO' \to \OO'[G^n]^{G^0}$ is an isomorphism by the universal coefficient theorem in rational cohomology, see \cite[Proposition I.4.18 (a)]{Jantzen}. The case of a general $\OO$-algebra is treated in \cite[Proposition 3.22]{quast}. We will use this fact whenever we need to extend the base ring $\OO$.

\subsection{\texorpdfstring{Deformations of $G$-pseudocharacters}{Deformations of G-pseudocharacters}}\label{sec_def_GPC}
We adopt the deformation theoretic setup of subsection \ref{def_probs}.
Let $\kappa$ be either a finite or a local field, which is an $\OO$-algebra, 
equipped with its natural topology. Let $\Lambda$ be the coefficient ring for $\kappa$ and let $\Aa_{\Lambda}$ be the category of local artinian $\Lambda$-algebras with residue field $\kappa$, equipped with the natural topology.  

In this section $G$ is a generalised reductive group over $\OO$.
We consider the following deformation problem for $G$-pseudocharacters.

\begin{defi} \label{pseudodeffunctor} Let $\overline{\Theta} \in \cPC^{\Gamma}_G(\kappa)$ be a continuous $G$-pseudocharacter of $\Gamma$ with values in $\kappa$. We define the \emph{deformation functor}
$$    D_{\overline{\Theta}} : \Aa_{\Lambda} \to \Set, \quad 
    A \mapsto \{\Theta \in \cPC_G^{\Gamma}(A) \mid \Theta \otimes_A \kappa = \overline{\Theta}\}$$
that sends an object $A \in \Aa_{\Lambda}$ to the set of continuous $G$-pseudocharacters $\Theta$ of $\Gamma$ over $A$ with $\Theta \otimes_A \kappa = \overline{\Theta}$.
\end{defi}

The deformation problem for $G$-pseudocharacters admits a universal deformation ring, which we will denote by $\RpsG$ when $\Thetabar = \Theta_{\rhobar}$. We will denote the universal $G$-pseudocharacter by $\Theta^u_G$ and will drop the index $G$, when it is clear which group scheme we are working with.

\subsection{\texorpdfstring{Comparison of $X^{\gps,\tau}_{G,\rhobarss}$ with $X^{\ps}_G$}{Comparison of XgpsGrhobarss with XpsG}}\label{comparison} In this section we assume that $\Gamma$ is
a profinite topologically finitely generated group. It has been shown by the second author in \cite[Theorem A]{quast} that $\RpsG$ is a complete noetherian local $\Lambda$-algebra. 

Let $A$ be a commutative $\RpsGLd$-algebra. We have a natural transformation
$$ \vartheta_A : X^{\gen,\tau}_G(A) \to \PC_G^{\Gamma}(A), \quad\rho \mapsto \Theta_{\rho}, $$
that attaches to a representation $\rho$ its $G$-pseudocharacter $\Theta_{\rho}$.
By \cite[Theorem 3.20]{quast} $\PC_G^{\Gamma}$ is representable by an affine scheme, so $\vartheta : X^{\gen,\tau}_G \to \PC_G^{\Gamma}$ factors over $\theta : X^{\gps,\tau}_G \to \PC_G^{\Gamma}$.

\begin{lem}\label{abstz} Let $\rho:\Gamma \rightarrow G(A^{\gen, \tau}_{G, \rhobarss})$ be the universal
representation. Then $\Theta_{\rho}\in \cPC_G^{\Gamma}(R^{\gps, \tau}_{G, \rhobarss})$ for 
the $\mm$-adic topology on $R^{\gps, \tau}_{G, \rhobarss}$.
\end{lem}
\begin{proof} \Cref{XgenGA} implies that $\rho$ is continuous for the  $\mm_{\RpsGLd}$-adic topology on $\AgenGtaurhobarss$.
    The  $G$-pseudocharacter $\Theta_{\rho}$ takes values in $R^{\gps,\tau}_{G,\rhobarss}$ and is thus continuous for the subspace topology on $R^{\gps,\tau}_{G,\rhobarss}$. Since 
    $$\mm_{\RpsGLd}^n R^{\gps, \tau}_{G, \rhobarss}\subseteq \mm_{\RpsGLd}^n \AgenGtau \cap R^{\gps,\tau}_{G, \rhobarss}, \quad \forall n\ge 1,$$
$\Theta_{\rho}$ is continuous for the $\mm_{\RpsGLd}$-adic topology on $R^{\gps,\tau}_{G, \rhobarss}$.
Since by \Cref{GGLd_finite}, $R^{\gps,\tau}_{G,\rhobarss}$ is a finite $\RpsGLd$-algebra, this topology coincides with the topology defined by the maximal ideal of $R^{\gps,\tau}_{G,\rhobarss}$.
\end{proof}

Let $\Theta_{\rho}$ be the $G$-pseudocharacter
of the universal representation over $A^{\gen, \tau}_{G, \rhobarss}$ as in Lemma \ref{abstz}.
Recall that $R^{\gps,\tau}_{G,\rhobarss}$ has residue field $k$ and 
$\Theta_{\rho}\otimes_{R^{\gps,\tau}_{G,\rhobarss}} k = \Theta_{\rhobar}$. 
It follows from Lemma  \ref{abstz}, that $\Theta_{\rho}$ 
is indeed a continuous deformation of $\Theta_{\rhobarss}$, so we have a local homomorphism $R^{\ps}_G \to R^{\gps,\tau}_{G,\rhobarss}$ which induces a map
\begin{equation}\label{XgpsXpsmap}
    \nu : X^{\gps,\tau}_{G,\rhobarss} \to X^{\ps}_G 
\end{equation}
of $\OO$-schemes.

\begin{prop}\label{nu_fin_u} $\nu : X^{\gps,\tau}_{G,\rhobarss} \to X^{\ps}_G$ is a finite universal homeomorphism.
\end{prop} 

\begin{proof}
    By \stackcite{04DF}, it suffices to show, that $\nu$ is finite, universally injective and surjective.
    By \Cref{GGLd_finite} $R_{G,\rhobarss}^{\gps,\tau}$ is a finite $\RpsGLd$-module.
    From the embedding $\tau : G \to \GL_d$ we obtain a factorisation $\RpsGLd \to R^{\ps}_G \to R^{\gps,\tau}_{G,\rhobarss}$.
    In particular, $R^{\gps,\tau}_{G,\rhobarss}$ is a finite $R^{\ps}_G$-module.

    The closed point of $X^{\gps,\tau}_{G,\rhobarss}$ is mapped to the closed point of $X^{\ps}_G$, so we only need to show surjectivity of $\dot \nu : U^{\gps,\tau}_{G, \rhobarss} := X^{\gps,\tau}_{G,\rhobarss} \setminus \{*\} \to U^{\ps}_{G} := X^{\ps}_G \setminus \{*\}$. By \stackcite{02J6} $U^{\gps,\tau}_{G, \rhobarss}$ and $U^{\ps}_{G}$ are Jacobson. By Chevalley's theorem \stackcite{054K} the image of $\dot \nu$ is locally constructible and thus also its complement. Therefore for surjectivity of $\nu$ it is sufficient to show, that the closed points of $U^{\ps}_{G}$ are all contained in the image of $\dot\nu$.

Let $y \in U^{\ps}_{G}$ be a closed point.
    By \cite[Lemma 3.17]{BIP_new} $\kappa(y)$ is a local field. 
    By \cite[Theorem 5.7]{quast}, $R^{\ps}_G$ is noetherian, and thus $\kappa(y)$ with its natural topology is a topological $\RpsG$-algebra and the image of $y: \RpsG \rightarrow \kappa(y)$ 
    is contained in the ring of integers $\OO_{\kappa(y)}$ by \cite[Lemma 3.17]{BIP_new}. 
    We obtain a continuous $G$-pseudocharacter $\Theta_y \in \cPC_G^{\Gamma}(\OO_{\kappa(y)})$, which is a deformation 
    of $\Theta_{\rhobar}$ to $\OO_{\kappa(y)}$. By the continuous reconstruction theorem 
    \cite[Theorem 3.8]{quast}, there is a continuous representation $\rho' : \Gamma \to G(\overline{\kappa})$ with $\Theta_{\rho'} = \Theta_y$, where $\kappabar$ is an algebraic closure of $\kappa(y)$. Lemma \ref{geom_pts}
    implies that $\rho'\in X^{\gen, \tau}_G(\overline{\kappa})$. The $G(\kappabar)$-conjugacy class of $\rho'$ defines a point $y'\in X^{\git,\tau}_G(\kappabar)$ which maps to $y$. We have to show that $y'$ lies on 
    the same connected component of $X^{\git, \tau}_G$ as $\rhobarss$. Since $X^{\gen,\tau}_G$ is of finite 
    type over $X^{\ps}_{\GL_d}$, $\rho'$ takes values in $G(\kappa')$ for some finite 
    extension of $\kappa(y)$ in $\kappabar$. Proposition \ref{building}
    implies that after conjugating $\rho'$ with an element of $G^0(\kappabar)$ we may assume 
    that $\rho'$ takes values in $G(\OO_{\kappabar})$. Let $\rhobar'$ be the composition 
    of $\rho'$ with the reduction map $G(\OO_{\kappabar})\rightarrow G(\kbar)$. Then 
    $\Theta_{\rhobar'}= \Theta_y\otimes_{\OO_{\kappa(y)}} \kbar = \Theta_{\rhobarss}$. 
    The uniqueness part of the reconstruction theorem \cite[Theorem 3.7]{quast} implies that $G$-semisimplification 
    of $\rhobar'$ is equal to $\rhobarss$, and Lemma \ref{connected_comp} implies that 
    $y'\in X^{\git,\tau}_{G, \rhobarss}$.

    For universal injectivity it suffices by \stackcite{01S4} to show that the diagonal morphism $X^{\gps,\tau}_{G,\rhobarss} \to X^{\gps,\tau}_{G,\rhobarss} \times_{X^{\ps}_G} X^{\gps,\tau}_{G,\rhobarss}$ is surjective. Let $x \in X^{\gps,\tau}_{G,\rhobarss} \times_{X^{\ps}_G} X^{\gps,\tau}_{G,\rhobarss}$. We obtain by \Cref{closed_orbit} and \Cref{closed_orbit_fibre} two $G$-semisimple representations $\rho_x, \rho_x' \in X^{\gen}_{G, \rhobarss}(\overline{\kappa(x)})$, such that their associated $G$-pseudocharacters coincide. This means by the uniqueness part of the reconstruction theorem  and \Cref{same_image}, their image in $\XgpsG(\overline{\kappa(x)})$ coincides and therefore $x$ is in the image of the diagonal.
\end{proof}

\begin{lem}\label{nu_cl_im}
    $\nu \otimes_{\Zp} \Qp : X^{\gps,\tau}_{G,\rhobarss}[1/p] \to X^{\ps}_G[1/p]$ is a closed immersion.
\end{lem}

\begin{proof} Let $\mathcal F_N$ be a free group generated by elements $x_1, \ldots, x_N$. 
Let $\Rep_G^{\mathcal F_N, \square}$ be the affine $\RpsGLd$-scheme, which represents the functor, that maps a commutative $\RpsGLd$-algebra 
$A$ to the set of (abstract) homomorphisms $\rho : \mathcal F_N \to G(A)$. Mapping $x_i$ to $\gamma_i$ induces 
a group homomorphism $\mathcal F_N\rightarrow \Gamma$ and it follows from Remark \ref{FN},
that we have a closed immersion $X^{\gen,\tau}_G \rightarrow \Rep_G^{\mathcal F_N, \square}$ and hence
a closed immersion $X^{\gen,\tau}_{G, \rhobarss} \rightarrow \Rep_G^{\mathcal F_N, \square}$.

By \cite[Proposition 2.11 (i)]{emerson2023comparison} the vertical map on the left hand side of the following diagram is an isomorphism.
    \begin{equation}\label{diagramclimm}
    \begin{tikzcd}
        \RpsGLd[1/p][\PC_G^{\mathcal F_N}] \dar{\cong} \rar & R^{\ps}_G[1/p] \dar{\nu^*} \\
        \RpsGLd[1/p][\Rep_G^{\mathcal F_N, \square}]^{G^0} \rar[two heads] & R^{\gps,\tau}_{G,\rhobarss}[1/p]
    \end{tikzcd}
    \end{equation}
    The bottom map is surjective, since $\RpsGLd[1/p][\Rep_G^{\mathcal F_N, \square}] \to A^{\gen,\tau}_{G,\rhobarss}[1/p]$ is surjective as explained above and taking $G^0$-invariants is exact in characteristic $0$.
    Since the diagram commutes the right vertical arrow is surjective, and we obtain the claim. 
\end{proof}

\begin{cor}\label{nu_iso_red} The map $\nu \otimes_{\Zp} \Qp$ induces an isomorphism of underlying reduced schemes. 
In particular,  if $X^{\ps}_G[1/p]$ is reduced, then $\nu \otimes_{\Zp} \Qp$ is an isomorphism.
\end{cor}

\begin{proof} Proposition \ref{nu_fin_u} implies that $\nu \otimes_{\Zp} \Qp$ is a universal homeomorphism 
and Lemma \ref{nu_cl_im} that it is a closed immersion, which implies the first claim. 
The second claim is then immediate. 
\end{proof} 

\begin{remar} In \cite{inf_laf}, building on the results of this paper, we show that $\nu\otimes_{\Zp} \Qp$ is an isomorphism. 
\end{remar}

\section{Functoriality}\label{sec_funct}

The main result of this section is Proposition \ref{XgenGtau_rhobarss_A}, 
which gives an intrinsic description of 
$X^{\gen, \tau}_{G, \rhobarss}$ as a functor. 
As a consequence 
in  \Cref{ind_funct} we show that
$X^{\gen, \tau}_{G, \rhobarss}$ is independent of the choice of an embedding $\tau: G\hookrightarrow \GL_d$ and  is 
functorial in $G$.  In  \Cref{sec_finite} 
we investigate functoriality in $G$ under finite maps. In  \Cref{sec_prod}
we study the case when  $G$ is a product of two generalised reductive $\OO$-groups. In \Cref{funct_Gamma}
we study the morphisms induced  by restriction of pseudocharacters and representations to an open subgroup of $\Gamma$. 

We assume in this section that $\Gamma$ is topologically finitely generated. This assumption 
ensures that $R^{\ps}_G$ is noetherian. If $\Gamma$ is only assumed to satisfy Mazur's $p$-finiteness
condition\footnote{In \cite{noeth_def}, using the results of this paper, we show that if $\Gamma_F$ satisfies Mazur's $p$-finiteness condition then $R^{\ps}_G$ is noetherian and the results of this section still hold in this more general setting.} then results of this section can be reformulated by replacing $\Gamma$ by a topologically
finitely generated quotient as in Lemma \ref{fg_quotient}. We note that $\Gamma_F$ is topologically finitely generated by \cite[Satz 3.6]{Jannsen1982}. 

\subsection{\texorpdfstring{Independence of $\tau$ and functoriality in $G$}{Independence of tau and functoriality in G}}\label{ind_funct}

\begin{lem}\label{Gps_comp} Let $A$ be an $\RpsG$-algebra and let $\rho\in X^{\gen, \tau}_{G}(A)$. Then 
$\rho\in X^{\gen, \tau}_{G,\rhobarss}(A)$ if and only if the $G$-pseudocharacter $\Theta_{\rho}$ is equal to the specialisation of $\Theta^u$ at $A$.
\end{lem}

\begin{proof} Let $\rho \in \XgenGtaurhobarss(A)$.
    Then $\rho = \rho^u \otimes_{\AgenGtaurhobarss} A$ where $\rho^u \in \XgenGtaurhobarss(\AgenGtaurhobarss)$ is the universal representation.
    Taking $G$-pseudocharacters, we obtain $\Theta_{\rho} = \Theta_{\rho^u} \otimes_{\AgenGtaurhobarss} A$.
    Using the map $\XgenGtaurhobarss \to X^{\ps}_G$ given by composing the map in \eqref{XgpsXpsmap} with the GIT quotient map, we see that $\Theta_{\rho^u} = \Theta^u \otimes_{R^{\ps}_G} \AgenGtaurhobarss$, so $\Theta_{\rho} = \Theta^u \otimes_{R^{\ps}_G} A$.
    
    Conversely, let $\rho \in \XgenGtau(A)$ satisfy $\Theta_{\rho} = \Theta^u \otimes_{R^{\ps}_G} A$.
    By factoring $\rho$ as $\Gamma \to G(\AgenGtau) \to G(A)$ over the universal representation on $\AgenGtau$ and using that $\AgenGtau$ is of finite type over $\RpsGLd$, we see that $\rho$ actually takes values in a finitely generated $\RpsG$-subalgebra $A' \subseteq A$.
    We want to show, that the image of $\Spec(A') \to \XgenGtau$ is entirely contained in the component $\XgenGtaurhobarss$, which amounts to showing that an idempotent $e \in \AgenGtau$ associated to a component $X^{\gen, \tau}_{G, \rhobarss_2}$ associated with a residual representation $\rhobarss_2$ which is non-conjugate to $\rhobarss$ acts as zero on $A'$. By contradiction, we assume that $eA'$ is nonzero. As $A'$ is of finite type over $R^{\ps}_G$, there is a dimension $\leq 1$ prime ideal $\pp \in \Spec(eA')$ with residue field $\kappa(\pp)$ finite or local by \cite[Lemma 3.17]{BIP_new}.
    By \Cref{geom_pts}, we know that the representation $\rho_{\pp} : \Gamma \to G(\overline{\kappa(\pp)})$ is continuous and that $\Theta_{\rho_{\pp}} = \Theta^u \otimes_{R^{\ps}_G} \overline{\kappa(\pp)}$. If $\kappa(\pp)$ is a finite field then we deduce that the $G$-semisimplification 
    of $\rho_{\pp}$ is $\rhobarss$, which yields a contradiction to the choice of $e$ via Lemma \ref{get_comp}.
    If $\kappa(\pp)$ is a local field, by \Cref{building} the image of $\rho_{\pp}$ lies up to conjugation in the ring of integers of a finite extension of $\kappa(\pp)$. The $G$-semisimplification of the reduction 
    modulo the maximal ideal is $\rhobarss$, since $\Theta_{\rho_{\pp}}$ specialises $\Theta^u$, which yields a contradiction to the choice of $e$ via Lemma \ref{get_comp}. 
    We conclude, that $\rho \in \XgenGtaurhobarss(A')\subseteq \XgenGtaurhobarss(A)$.
\end{proof}

\begin{lem}\label{factors_through_Eu_G} Let $R\rightarrow A$ be a map of $\RpsG$-algebras such that $R$ is a noetherian profinite topological $\RpsG$-algebra. Let  $\rho : \Gamma \to G(A)$ be a 
representation satisfying $\Theta_{\rho}=\Theta^u_{|A}$.
Then the following are equivalent:
\begin{enumerate} 
\item $(\tau \circ \rho)^{\lin}|_{R^{\ps}_G[\Gamma]}$ factors through $R^{\ps}_G \otimes_{\RpsGLd} E^u$;
\item $\rho$ is $\RpsG$-condensed;
\item $\rho$ is $R$-condensed. 
\end{enumerate}
\end{lem}
\begin{proof} It follows from Lemma \ref{cond_rep_funct} that $\rho$ is $R$-condensed if and only if $\tau\circ \rho$
is $R$-condensed. As explained in  section \ref{comparison} 
the assumption $\Theta_{\rho}=\Theta^u_{|A}$ implies that 
$\det \circ (\tau\circ \rho)^{\lin} = D^{u}_{|A}$. The equivalence now follows from Lemma \ref{factors_through_Eu}.
\end{proof} 

\begin{prop}\label{XgenGtau_rhobarss_A} Let $A$ be an $\RpsG$-algebra. There is canonical bijection between the following sets:
\begin{enumerate}
\item $X^{\gen, \tau}_{G, \rhobarss}(A)$;
\item the set of representations $\rho : \Gamma \to G(A)$, such that the $G$-pseudocharacter $\Theta_{\rho}$ is equal to the specialisation of $\Theta^u$ at $A$ and the equivalent conditions 
of Lemma \ref{factors_through_Eu_G} hold.
\end{enumerate}
\end{prop}
\begin{proof} If $\rho\in X^{\gen, \tau}_{G, \rhobarss}(A)$ then $\Theta_{\rho}=\Theta^u_{|A}$ 
by Lemma \ref{Gps_comp}. Moreover, $\tau\circ \rho$ is $\RpsGLd$-condensed by Lemma \ref{XgenGA} (1). Lemmas \ref{cond_rep_funct} and \ref{B4} imply that $\rho$ is $R^{\ps}_G$-condensed.

If $\rho:\Gamma\rightarrow G(A)$ satisfies the conditions of part (2) then 
$\det \circ (\tau\circ \rho)^{\lin}= D^u_{|A}$ and $\tau\circ \rho$ is $R^{\ps}_G$-condensed. 
Thus $\rho$ satisfies part (4) of Lemma \ref{factors_through_Eu} with $R=R^{\ps}_G$ and thus $\rho\in X^{\gen, \tau}_{G}(A)$ by Lemma \ref{XgenGA} (1). Since $\Theta_{\rho}=\Theta^u_{|A}$, Lemma \ref{Gps_comp} implies that 
$\rho\in X^{\gen, \tau}_{G, \rhobarss}(A)$.
\end{proof}

\begin{cor}\label{indep_tau} The $\RpsG$-algebra $A^{\gen, \tau}_{G, \rhobarss}$  does not depend on the chosen embedding $\tau: G\hookrightarrow \GL_d$.
\end{cor} 
\begin{proof} The assertion follows from Proposition \ref{XgenGtau_rhobarss_A}, as the condition 
in part (2) of Lemma \ref{factors_through_Eu_G} is independent of $\tau$. 
\end{proof}

 As a consequence of the Corollary we will drop $\tau$ from the notation and write $A^{\gen}_{G, \rhobarss}$, $R^{\gps}_{G, \rhobarss}$,
 $\XgenGrhobarss$, $X^{\gps}_{G, \rhobarss}$ 
 for $A^{\gen,\tau}_{G, \rhobarss}$, $R^{\gps,\tau}_{G, \rhobarss}$, $\XgenGtaurhobarss$ and $X^{\gps,\tau}_{G, \rhobarss}$, respectively.

We now turn to functoriality.

\begin{prop}\label{functoriality_2}
    Let $\varphi : G \to H$ be a homomorphism of generalised reductive $\OO$-group schemes.
    Then there is a natural map\footnote{As pointed out by referee the target of the map should be denoted by $X^{\gen}_{H, (\varphi \circ \rhobarss)^{\mathrm{ss}}}$, as $\varphi\circ \rhobarss$ need not be $H$-semisimple in general. For ease of reading, we will continue to use this abuse of notation.} of $\RpsH$-schemes $\XgenGrhobarss \to X^{\gen}_{H, \varphi \circ \rhobarss}$, such that for every $\RpsG$-algebra $A$, the map $\XgenGrhobarss(A) \to X^{\gen}_{H, \varphi \circ \rhobarss}(A)$ is given by $\rho \mapsto \varphi \circ \rho$.
\end{prop}

\begin{proof} Let $\Theta^u_G$ be the universal $G$-pseudocharacter deforming $\Theta_{\rhobarss}$
and let $\Theta^u_H$ be the universal $H$-pseudocharacter deforming $\Theta_{\varphi\circ\rhobarss}$. 
Since $\varphi(\Theta_G^u)$ is a deformation of $\Theta_{\varphi\circ\rhobarss}$ to $R^{\ps}_G$, we obtain a natural homomorphism of local $\OO$-algebras $R^{\ps}_H \rightarrow R^{\ps}_G$ such that 
$\varphi(\Theta^u_G)= \Theta^u_H \otimes_{R^{\ps}_H} R^{\ps}_G$. 
    
    If $\rho \in \XgenGrhobarss(A)$ then $\Theta_{\rho}= \Theta^u_G \otimes_{R^{\ps}_G} A $ and
    by applying $\varphi$, we get
    \begin{equation}\label{theta_good}
    \Theta_{\varphi \circ \rho}= \varphi(\Theta_{\rho})=\varphi(\Theta^u_G)\otimes_{\RpsG} A =
    \Theta^u_H\otimes_{\RpsH} A.
    \end{equation}
 Lemma \ref{cond_rep_funct} implies that $\varphi\circ \rho$ is $\RpsG$-condensed. Since
 $\varphi\circ \rho$ satisfies \eqref{theta_good}, Lemma \ref{factors_through_Eu_G} (3) applied with $R=\RpsG$ and Proposition \ref{XgenGtau_rhobarss_A}
 imply that $\varphi\circ \rho\in X^{\gen}_{H, \varphi\circ \rhobarss}(A)$.
\end{proof}

\subsection{Finite maps}\label{sec_finite} Let $\varphi: G\rightarrow H$ be a morphism of generalised reductive $\OO$-group schemes. 
By functoriality of $X^{\gen}_{G,\rhobarss}$ proved in Propositions \ref{functoriality_2} we obtain a 
morphism $f: X^{\gen}_{G, \rhobarss}\rightarrow X^{\gen}_{H, \varphi\circ \rhobarss}$. 

\begin{prop}\label{finite_functoriality} If $\varphi$ is finite then the map $X^{\ps}_{G}\rightarrow X^{\ps}_H$ is also finite. 
\end{prop}

\begin{proof}
    By topological Nakayama's lemma it is enough to show, that the fibre ring $C := k \otimes_{R^{\ps}_{H}} R^{\ps}_{G}$ is a finite-dimensional $k$-vector space.
    Since $\Gamma$ is topologically finitely generated, we know by \cite[Theorem 5.7]{quast}, that $R^{\ps}_{G}$ is noetherian, hence $C$ is a complete local noetherian ring and we are left to show, that $C$ has Krull dimension $0$.
    
    If $C$ is not $0$-dimensional then there is a point $x \in \Spec(C)$ of dimension $1$.
    The residue field $\kappa(x)$ is a local field of characteristic $p$ by \cite[Lemma 3.17]{BIP_new} and the corresponding homomorphism $C \to \kappa(x)$ gives a $G$-pseudocharacter $\Theta \in \PC_{G}^{\Gamma}(\kappa(x))$, such that $\varphi \circ \Theta \in \PC_{H}^{\Gamma}(\kappa(x))$ arises by scalar extension from the residual pseudocharacter attached to $\varphi \circ \rhobar$.

    For all $m \geq 0$, the map $(\varphi \circ \Theta)_m = \Theta_m \circ \varphi^*$ is the composition
    $$ \OO[H^m]^{H^0} \xrightarrow{\varphi^*} \OO[G^m]^{G^0} \xrightarrow{\Theta_m} \Map(\Gamma^m, \kappa(x)) $$
    and has image in $\Map(\Gamma^m, k)$.
    By \Cref{cor_cotner} $\varphi^*$ is finite, so $\Theta_m$ has image in $\Map(\Gamma^m, \overline{k})$.
    This implies that $\kappa(x)$ is algebraic over $k$ and contradicts the assumption that $x$ is a dimension $1$ point.
\end{proof}

\begin{prop}\label{finite_maps} If $\varphi: G \rightarrow H $ is finite then  $f:X^{\gen}_{G, \rhobarss}\rightarrow X^{\gen}_{H, \varphi\circ \rhobarss}$  is also finite. 
\end{prop}

\begin{proof} The induced map $X^{\ps}_G \rightarrow X^{\ps}_H$ is finite by \Cref{finite_functoriality}. Since 
$G\rightarrow H$ is finite by assumption the map $G^N_{X^{\ps}_G}\rightarrow H^N_{X^{\ps}_H}$ is also 
finite. We have a commutative diagram 
\begin{equation}\label{val_prop_gen}
        \begin{tikzcd}
            X^{\gen}_{G, \rhobarss} \ar[r]\ar[d, "f"]  & G^N_{X^{\ps}_G}\ar[d]\\
            X^{\gen}_{H, \varphi\circ\rhobarss}  \ar[r] & H^N_{X^{\ps}_H}.
            \end{tikzcd}
\end{equation}
with horizontal arrows closed immersions. Since closed immersions are finite maps and the 
composition of finite maps are again finite, the diagram implies the assertion.
\end{proof}

\subsection{Products}\label{sec_prod}
In this subsection suppose that $G=G_1\times G_2$, where the fibre product is taken over
$\Spec \OO$ and $G_1$ and $G_2$ are generalised reductive group schemes over $\OO$. 
Let $p_i: G\rightarrow G_i$ denote the projection map onto the $i$-th component. 
Then $\rhobarss$ is of the form $\rhobarss_1\times \rhobarss_2$, where $\rhobarss_i: \Gamma\rightarrow G_i(k)$
is $G_i$-semisimple.

\begin{lem}\label{prodPC} The map $\Theta\mapsto (p_1\circ \Theta, p_2\circ \Theta)$ induces a bijection of functors $\PC_G^{\Gamma}\overset{\sim}{\rightarrow} \PC_{G_1}^{\Gamma}\times \PC_{G_2}^{\Gamma}$ and 
$\cPC_G^{\Gamma}\overset{\sim}{\rightarrow}\cPC_{G_1}^{\Gamma}\times \cPC_{G_2}^{\Gamma}$.
\end{lem}

\begin{proof} For an integer $m \geq 0$, we have $\OO[G^m]^{G^0} = \OO[G_1^m]^{G_1^0} \otimes_{\OO} \OO[G_2^m]^{G_2^0}$, since the $\OO[G_i^m]$ are flat $\OO$-modules. 

    Recall that a $G$-valued pseudocharacter is a family of $\OO$-algebra homomorphisms $\Theta_m : \OO[G^m]^{G^0} \to \Map(\Gamma^m, A)$ satisfying certain compatibility conditions.
    Restriction along the map $\OO[G_i^m]^{G_i^0} \to \OO[G^m]^{G^0}$ defines
    $G_i$-pseudocharacters $\Theta^{(i)}$ satisfying the same compatibility conditions; this can be directly seen from the description of pseudocharacters in terms of functors on the category of finitely generated free groups \cite[Proposition 3.19]{quast}. Conversely given $\Theta^{(i)} \in \PC_{G_i}^{\Gamma}(A)$, we can define a $G$-pseudocharacter $\Theta^{(1)} \otimes \Theta^{(2)}$ by defining $(\Theta^{(1)} \otimes \Theta^{(2)})_m : \OO[G^m]^{G^0} \to \Map(\Gamma^m, A)$ as the tensor product $\Theta^{(1)}_m \otimes \Theta^{(2)}_m : \OO[G_1^m]^{G_1^0} \otimes_{\OO} \OO[G_2^m]^{G_2^0} \to \Map(\Gamma^m, A)$. This implies that the first map is bijective. 
    
    If $A$ is a topological $\OO$-algebra then a map $f: \Gamma^m \rightarrow G(A)$ is continuous
if and only if $p_1\circ f$ and $p_2\circ f$ are continuous. This implies the second assertion.
\end{proof}

\begin{lem}\label{prod_ps} $R^{\ps}_G \cong R^{\ps}_{G_1}\wtimes_{\OO} R^{\ps}_{G_2}$.
\end{lem}

\begin{proof}
    Let $\Thetabar_i$ be the $G_i$-pseudocharacter associated to $\rhobarss_i$; $i=1,2$.
    The ring $R^{\ps}_{G_1}\wtimes_{\OO} R^{\ps}_{G_2}$ pro-represents the product functor $D_{\Thetabar_1} \times D_{\Thetabar_2} : \Aa_{\OO} \to \Set$.
    If $A\in \Aa_{\OO}$ then the bijection between $\cPC_{G_1}^{\Gamma}(A) \times \cPC_{G_2}^{\Gamma}(A)$ and $\cPC_G^{\Gamma}(A)$
    established in \Cref{prodPC} induces a bijection between 
    $D_{\Thetabar_1}(A) \times D_{\Thetabar_2}(A)$ and $D_{\Thetabar}(A)$, which implies the 
    assertion. 
   \end{proof}

The map from the usual tensor product into the completed tensor product induces
a morphism of schemes $X^{\ps}_G\rightarrow X^{\ps}_{G_1}\times X^{\ps}_{G_2}$ over $\Spec \OO$.

\begin{prop}\label{products} The  natural maps  $X^{\gen}_{G, \rhobarss} \to X^{\gen}_{G_i, \rhobarss_i}$ constructed in \Cref{functoriality_2} induce an isomorphism of schemes over $X^{\ps}_G$: 
\begin{align}
    X^{\gen}_{G, \rhobarss} \eqto ( X^{\gen}_{G_1, \rhobarss_1}\times X^{\gen}_{G_2, \rhobarss_2})\times_{X^{\ps}_{G_1}\times X^{\ps}_{G_2}} X^{\ps}_G, \label{prodiso}
\end{align}
where the fibre products without subscript are taken over $\Spec \OO$.
\end{prop}

\begin{proof}
    Let $A$ be an $\RpsG$-algebra. By Proposition \ref{XgenGtau_rhobarss_A} and Lemma \ref{factors_through_Eu_G} (3) applied with 
    $R=\RpsG$,  the $A$-valued points of the right hand side of \eqref{prodiso} are pairs of 
    $\RpsG$-condensed representations $\rho_1 : \Gamma \to G_1(A)$ and $\rho_2 : \Gamma \to G_2(A)$, 
    such that $\Theta^u_{G_i} \otimes_{R^{\ps}_{G_i}} A = \Theta_{\rho_i}$ for $i=1,2$. 
    \Cref{B3} implies, that the representation $\rho := \rho_1 \times \rho_2 : \Gamma \to G(A)$ is $\RpsG$-condensed. It follows from Lemma \ref{prod_ps} that the 
    bijection $$\PC^{\Gamma}_{G}(\RpsG)\eqto
    \PC^{\Gamma}_{G}(\RpsG)\times \PC^{\Gamma}_{G}(\RpsG)$$ maps 
    $\Theta^u_G$ to $(\Theta^u_{G_1}\otimes_{R^{\ps}_{G_1}} \RpsG,
    \Theta^u_{G_2}\otimes_{R^{\ps}_{G_2}} \RpsG)$. Hence, $(\Theta^u_G)_{|A}\in \PC^{\Gamma}_{G}(A)$
    is mapped to $(\Theta_{\rho_1}, \Theta_{\rho_2})\in \PC^{\Gamma}_{G}(A)\times \PC^{\Gamma}_{G}(A)$ and hence $(\Theta^u_G)_{|A}=\Theta_{\rho}$ by Lemma \ref{prod_ps}. Proposition \ref{XgenGtau_rhobarss_A}
    implies that $\rho\in  X^{\gen}_{G, \rhobarss}(A)$ and hence \eqref{prodiso} is surjective. 
    Since $p_1\circ \rho$ and $p_2\circ \rho$ uniquely determine $\rho\in  X^{\gen}_{G, \rhobarss}(A)$, \eqref{prodiso} is 
    also injective. 
    \end{proof}

\subsection{\texorpdfstring{Functoriality in $\Gamma$}{Functoriality in Gamma F}}\label{funct_Gamma}
Let $\Gamma'$ be an open subgroup of $\Gamma$. The restriction of $G$-pseudocharacters of $\Gamma$ to $\Gamma'$ defines 
a morphism $X^{\ps, \Gamma}_{G}\rightarrow X^{\ps,\Gamma'}_{G}$. Similarly, \Cref{rest_open} implies that the restriction of representations
of $\Gamma$ to $\Gamma'$ induces a morphism $X^{\gen,\Gamma}_{G, \rhobarss}\rightarrow X^{\gen,\Gamma'}_{G, \rhobarss}$.

\begin{lem}\label{finite_EF} The morphism $X^{\ps, \Gamma}_{G}\rightarrow X^{\ps,\Gamma'}_{G}$ is finite.
\end{lem}
\begin{proof} If $G=\GL_d$ then this is proved in \cite[Lemma A.3]{BIP_new}. 
In the general case, one could either use that $R^{\ps}_{\GL_d}\rightarrow R^{\ps}_G$ is finite, 
which follows from Propositions \ref{GGLd_finite} and \ref{nu_fin_u}, or observe that the proof of 
\cite[Lemma A.3]{BIP_new} carries over to the more general setting.
\end{proof} 

\begin{lem}\label{type_type_EF} The morphism $X^{\gen,\Gamma}_{G, \rhobarss}\rightarrow X^{\gen,\Gamma'}_{G, \rhobarss}$ is of finite type. 
\end{lem} 
\begin{proof} The assertion follows from Lemma \ref{finite_EF} and the fact that $X^{\gen}_{G,\rhobarss}$ is of finite type over $X^{\ps}_G$.
\end{proof}

 \section{Absolutely irreducible locus}\label{sec_abs_irr}
Let $G$ be a generalised reductive $\OO$-group scheme and let $\pi_G: G\rightarrow 
G/G^0$ be the quotient map. Naively, the absolutely irreducible locus should be an open  
subscheme $V$ of $X^{\gen}_{G, \rhobarss}$ such that a geometric point $x$ of $X^{\gen}_{G, \rhobarss}$ lies in $V$ if and only if the image of $\rho_x$ is not contained in any 
proper $R$-parabolic subgroup of $G_{\kappa(x)}$. 
It turns out that if $\pi_G\circ \rhobarss: \Gamma \rightarrow (G/G^0)(\kbar)$ 
is not surjective then the naively defined absolutely irreducible locus might be empty, see Example \ref{ex_bad_def}. To overcome this problem  
we first show that there is a closed generalised 
reductive subgroup scheme $H$ of $G$, such that $H^0=G^0$, $\rhobarss$ factors through $H$,  
$\pi_H\circ \rhobarss: \Gamma \rightarrow (H/H^0)(\overline{k})$ is surjective and $X^{\gen}_{G, \rhobarss}= X^{\gen}_{H,\rhobarss}$.  In particular, we may always assume that
$G/G^0$ is a finite constant group scheme. 

We then define the absolutely irreducible locus 
in $X^{\gen}_{G, \rhobarss}$ under the assumption 
that $G$ is split and $\pi_G \circ \rhobarss: \Gamma \rightarrow (G/G^0)(\overline{k})$ is surjective, see Proposition \ref{irr=stable} and Definition \ref{defi_abs_irr}. It is convenient to define it as a subscheme of $X^{\gen}_{G, \rhobarss}\setminus Y_{\rhobarss}$, so that points in $X^{\gen}_{G, \rhobarss}(\kbar)$ are never in the absolutely 
irreducible locus, even if the image of the corresponding representation $\rho_x: \Gamma \rightarrow G(\kbar)$ is not contained in any proper R-parabolic subgroup of $G$. 

The main result proved in Corollary \ref{UV_GG} says that if the absolutely irreducible locus 
is non-empty then the difference between its dimension and the dimension of its GIT quotient by $G^0$ 
is equal to $\dim G_k - \dim Z(G_k)$. We also  define schemes $V_{LG}$ and $U_{LG}$ for an R-Levi $L$, which along with Corollary \ref{UV_GG} play an important role in the inductive argument in 
\Cref{sec_complete_intersection}.

The results of  this section apply to any  profinite group $\Gamma$ satisfying Mazur's $p$-finiteness condition. Lemma \ref{fg_quotient}
implies that we may replace $\Gamma$ with a topologically finitely generated
quotient without changing $X^{\gen}_{G, \rhobarss}$. Hence, without loss of 
generality we assume that $\Gamma$ is topologically finitely generated. 

\subsection{Trimming the component group}\label{sec_trimming} Let $\rho_z: \Gamma \rightarrow G(\kbar)$ be a representation
corresponding to a point $z\in X^{\gen}_{G, \rhobarss}(\kbar)$. Then the composition 
$\pi_G\circ \rho_z: \Gamma \rightarrow (G/G^0)(\kbar)$ depends only on the $G$-semisimplification 
of $\rho_z$. Since $\rho_z^{\mathrm{ss}}=\rhobarss$ by Lemma \ref{get_comp} the map $\pi_G \circ \rhobarss$ 
is well defined. Let $\Delta$ be the image of 
$$\pi_G \circ \rhobarss: \Gamma \rightarrow (G/G^0)(\overline{k}).$$ 
Since $G$ is generalised reductive $G/G^0$ is finite \'etale over $\OO$ by Remark \ref{rem_fin_etale}.
The map $\mathfrak S\mapsto \mathfrak S(\kbar)$ induces an equivalence of categories between finite \'etale $\OO$-groups schemes and the category of finite groups together with 
$\Gal(\overline{k}/k)$-action by group automorphisms. Since $X^{\gen}_{G, \rhobarss}(k)$ is non-empty, 
$\Delta$ is contained in $(G/G^0)(k)$ and hence
the constant group scheme $\underline{\Delta}$ defined by $\Delta$ is a finite \'etale subgroup scheme
of $G/G^0$.

By construction 
 $X^{\gen,\tau}_G$ is a closed subscheme of $G_{\XpsGLd}^N$. The map $\pi_G$ induces a $G^0$-equivariant morphism
of schemes $\nu: X^{\gen,\tau}_G\rightarrow (G/ G^0)^N_{\XpsGLd}$ for the trivial action on the target. 
If $A$ is an $\RpsGLd$-algebra 
then the map $\nu$ on $A$-points is given by 
\begin{equation}\label{nu}
\nu: X^{\gen,\tau}_G(A)\rightarrow (G/G^0)_{\XpsGLd}^N(A), \quad \rho\mapsto (\pi_G(\rho(\gamma_1)), \ldots, 
\pi_G(\rho(\gamma_N))),
\end{equation}
where $\gamma_1, \ldots, \gamma_N$ is a generating set for $\Gamma$ chosen before \Cref{defiXgenG}.
Since $\nu$ is $G^0$-equivariant it factors through $X^{\gps,\tau}_G\rightarrow  (G/ G^0)^N_{\XpsGLd}$.
We  view an $N$-tuple 
$\underline{\delta}=(\delta_1, \ldots, \delta_N)\in \Delta^N$ as a closed constant subscheme 
of $\underline{\Delta}^N_{\XpsGLd}$ over $\XpsGLd$ and define 
$$X^{\gps,\tau}_{G, \underline{\delta}}:= X^{\gps,\tau}_G\times_{(G/G^0)^N_{\XpsGLd}} \underline{\delta}.$$
If $X$ is a scheme over $X^{\gps, \tau}_G$ we let 
\begin{equation}\label{eq_notation}
X_{\underline{\delta}}:= X\times_{\XgpsGtau} X^{\gps,\tau}_{G, \underline{\delta}}, \quad  
X_{\rhobarss}:= X \times_{\XgpsGtau} X^{\gps}_{G, \rhobarss}.
\end{equation}

\begin{lem}\label{connected_Delta} If $\underline{\delta}=(\pi_G(\rhobarss(\gamma_1)), \ldots, \pi_G(\rhobarss(\gamma_N)))$ then 
 $$ X^{\gps}_{G, \rhobarss} = ((\XgpsGtau)_{\rhobarss})_{\underline{\delta}}
 = ((\XgpsGtau)_{\underline{\delta}})_{\rhobarss}.$$ 
 In particular, for any scheme $X$ over $\XgpsGtau$ we have $(X_{\underline{\delta}})_{\rhobarss}= X_{\rhobarss}$.
 \end{lem}
\begin{proof} If $R$ is a complete local noetherian $\OO$-algebra with residue field 
$k$ then the map $R\rightarrow k$ induces a bijection $(G/G^0)(R)\overset{\sim}{\longrightarrow} (G/G^0)(k)$ by \cite[Exp.\,I, Cor.\,6.2]{SGA1}.  
We apply this observation with $R=R^{\gps}_{G,\rhobarss}$ and let 
$\underline{\delta}\in (G/G^0)^N(k)$ be the $N$-tuple 
 corresponding to $\alpha: X^{\gps}_{G, \rhobarss} \rightarrow \XgpsGtau \overset{\nu}{\rightarrow} 
(G/G^0)^N_{\XpsGLd}$. Since the closed point in $X^{\gps}_{G, \rhobarss}$ maps to 
$(\pi_G(\rhobarss(\gamma_1)), \ldots, \pi_G(\rhobarss(\gamma_N)))\in (G/G^0)^N_{\XgpsGtau}(k)$ under $\alpha$ 
we conclude that $\underline{\delta}=(\pi_G(\rhobarss(\gamma_1)), \ldots, \pi_G(\rhobarss(\gamma_N)))$.

The image of $\alpha$ is contained in $\underline{\delta}$, which implies the first 
equality. The second equality follows from the symmetry of a tensor product. The last part 
 follows from a manipulation with fibre products. 
\end{proof}

\begin{lem}\label{shrink} Let $H$ be a closed subgroup scheme of $G$, such that $H^0=G^0$. If 
$(\pi_G\circ\rhobarss)(\Gamma) \subseteq (H/H^0)(k)$ then 
$ X^{\gps}_{H, \rhobarss}= X^{\gps}_{G, \rhobarss}$ and $X^{\gen}_{H, \rhobarss}= X^{\gen}_{G, \rhobarss}.$
Moreover, there exists $H$ as above, such that $\pi_H \circ \rhobarss: \Gamma \rightarrow (H/H^0)(\kbar)$
is surjective. 
\end{lem} 
\begin{proof} We may write $X^{\gen,\tau}_{H}$ as the fibre product of $X^{\gen,\tau}_G$ and $(H/H^0)^N$ 
over $(G/G^0)^N$. The assumption implies that  $X^{\gen,\tau}_{H}$ contains $X^{\gen,\tau}_{G, \underline{\delta}}$ as a closed subscheme with $\underline{\delta}$ as in Lemma \ref{connected_Delta}.
The assertion follows from Lemma \ref{connected_Delta}. Since $H^0=G^0$ by passing to GIT quotients 
we obtain $X^{\gps}_{H, \rhobarss}=X^{\gps}_{G, \rhobarss}$. For the last part 
we may take $H$ to be the preimage of $\underline{\Delta}$ in $G$. 
\end{proof}
 
If $H$ is a closed generalised reductive subgroup scheme of $G$ (resp.\,$G_k$) we let $X^{\gps,\tau}_{HG}$ (resp.\,$\overline{X}^{\gps,\tau}_{HG}$) be the scheme theoretic image of 
$\XgpsHtau\rightarrow \XgpsGtau$ (resp.\,$\XbargpsHtau \rightarrow \XbargpsGtau$). Since these maps are 
finite by Corollary \ref{finite_HG} we have 
\begin{equation}\label{inductive_bound} 
\dim X^{\gps,\tau}_{HG} \le \dim \XgpsHtau, \quad \dim \overline{X}^{\gps,\tau}_{HG} \le \dim \XbargpsHtau. 
\end{equation}

\begin{remar}\label{remar_XgittauHG_bar}
    We note that if $H$ is a closed generalised reductive subgroup of $G$ then the special fibre of $X^{\gps,\tau}_{HG}$ is homeomorphic to $\overline{X}^{\gps,\tau}_{H_kG}$, see Remark \ref{homeo}. We will denote both by $\overline{X}^{\gps,\tau}_{HG}$. This is harmless, since in our arguments we only care about dimensions.
\end{remar}

We denote by $X^{\gen}_{H, \rhobarss}$ (resp. $X^{\gps}_{H, \rhobarss}$, $X^{\gps}_{HG, \rhobarss}$) the schemes $(X^{\gen, \tau}_H)_{\rhobarss}$ (resp. $(X^{\gps, \tau}_H)_{\rhobarss}$, $(X^{\gps, \tau}_{HG})_{\rhobarss}$) with the notation introduced in \eqref{eq_notation}.
 \Cref{connected_Delta} implies that these schemes are independent of $\tau$ and only depend on the map $H \to G$.

\begin{lem}\label{cover_Hi} Let $H$ be a generalised reductive subgroup scheme of $G$ (resp.~$G_k$). If 
$X^{\gps}_{HG, \rhobarss}$ (resp.~$\Xbar^{\gps}_{HG, \rhobarss}$) is non-empty then 
there exist 
\begin{itemize}
\item[(a)] a finite unramified extension $L'$ of $L$ with ring of integers $\OO'$ and residue field $k'$;
\item [(b)] finitely many generalised reductive subgroup schemes $H_i$ of $H_{\OO'}$;
\item [(c)] continuous representations 
$\rhobar_i: \Gamma \rightarrow H_i(k')$;
\end{itemize}
such that the following hold: 
\begin{enumerate} 
\item $H_i^0=H^0_{\OO'}$ and $\pi_{H_i}\circ \rhobar_i: \Gamma \rightarrow (H_i/H_i^0)(\overline{k})$ is surjective;
\item $(X^{\gps}_{H, \rhobarss})_{\OO'}$ (resp.\,$(\Xbar^{\gps}_{H, \rhobarss})_{k'}$) is a disjoint union 
of $X^{\gps}_{H_i, \rhobarss_i}$ (resp.\,$\Xbar^{\gps}_{H_i, \rhobarss_i}$);
\item $(X^{\gen}_{H, \rhobarss})_{\OO'}$ (resp.\,$(\Xbar^{\gen}_{H, \rhobarss})_{k'}$)  is a disjoint union of $X^{\gen}_{H_i, \rhobarss_i}$ (resp.\,$\Xbar^{\gen}_{H_i, \rhobarss_i}$);
\item $(X^{\gps}_{HG, \rhobarss})_{\OO'}$ (resp.\,$(\Xbar^{\gps}_{HG, \rhobarss})_{k'}$) is a union of scheme theoretic images 
$$X^{\gps}_{H_i, \rhobarss_i}\rightarrow (X^{\gps,\tau}_G)_{\OO'},\quad 
\text{(resp. }\Xbar^{\gps}_{H_i, \rhobarss_i}\rightarrow (X^{\gps,\tau}_G)_{k'}).$$
\end{enumerate} 
\end{lem}

\begin{proof} The map $X^{\gps,\tau}_H \rightarrow X^{\gps,\tau}_G$ corresponds to the map 
between finite products of local rings $\prod_{i\in I} R_i \rightarrow \prod_{j\in J} R_j'$ by \eqref{semi-local}. The representation $\rhobarss$ defines an idempotent $e\in\prod_i R_i$, such that 
$e(\prod_i R_i)= R^{\gps}_{G, \rhobarss}$. Then $X^{\gps}_{H, \rhobarss}$ is the 
spectrum of $e(\prod_{j\in J} R_j')$. 
Since $R_j'$ are local rings we have either $e(R_j')=R'_j$ or $e(R'_j)=0$. Thus $e(\prod_{j\in J} R_j')=
\prod_{j\in J'} R'_j$ for some subset $J'\subseteq J$. We may identify $J'$ with the
set of $\Gal(\kbar/k)$-orbits in $X^{\gps}_{H, \rhobarss}(\kbar)$.

After replacing $L$ by a finite unramified extension we may ensure that every 
connected component of $X^{\gen}_{H, \rhobarss}$ has a $k$-rational point, see 
Remark \ref{rem_k_rat_comp}. Then \eqref{semi-local} yields decompositions:
$$X^{\gps}_{H, \rhobarss}\cong \coprod_i X^{\gps}_{H, \rhobar^{\mathrm{ss}}_i},\quad X^{\gen}_{H, \rhobarss}\cong \coprod_i X^{\gen}_{H, \rhobar^{\mathrm{ss}}_i}$$
for finitely many continuous representations $\rho_i:\Gamma \rightarrow H(k)$.
The existence of $H_i$ follows from Lemma \ref{shrink}. Parts (2), (3) and (4) then follows immediately. If $H$ is a closed subgroup of $G_k$ the argument is the same.
\end{proof}

\subsection{Absolutely irreducible locus}\label{sec_abs_irr_loc} In this subsection we omit the fixed embedding 
$\tau:G\hookrightarrow \GL_d$ from the notation and write $X^{\gen}_G$, $X^{\gps}_G$, 
$X^{\gen}_{HG}$ and $X^{\gps}_{HG}$ for $X^{\gen,\tau}_G$, $X^{\gps,\tau}_G$, 
$X^{\gen,\tau}_{HG}$ 
and $X^{\gps,\tau}_{HG}$. We note that as a consequence of Corollary \ref{indep_tau} 
$X^{\gen}_{HG,\rhobarss}$ and $X^{\gps}_{HG,\rhobarss}$ are independent of $\tau$. In particular, 
the schemes $U_{LG,\rhobarss}$ and $V_{LG, \rhobarss}$ defined below are independent of $\tau$. 

It follows from Proposition \ref{GGLd_finite} that $\XgpsG$ has only finitely many closed 
points, we denote this subscheme by $\pts$. We let $Y$ be the preimage of these 
closed points in $\XgenG$. Then $Y$ is a closed 
subscheme  of $\XgenG$ contained in $\XbargenG$. 

We assume for the rest of this subsection that $G^0$ is split and fix a maximal split torus $T$ 
defined over $\OO$. Let $L$ be an R-Levi subgroup of $G$ containing $T$. Then we define
$$U_{LG}:=X^{\gps}_{LG} \setminus ( \pts \cup \bigcup_{M} X^{\gps}_{MG}),\quad \Ubar_{LG}:=\Xbar^{\gps}_{LG} \setminus ( \pts \cup \bigcup_{M} \Xbar^{\gps}_{MG})$$
where the union is taken over all R-Levi subgroups $M$ of $G$, such that $T\subseteq M \subsetneq L$. Lemma \ref{Rlevi} implies that there are only finitely many such $M$. Thus $U_{LG}$ is 
an open subscheme of $X^{\gps}_{LG} \setminus \pts$ and  $\Ubar_{LG}$ is 
an open subscheme of $\Xbar^{\gps}_{LG} \setminus \pts$. Hence, $U_{LG}$ and $\Ubar_{LG}$ are  locally closed
subschemes of $X^{\gps}_G\setminus \pts$.  We let $V_{LG}$ be the preimage of $U_{LG}$ 
in $\XgenG$ and $\Vbar_{LG}$ be the preimage of $\Ubar_{LG}$ 
in $\XgenG$. We then have 
\begin{equation} 
\XgenG= Y \cup \bigcup_{L} V_{LG}, \quad \XbargenG= Y\cup \bigcup_{L} \overline{V}_{LG},
\end{equation} 
where the union is taken over all $R$-Levi subgroups $L$ of $G$ containing $T$.

\begin{defi} We say that a geometric point $x\in X^{\gen}_G$ is \emph{$G$-stable}, if 
the orbit $G\cdot x$ is closed and $\dim Z_G(x) = \dim Z(G)_{\kappa(x)}$, where $Z_G(x)$ is the $G_{\kappa(x)}$-centraliser of $x$. 
\end{defi}

\begin{remar} The action of $G$ on $X^{\gen}_G$ factors through the action of $G/Z(G)$. The 
stability condition in the definition is equivalent to $G\cdot x$ closed and 
$\dim G\cdot x = \dim 
(G/Z(G))_{\kappa(x)}$. Hence, it is equivalent to $G/Z(G)$-stability in the sense of  \cite[Definition 1]{seshadri}.
\end{remar} 

\begin{prop}\label{irr=stable} Let $x$ be a geometric point of $X^{\gen}_G\setminus Y$, and let $H$ be 
the Zariski closure of $\rho_x(\Gamma)$ in $G_{\kappa(x)}$. 
Then the following statements are equivalent:
\begin{enumerate}
\item $x\in V_{GG}$;
\item $\rho_x$ is $G$-irreducible;
\item $H$ is $G$-irreducible;
\item $x$ is $G$-stable.
\end{enumerate}
Moreover, $H$ is reductive and $\dim Z_G(H)= \dim Z(G_{\kappa(x)})$. 
\end{prop}
\begin{proof} 
The equivalence of (1) and (2) follows from Corollary  \ref{conjugate_Levi}. The equivalence
of (2) and (3) follows from Definition \ref{defi_G_irr}. The equivalence of 
(3) and (4) follows from \cite[Proposition 16.7]{Richardson1988ConjugacyCO}. The last part 
follows from \cite[Lemma 6.2]{martin}.
\end{proof}

\begin{cor}\label{dim_Z_same} Let $L$ be an R-Levi of $G$ and let $H$ be a closed generalised reductive subgroup of $L_{k'}$, where 
$k'$ is a finite extension of $k$. If $\Xbar^{\gps}_{HG}\cap U_{LG}$ is non-empty then  $\dim Z(H)\le \dim Z(L_k)$. 
\end{cor}
\begin{proof} After extending scalars we may assume that $k=k'$. Since $H$ is contained
in $L$ the map $\Xbar^{\gps}_H \rightarrow \Xbar^{\gps}_G$ factors through 
$\Xbar^{\gps}_H \rightarrow \Xbar^{\gps}_L\rightarrow \Xbar^{\gps}_G$. The assumption that 
$\Xbar^{\gps}_{HG}\cap U_{LG}$ is non-empty implies that $\Xbar^{\gps}_{HL} \cap U_{LL}$ is non-empty. 
It follows from Proposition \ref{irr=stable} that $H_{\kappa(x)}$ is not contained 
in any R-parabolic subgroup of $L_{\kappa(x)}$, where $x$ is any geometric point of $V_{LL}$, 
such that its image in $\Xbar^{\gps}_L$ is contained in $\Xbar^{\gps}_{HL}$. 
Moreover, we have 
$$\dim Z(H)= \dim Z(H_{\kappa(x)})\le \dim Z(L_{\kappa(x)}) = \dim Z(L_k),$$
where the inequality in the middle follows from \cite[Lemma 6.2]{martin}, as the centre of $H_{\kappa(x)}$ is 
contained in the $L_{\kappa(x)}$-centraliser of $H_{\kappa(x)}$,
and the other two follow from Lemma \ref{base_change_Z}. 
\end{proof} 

\begin{cor}\label{dim_fibre_GG} Let $X^{\gen}_{G, y}$ be the fibre at a 
point $y\in U_{GG}$. Then 
$$\dim X^{\gen}_{G, y}= \dim G -\dim Z(G)=\dim G_k - \dim Z(G_k).$$
\end{cor}

\begin{proof} Let $\kappa$ be an algebraic closure of $\kappa(y)$. If $x\in X^{\gen}_{G, y}(\kappa)$ then 
the $G$-orbit $G\cdot x$ is closed, as $x$ is $G$-stable 
by Proposition \ref{irr=stable}. Proposition \ref{closed_orbit} implies that the
$G^0$-orbit $G^0\cdot x$ is also closed. Since $X^{\gen}_{G, y}$ contains 
a unique closed $G^0$-orbit by Proposition \ref{orbit_ss} we conclude 
that $G^0(\kappa)$ acts transitively on $X^{\gen}_{G, y}(\kappa)$.
Thus the orbit map 
$$G^0_{\kappa}/ Z_{G^0_{\kappa}}(x) \rightarrow X^{\gen}_{G, y, \kappa}$$
induces a bijection on $\kappa$-points. Since both schemes are 
of finite type over $\kappa$, we deduce that 
$\dim X^{\gen}_{G, y}= \dim G^0_{\kappa}/ Z_{G^0_{\kappa}}(x)= \dim G^0_{\kappa} - \dim Z_{G^0_{\kappa}}(x)$.
 By passing to the underlying reduced subschemes of neutral components 
we obtain $(Z_{G^0_{\kappa}}(x)^0)^{\red} = (Z_{G_{\kappa}}(x)^0)^{\red}= (Z(G)_{\kappa}^0)^{\red}$, where the last equality 
holds as $x$ is $G$-stable.  Thus $\dim Z_{G^0_{\kappa}}(x)= \dim Z(G)_{\kappa}$ and we conclude
that 
$$\dim X^{\gen}_{G, y} = \dim G_{\kappa} - \dim Z(G)_{\kappa}.$$
The other assertions follow from \Cref{dimkisdimL} and \Cref{ZGexists}.
\end{proof}

\begin{cor}\label{UV_GG} 
Let $U$ be a non-empty locally closed subspace of $U_{GG}$ (resp.~$\Ubar_{GG}$) and 
let $V$ be the preimage of $U$ in $X^{\gen}_G$ (resp.~$\Xbar^{\gen}_G$). Then 
$$\dim V = \dim U +\dim G_k -\dim Z(G_k).$$
\end{cor}
\begin{proof} Corollary \ref{dim_fibre_GG} says that $\dim V_{y}=d$, where
$d:=\dim G_k -\dim Z(G_k)$, for all $y\in U$. The inequality 
$\dim V \le \dim U +d$
follows directly from \cite[Lemma 3.18 (6)]{BIP_new} (or alternatively \cite[\href{https://stacks.math.columbia.edu/tag/0BAG}{Tag 0BAG}]{stacks-project}). 

To prove the reverse inequality we may assume that $U$ is irreducible and, since 
$V\rightarrow U$ is surjective by \cite[Theorem 3 (ii)]{seshadri}, we may replace $V$ 
by an irreducible component, whose image contains the generic point of $U$. Lemma 2.1 
in \cite{BIP_new} implies that $V$ is $G^0$-invariant and \cite[Theorem 3 (iii)]{seshadri} 
implies that the image of $V$ in $U$ is closed, and hence $V\rightarrow U$ is surjective.
It follows from \cite[\href{https://stacks.math.columbia.edu/tag/02JU}{Tag 02JU}]{stacks-project}
applied with $s\in U$ a closed point such that $\dim U=\dim \OO_{U, s}$ and 
$x\in V$ a closed point, which maps to $s$, that $\dim V\ge \dim U+d$. 
\end{proof}

 \begin{lem}\label{V_LG_maxdim} Assume that $\pi_G\circ \rhobarss: \Gamma \rightarrow (G/G^0)(\kbar)$ is surjective. 
Let $L$ be an R-Levi of $G$ containing $T$ and let $P$ be an R-parabolic subgroup of $G$ with 
R-Levi $L$ and unipotent radical $N$. 
Assume that one of the following is non-empty: 
$X^{\gen}_{L, \rhobarss}$, $X^{\gps}_{L, \rhobarss}$, $V_{LG, \rhobarss}$, $U_{LG, \rhobarss}$. 
Then the following hold:
\begin{enumerate} 
\item $L/L^0\rightarrow G/G^0$ is an isomorphism;
\item if $\dim L=\dim G$ then $L=G$;
\item If $L\neq G$ then $\dim  N_k\ge 1$.
\end{enumerate} 
\end{lem}
\begin{proof} It follows from Lemma \ref{para} that $L^0= L\cap G^0$, so that the map in (1) is injective. The assumption in all the cases implies that $X^{\gps}_{L, \rhobarss}$ is non-empty. Lemma \ref{connected_Delta} implies that
$X^{\gps}_{L, \underline{\delta}}$ is non-empty, where $\underline{\delta}$ is as in Lemma 
\ref{connected_Delta}. Thus $\underline{\delta}\in (L/L^0)(k)^N$ and thus 
$\pi_G(\rhobarss(\gamma_i))\in (L/L^0)(k)$
for all $1\le i \le N$. Since $\pi_G\circ \rhobarss$ is continuous and surjects onto $(G/G^0)(\kbar)$ and $\gamma_1, \ldots \gamma_N$ generate
$\Gamma$ topologically, the entries of $\underline{\delta}$ generate $(G/G^0)(\kbar)$ as a group. Thus 
$(L/L^0)(k)=(G/G^0)(k)=(G/G^0)(\kbar)$. Thus $G/G^0$ is a constant group scheme and 
 we deduce part (1).

As explained in Lemma \ref{para},  $P^0$ is a parabolic of $G^0$
with Levi $L^0$. Thus $\dim L=\dim G$ implies that $\dim L^0= \dim G^0$ and hence $L^0=G^0$.
Combining this with part (1) we deduce part (2). 

Part (3) follows from  $\dim G_k - \dim L_k= \dim G_k^0 - \dim L_k^0 = 2 \dim N_k$. 
\end{proof} 

\begin{defi}\label{defi_abs_irr} If $\pi_G\circ \rhobarss: \Gamma \rightarrow (G/G^0)(\kbar)$ is surjective then we define 
the \emph{absolutely $G$-irreducible} locus in $X^{\gen}_{G, \rhobarss}$ as $V_{GG,\rhobarss}$. 
\end{defi}

\begin{remar}\label{rem_bad_def} The justification for calling $V_{GG,\rhobarss}$ absolutely $G$-irreducible locus comes 
from Proposition \ref{irr=stable}. We are not sure if this is a good definition if $\pi_G\circ 
\rhobarss$ is not surjective, and we discuss an example below. However, one can get around the problem by replacing $G$ by a closed subgroup scheme
$H$ such that $G^0=H^0$ and $\pi_G\circ \rhobarss$ surjects onto $(H/H^0)(\kbar)$ and using Lemma \ref{shrink}.
\end{remar}

\begin{examp}\label{ex_bad_def} Assume that $p>2$ and let $G$ be the normaliser of the subgroup of diagonal matrices in $\GL_2$. Then 
$G\cong \Gm^2\rtimes \ZZ/2\ZZ$ as in the Example \ref{ex_comp}. Let $E$ be a quadratic Galois extension 
of $F$ and let $c\in \Gal(E/F)$ be non-trivial. Let $\rhobar_1, \rhobar_2: \Gamma_F \rightarrow G(k)$ be 
representations, such that $\Gamma_E$ is mapped to identity and 
$$\rhobar_1(c)= \begin{pmatrix} 1 & 0 \\ 0 & -1\end{pmatrix}, \quad \rhobar_2(c)= \begin{pmatrix} 0 & 1\\ 1 & 0 \end{pmatrix}.$$
These representations are $G$-semisimple. 
Since $\rhobar_1(c)$ and $\rhobar_2(c)$ are conjugate in $\GL_2(k)$ the representations $\tau\circ \rhobar_1$ 
and $\tau\circ \rhobar_2$ will have the same $\GL_2$-pseudocharacter, where $\tau: G\hookrightarrow \GL_2$ is
the natural inclusion. The image of $\rhobar_1$ is contained
in $G^0(k)$ and $G^0$ is a proper $R$-parabolic subgroup of $G$, as explained in Example \ref{ex_comp}, hence 
$V_{GG, \rhobar_1}$ is empty, and hence the $G$-irreducible locus in $X^{\gen}_{G, \rhobar_1}$ is also empty. 
On the other hand, $V_{GG, \rhobar_2}$ is non-empty. We also would like to point out that in this 
case $Z(G^0)$ and $Z(G)$ have different dimensions, and this results in a difference between the dimensions
of $V_{G^0G, \rhobar_1}\sslash G^0$ and  $V_{GG, \rhobar_2}\sslash G^0$.
\end{examp}

\begin{remar} The referee suggested  getting around the problem 
explained in \Cref{rem_bad_def} and \Cref{ex_bad_def} by working only with $G$-normalisers of parabolic subgroups of $G^0$ (which are $R$-parabolic by \cite[Proposition 6.1]{BMR}) instead of all $R$-parabolic subgroups. This definition would
save some work in \Cref{sec_trimming}, but would not match the standard definition for $G$-irreducibility in the literature such as \cite{BMR}, which we use. If $\pi_G \circ \rhobarss$ is surjective then both definitions of irreducibility coincide. 
\end{remar}

 \section{Bounds for the dimensions of reductive subgroups}\label{bounds_red_sbgp}

 Let $G$ be a generalised reductive group over an an algebraically closed field $\kappa$.
 If $H$ is a closed reduced subgroup scheme of $G$, such that the unipotent 
 radical of $H^0$ is trivial then $H$ is also generalised reductive by \Cref{gen_red_field}. Let $G'$ denote 
 the derived subgroup of $G^0$. Then $G'$ is semisimple and $G^0= Z(G^0) G'$, \cite[14.2]{borel}. Moreover, $G'$ is equal to its 
 own derived subgroup, \cite[Example 1.1.16]{conrad}.
 
 \begin{prop}\label{bound_subred} 
 Let $H$ be a closed reductive subgroup of 
 $G$. Then one of the following holds: 
 \begin{itemize} 
 \item[(i)] $H$ contains $G'$;
 \item[(ii)] $\dim G' - \dim (H\cap G')\ge 2$.
 \end{itemize}
 \end{prop}
 \begin{proof} After replacing $G$ by $G'$ we may assume that 
 $G$ is connected and semisimple. One may check that 
 the unipotent radical of $(G'\cap H)^0$ is normal in $H^0$. Since $H^0$ is reductive, $(G'\cap H)^0$ is also reductive. Thus after replacing $H$ by 
 $H\cap G'$ we may assume that $H$ is closed 
 reductive subgroup 
 of $G$.
 
 Let $B_0=T_0 U_0$ be a Borel subgroup of $H$ with the unipotent radical $U_0$ and a maximal torus $T_0$. Let $B=T U$ be a Borel 
 subgroup of $G$ with the unipotent radical $U$ and a maximal torus $T$, such that $U$ contains
 $U_0$ and $T$ contains $T_0$. We have $\dim G = 2\dim U + \dim T$ and $\dim H = 2 \dim U_0 +\dim T_0$. If $U\neq U_0$ then $\dim U >\dim U_0$ and we deduce that $\dim G -\dim H \ge 2$. 
 
 Let us assume that $U=U_0$ and let $U^-_0$ be the unipotent radical of a Borel subgroup of $H$, which is opposite to $B_0$ with respect to $T_0$. Since $U_0=U$, $U_0^-$ is a
 unipotent radical of a Borel subgroup of $G$. 
 Since $U_0^-\cap U_0=\{e\}$,  \cite[Proposition 4.10 e)]{BoTi} implies that $U^-_0$ is a unipotent radical of a parabolic subgroup of $G$, which is opposite to $B$ with respect to $T$. Since $G$ is by assumption semisimple, $U_0$ and $U_0^{-}$ generate
 it as a group by \cite[Theorem 8.1.5(i)]{springer}. Thus $G=H$.
 \end{proof} 
 
 \begin{cor}\label{cor_one} Assume that we are 
 in the situation of part (ii) of Proposition \ref{bound_subred}. Then $\dim G -\dim H \ge 2$.
 \end{cor}
 \begin{proof} The $G'$-orbit of 
 $H$ in $G/H$ is isomorphic to $G'/(G'\cap H)$ and hence 
 $$\dim G/H \ge \dim G'/(G'\cap H)\ge 2.$$ 
 \end{proof}
 
 \begin{cor}\label{cor_two} Let $V$ be a finite dimensional
 representation of $G$ and let $H$ be a closed reductive subgroup of $G$. If $V^{G'}=0$ and $V^H\neq 0$ then $\dim G - \dim H\ge 2$.
 \end{cor}
 \begin{proof} If $H$ contains $G'$ then $V^H$ is 
 contained in $V^{G'}$ and so $V^H\neq 0$ implies that 
 $V^{G'}\neq 0$. The assertion follows from Corollary \ref{cor_one}.
  \end{proof} 
 
 Since $G'$ is a normal subgroup of $G$, $G$ acts on $\Lie(G')$ by conjugation. 
 
 \begin{lem}\label{coadjoint} If $\pi_1(G')$ is \'etale then $((\Lie G')^*)^{G'}=0$.
 \end{lem}
 \begin{proof} We may assume that $\Lie G'$ is non-zero. 
 Let $G'_{\sic} \rightarrow G'$ be the universal simply connected cover of $G'$. Since $\pi_1(G')$ is \'etale the exact sequence:
 $$ 0 \rightarrow \pi_1(G')\rightarrow G'_{\sic}\rightarrow 
 G'\rightarrow 0$$ 
induces a $G'_{\sic}$-equivariant  isomorphism of Lie algebras $\Lie G'_{\sic}\cong \Lie G'$. We thus may assume 
that $G'$ is simply connected.  Let $T$ be a maximal torus in $G'$.

 If $\dim T=1$ then $G'=\SL_2$ by  \cite[7.2.4]{springer}, as the assumption that $G'$ is simply connected rules out $\PGL_2$. If $\cha(\kappa) \neq 2$ then the action of $G'$
 on $\Lie G'$  is an irreducible representation, isomorphic to 
 $\Sym^2(\Std)$, where $\Std$ is the standard $2$-dimensional representation of $\SL_2$.
If $\cha(\kappa)=2$ then $\Lie G'$  is a non-split extension 
 $0\rightarrow \Eins \rightarrow \Lie G'
 \rightarrow V\rightarrow 0$, with $\Eins$
 equal to the subspace of scalar matrices and $V$ an irreducible 
 $2$-dimensional representation.
 In both cases we obtain the claim. 
 
 Let $\dim T$ be arbitrary and let $\chi: \Lie G'\rightarrow \Eins$ be $G'$-invariant. 
 Let us abbreviate $\mathfrak g= \Lie G'$.
 We have $\mathfrak g= \mathfrak g^T \oplus \bigoplus\nolimits_{\alpha\in \Phi} \mathfrak g_{\alpha}$, where $\Phi$ is the set of roots, and $\mathfrak g^T= \Lie(T)$, \cite[13.18(1)]{borel}.  
 Let $\alpha\in \Phi$ and let $G_{\alpha}$ be the centraliser of $(\Ker \alpha)^0$ in $G'$. Then $G_{\alpha}$ is reductive of semisimple rank $1$, \cite[13.18(4)]{borel}. Let $G'_{\alpha}$ be the derived subgroup of $G_{\alpha}$. The assumption 
 that $G'$ is simply connected implies that 
 $G'_{\alpha}\cong \SL_2$ by \cite[5.3.9]{conrad}.
 It follows from 
 \cite[1.2.7]{conrad} that $\Lie G_{\alpha}'= 
 \mathfrak g_{-\alpha}\oplus \Lie (\alpha^{\vee}(\Gm))\oplus
 \mathfrak g_{\alpha}$. The restriction of $\chi$ to 
 $\Lie G'_{\alpha}$ is $G'_{\alpha}$-invariant and hence zero 
 by the previous part. 
 Hence, $\chi$ is zero on $\Lie (\alpha^{\vee}(\Gm))$ and $\mathfrak g_{\alpha}$ for every coroot $\alpha^{\vee}$. 

 Let $T_1=\prod_{\alpha\in\Delta}\Gm$, where $\Delta$ is the set of positive simple roots 
 with respect to a Borel subgroup containing $T$. Since $G'$ is simply connected 
 the map $\prod_{\alpha\in \Delta} \alpha^{\vee}: T_1\rightarrow T$ induces an isomorphism
 on the character groups, \cite[Example 1.3.10]{bcnrd}, and hence an isomorphism between tori. Thus  
 the subspaces $\Lie(\alpha^{\vee}(\Gm))$ for $\alpha\in \Delta$ span $\Lie (T)$. Thus $\chi$ vanishes on $\mathfrak g^T$ and $\mathfrak g_{\alpha}$ for every $\alpha\in \Phi$ and hence $\chi=0$.
 \end{proof} 

\begin{lem}\label{sic} Let $G'_{\sic}\rightarrow G'$ be the simply connected central cover of $G'$. Then there is 
a natural action of $G$ on $\Lie ( G'_{\sic})$ such that the map 
$\Lie ( G'_{\sic})\rightarrow \Lie G'$ is $G$-equivariant. Moreover, $G'$-invariants in 
$(\Lie ( G'_{\sic}))^{\ast}$ are zero. 
\end{lem}

\begin{proof}
    The action has been defined over $\OO$ in \Cref{Gprimesc_action} and the same holds over $\kappa$.
    The $G'_{\sic}$-invariants in $(\Lie(G'_{\sic}))^*$ are zero by \Cref{coadjoint}. 
    The assertion follows as $G'_{\sic} \to G'$ is a central extension and therefore this action factors through the action of $G'$.
\end{proof}

 Let $P$ be an $R$-parabolic subgroup of $G$,  let $U$ its unipotent radical and let $L$ be an $R$-Levi subgroup of $P$. Then 
 $P^0= P\cap G^0$ is a parabolic subgroup of $G^0$ and
 $L^0=L\cap G^0$ is a Levi subgroup of $P^0$, \cite[Proposition 5.2]{martin}. In particular, $L^0$
 is reductive. We let $L'$ be the derived subgroup of $L^0$. 
 
 \begin{lem}\label{HL} Let $H$ be a closed reductive subgroup of $L$.
 If $((\Lie U)^*)^H\neq 0$ then one of the following holds:
 \begin{itemize}
 \item[(i)] $\dim L - \dim H \ge 2$;
 \item[(ii)] $\dim L -\dim H =1$ and $\dim_{\kappa} ((\Lie U)^*)^H=1$.
 \end{itemize}
 \end{lem}
 \begin{proof} We may assume that  $H$ contains $L'$ since otherwise (i) holds by Corollary \ref{cor_one}. By replacing $H$ by $H\cap L^0$
 and $L$ by $L^0$ we may assume that $L$ is connected. Then 
 $L=Z(L)H$. Let $T$ be a maximal torus in $L$. Then $T$ contains
 $Z(L)$ and hence $L/H \cong T/(T\cap H)$. The action of $T$ 
 on $\Lie U$ is semisimple, multiplicity free and $(\Lie U)^T=0$, \cite[13.18]{borel}. 
 Thus $\dim L > \dim H$. If $\dim_{\kappa} ((\Lie U)^*)^H\ge 2$, then 
 $\dim_{\kappa} ((\Lie U)^*)^{T\cap H}\ge 2$. Let $\alpha,\beta\in \Phi$ be distinct roots such that the dual characters $\alpha^*$ and 
 $\beta^*$
 appear in $((\Lie U)^*)^{T\cap H}$. We note that both 
 $\alpha$ and $\beta$ are positive with respect to a Borel
 subgroup of $G'$ containing $T$ and $U$. Thus $\alpha\neq -\beta$. Now $\{\alpha, -\alpha\}$ are the roots of
 the centraliser of $(\Ker \alpha)^0$ in $G'$, \cite[13.18(4d)]{borel}. Thus $\Ker \alpha \neq \Ker \beta$
 and hence $\dim T - \dim (\Ker \alpha \cap \Ker \beta) = 2$. Thus $\dim T -\dim T\cap H\ge 2$, as $T\cap H$ is contained in $\Ker \alpha \cap\Ker \beta$. 
 Thus part (i) holds. 
 \end{proof}
 
 \begin{cor}\label{Lzero} $((\Lie U)^*)^{L^0}=0$. 
 \end{cor}
 \begin{proof} This follows from Lemma \ref{HL} as $\dim L=\dim L^0$. 
 \end{proof} 
 
 \begin{examp} Let $G=G'= G_1\times \SL_2$, where $G_1$ is any semisimple group, 
 $L= G_1\times \Gm$ and $P= G_1\times B$, where $B$ is the subgroup of upper triangular 
 matrices in $\SL_2$. If $H=G_1\times\{1\}$ then part (ii) of Lemma \ref{HL} holds.
 \end{examp}

\section{\texorpdfstring{R-Levi subgroups of codimension $2$}{R-Levi subgroups of codimension 2}}\label{RLevi2} 

Let $G$ be a generalised reductive group over an algebraically closed field $\kappa$. 
Let $P$ be an R-parabolic subgroup of $G$ with R-Levi $L$ and unipotent radical $U$. 
We have $\dim G - \dim L = 2 \dim U$. In this section we assume that $\dim U=1$ and 
the natural map $L/L^0 \rightarrow G/G^0$ is an isomorphism. The main result of this 
section is Proposition \ref{vanish_H2_please}, which is a criterion for a representation 
$\rho: \Gamma_F \rightarrow L(\kappa)$ to have $h^0(\Gamma_F, (\Lie G'_{\sic})^*(1))=0$, 
where $G'_{\sic}$ is the simply connected central cover of $G'$ and the action of $G$ on 
$\Lie G'_{\sic}$ is given by Lemma \ref{Gprimesc_action}.
Proposition \ref{vanish_H2_please} is not needed to prove the complete intersection 
property of deformation rings, but is first used in  \Cref{sec:bound_special} and is needed to compute their irreducible components. 

Let $\varphi: G \rightarrow \Aut(G^0)$ be the group homomorphism induced by the action of
$G$ on $G^0$ by conjugation. If $H$ is a subgroup of $G$ then we let 
$\overline{H}$ be the image  of $H$ in $G/\Ker \varphi$. We have $G^0\cap \Ker \varphi = Z(G^0)$ 
and $U \cong \overline{U}$. Since the image of any cocharacter 
$\lambda: \Gm \rightarrow G$ is contained in $G^0$, $\Ker \varphi$ is contained 
in every R-parabolic and R-Levi subgroup of $G$. Thus $L$ is the preimage of $\overline{L}$ and 
$P$ is the preimage of $\Pbar$. 

We will let $B_2$, $T_2$ and $U_2$ be subgroups of $\PGL_2$ consisting of upper-triangular, 
diagonal and unipotent upper-triangular  matrices, respectively. We will identify $T_2$ 
with $\Gm$ via the character $\alpha$, which sends $\bigl(\begin{smallmatrix} a & 0 \\ 0 & d \end{smallmatrix}\bigr)$ to $ad^{-1}$.

\begin{prop}\label{tasho1} 
There is a generalised reductive group $G_1$ such that $G^0_1$ is semisimple, $\Gbar\cong G_1\times \PGL_2$ and this isomorphism induces isomorphisms:
$$ \Pbar\cong G_1\times B_2 
,\quad \Lbar \cong G_1\times T_2,\quad \Ubar \cong \{1\} \times U_2.$$
\end{prop}
\begin{proof} We fix a Borel subgroup $B^0$ of $G^0$ contained in $P$ and a maximal split torus 
$T$ contained in $B^0\cap L$. Let $\Phi(G^0)$ be the set of roots of $G^0$ with respect to $T$,
and let $\Phi^+(G^0)$ be the set of positive roots with respect to $(B^0,T)$. We define 
$\Phi(L^0)$, $\Phi^+(L^0)$ in the same way with respect to $(B^0\cap L^0, T)$. We then have
\begin{equation}\label{dimGroots}
\dim G =\dim_{\kappa} \Lie G = \dim T +\sum_{\alpha\in \Phi(G^0)} \dim_{\kappa} \mathfrak g_{\alpha},
\end{equation}
\begin{equation}\label{dimLroots}
\dim L = \dim_{\kappa} \Lie L = \dim T + \sum_{\alpha\in \Phi(L^0)} \dim_{\kappa} \mathfrak g_{\alpha}.
\end{equation}
If $\Phi^+(L^0)$ contains all simple roots of $\Phi^+(G^0)$ then $\Phi(L^0)=\Phi(G^0)$ and hence $L^0=G^0$. Thus there is a 
simple root $\beta\in \Phi^+(G^0)$, such that $\beta\not\in \Phi^+(L^0)$. It follows from 
\eqref{dimGroots}, \eqref{dimLroots} and the assumption $\dim G-\dim L=2$ that $\Phi(G^0)$ is a disjoint union of $\Phi(L^0)$ and $\{\beta, -\beta\}$. Let $U_{\alpha}$ be the root subgroup of $G^0$ corresponding to $\alpha$. If $\alpha\in \Phi(L^0)$ then $r \alpha \pm s\beta$ is not a root of $G$ for any integers $r, s>0$, 
it follows from assertion $(\ast)$ in the proof of \cite[14.5]{borel} that the commutator $(U_{\alpha}, U_{\pm \beta})$ vanishes.  

We have $\Gbar{}^0\cong G^0/Z(G^0)$, and hence $\Gbar{}^0$ is semisimple by \cite[Corollary 14.2(a)]{borel}. Let $H_1$ be the closed subgroup of $\Gbar{}^0$ generated by the subgroups $U_{\alpha}$ for 
$\alpha\in \Phi(L^0)$ and let $H_2$ be the closed subgroup of $\Gbar{}^0$ generated by $U_{\beta}$ and $U_{-\beta}$. Arguing as in the proof of \cite[Proposition 14.2 (2)]{borel} we deduce that 
$H_1$ centralises $H_2$, $H_1$ and $H_2$ generate $\Gbar{}^0$ as a group, and $H_1\cap H_2$ is contained 
in the centre of $\Gbar{}^0$. Since the centre of $\Gbar{}^0$ is trivial, we deduce that 
$\Gbar{}^0\cong H_1\times H_2$. Moreover, by \cite[Lemma 1.2.3]{bcnrd} there is a central isogeny 
$H_2\rightarrow \PGL_2$. The centre of $H_2$ lies in the centre of $\Gbar{}^0$ and hence
is trivial, so that $H_2\cong \PGL_2$.

For each simple root $\alpha$  in $\Phi^+(G^0)$ we fix an isomorphism $p_{\alpha}: \Ga \overset{\cong}{\longrightarrow} U_{\alpha}$. Let $\Theta$ be the subgroup of $\Aut(G^0)$ consisting of automorphisms
that preserve the pinning $((G^0, B^0, T), \{p_{\alpha}\}_{\alpha})$. Theorem 7.1.9 in \cite{bcnrd}
gives an isomorphism of group schemes over $\kappa$: 
$$\Aut(G^0)\cong \Gbar{}^0\rtimes \Theta.$$
If $\Theta'$ is a subgroup of $\Theta$ then there is a unique subgroup $H$ of $\Aut(G^0)$, such that $H$ contains $\Gbar{}^0$, and the image of $H$ in $\Aut(G^0)/\Gbar{}^0$ is equal to $\Theta'$, 
namely the inverse image of $\Theta'$. Uniqueness implies that $H\cong \Gbar{}^0\rtimes \Theta'$.
Let $\Delta$ be the unique subgroup of $\Theta$, such that $\Delta$ maps isomorphically onto $\overline{G}/\Gbar{}^0 \subseteq \Aut(G^0)/\Gbar{}^0$. Then by the previous remark 
$\overline{G}\cong \Gbar{}^0\rtimes \Delta$. We claim that every  $\delta\in \Delta$
maps $\Phi^+(L^0)$ to itself. The claim implies that $\Delta$ fixes $\beta$, and since it fixes the 
pinning $\Delta$ acts trivially on $H_2$. Thus $\overline{G}\cong (H_1\rtimes \Delta) \times \PGL_2$
and the proposition follows.

To prove the claim we use the assumption that the map $L/L^0\rightarrow G/G^0$ is an isomorphism. 
Thus given $\delta\in \Delta$ we may find $g\in L(\kappa)$, $h_1\in H_1(\kappa)$ and $h_2\in H_2(\kappa)$, 
such that for all $u\in U_{\alpha}$ we have $\delta(u)= (h_1, h_2) g u g^{-1} (h_1, h_2)^{-1}$.
The image of $gug^{-1}$ in $\Gbar{}^0$ lies in $H_1$ as $L$ normalises $H_1$. (It follows from 
the argument in the proof of \cite[Proposition 14.10 (3)]{borel} that the closed
subgroup of $L$ generated by $U_{\alpha}$ for all $\alpha\in \Phi(L^0)$ is the derived 
subgroup of $L^0$ and thus is normal in $L$.) Thus $\delta(u)$ lies 
in $H_1$, as $h_2$  centralises $H_1$. Hence $\delta(\alpha)\neq \pm \beta$ and this finishes the proof of 
the claim. 
\end{proof}

\begin{cor}\label{tasho2} The map $G'\rightarrow \Gbar$ induces an isomorphism 
of $G$-representations: 
$$ \Lie G'_{\sic} \cong \Lie \Gbar_{\sic} \cong \Lie \SL_2 \oplus \Lie (G_1')_{\sic}.$$
\end{cor} 
\begin{proof} Since $\Gbar{}^0 \cong G^0/Z(G^0)$, the map $G'\rightarrow \Gbar{}^0$
is a central isogeny and hence they have the same simply connected central cover, which 
induces the first isomorphism. The second isomorphism follows from 
Proposition \ref{tasho1}.
\end{proof}

Let $p_2: G\rightarrow \PGL_2$ be the composition of the quotient map 
$G\rightarrow \Gbar$ with the projection map $\Gbar\rightarrow \PGL_2$ onto the second component. 
Let $H_0$ be the preimage of $\mu_{p-1}\subset T_2$ in $G$ under $p_2$. 

\begin{lem}\label{bound_centre_Hzero} $H_0$ 
is generalised reductive and $\dim Z(H_0)< \dim Z(L)$. 
\end{lem} 

\begin{proof} 
    If $H_0$ is contained in a proper R-parabolic of $L$ then the image 
    of $H_0$ is contained in a proper R-parabolic of $\overline{L}$. This would imply that 
    $Z(\Lbar) \overline{H}_0$ is contained in a proper R-parabolic of $\overline{L}$, which 
    is not possible as Proposition \ref{tasho1} implies that  $Z(\Lbar) \overline{H}_0= \Lbar$. 
    Lemma 6.2 in \cite{martin} implies that $H_0$ is reductive and $Z_{L}(H_0)^0= Z(L)^0$.
    Since $H_0\subseteq L$ we have $Z(H_0)\subseteq Z_{L}(H_0)$ and hence 
    $Z(H_0)^0 \subseteq Z(L)^0$. If $\dim Z(H_0)= \dim Z(L)$ then the quotient would be 
    $0$-dimensional and connected and hence $Z(H_0)^0 = Z(L)^0$. Thus it is enough to produce 
    a cocharacter $\lambda: \Gm\rightarrow Z(L)$, such that its image is not contained in $H_0$. 
    
    We have an exact sequence $0\rightarrow Z(G^0)\rightarrow L^0 \rightarrow \overline{L}{}^0\rightarrow 0$. 
    Since $L^0$ is reductive we have a surjection $Z(L^0)\twoheadrightarrow Z(\overline{L}{}^0)$. The group 
    of components of $L$, which we denote by $\Delta$,  acts on both groups and induces a finite map 
    on invariants $Z(L^0)^{\Delta} \rightarrow Z(\overline{L}{}^0)^{\Delta}$. 
    Hence, the quotient is $0$-dimensional and hence every cocharacter $\mu:\Gm\rightarrow Z(\overline{L}{}^0)^{\Delta}$
    can be lifted to a cocharacter $\tilde{\mu}: \Gm\rightarrow Z(L^0)^{\Delta}$. Let $\mu: \Gm \rightarrow 
    \Lbar=G_1\times \Gm$ be the cocharacter $t\mapsto (1,t)$. The image of $\mu$ is contained in 
    $Z^0(\Lbar)^{\Delta}$. Its lift $\tilde{\mu}$ takes values in $Z(L^0)^{\Delta}= Z(L)\cap L^0$, see 
    the proof of Proposition \ref{dimkisdimL}. If the image 
    of $\tilde{\mu}$ is contained in $H_0$, then the image of $\mu$ would be contained in $\overline{H}_0$, which 
    is not the case. 
\end{proof} 

The following Lemma is essentially \cite[Lemma 4.2]{BIP_new}, but made to work for all $F$ and not just 
$\mathbb{Q}_2$. 

\begin{lem}\label{BIP_Qtwo} Let $\rho: \Gamma_F \rightarrow B_2(\kappa)$ be a representation and let $\psi: \Gamma_F \rightarrow \kappa^{\times}$ be the character
$\psi: \Gamma_F \overset{\rho}{\longrightarrow} B_2(\kappa)\rightarrow B_2(\kappa)/U_2(\kappa) \overset{\alpha}{\longrightarrow} \Gm(\kappa).$
If the following hold:
\begin{enumerate}
\item $\psi \neq \omega^{\pm 1}$, where $\omega= \chi_{\cyc} \otimes_{\Zp} \kappa$;
\item $\rho(\Gamma_F)$ is not contained in any maximal split torus in $B_2$;
\end{enumerate}
then $h^0(\Gamma_F, (\Lie \SL_2)^*(1))=0$, where the action of $\Gamma_F$ on $\Lie \SL_2$ is
given by $\rho$ composed with the adjoint action of $\PGL_2$.
\end{lem}
\begin{proof} Let $M_2(\kappa)$ be the space of $2\times 2$-matrices over $\kappa$. We may identify 
$\Lie \SL_2$ with the subspace of trace zero matrices. The quadratic form $(A, B)\mapsto 
\tr(AB)$ identifies $(\Lie \SL_2)^*$ with the quotient of $M_2(\kappa)$ by the scalar matrices, 
and this isomorphism is $\PGL_2$-equivariant for the adjoint action on $M_2(\kappa)$. 
We denote the resulting representation of $\Gamma_F$ by $\overline{\ad} \rho$. 
For $i, j\in \{1,2\}$ let $e_{ij}\in M_2(\kappa)$ be the matrix with $ij$-entry equal to $1$ 
and the rest of entries equal to zero, and let $\bar{e}_{ij}$ be the image of $e_{ij}$ in 
$\overline{\ad} \rho$. Then  $\bar{e}_{12}$, $\bar{e}_{11}$, $\bar{e}_{21}$ is a basis of 
$\overline{\ad} \rho$ as a $\kappa$-vector space. 

The subgroup $\bigl( \begin{smallmatrix} 1 & \ast\\ 0 & \ast\end{smallmatrix}\bigr)\subset \GL_2$
maps isomorphically onto $B_2$. Thus there is a representation $\tilde{\rho}: \Gamma_F \rightarrow
\GL_2(\kappa)$ such that 
$ \tilde{\rho}(\gamma)= \left(\begin{smallmatrix} 1 & b(\gamma)\\ 0 & \psi^{-1}(\gamma)\end{smallmatrix}\right)$, for all $\gamma\in \Gamma_F,$
and $\rho(\gamma)$ is equal to $\tilde{\rho}(\gamma)$ modulo the scalar matrices. A direct computation of $\tilde{\rho}(\gamma) e_{ij} \tilde{\rho}(\gamma)^{-1}$, shows that 
the action of $\Gamma_F$ on $(\Lie \SL_2)^*(1)$ is given on the basis as follows: 
\begin{displaymath}
\begin{split}
 &\gamma \cdot \bar{e}_{12}= \omega(\gamma)\psi(\gamma)  \bar{e}_{12}, \quad 
\gamma \cdot \bar{e}_{11}= \omega(\gamma) \bar{e}_{11}- \omega(\gamma) \psi(\gamma)b(\gamma) \bar{e}_{12},\\
&\gamma \cdot \bar{e}_{21} = \omega(\gamma) \psi^{-1}(\gamma)\bar{e}_{21} -\omega(\gamma) \psi(\gamma) b(\gamma)^2 \bar{e}_{12}.
\end{split}
\end{displaymath}
Thus the flag $\langle \bar{e}_{12}\rangle \subset \langle \bar{e}_{12}, \bar{e}_{11}\rangle
\subset \langle \bar{e}_{12}, \bar{e}_{11}, \bar{e}_{21}\rangle$ is $\Gamma_F$-stable and
the action on the graded pieces is given by the characters $\omega \psi$, $\omega$, $\omega \psi^{-1}$. 
Since $\omega \psi^{-1}$ and $\omega \psi$ are non-trivial by assumption, if 
$h^0(\Gamma_F, (\Lie \SL_2)^*(1))\neq 0$ then $\omega$ is trivial and the $\Gamma_F$-invariants
are of the form $\langle \bar{e}_{11} + \lambda \bar{e}_{12}\rangle$ for some $\lambda\in \kappa$. 
A calculation shows that this is equivalent to $\lambda\psi(\gamma)= b(\gamma)+\lambda$ for all 
$\gamma\in \Gamma_F$, and a further calculation shows that this is equivalent to 
$\bigl (\begin{smallmatrix} 1 & \lambda\\ 0 & 1\end{smallmatrix} \bigr)^{-1} \tilde{\rho}(\gamma)
\bigl (\begin{smallmatrix} 1 & \lambda\\ 0 & 1\end{smallmatrix} \bigr)$ being a diagonal matrix 
for all $\gamma\in \Gamma_F$, which would contradict assumption (2). 
\end{proof}

\begin{prop}\label{vanish_H2_please}  Let $\rho: \Gamma_F\rightarrow P(\kappa)$ 
be a representation such that the following hold: 
\begin{enumerate}
\item $\rho(\Gamma_F)$ is not contained in an R-Levi of $P$;
\item the $G$-semisimplification of $\rho$ is not contained in $H_0$;
\item $h^0(\Gamma_F, (\Lie (G'_1)_{\sic})^*(1))=0$. 
\end{enumerate}
Then $h^0(\Gamma_F, (\Lie G'_{\sic})^*(1))=0$. 
\end{prop}
\begin{proof} Lemma \ref{tasho2} gives us an isomorphism of $G$-representations 
$$(\Lie G'_{\sic})^*\cong (\Lie \SL_2)^* \oplus (\Lie (G_1')_{\sic})^*,$$ and assumption (3) implies that it is 
enough to show that $h^0(\Gamma_F, (\Lie \SL_2)^*(1))=0$.

Let $\rho_2: \Gamma_F \rightarrow \PGL_2(\kappa)$ be the composition of $\rho$ with $p_2$. 
The action of $\Gamma_F$ on $(\Lie \SL_2)^*$ factors through $\rho_2$. Thus it is enough to 
verify that $\rho_2$ satisfies the assumptions of Lemma \ref{BIP_Qtwo}.

If $\rho_2(\Gamma_F)$ is contained in a maximal split torus of $B_2$ then there exists $u_2\in U_2(\kappa)$ 
such that $u_2 \rho_2(\Gamma_F) u^{-1}_2$ is contained in $T_2(\kappa)$. Since $L= p_2^{-1}(T_2)$ and $p_2: U\overset{\cong}{\longrightarrow} U_2$, there is $u\in U(\kappa)$ such that
$u \rho(\Gamma_F) u^{-1}$ is contained in $L(\kappa)$. Since $u^{-1} L u$ is an R-Levi  of $G$ we obtain 
 a contradiction to assumption (1). 

Let $\psi:\Gamma_F \rightarrow \kappa^{\times}$ be a character associated to $\rho_2$ in Lemma 
\ref{BIP_Qtwo}. If $\psi= \omega$ or $\psi= \omega^{-1}$ then $\psi^{p-1}=1$. Thus 
$\rho_2(\Gamma_F)$ is contained in $\mu_{p-1}(\kappa) U_2(\kappa)$. In the proof of Lemma 
\ref{bound_centre_Hzero} we have constructed a cocharacter $\lambda: \Gm\rightarrow Z(L)$ such that 
$p_2\circ \lambda: \Gm\rightarrow T_2$ is the unique cocharacter satisfying $\alpha\circ p_2\circ\lambda=\id$. Then $\lim_{t\rightarrow 0} p_2(\lambda(t)) U_2 p_2(\lambda(t))^{-1}= \{1\}$.
Since $p_2: G\rightarrow \PGL_2$ is a group homomorphism, which induces an isomorphism between $U$ and $U_2$, we conclude that 
$$\lim_{t\rightarrow 0} \lambda(t) \rho(\Gamma_F) \lambda(t)^{-1}\subset  
\lim_{t\rightarrow 0} \lambda(t) H_0(\kappa)\lambda(t)^{-1}= H_0(\kappa),$$
as $H_0$ is a subgroup of $L$ and $\lambda$ takes values in $Z(L)$. But this contradicts the assumption 
(2).
\end{proof}

 \section{Dimension of fibres}\label{sec_dim_fib}

Let $\pp$ be a prime ideal of $R^{\ps}_G$ such that $\dim R^{\ps}_G/\pp \le 1$. 
Then $\kappa(\pp)$ is either a finite or  a local field, which we equip with its natural topology. 
We denote by $y$ the homomorphism 
$y: R^{\ps}_{G} \rightarrow \kappa(\pp)$ and by $\ybar$ the composition of $y$ with an 
embedding of $\kappa(\pp)$ into its algebraic closure $\overline{\kappa(\pp)}$. We equip 
$\overline{\kappa(\pp)}$ with its natural topology, extending the topology on $\kappa(\pp)$. 
The goal of this section is to bound the dimension of 
$$ X^{\gen}_{G, y}:= X^{\gen}_{G, \rhobarss} \times_{X^{\ps}_G} y = \XgenGtau\times_{X^{\ps}_{G}} y$$
from above in Proposition \ref{bound_dim_fibre}. 

\begin{lem}\label{fib_gps_ps} There is a unique point $\zbar\in X^{\gps}_{G, \rhobarss}(\overline{\kappa(\pp)})$ mapping to $\ybar\in \XpsG(\overline{\kappa(\pp)})$. 
Moreover, the map 
$$X^{\gen}_{G, \rhobarss}\times_{X^{\gps}_{G, \rhobarss}} \zbar\rightarrow X^{\gen}_{G, \rhobarss}\times_{X^{\ps}_{G}} \ybar$$
induces an isomorphism of underlying reduced subschemes.
\end{lem}
\begin{proof} Proposition \ref{nu_fin_u} implies that $X^{\gps}_{G, \rhobarss}\times_{X^{\ps}_{G}} \ybar$
is finite over $\ybar$ and its underlying topological space is a point. This implies the assertion. 
\end{proof}

We assume\footnote{The assumption is for convenience only as in Proposition \ref{bound_dim_fibre} we consider geometric fibres.} that $G$ is split over $\OO$ and let $T$ be a maximal split torus in $G$ defined over $\OO$. Lemma \ref{fib_gps_ps} and Corollary \ref{conjugate_Levi} imply that 
there is $\xbar\in X^{\gen}_{G, y}(\overline{\kappa(\pp)})$, an R-Levi subgroup $L$ of $G$ containing $T$ and a continuous $L_{\ybar}$-irreducible representation 
$\rho: \Gamma_F \rightarrow L(\kappa(\ybar))$ such that $\rho=\rho_{\xbar}$ and the $G^0$-orbit of $\xbar$ is
closed in $X^{\gen}_{G, \ybar}$. 

\begin{remar}\label{rel_to_ULG} Proposition \ref{nu_fin_u} implies that there is a unique point 
$y'\in X^{\gps}_{G, \rhobarss}$  mapping to $y$. It follows from  Lemma \ref{fib_gps_ps}  
that the fibres of $X^{\gen}_{G, \rhobarss}$ at $y$ and $y'$ have the same dimension.
Moreover, if $\kappa(\pp)$ is a local field then $y'$ is a closed point of the subscheme $U_{LG, \rhobarss}$ defined in section \ref{sec_abs_irr}. Conversely, if $y$ is the image of a closed point of $U_{LG, \rhobarss}$ then $\kappa(y)$ is a local field and we may take  the R-Levi in the above discussion to be equal to $L$.
\end{remar} 
 
 Let $\qq$ be the kernel of $\xbar: A^{\gen}_{G, \rhobarss} \rightarrow \overline{\kappa(\pp)}$. We let 
 $\kappa:=\kappa(\qq)$, let $x$ be the homomorphism $x: A^{\gen}_G \rightarrow \kappa$. Then the image of 
 $\rho$ is contained in $L(\kappa)$. 
 We write $\kappabar$ for $\overline{\kappa(\pp)}$. We consider $X^{\gen}_{G, y}$ as a scheme over 
 $\kappa$ by considering $y: R^{\ps}_{G} \rightarrow \kappa(\pp)\hookrightarrow \kappa$.
 
Let $P$ be any R-parabolic of $G$ which admits $L$ as its R-Levi  and let $U$ be the unipotent radical of $P$. 
\begin{defi}\label{def_XgenPrho}  Let $X^{\gen}_{P, \rho}: \kappa\hyphen\alg \to \Set$ be the functor which sends 
$A$ to the set of  $R^{\ps}_{G}$-condensed representations $\rho': \Gamma_F\rightarrow P(A)$ such that 
\begin{equation}\label{cong_mod_U}
\rho'(\gamma)\equiv\rho(\gamma)\otimes_{\kappa} A \pmod{U(A)}, \quad \forall \gamma\in \Gamma.
\end{equation}
\end{defi}

\begin{lem} The functor $X^{\gen}_{P, \rho}$ is represented by a closed subscheme of $X^{\gen}_{G, y}$.
\end{lem} 
\begin{proof} It follows from Proposition \ref{XgenGtau_rhobarss_A} that for a $\kappa$-algebra $A$, 
$X^{\gen}_{G, y}(A)$ is the set of $R^{\ps}_{G}$-condensed representations
$\rho':\Gamma_F\rightarrow G(A)$ such that $\Theta_{\rho'}= \Theta_{\rho}\otimes_{\kappa} A$.
If $\rho'\in X^{\gen}_{P, \rho}(A)$ then \eqref{cong_mod_U} implies that 
$\Theta_{\rho'}= \Theta_{\rho}\otimes_{\kappa} A$, and hence $\rho'\in X^{\gen}_{G, y}(A)$.

Since $X^{\gen,\tau}_G$ is a closed subscheme of $\XpsGLd\times G^N$ by Remark \ref{FN},
the map $\rho' \mapsto (\rho'(\gamma_1), \ldots, \rho'(\gamma_N))$ identifies 
$X^{\gen}_{G, y}$ with a closed subscheme of $G^N_{\kappa}$. Lemma \ref{gener_cond} implies that 
$(X^{\gen}_{G, y}\times_{G^N_{\kappa}} P^N_{\kappa})(A)$ is the set of $R^{\ps}_G$-condensed 
representations $\rho':\Gamma_F \rightarrow P(A)$ such that $\Theta_{\rho'}=\Theta_{\rho} \otimes_{\kappa} A$. The proof of Lemma \ref{gener_cond} shows that an 
$\RpsG$-condensed representation $\rho': \Gamma_F \rightarrow P(A)$ satisfies \eqref{cong_mod_U} 
if and only if 
$$\rho'(\gamma_i)\equiv \rho(\gamma_i)\pmod{U(A)}, \quad \forall 1\le i\le N.$$
We thus may identify $X^{\gen}_{P, \rho}$ with 
$X^{\gen}_{G, y}\times_{G^N_{\kappa}} P^N_{\kappa} \times_{(P/U)_{\kappa}^N} \Spec \kappa$, 
where $\Spec \kappa \rightarrow (P/U)_{\kappa}^N$ corresponds to a closed point 
$(\rho(\gamma_1), \ldots, \rho(\gamma_N))\in (P/U)_{\kappa}^N(\kappa)$. The claim follows as 
the  fibre products are taken along closed immersions. 
\end{proof}

\begin{lem}\label{TxZ1} Let $x\in X^{\gen}_{P, \rho}(\kappa)$ be the point corresponding to $\rho$
and let $T_x X^{\gen}_{P, \rho}$ be the tangent space at $x$. There is a natural isomorphism
\begin{equation}\label{TxXPrho}
T_x X^{\gen}_{P, \rho}  \overset{\cong}{\longrightarrow} Z^1(\Gamma_F, (\Lie U)_{\kappa})
\end{equation}
of $\kappa$-vector spaces, where the action of $\Gamma_F$ on $(\Lie U)_{\kappa}$ is given via the homomorphism 
$\rho: \Gamma_F \rightarrow L(\kappa)$ composed with the adjoint action of $L$ on $\Lie U$. 
\end{lem}
\begin{proof} We may identify $T_x X^{\gen}_{P, \rho}$ with the set of 
$\rho'\in X^{\gen}_{P, \rho}(\kappa[\varepsilon])$ such that $\rho'\equiv \rho \pmod{\varepsilon}$, which implies that we may write $\rho'(\gamma)= (1 + \varepsilon \beta(\gamma)) \rho(\gamma)$ 
for a unique function $\beta: \Gamma_F \rightarrow (\Lie U)_{\kappa}$.
Lemma \ref{cont_vs_cond} implies that $\rho': \Gamma_F \rightarrow P(A)$ is 
$\RpsG$-condensed if and only if $\rho'$ is continuous for the topology on $\kappa[\varepsilon]$ induced
by $\kappa$, which is equivalent to continuity of $\beta$. A standard calculation shows that
$\rho'$ is a representation if and only if $\beta$ defines a $1$-cocycle. 
Mapping $\rho'$ to $\beta\in Z^1(\Gamma_F, (\Lie U)_{\kappa})$ induces the claimed isomorphism.
\end{proof}

\begin{lem}\label{Prho_geom_pts} The set $X^{\gen}_{P, \rho}(\kappabar)$ is naturally in bijection with the 
set of continuous representations $\rho':\Gamma_F \rightarrow P(\kappabar)$, such that 
$\rho' \equiv \rho \pmod{U(\kappabar)}$. 
\end{lem}

\begin{proof} The assertion follows from Lemma \ref{cont_vs_cond}.
\end{proof}

Let $\lambda: \Gm \rightarrow G$ be a cocharacter such that $P=P_{\lambda}$ and 
$L=L_{\lambda}$. Since $\lambda$ centralises $L$ and normalises $P$ the 
adjoint action of $\Gm$ on $\XgenG$ induces an action of $\Gm$ on $X^{\gen}_{P, \rho}$. 

\begin{lem}\label{clorbit_Prho} The unique closed $\Gm$-orbit in 
$X^{\gen}_{P, \rho}(\kappabar)$ is the singleton $\{\rho\}$.
\end{lem}
\begin{proof} Since $\lim_{t\rightarrow 0} \lambda(t) U \lambda(t)^{-1}= 1$, 
using Lemma \ref{Prho_geom_pts}, we deduce that the closure of any $\Gm$-orbit 
in $X^{\gen}_{P, \rho}(\kappabar)$ will contain $\{\rho\}$, which is a closed 
$\Gm$-orbit, as it is zero dimensional and $\rho(\Gamma_F)\subseteq L(\kappa)$. 
\end{proof}

 \begin{prop}\label{bound_P} Let $x\in X_{P, \rho}^{\gen}$ be the closed point corresponding to $\rho$. Then 
 $$\dim X^{\gen}_{P,\rho} \le \dim_{\kappa} T_{x} X^{\gen}_{P, \rho}= \dim U_{\kappa} + \dim U_{\kappa} [F:\Qp] + h^0(\Gamma_F, (\Lie U)_{\kappa}^*(1)),$$
 where the action of $\Gamma_F$ on $\Lie U$ is given by the composition of $\rho$ with the adjoint action of $L$ on $\Lie U$. 
 \end{prop}

 \begin{proof} The first inequality follows from \cite[Lemma 2.2]{BIP_new} together 
with Lemma \ref{clorbit_Prho}. (Its proof shows that $x$ lies on every 
irreducible component of $X_{P, \rho}^{\gen}$.) 

Lemma \ref{TxZ1} identifies $T_xX^{\gen}_{P,\rho}$ with $Z^1(\Gamma_F, (\Lie U)_{\kappa})$.
The assertion follows from Lemma \ref{dimZ1}, noting that $\dim U_{\kappa}$  is equal to $\dim_{\kappa} (\Lie U)_{\kappa}$. 
\end{proof}

\begin{lem}\label{conj_P2} Let $x'\in X^{\gen}_{G, y}(\kappabar)$. 
Then there is an R-parabolic subgroup $P$ of $G$
with R-Levi equal to $L$, $g\in G^0(\kappabar)$ and $x''\in X^{\gen}_{P, \rho}(\kappabar)$ such that 
$$ \rho_{x'}(\gamma)= g \rho_{x''}(\gamma) g^{-1}, \quad \forall \gamma\in \Gamma_F.$$ 
\end{lem} 
\begin{proof} The assertion follows from Lemma \ref{fib_gps_ps}, Corollary \ref{conjugate_Levi} and Lemma \ref{base_change_R_par}.
\end{proof}

\begin{defi}\label{defi_defect} We define the defect of the fibre $X^{\gen}_{G,y}$ to be 
\begin{equation}\label{delta_def}
 \delta(y):= \max_{P} h^0(\Gamma_F, (\Lie U)_{\kappa}^*(1)),
\end{equation}
where $P$ ranges over all R-parabolic subgroups of $G$ with R-Levi $L$, and $U$ is the unipotent radical of $P$ and $\Gamma_F$ acts on $(\Lie U)_{\kappa}^*$ via 
$\rho$ composed with the adjoint action of $L$.
\end{defi}
\begin{remar} Lemma \ref{Rlevi} implies that only finitely many $P$ appear in \eqref{delta_def}.
\end{remar}
\begin{remar} If we replace $(L, \rho)$ in \eqref{delta_def} by a conjugate then $\delta(y)$ 
does not change. Since Lemma \ref{conj_P2} implies that the pair $(L, \rho)$ is uniquely determined by $y$ up to $G^0$-conjugation the defect $\delta(y)$ depends only on $y$.   
\end{remar}

\begin{remar}\label{dim_U_same} Let $\lambda: \Gm \rightarrow G_{\kappa}$ be a cocharacter.
Let $M=L_{\lambda}$, $U=U_{\lambda}$ and $U^-= U_{-\lambda}$.
Then 
$ \Lie G_{\kappa} = \Lie U^-\oplus \Lie M \oplus \Lie U.$
Moreover, the theory of root systems gives us $\dim_{\kappa} \Lie U^{-}=\dim_{\kappa} \Lie U$. 
Hence, $\dim_{\kappa} \Lie U=\dim U = \frac{1}{2} ( \dim G_{\kappa} - \dim M).$
Thus for all the unipotent radicals $U$ appearing in \eqref{delta_def}  we have 
 $\dim_{\kappa} \Lie U_{\kappa}= \frac{1}{2} ( \dim G_{\kappa} - \dim L_{\kappa}).$
\end{remar} 

\begin{lem}\label{bound_defect} $\delta(y)\le \frac{1}{2}(\dim G_{\kappabar} - \dim L_{\kappabar})$
\end{lem}
\begin{proof} Since the dimension of $\Gamma_F$-invariants is less or equal to the dimension 
of the vector space, we have $h^0(\Gamma_F, (\Lie U)_{\kappa}^*(1))\le \dim_{\kappa} \Lie U_{\kappa}$. 
The bound follows from Remark \ref{dim_U_same} by observing that extending scalars to the algebraic 
closure does not change the dimension.
\end{proof}

\begin{prop}\label{bound_dim_fibre}
$$\dim X^{\gen}_{G,y}\le \dim G_{\kappabar} -\dim Z(L_{\kappabar}) +\tfrac{1}{2}(\dim G_{\kappabar} - \dim L_{\kappabar})[F:\Qp] + \delta(y).$$
\end{prop}
\begin{proof} The dimension does not change if we replace $X^{\gen}_{G,y}$ by the 
geometric fibre $X^{\gen}_{G,\ybar}$. The action of $G^0$ on $X^{\gen}_{G, \rhobarss}$
induces an action of $G^0_{\kappabar}$ on $X^{\gen}_{G,\ybar}$. 
Moreover, $X^{\gen}_{P, \rho}\times_y \ybar$ is invariant under the action of 
the subgroup $Z(L_{\kappabar})^0 U_{\kappabar}$. Thus we obtain a map of schemes
\begin{equation}\label{cow_prod}
\coprod_{ P} G^0_{\kappabar} \times^{Z(L_{\kappabar})^0 U_{\kappabar}} (X^{\gen}_{P, \rho}\times_y \ybar)\rightarrow 
X^{\gen}_{G, \ybar},
\end{equation}
where the disjoint union is taken over all the R-parabolics $P$ of $G$ with 
R-Levi $L$. We note that there are
only finitely many such $P$ by Lemma \ref{Rlevi}. It follows from Lemma \ref{conj_P2} that \eqref{cow_prod}
induces a surjection on $\kappabar$-points. Since both the source of and the target of \eqref{cow_prod}
are of finite type over $\kappabar$, \cite[Lemma 3.14]{BIP_new} implies that 
\begin{displaymath}
\begin{split}
\dim X^{\gen}_{G, \ybar}\le &\max_{P} 
\dim (G_{\kappabar}^0 \times^{Z(L_{\kappabar})^0 U_{\kappabar}}(X^{\gen}_{P}\times_y \ybar)^{\red})\\
=& 
\dim G_{\kappabar}^0  - \dim Z(L_{\kappabar})^0 - \dim U_{\kappabar}+ \max_{P} \dim (X^{\gen}_{P, \rho}\times_y \ybar)\\
\le& 
\dim G_{\kappabar} - \dim Z(L_{\kappabar}) +\tfrac{1}{2}(\dim G_{\kappabar} - \dim L_{\kappabar}) [F:\Qp] + \delta(y),
\end{split}
\end{displaymath}
where in the last two lines we have used Remark \ref{dim_U_same} and in the last inequality we have used Proposition  \ref{bound_P}.
\end{proof}

\begin{remar} If $y'$ is a closed point of $U_{LG, \rhobarss}$ then Remark \ref{rel_to_ULG} implies that the bound in 
Proposition \ref{bound_dim_fibre} also applies to the fibre of $X^{\gen}_{G,\rhobarss}$ at $y'$. We will use this without further comment  in section \ref{sec_complete_intersection}. 
\end{remar}

\begin{lem}\label{tiny_levi} Let $M$ be an R-Levi of $G_{\kappabar}$. Assume that $\dim G_{\kappabar}>1$. If $\dim M=1$ then $Z(M)^0= M^0$. In particular, 
$\dim Z(M)=1$. Moreover, $\dim G_{\kappabar}=3$.
\end{lem}
\begin{proof} The assumption on $M$ implies that 
the semisimple rank of $G^0_{\kappabar}$ is $1$ and
hence $G^0_{\kappabar}$ is isomorphic to either $\SL_2$ or $\PGL_2$, \cite[Theorem 20.33]{milne_alg}, which implies the Lemma. 
\end{proof}

\begin{cor}\label{bound_Y}  Let $y:\Spec \kbar\rightarrow X^{\ps}_{G} $ be a geometric point above the closed point of $X^{\ps}_{G}$ and let $Y_{\rhobarss}:= X^{\gen}_{G,y}$. If $\dim G_{\kbar}>1$ then 
\begin{equation}\label{threight}
\dim G_{\kbar}(1+[F:\Qp]) - \dim Y_{\rhobarss} \ge [F:\Qp] \dim L_{\kbar} + 
\dim Z(L_{\kbar}) \ge [F:\Qp]+1.
\end{equation}
Moreover, if \eqref{threight} is an equality\footnote{ By which we mean both inequalities are equalities.} then $\dim G_{\kbar}-\dim L_{\kbar}=2$ and $L$ is $\rhobarss$-compatible in the sense of 
\Cref{ind_hyp} below. Otherwise, 
\begin{equation}\label{thrnine}
\dim G_{\kappabar}(1+[F:\Qp])- \dim Y_{\rhobarss}\ge 2[F:\Qp]+1.
\end{equation}
\end{cor}
\begin{proof} The first inequality follows from Proposition \ref{bound_dim_fibre} and Lemma \ref{bound_defect}.
The second inequality follows from Lemma \ref{tiny_levi}.  

If \eqref{threight} is an equality then  
$\dim L_{\kbar}=1$. The last part of Lemma \ref{tiny_levi} says that $\dim G_{\kappabar}=3$.
So the difference of dimensions is $2$. The compatibility of $L_k$ with $\rhobarss$ follows \Cref{L_compatible} below, since $\Xbar^{\gps}_{L, \rhobarss}$ is non-empty by construction of $L$.

If $\dim L_{\kbar}\neq 1$ then $\dim L_{\kbar}\ge 2$ and \eqref{threight} implies \eqref{thrnine}.
\end{proof}

\section{Complete intersection} \label{sec_complete_intersection}
In this section we will show that the deformation rings $R^{\square}_{G, \rhobar}$ are 
complete intersections of expected dimension. As explained in the introduction it is 
enough to bound the dimension of their special fibres from above. We will achieve this 
by bounding the dimension of $\Xbar^{\gen}_{G, \rhobarss}$. 

\subsection{Induction hypothesis}\label{ind_hyp}

Let us recall the setup. We start with a continuous representation $\rhobar: \Gamma_F \rightarrow G(k)$, such that 
the group homomorphism $\pi_G\circ\rhobarss: \Gamma_F \rightarrow (G/G^0)(\overline{k})$ is surjective, where
$\pi_G: G\rightarrow G/G^0$ is the quotient map. 

We will bound the dimension of $\Xbar^{\gen}_{G, \rhobarss}$ arguing by  induction 
 that the assertion 
\begin{equation}\label{dim_pair} 
\dim \Xbar^{\gps}_{H, \rhobar_1^{\mathrm{ss}}} \le \dim Z(H) + [F:\Qp] \dim H \tag{DIM}.
\end{equation} 
holds for all \emph{pairs} $(H, \rhobar_1)$, where $H$ is a generalised reductive group
scheme defined over a finite extension of $k$ and $\rhobar_1: \Gamma_F\rightarrow H(\kbar)$ is a continuous $H$-semisimple representation
such that $\pi_H\circ \rhobar_1: \Gamma_F\rightarrow (H/H^0)(\overline{k})$ is surjective. We refer the reader to \Cref{sec_intro_bd_dim} for the sketch of the proof strategy. Instead of reading this section linearly we suggest to first look at \Cref{some_prop}, where the induction step is carried out, and then 
look up the statements needed for its proof. 

The induction argument takes place entirely in the special fibre and the group schemes $H$ appearing in the induction need not be special fibres of 
generalised reductive group schemes defined over $\OO$. We note that the 
schemes $\Xbar^{\gen}_{H, \rhobar_1^{\mathrm{ss}}}$, $\Xbar^{\git}_{H,\rhobar_1^{\mathrm{ss}}}$ and $\Xbar^{\ps}_H$ are still defined for such $H$ by replacing $\OO$ with $k$ in all the definitions.
All the statements proved for $\Xbar^{\gen}_{G, \rhobarss}$, $\Xbar^{\git}_{G, \rhobarss}$ and $\Xbar^{\ps}_G$ when 
$G$ is a generalised reductive group scheme over $\OO$, also hold 
for $\Xbar^{\gen}_{H, \rhobarss_1}$, $\Xbar^{\git}_{H, \rhobarss_1}$ and $\Xbar^{\ps}_H$. In particular, 
if $x$ is a closed point of $Y_{\rhobarss_1}$ or a closed point of 
$\Xbar^{\gen}_{H, \rhobarss_1}\setminus Y_{\rhobarss_1}$ then the arguments of \Cref{333}, 
\Cref{local_ring_def_ring} carry over and relate the completed
local ring of $\Xbar^{\git}_{H,\rhobar_1^{\mathrm{ss}}}$ at $x$ to the deformation 
ring $\Rbar{}^{\square}_{\rho_x, H}$, which 
parameterises deformations $\rho: \Gamma_F\rightarrow H(A)$ of $\rho_x$ for 
$A\in \Aa_{\kappa(x)}$. If $H$ is a generalised reductive group scheme
over $\OO$ then $\Rbar{}^{\square}_{\rho_x, H}=R^{\square}_{\rho_x, H}/\varpi$.

We say that a pair $(H,\rhobar_1)$ is \emph{$\rhobarss$-compatible}
if $H$ is a closed subgroup scheme of $G_{\kbar}$ and the $G$-semisimplification of 
$\rhobar_1$ is $G^0(\overline{k})$-conjugate to $\rhobarss$.
We say that a closed generalised reductive subgroup $H$ of $G_{\kbar}$ is 
$\rhobarss$-compatible if it appears in some $\rhobarss$-compatible 
pair $(H,\rhobar_1)$.

\begin{lem}\label{L_compatible} Let $L$ be an R-Levi of $G$. If $\Xbar^{\gps}_{LG, \rhobarss}$ 
is non-empty then $L_k$ occurs in a $\rhobarss$-compatible pair. Moreover, if $L\neq G$ then $\dim L_k <\dim G_k$.
\end{lem} 
\begin{proof} This follows from Lemma \ref{V_LG_maxdim}.
\end{proof}

\begin{lem}\label{useful_U} Let $U$ be a non-empty locally closed subscheme of 
$X^{\gps}_{G, \rhobarss}\setminus\{ \ast\}$, 
where $\ast$ is the closed point, and let $Z$ be the closure of $U$ in $X^{\gps}_{G, \rhobarss}$.
Then $\dim Z= \dim U+1$.
\end{lem}

\begin{proof} This follows from \cite[Lemma 3.5 (5)]{BIP_new} applied
to $\Spec R=\Spec S= Z$.
\end{proof} 

\begin{lem}\label{useful_V} Let $Y_{\rhobarss}$ be the preimage in $X^{\gen}_{G, \rhobarss}$ of the closed point in $X^{\gps}_{G, \rhobarss}$ and let 
$V$ be an open  non-empty $G^0$-invariant subscheme of $X\setminus Y_{\rhobarss}$, where $X$ is either $X^{\gen}_{G, \rhobarss}$ or  $\Xbar^{\gen}_{G, \rhobarss}$.
Let $Z$ be the closure of $V$ in $X$. Then $\dim Z=\dim V+1$. 
\end{lem}
\begin{proof} Since $V$ is $G^0$-invariant, its closure $Z$ is also
$G^0$-invariant. It follows from the proof of \cite[Lemma 3.21]{BIP_new} that every irreducible component 
of $Z$ meets $Y_{\rhobarss}$ non-trivially. The assertion follows from \cite[Lemma 3.18 (5)]{BIP_new}. 
\end{proof} 

\begin{lem}\label{dim_closed_pts} 
$\dim X^{\gen}_{G,\rhobarss}= \max_{x} \dim R^{\square}_{\rho_x, G}$,
$\dim \Xbar^{\gen}_{G,\rhobarss}=  \max_{x} \dim R^{\square}_{\rho_x, G}/\varpi$, 
where the maximum is taken over all the closed points $x\in Y$. 
\end{lem}
\begin{proof} Let $X$ be either $X^{\gen}_{G,\rhobarss}$ or $\Xbar^{\gen}_{G,\rhobarss}$
and let $x$ be a closed point of $X$. If 
$x\not\in Y$ then $x$ is a closed point in $X\setminus Y$ and 
since $\dim X\setminus Y = \dim X -1$ by Lemma \ref{useful_V} 
we have $\dim \OO_{X,x}< \dim X$. Thus $\dim X= \max_{x} \dim \OO_{X,x}$, 
where the maximum is taken over the closed points $x\in Y$. Since  $\dim \OO_{X,x} = \dim \hat{\OO}_{X,x}$ the assertion 
follows from \Cref{local_ring_def_ring}. 
\end{proof}

\begin{prop}\label{mother_bound} Let $L$ be an R-Levi of $G$ and let $H$ be a closed reductive subgroup of $L_{k'}$, where 
$k'$ is a finite extension of $k$. If $L=G$ then we assume that $H\neq G_{k'}$. 
Let $U$ be an open subscheme of $\Xbar^{\gps}_{HG}\cap U_{LG, \rhobarss}$, let 
$V$ be the preimage of $U$ in $\Xbar^{\gen}_{G, \rhobarss}$, and let $Z$ be the closure of $V$.
If \eqref{dim_pair} holds for all $\rhobarss$-compatible pairs $(H_1, \rhobar_1)$ with $\dim H_1< \dim G_k$ then
\begin{equation}\begin{split}\label{seventh_ofmany}
\dim G_k ([F:\Qp]+1) -&\dim Z \ge \\&(\dim L_k - \dim H)[F:\Qp] + \dim N_k[F:\Qp]-\delta(U), 
\end{split}
\end{equation}
where $\delta(U)=\max_y \delta(y)$, where the maximum is taken over all the closed points 
$y\in U$ and $\delta(y)$ denotes the defect of the fibre at $y$ and $N$ is a unipotent 
radical of any R-parabolic subgroup of $G$ with R-Levi $L$. 
\end{prop}
\begin{proof} After extending scalars we may assume that $k=k'$. We may assume that 
$\Xbar^{\gps}_{H, \rhobarss}$ is non-empty, since otherwise both $V$ and $Z$ are empty and the claim 
holds trivially. Using Lemma \ref{cover_Hi} we may further assume that the map $\rhobarss: \Gamma_F \rightarrow 
(H/H^0)(k)$ is surjective. If $\dim H=\dim G_k$ then Lemma \ref{V_LG_maxdim} implies that 
$H=G_k$ contradicting our assumption. Thus $\dim H <\dim G_k$ and by hypothesis we have
\begin{equation}\label{zeroth_ofmany}
\dim \Xbar^{\gps}_H \le \dim H [F:\Qp] + \dim Z(H).
\end{equation}
Since $V$ is a preimage of $U$, it is $G^0$-invariant and Lemma \ref{useful_V} implies that 
\begin{equation}\label{first_ofmany} 
\dim Z = \dim V +1.
\end{equation}
From \cite[Lemma 3.18]{BIP_new} we deduce that 
\begin{equation}\label{second_ofmany}
\dim V \le \dim U + \max_{y} \dim X^{\gen}_{G, y},
\end{equation}
where the maximum is taken over all closed points $y$ of $U$. Since $U$ is open in the 
punctured spectrum of a local ring, and $\Xbar^{\gps}_H \rightarrow \Xbar^{\gps}_G$ is finite 
by Corollary  \ref{finite_HG}, we obtain
\begin{equation}\label{third_ofmany}
 \dim U +1\le \dim \Xbar^{\gps}_{HG, \rhobarss} \le \dim \Xbar^{\gps}_{H, \rhobarss} 
\le
\dim H [F:\Qp] + \dim Z(H),
\end{equation}
 where the last inequality is given by \eqref{zeroth_ofmany}. Proposition \ref{bound_dim_fibre} implies that 
\begin{equation}\label{fifth_ofmany} 
\dim X^{\gen}_{G,y} \le \dim G_k+ \dim N_k [F:\Qp] + \delta(y) - \dim Z(L_k), 
\end{equation}
where $\delta(y)$ is the defect of the fibre. 
Since $\dim Z(H)\le \dim Z(L_k)$ by Corollary \ref{dim_Z_same}  putting 
all the inequalities together we obtain 
\begin{equation}\label{sixth_ofmany} 
\dim Z\le \dim H [F:\Qp] + \dim G_k +\dim N_k [F:\Qp] +\delta(U).
\end{equation}
where $\delta(U)=\max_y \delta(y)$. Since $\dim G_k = \dim L_k + 2 \dim N_k$ we 
obtain \eqref{seventh_ofmany}.
\end{proof}

\begin{cor}\label{bound_1} Let $L$ be an R-Levi of $G$ and let $H$ be a closed reductive subgroup of $L_{k'}$, where 
$k'$ is a finite extension of $k$. If $L=G$ then we assume that $H\neq G_{k'}$. 
Let $V$ be the preimage of $U=\Xbar^{\gps}_{HG}\cap U_{LG, \rhobarss}$ in $\Xbar^{\gen}_{G, \rhobarss}$, and let $Z$ be the closure of $V$ in $\Xbar^{\gen}_{G, \rhobarss}$. 
If \eqref{dim_pair} holds for all $\rhobarss$-compatible pairs $(H_1, \rhobar_1)$ with $\dim H_1< \dim G_k$ then
$$\dim G_k ([F:\Qp]+1) - \dim Z\ge (\dim L_k - \dim H)[F:\Qp].$$
\end{cor}
\begin{proof} Remark \ref{rel_to_ULG} and Lemma \ref{bound_defect} imply that 
\begin{equation}\label{fourth_ofmany} 
 \delta(U) \le \frac{1}{2}(\dim G_k - \dim L_k) =\dim N_k.
\end{equation}
The assertion follows from Proposition \ref{mother_bound}.
\end{proof} 

\begin{cor}\label{bound_2} Let $L$ be an R-Levi of $G$, such that $L\neq G$, 
let $H$ be a closed reductive subgroup of $L_{k'}$, where 
$k'$ is a finite extension of $k$. Let $U$ be an open subscheme of $U_{LG, \rhobarss}\cap \Xbar^{\gps}_{HG}$, 
let $V$ be a preimage of $U$ in $\Xbar^{\gen}_{G, \rhobarss}$ and let $Z$ be its closure. 
If \begin{equation}\label{lostcount} 
 \delta(U) \le \dim L_k -\dim H.
\end{equation} and
\eqref{dim_pair} holds for all $\rhobarss$-compatible pairs $(H_1, \rhobar_1)$ with $\dim H_1< \dim G_k$ then
$$\dim G_k ([F:\Qp]+1) - \dim Z\ge \dim N_k [F:\Qp]\ge 1,$$
where $N$ is the unipotent radical of any R-parabolic of $G$ with R-Levi $L$. Moreover, 
if either $F\neq \Qp$ or $\dim N_k \neq 1$ then the lower bound can be changed to $2$. 
\end{cor}
\begin{proof} 
We may assume that $\Xbar^{\gps}_{LG, \rhobarss}$ is non-empty. Lemma \ref{V_LG_maxdim} implies that $\dim L< \dim G$ and $\dim N_k \ge 1$. 
As in the proof of \Cref{bound_1} the assertion follows from Proposition \ref{mother_bound}, but using \eqref{lostcount} instead of  
  \eqref{fourth_ofmany}: 
\begin{equation}\begin{split}
\dim G_k ([F:\Qp]+1) -&\dim Z \ge \\&(\dim L_k - \dim H)([F:\Qp]-1) + \dim N_k[F:\Qp].
\end{split}
\end{equation}
Since, as explained above, $\dim N_k \ge 1$ this implies the assertion.
\end{proof}

\subsection{\texorpdfstring{$W$-special locus}{W-special locus}}
In this subsection we define the $W$-special locus, where $W$ is a representation of $G$ on a 
free $\OO$-module of finite rank. Using it we will show that if $L$ is an $R$-Levi of $G$ 
then the locus of $\Ubar_{LG, \rhobarss}$, where the defect of the fibre is non-zero is contained in a finite 
union of subschemes $\Xbar^{\gps}_{HG}\cap U_{LG, \rhobarss}$, where $H$ is a closed reductive subgroup 
of $L_k$ of smaller dimension. We will combine this with Corollaries \ref{bound_1} and \ref{bound_2} 
to carry out the induction argument in \Cref{bound_space}.

\begin{lem}\label{V_iks} Let $R$ be a noetherian $\OO$-algebra, let $G$ be an affine group scheme 
over $\Spec R$ and let $W$ be a representation of $G$ on a free $R$-module of finite rank. 
Let $A$ be a noetherian $R$-algebra and let $\rho: \Gamma_F\rightarrow G(A)$ 
be a group homomorphism. Then for each $j\ge 0$ there is a closed reduced subscheme $X_{W, j}$ of $X:=\Spec A$ such that for 
$x\in X$ we have 
$$x\in X_{W,j}  \iff \dim_{\kappa(x)} W_x(1)^{\Gamma_F}\ge j+1,$$ where $W_x= W\otimes_{R} \kappa(x)$ with the 
$\Gamma_F$ action given by the specialisation of $\rho$ at $x$. 
\end{lem} 
\begin{proof} 
Let $\check{W}:= \Hom_R(W, R)$ equipped with the contragredient left action of  $G$. Then 
$$ \check{W}_x:= \check{W}\otimes_R \kappa(x)\cong \Hom_{\kappa(x)}(W_x, \kappa(x)), \quad \forall 
x\in X.$$
In particular, $\Gamma_F$-invariants in $W_x(1)$ are non-zero if and only if $\Gamma_F$-coinvariants 
in $\check{W}_x(-1)$ are non-zero.

Evaluating $\check{W}$ at $A$ we
get a free finite $A$-module $\check{W}(A)=\Hom_A(W(A), A)$.  
The action of $\Gamma_F$ on $\check{W}(A)$ is given by $g \cdot \ell:= \ell\circ \rho(g)^{-1}$. 
Let $M$ be the $A$-submodule of  $\check{W}(A)$ generated by elements of the form 
$ \chi_{\cyc}(g)^{-1} g \cdot \ell  - \ell$ for all $g\in \Gamma_F$ and all $\ell \in \check{W}(A)$.
Let $X_W$ be the support of $\check{W}(A)/ M$ in $X$. Since $A$ is noetherian $X_W$ is closed in $X$ 
and we equip it with the reduced subscheme structure. 

For all $x\in X$ we have an exact sequence 
$$ M\otimes_A \kappa(x) \rightarrow \check{W}_x \rightarrow (\check{W}(A)/M)\otimes_A \kappa(x)\rightarrow 0.$$
Thus $(\check{W}(A)/M)\otimes_A \kappa(x)$ is isomorphic to $\Gamma_F$-coinvariants of 
$\check{W}_x(-1)$. We conclude that $x\in X_W$ if and only if $\Gamma_F$-coinvariants of 
$\check{W}_x(-1)$ are non-zero, and hence $X_{W,0}=X_W$.  It follows from upper-semicontinuity of the dimension
function, \cite[Example III.12.7.2]{hartshorne} that $X_{W,j}$ is a closed subscheme of $X_{W}$ for $j\ge 1$.
\end{proof} 

\begin{defi}\label{defi_W_special} In the situation of Lemma \ref{V_iks} we call $X_{W,0}$ the \emph{$W$-special 
locus} in $X$ and its complement $X\setminus X_{W,0}$ the \emph{$W$-nonspecial locus} in $X$. 
We will refer to $X_{W,j}$ as the $W$-\emph{special locus of level} $j$.
\end{defi}

\begin{lem}\label{inv_non} Let us assume the setup of Lemma \ref{V_iks} and let 
us further assume that $\chara (R) =p$. Let $x\in X_{W,j}$ and let $H$ 
be the Zariski closure of $\rho_x(\Gamma_F)$ in $G_{\kappa(x)}$. Then 
$\dim_{\kappa(x)} W_x^{H^0}\ge j+1$. 
\end{lem}
\begin{proof} Since $\chara (R) =p$ the restriction of 
$W_x(1)$ to $\Gamma_{F(\zeta_p)}$ is isomorphic to the restriction of $W_x$ to $\Gamma_{F(\zeta_p)}$. 
Since $x\in X_{W,j}$ we have 
$$\dim_{\kappa(x)} W_x^{\Gamma_{F(\zeta_p)}}\ge \dim_{\kappa(x)} W_x(1)^{\Gamma_F}\ge j+1$$ and thus 
$\dim_{\kappa(x)} W_x^{H_1}\ge j+1$, where $H_1$ is the closure 
of $\rho(\Gamma_{F(\zeta_p)})$ in $G_{\kappa(x)}$. 
Since   $\Gamma_{F(\zeta_p)}$ is of finite index in $\Gamma_F$, 
we get that $H_1$ is of finite index in $H$, and thus $H_1^0=H^0$.
\end{proof}

From now on let $G$ be a generalised reductive  $\OO$-group scheme as in the previous 
sections. Then its Lie algebra $\Lie G$ is a finite free $\OO$-module equipped with a $G$-action. 
If $P$ is an R-parabolic with R-Levi $L$ 
and unipotent radical $U$, then $\Lie L$ and $\Lie U$ are $\OO$-submodules of $\Lie G$ by 
\cite[Theorem 4.1.7, part 4]{conrad}. Thus they are free $\OO$-modules of finite rank. Moreover, they are equipped 
with the adjoint action of the R-Levi $L$.

\begin{prop}\label{setup} Let us assume the setup of Lemma \ref{V_iks} in one of the following cases:
\begin{enumerate}
\item $R=\OO$, $A=A^{\gen}_{G, \rhobarss}$, $X=\XgenGrhobarss$;
\item $R=k$, $A=A^{\gen}_{G, \rhobarss} /\varpi$, $X= \overline X^{\gen}_{G, \rhobarss}$;
\item $R=L$, $A= A^{\gen}_{G, \rhobarss}[1/p]$, $X= \XgenGrhobarss[1/p]$.
\end{enumerate}
and let
$\rho: \Gamma_F \rightarrow G(A)$ be the universal
representation. Then the $W$-special locus of level $j$ in $X$ is $G^0$-invariant.
\end{prop} 
\begin{proof} It follows from the universal property of $A^{\gen}_{G,\rhobarss}$ that 
if $g\in G^0(A)$ then there is an isomorphism of $\RpsG$-algebras
$\varphi_g: A\rightarrow A$ such that for all $\gamma\in \Gamma_F$ we have 
\begin{equation}\label{conjugate}
g \rho(\gamma) g^{-1} = G(\varphi_g)( \rho(\gamma)),
\end{equation}
where $G(\varphi_g): G(A)\rightarrow G(A)$ is the group homomorphism 
induced by $\varphi_g: A\rightarrow A$. 

 Let $x: A\rightarrow \kappa$ 
be a homomorphism of $\RpsGLd$-algebras, where $\kappa$ 
is a field. Then 
$y:= g(x)= x\circ \varphi_g^{-1}: A\rightarrow \kappa$. In particular, the residue 
fields of $x$ and $y$ are the same and we may assume that they are equal to $\kappa$. 
If $h\in G(A)$ then 
let $h_x= G(x)(h), h_y= G(y)(h)\in G(\kappa)$ be the specialisations of $h$ at $x$ and $y$ respectively.
It follows from \eqref{conjugate} that 
\begin{equation}\label{conjugate2}
\rho_x(\gamma)= G(x)(\rho(\gamma))= G(y)( G(\varphi_g)( \rho(\gamma)))= G(y)(g \rho(\gamma) g^{-1})= g_y \rho_y(\gamma) g_y^{-1}.
\end{equation}

We may identify $W_x= W_y= W_{\kappa}$ as $\kappa$-vector spaces. 
If $v\in W_{\kappa}$ such that $\rho_x(\gamma)v = \chi_{\cyc}(\gamma)^{-1} v$ for all 
$\gamma\in \Gamma_F$ then it follows from \eqref{conjugate2} that 
 $\rho_y(\gamma) ( g_y^{-1} v) = \chi_{\cyc}(\gamma)^{-1}  g_y^{-1} v$ for all 
$\gamma\in \Gamma_F$. Hence, $g_y^{-1}: W_{\kappa} \rightarrow W_{\kappa}$ induces 
an isomorphism between $W_{x}(1)^{\Gamma_F}$ and $W_y(1)^{\Gamma_F}$. Thus 
$x\in X_{W,j}$ if and only if $y\in X_{W, j}$. 
\end{proof} 

\begin{cor}\label{setup_prime} Assume the setup of Proposition \ref{setup} and let $X'$ be a closed $G^0$-invariant 
subscheme of $X$. Then the $W$-special locus of level $j$ in $X'$ is $G^0$-invariant.
\end{cor}
\begin{proof} The fibre product $X_{W,j}\times_X X'$ is a closed  $G^0$-invariant subscheme of $X'$ and
its underlying reduced subscheme coincides with $X'_{W,j}$. Thus $X'_{W,j}$ is also $G^0$-invariant.
\end{proof} 

\begin{lem}\label{local_ring_dim} Let $Z$ be a closed non-empty $G^0$-invariant subscheme of $X^{\gen}_{G, \rhobarss}$ and let 
$Z'$ be an irreducible component of $Z$. Then $Z'$ is $G^0$-invariant and $Z'\cap Y_{\rhobarss}$ is non-empty. 
Moreover, if $x$ is a closed point of $Z'$ then one of the following holds:
\begin{enumerate}
\item if $x\in Y_{\rhobarss}$ then $\dim \OO_{Z', x} = \dim Z'$;
\item if $x\not\in Y_{\rhobarss}$ then $\dim \OO_{Z', x} = \dim Z'-1$.
\end{enumerate}
\end{lem} 
\begin{proof} The proof is the same as the proof of \cite[Lemma 3.21]{BIP_new}.
\end{proof}

\begin{prop}\label{dim_sp_gen} If $Z$ is a closed $G^0$-invariant subscheme of $X^{\gen}_{G, \rhobarss}$ then 
$$\dim Z[1/p]\le \dim \Zbar,\quad \dim Z\le \dim \Zbar+1,$$
where $\Zbar$ is its special fibre and  $Z[1/p]$ is its generic fibre. 
\end{prop} 
\begin{proof} The proof is the same as the proof of \cite[Lemma 3.23]{BIP_new}.
\end{proof} 

\begin{cor}\label{gen_sp_dim} Let $X'$ be a closed $G$-invariant subscheme of $X^{\gen}_{G, \rhobarss}$, let $W$ be a representation
of $G$ on a free $\OO$-module of finite rank. Then 
$$\dim X'_{W_L}\le \dim X'_{W_k}, \quad \dim X'_{W}\le \dim X'_{W_k}+1$$
where $X'_{W_L}$ is the $W_L$-special locus in $X'[1/p]$ and $X'_{W_k}$ is the $W_k$-special locus in 
$\Xbar'$ and $X'_W$ is the $W$-special locus in $X'$.
\end{cor}
\begin{proof}
Let $X'_W$ be the $W$-special locus in $X'$. Then it follows from the definition of the $W$-special locus 
that $X'_{W_L}= X'_W[1/p]$ and $X'_{W_k}$ is the underlying reduced subscheme of 
the special fibre of $X'_W$. 
The assertion follows from Corollary \ref{setup_prime} and Proposition \ref{dim_sp_gen}.
\end{proof}

\begin{prop}\label{WGH} Assume the setup of Proposition  \ref{setup} (2) so that 
$R=k$ and $A=A^{\gen}_{G, \rhobarss}/\varpi$. If the $W$-special locus of 
level $j$ in $\Xbar^{\gen}_{G, \rhobarss}$ is nonempty then 
the following assertions hold: 
\begin{enumerate} 
\item The image $C_{W,j}$ in $\XbargpsG$ of the $W$-special locus of level $j$ in $\Xbar^{\gen}_{G, \rhobarss}$ is closed;
\item there exists a finite extension $k'$ of $k$ and 
finitely many reductive subgroups $H_1, \ldots, H_r$ of $G_{k'}$ such that 
$C_{W,j}$ is contained in the union of $\Xbar^{\gps}_{H_iG, \rhobarss}$ and 
$$ \dim_{k'} W_{k'}^{H_i^0} \ge j+1, \quad \forall 1\le i\le r.$$
\item if $W^{G^0}=0$ then $H_i$ in (2) satisfy
$\dim H_i\le \dim G_k-1$;
\item if $W^{G'}=0$ then $H_i$ in (2) satisfy $\dim H_i\le \dim G_k -2$. 
\end{enumerate}
\end{prop} 

\begin{proof} 
Since the $W$-special locus of level $j$ is closed in $\Xbar^{\gen}_{G, \rhobarss}$ by Lemma \ref{V_iks} and $G^0$-invariant by \Cref{setup}, 
its image in $\Xbar^{\gps}_{G,\rhobarss}$ is also closed by \cite[Theorem 3 (iii)]{seshadri}. 
To prove part (2),  let $\eta_1, \ldots, \eta_r$ be the 
generic points of $C_{W,j}$. It is enough to find $H_i$ satisfying the required bounds on the
dimension such that $\eta_i \in \Xbar^{\gps}_{H_iG, \rhobarss}$. 

To ease the notation let $X$ be the $W$-special locus of level $j$ in $\Xbar^{\gen}_{G, \rhobarss}$, let $\eta=\eta_1$ and let $\bar{\eta}$ be a geometric point above $\eta$. It follows from 
\cite[Theorem 3]{seshadri} that there is $x\in X(\kappa(\bar{\eta}))$ above $\bar{\eta}$ such that the 
orbit $G^0 \cdot x$ is closed in $X_{\bar{\eta}}$. Since $X$ is closed in $\Xbar^{\gen}_{G, \rhobarss}$, Lemma \ref{closedimmersionforGLd} together with \eqref{def_XgenG}
imply that $X_{\bar{\eta}}$ is a closed $G^0$-invariant subscheme of $G^N_{\bar{\eta}}$, where $x\in X_{\bar{\eta}}$ 
is mapped to the $N$-tuple $y=(\rho_x(\gamma_1), \ldots, \rho_x(\gamma_N))$, where $\gamma_i$ are elements of $\Gamma_F$ chosen before Lemma \ref{closedimmersionforGLd}. 
Hence, the orbit $G^0 \cdot y$ is closed in $G^N_{\bar{\eta}}$ under the conjugation action. Let $H_x$ be the smallest closed subgroup 
of $G_{\bar{\eta}}$ containing the entries of $y$. It follows from \cite[Theorem 3.1]{BMR} that $H_x$ is a strongly reductive subgroup of $G_{\bar{\eta}}$, and thus 
is reductive by \cite[Section 6]{martin}. It follows from Lemma \ref{dense_again} that  $H_x$ is equal to the Zariski closure of $\rho_x(\Gamma_F)$ in $G_{\bar{\eta}}$. Since $x\in X$, the dimension of the $H_x^0$-invariants
of $W_x$ is at least $j+1$ by Lemma \ref{inv_non}. If $\dim H_x= \dim G_k$ then $H_x^0= G_{\etabar}^0$ and hence 
$W^{G^0}\neq 0$. If $W^{G'}=0$ then $\dim H_x \le \dim G_k - 2$ by \Cref{cor_two}.

By \cite[Theorem 10.3]{martin} there is a closed reductive 
subgroup $H$ of $G$ defined over $\bar{k}$ and $g\in G(\kappa(\etabar))$ such that $H_x= g (H_{\etabar}) g^{-1}$. Since $H$ is defined by finitely many polynomials with coefficients in $\bar{k}$, $H$ is 
defined over a finite extension of $k$. The point $x':= g(x)\in \XbargenH(\kappa(\etabar))$ maps to $\eta$ in $\Xbar^{\gps}_{G, \rhobarss}$, thus $\eta$ lies in $\Xbar^{\gps}_{HG, \rhobarss}$. 
\end{proof}

\subsection{Bounding the dimension of the space}\label{bound_space} 
In this subsection we carry out the induction to prove \eqref{dim_pair}.
It follows from \cite[Theorem 9.3 (3)]{defT} that \eqref{dim_pair} holds
if $H^0$ is a torus. Thus we may assume that $\dim H\ge 3$. 

\begin{prop}\label{bound_Levi} 
    Let $L$ be an R-Levi of $G$. Let $\Vbar_{LG, \rhobarss}$ be 
    the preimage of $\Ubar_{LG, \rhobarss}$ in $\Xbar^{\gen}_{G, \rhobarss}$ and let $\Zbar_{LG, \rhobarss}$ be 
    its closure. Assume that $L\neq G$ and $\Ubar_{LG, \rhobarss}$ is non-empty.
    If \eqref{dim_pair} holds for all $\rhobarss$-compatible pairs $(H_1, \rhobar_1)$ with $\dim H_1< \dim G_k$ then
    \begin{equation}\label{eq_result}
    \dim G_k ([F:\Qp]+1) - \dim \Zbar_{LG, \rhobarss} \ge\min(2, \dim N_k) [F:\Qp]\ge 1, 
    \end{equation}
    where $N$ is the unipotent radical 
    of an R-parabolic subgroup with R-Levi $L$. Moreover, if either $F\neq \Qp$ or $\dim N_k \neq 1$
    then the lower bound can be replaced by $2$. 
\end{prop}

\begin{proof} Lemma \ref{V_LG_maxdim} implies that $\dim N_k\ge 1$. Hence, $\min(2, \dim N_k)[F:\Qp]\ge 1$ with 
equality if and only if  $F=\Qp$ and $\dim N_k=1$.

It follows from Lemma \ref{Rlevi} that there are only finitely many R-parabolic subgroups 
$P$ with R-Levi $L$. For each such $P$ the Lie algebra of its unipotent radical $\Lie(\rad(P))$
is a free $\OO$-module of finite rank with the adjoint action of  $L$. Moreover, 
$\Lie(\rad(P))\otimes_{\OO} k = \Lie(\rad(P_k))$. Let $W_P$ denote the 
$k$-linear dual of $\Lie(\rad(P_k))$ and for each $j\ge 0$ let $C_{W_P, j}$ be the image in $\Xbar^{\gps}_{L, \rhobarss}$ of the 
$W_P$-special locus of level $j$ in $\Xbar^{\gen}_{L, \rhobarss}$. If $C_{W_P, j}$ is nonempty then it follows from Proposition \ref{WGH}, Corollary \ref{Lzero} and Lemma \ref{HL}
that (after possibly replacing $k$ by a finite extension) there exist finitely many closed reductive subgroups $H_{ij}$, for $1\le i \le r_{P,j}$, of $L_k$ defined over $k$, 
such that $C_{W_P, j}$ is contained in the union of $\Xbar^{\gps}_{H_{ij} L, \rhobarss}$ and
\begin{equation}\label{cantucci}
\dim_k W_P^{H^0_{ij}}\ge j+1
\end{equation} 
and one of the 
following holds:
\begin{itemize} 
\item[(a)] $\dim L_k -\dim H_{ij} \ge 2$, or
\item[(b)] $\dim L_k -\dim H_{ij}=1$ and $\dim_k W_P^{H^0_{ij}}=1$. 
\end{itemize} 
We note that if $j\ge 1$ then \eqref{cantucci} rules out (b), so that (a) holds. 

Let $\Vbar_{P,ij}$ be the preimage of $\Ubar_{P,ij}:=\Ubar_{LG, \rhobarss}\cap \Xbar^{\gps}_{H_{ij} L, \rhobarss}$  
in $\Xbar^{\gen}_{G, \rhobarss}$ and let $\Zbar_{P, ij}$ be its closure. If part  (a) holds then 
Corollary \ref{bound_1} gives us
\begin{equation}\label{eq_ij}
 \dim G_k ([F:\Qp]+1) - \dim \Zbar_{P,ij}\ge 2[F:\Qp].
\end{equation}
As remarked above part (a) holds for $j=1$. If part (b) holds for $j=0$ then 
we let $\Vbar'_{P,i}$ be the preimage in 
$\Xbar^{\gen}_{G, \rhobarss}$ of 
$$\Ubar'_{P,i}:= \Ubar_{P,i0} \setminus \bigcup_{Q} (\Ubar_{P,i0} \cap C_{W_Q, 1}),$$ 
where the union is taken over all R-parabolic subgroups $Q$ of $G$ with R-Levi $L$.
Let $\Zbar'_{P,i}$ be the closure of $\Vbar'_{P,i}$ in  $\Xbar^{\gen}_{G, \rhobarss}$.  

 Let $y: \Spec \kappa\rightarrow U'_{P,i}$ be a closed geometric point  and let $x\in \Xbar^{\gen}_{L, \rhobarss}(\kappa)$ lie in the preimage of $y$, such that $L\cdot x$ is a closed orbit in $\Xbar^{\gen}_{L, \rhobarss}$. Let $\rho_x: \Gamma_F \rightarrow L(\kappa)$ be the 
 specialisation of the universal representation of $\Gamma_F$ over $\Xbar^{\gen}_{L, \rhobarss}$
 at $x$. Let $x'$ be the image of $x$ in $\Xbar^{\gen}_{G, \rhobarss}$ and let $\rho_{x'}$ 
 be the specialisation of the universal representation over $\Xbar^{\gen}_{G, \rhobarss}$
 at $x'$. Then $\rho_{x'}$ is equal to the composition of $\rho_x$ with the embedding
 $L(\kappa)\hookrightarrow G(\kappa)$. Since the orbit $L\cdot x$ is closed $\rho_x$ is 
 $L$-semisimple and hence $\rho_{x'}$ is $G$-semisimple and the orbit $G\cdot x'$ is 
 closed in $X^{\gen}_{G,y}$. We let $\Gamma_F$ act on $W_{Q, \kappa}$ via $\rho_x$. 
 If $h^0(\Gamma_F, W_{Q, \kappa}(1))\ge 2$ then $x'$ is in $W_{Q}$-special locus of level $1$. 
 Since $x'$ maps to $y$, we would obtain that $y\in C_{W_Q, 1}$, which yields a contradiction. 
 Hence $h^0(\Gamma_F, W_{Q, \kappa}(1))\le 1$, which implies that $\delta(y)\le 1$ and 
hence $\delta(U'_{P,i})\le 1$. Since  part (b) holds by assumption we deduce 
that $\delta(U'_{P,i})\le \dim L_k -\dim H_{i0}$. Corollary \ref{bound_2} implies that 
\begin{equation}\label{eq_prime_i}
\dim G_k([F:\Qp] +1) - \dim \Zbar'_{P,i} \ge \dim N_k [F:\Qp].
\end{equation}
Since $Z_{P, i0} \subseteq \Zbar'_{P,i} \cup \bigcup_{Q,k} \Zbar_{Q,k1}$ we deduce from \eqref{eq_ij} and \eqref{eq_prime_i} that 
\begin{equation}\label{eq_i_zero}
\dim G_k([F:\Qp] +1) - \dim \Zbar_{P,i0} \ge \min(2, \dim N_k)[F:\Qp].
\end{equation}
Let $\Vbar''$ be the preimage of $\Ubar'':= \Ubar_{LG, \rhobarss} \setminus \bigcup_P (\Ubar_{LG, \rhobarss}\cap C_{W_P,0})$ 
and let $\Zbar''$ be its closure. The same argument as above shows that $\delta(\Ubar'')=0$. Corollary 
\ref{bound_2} applied with $H=L$ gives
\begin{equation}\label{eq_double_prime} 
\dim G_k([F:\Qp] +1) - \dim \Zbar'' \ge \dim N_k [F:\Qp].
\end{equation}
Since $\Zbar_{LG, \rhobarss} = \Zbar''\cup \bigcup_{P,i} \Zbar_{P,i0}$ we conclude from \eqref{eq_i_zero} and
\eqref{eq_double_prime} that \eqref{eq_result} holds. 
\end{proof}

Let $G'_{\sic}\rightarrow G'$ be the simply connected central cover of $G'$. We have 
shown in \Cref{Gprimesc_action} that $\Lie G'_{\sic}$ is naturally a $G$-representation.

\begin{prop}\label{bound_special} Let $\Xbar^{\spcl}_{G,\rhobarss}$ be the $(\Lie G_{\sic}')^*$-special locus in $\Xbar^{\gen}_{G, \rhobarss}$. 
Let $\Vbar^{\spcl}_{GG, \rhobarss}=\Vbar_{GG, \rhobarss}\cap \Xbar^{\spcl}_{G,\rhobarss}$  and let 
 $\Zbar^{\spcl}_{GG, \rhobarss}$ be its closure. If \eqref{dim_pair} holds for all $\rhobarss$-compatible pairs $(H_1, \rhobar_1)$ with $\dim H_1< \dim G_k$ then
\begin{equation}\label{eq_spcl}
\dim G_k ([F:\Qp]+1) - \dim \Zbar_{GG, \rhobarss}^{\spcl} \ge 2 [F:\Qp]. 
\end{equation}
\end{prop} 

\begin{proof} We may assume that $\Vbar^{\spcl}_{GG, \rhobarss}$ is non-empty. Let $W=(\Lie G_{\sic}')^*$ and let $C_{W}$ be the image of $\Xbar^{\spcl}_{G,\rhobarss}$ in $\Xbar^{\gps}_{G,\rhobarss}$. 
Since $W^{G'}=0$ by \Cref{sic}, Proposition \ref{WGH} implies that 
there exist finitely many reductive subgroups $H_i$ of $G_k$, such that 
$C_W$ is contained in the union of $\Xbar^{\gps}_{H_iG, \rhobarss}$ and 
$\dim G_k - \dim H_i\ge 2$. The assertion follows from Corollary \ref{bound_1}.
\end{proof}

\begin{defi}\label{defi_Vnspcl} We define the \emph{absolutely irreducible non-special locus} $V^{\nspcl}_{GG, \rhobarss}$ 
in $X^{\gen}_{G, \rhobarss}$ as the complement of the $(\Lie G'_{\sic})^*$-special locus in $V_{GG, \rhobarss}$. We 
define the \emph{absolutely irreducible non-special locus} $\Vbar^{\nspcl}_{GG, \rhobarss}$ 
in $\Xbar^{\gen}_{G, \rhobarss}$ as the complement of the $(\Lie G'_{\sic})_k^*$-special locus in $\Vbar_{GG, \rhobarss}$.
\end{defi}

\begin{remar} $\Vbar^{\nspcl}_{GG, \rhobarss}$ coincides 
with the reduced special fibre of $V^{\nspcl}_{GG, \rhobarss}$.
\end{remar}

Recall that $Y_{\rhobarss}$ is the preimage in $X^{\gen}_{G, \rhobarss}$ of the 
closed point in $X^{\gps}_{G, \rhobarss}$.

\begin{prop}\label{cpnhg} If \eqref{dim_pair} holds for all $\rhobarss$-compatible pairs $(H_1, \rhobar_1)$ with $\dim H_1< \dim G_k$ then 
the complement of $\Vbar^{\nspcl}_{GG, \rhobarss}$  in $\Xbar^{\gen}_{G, \rhobarss}$ 
 has positive codimension. Moreover, 
 \begin{equation}\label{fixing_eq}
 \dim \Xbar^{\gen}_{G,\rhobarss}= \dim \Vbar^{\nspcl}_{GG, \rhobarss}+1= \dim \Xbar^{\git}_{G, \rhobarss} + \dim G_k -\dim Z(G_k).
 \end{equation}
 In particular, $\Vbar^{\nspcl}_{GG, \rhobarss}$  is non-empty.
\end{prop} 

\begin{proof} We have $\Vbar^{\nspcl}_{GG, \rhobarss}$ is open in $\Xbar^{\gen}_{G, \rhobarss}$ and 
its complement $C$ is contained in $Y_{\rhobarss}\cup \Zbar^{\spcl}_{GG, \rhobarss} \cup \bigcup_L \Zbar_{LG, \rhobarss}$, where $L$ runs over R-Levi subgroups of $G$, such that $L\neq G$ and $L$ contains
a fixed maximal split torus. It follows from Propositions \ref{bound_Levi}, \ref{bound_special}
and \ref{bound_Y} that 
\begin{equation} \label{bound_C}
\dim G_k([F:\Qp]+1) - \dim C \ge 1.
\end{equation}
If $\Vbar^{\nspcl}_{GG, \rhobarss}$ is empty then $C= \Xbar^{\gen}_{G, \rhobarss}$ and we deduce
that 
\begin{equation}\label{wrong_bound}
\dim G_k([F:\Qp]+1) > \dim \Xbar^{\gen}_{G, \rhobarss}.
\end{equation}
The representation $\rhobar: \Gamma_F \rightarrow G(k)$ defines a closed point $x\in Y_{\rhobarss}(k)$. The completion of the local ring at $x$ with respect to the maximal ideal 
is isomorphic to $R^{\square}_{G,\rhobar}/\varpi$ 
by Lemma \ref{local_ring_def_ring}. Since $\dim R^{\square}_{G,\rhobar}/\varpi \ge \dim G_k([F:\Qp]+1)$ by 
Corollary \ref{represent_kappa}, we conclude that 
\begin{equation}\label{bound_d}
d:=\dim \Xbar^{\gen}_{G, \rhobarss}\ge  \dim G_k([F:\Qp]+1)
\end{equation}
contradicting \eqref{wrong_bound}. Thus $\Vbar^{\nspcl}_{GG, \rhobarss}$ is non-empty
and it follows from \eqref{bound_C} and \eqref{wrong_bound} that $C$ has positive codimension.

We may write $\Xbar^{\gen}_{G, \rhobarss} = \Zbar^{\nspcl}_{GG, \rhobarss} \cup C$, where  $\Zbar^{\nspcl}_{GG, \rhobarss}$ is the closure of $\Vbar^{\nspcl}_{GG, \rhobarss}$ in $\Xbar^{\gen}_{G, \rhobarss}$. We deduce that $\dim \Zbar^{\nspcl}_{GG, \rhobarss}=d$, 
as $C$ has positive codimension.
Lemma \ref{useful_V} implies the first equality in \eqref{fixing_eq}. 
We have $d-1=\dim \Vbar^{\nspcl}_{GG, \rhobarss} \le \dim \Vbar_{GG, \rhobarss} \le d -1$ by Lemma \ref{useful_V}. Thus $\dim \Vbar_{GG, \rhobarss}= d-1$. 
It follows from Corollary \ref{UV_GG} that 
$\dim \Ubar_{GG, \rhobarss} = d - 1 -\dim G_k  +\dim Z(G_k).$
Let  $\Zbar$ denote the closure of $\Ubar_{GG, \rhobarss}$ in $\Xbar^{\gps}_{G, \rhobarss}$. 
Lemma \ref{useful_U} implies 
\begin{equation}\label{eq_dim_Zbar}
\dim \Zbar= \dim \Ubar_{GG, \rhobarss}+1= d- \dim G_k +\dim Z(G)_k.
\end{equation}
We have $\Xbar^{\gps}_{G, \rhobarss} = \Zbar \cup \bigcup_L \Xbar^{\gps}_{LG, \rhobarss}$, where 
the union is taken over all R-Levi subgroups $L$ of $G$,   containing a fixed maximal split torus in $G$ and not equal to $G$. If $\Xbar^{\gps}_{LG, \rhobarss}$ is non-empty then 
$L_k$ is $\rhobarss$-compatible by Lemma \ref{L_compatible} and $\dim L_k<\dim G_k$ as $L\neq G$.
Thus we may apply \eqref{dim_pair} to deduce that 
\begin{equation}\label{last_eq}
\dim \Xbar^{\gps}_{LG, \rhobarss}\le \dim L_k [F:\Qp]+ \dim Z(L_k).
\end{equation}
Lemma  \ref{V_LG_maxdim} implies that $\dim N_k \ge 1$, where $N$ is the unipotent radical of any R-parabolic with R-Levi $L$. It follows
from Lemma \ref{RRX}, \eqref{bound_d}, \eqref{eq_dim_Zbar} and \eqref{last_eq} that $\dim \Zbar > \dim \Xbar^{\gps}_{LG, \rhobarss}$ and thus $\dim \Xbar^{\gps}_{G, \rhobarss}= \dim \Zbar$, which implies the last equality in \eqref{fixing_eq}.
\end{proof}

\begin{lem}\label{abs_irr_nspcl} Let $x$ be a closed point of $X^{\gen}_{G, \rhobarss}\setminus Y_{\rhobarss}$ 
and let 
$\rho_x: \Gamma_F \rightarrow G(\kappa(x))$ be the specialisation of the universal Galois 
representation over $X^{\gen}_{G, \rhobarss}$ at $x$. If 
$\pi_1(G')$ is \'etale or 
$x\in X^{\gen}_{G,\rhobarss}[1/p]$ then $x\in V^{\nspcl}_{GG, \rhobarss}$ 
 if and only if $\rho_x$ is absolutely $G$-irreducible and $H^2(\Gamma_F, \ad^0 \rho_x)=0$. 
\end{lem}
\begin{proof} If $\pi_1(G')$ is \'etale or $x\in X^{\gen}_{G,\rhobarss}[1/p]$ then
$(\Lie G'_{\sic})_x=(\Lie G')_x$.  
It follows from Proposition \ref{irr=stable} and Lemma \ref{V_iks} that 
$x\in V^{\nspcl}_{GG, \rhobarss}$ if and only if $\rho_x$ is absolutely 
$G$-irreducible and $H^0(\Gamma_F, (\ad^0 \rho_x)^*(1))=0$. Since $x$ is closed in 
$X^{\gen}_{G, \rhobarss}\setminus Y_{\rhobarss}$ its residue field is a local field 
by \cite[Lemmas 3.17, 3.18]{BIP_new}. Local Tate duality gives us an isomorphism 
$H^0(\Gamma_F, (\ad^0 \rho_x)^*(1))\cong H^2(\Gamma_F,\ad^0 \rho_x)^*$, which implies the assertion. 
\end{proof}

\begin{lem}\label{dim_Vnspcl} Assume that $\pi_1(G')$ is \'etale. If $V^{\nspcl}_{GG, \rhobarss}$ is non-empty then it is flat over $\Spec \OO$ of relative dimension $\dim G_k([F:\Qp]+1) -1$, which is also equal to the dimension of $\Vbar^{\nspcl}_{GG, \rhobarss}$.
\end{lem}
\begin{proof} Let $x$ be a closed point of $V:=V^{\nspcl}_{GG, \rhobarss}$. Since 
$H^2(\Gamma_F, \ad^0\rho_x)=0$ by Lemma \ref{abs_irr_nspcl}, Proposition \ref{present_over_RH} implies that 
\begin{equation}
R^{\square}_{G, \rho_x}\cong R^{\square}_{G/G', \varphi\circ \rho_x}\br{x_1,\ldots, x_r},
\end{equation}
where $\varphi: G \rightarrow G/G'$ is the quotient map and $r= \dim G'_k ([F:\Qp]+1)$. It is proved 
in \cite[Corollary 9.5]{defT}   that $X^{\gen}_{G/G', \varphi\circ \rhobarss}$ is 
$\OO$-flat of relative dimension $\dim (G/G')_k([F:\Qp] + 1)$. It follows from 
\Cref{333} that 
$R^{\square}_{G/G', \varphi\circ \rho_x}$ is flat over the coefficient ring $\Lambda$ for $\kappa(x)$ 
of relative dimension $\dim (G/G')_k([F:\Qp] + 1)$. Thus $R^{\square}_{G, \rho_x}$ is flat over $\Lambda$ of relative dimension $\dim G_k ([F:\Qp]+1)$. 

If the characteristic of $\kappa(x)$ is zero then $\Lambda$ is a finite field extension of $L$
and $R^{\square}_{G, \rho_x}\cong \hat{\OO}_{V,x}$ by Lemma \ref{local_ring_def_ring}. 
Thus $\OO_{V,x}$ is flat over $\OO$ and $\dim \OO_{V,x}= \dim \hat{\OO}_{V,x}= \dim G_k ([F:\Qp]+1)$. 

If the characteristic of $\kappa(x)$ is $p$ then $\Lambda$ is a DVR, which is flat over $\OO$ 
and $R^{\square}_{G, \rho_x}\cong \hat{\OO}_{V,x}\br{T}$ by Lemma \ref{local_ring_def_ring}.
Thus $\OO_{V,x}$ is flat over $\OO$ and $\dim \OO_{V,x}= \dim \hat{\OO}_{V,x}= \dim G_k ([F:\Qp]+1)$. 

Hence, $V$ is flat over $\Spec \OO$ and 
$\dim V= \max_{x} \dim \OO_{V,x}= \dim G_k ([F:\Qp]+1),$
where the maximum is taken over all the closed points $x$ in $V$. We have to subtract $1$ 
to get the relative dimension over $\OO$, which is equal to the dimension of the special fibre. 
\end{proof}

\begin{cor}\label{etale_bound_space} Assume that $\pi_1(G')$ is \'etale. If \eqref{dim_pair} holds for all $\rhobarss$-compatible pairs $(H_1, \rhobar_1)$ with $\dim H_1< \dim G_k$ then the following hold:
\begin{enumerate} 
\item $\dim \Xbar^{\gen}_{G, \rhobarss}= \dim G_k ([F:\Qp] +1)$;
\item $\dim \Xbar^{\gps}_{G, \rhobarss}= \dim G_k [F:\Qp] + \dim Z(G_k)$.
\end{enumerate} 
\end{cor} 
\begin{proof} The assertion follows from \Cref{cpnhg} and \Cref{dim_Vnspcl}.
\end{proof}

\begin{prop}\label{some_prop}
    If \eqref{dim_pair} holds for all pairs $(H_1, \rhobar_1)$ with $\dim H_1 < \dim G_k$, then \eqref{dim_pair} 
    holds for $(G_k, \rhobarss)$. 
\end{prop}

\begin{proof} 
    For the purpose of establishing \eqref{dim_pair}, after passing to a finite extension of $\OO$ we may assume, that $G^0$ is split, see \Cref{rem_extend_scalars}. 
    Let $\mu := Z(G^0) \cap G' = Z(G')$, $G_1 := G/\mu$ and let $\varphi_1 : G \to G_1$ be the projection map. 
    Since $Z(G^0)$ is a split torus, $\mu$ is a finite diagonalisable group scheme and we may decompose 
    $\mu$ canonically  into a $p$-part and a prime-to-$p$ part, which are both normal in $G$. Let $x$ be a closed point of $Y_{\rhobarss}$.
   \Cref{fin_et} and \Cref{fin_conn} imply that $\dim \Rbar^{\square}_{G, \rho_x}= \dim \Rbar^{\square}_{G_1, \varphi_1 \circ \rho_x}$.

    As $G'$ is the fppf sheafification of $A \mapsto [G^0(A), G^0(A)]$, the map  $G' \to G_1'$ is surjective with kernel $Z(G')$.
    So $G_1'$ is semisimple of adjoint type and we have $Z(G_1^0) \cap G_1' = Z(G_1') = 1$; this can be checked on geometric fibres using that semisimple groups over algebraically closed fields are centerless.
    We can apply \Cref{cor_nice_torus} to $Z(G_1^0)$ in $G_1$ and $G_2 := G_1/Z(G_1^0) = G/Z(G^0)$ and obtain
    \begin{align}
        \dim \Rbar^{\square}_{G_1, \varphi_1 \circ \rho_x}= \dim \Rbar^{\square}_{G_2, \varphi \circ \rho_x} + ([F:\Qp]+1) \dim Z(G^0)_k, \label{torus_transfer}
    \end{align}
   where $\varphi : G \to G/Z(G^0)$ is the projection map. 
   
    By \Cref{letter_bcnrd} we find a surjection of generalised reductive $k$-group schemes $G_3 \to (G_2)_k$ with finite connected kernel, such that $\pi_1(G_3^0)$ is \'etale.
    Since $G_3^0$ is semisimple, we have $G_3' = G_3^0$. Thus \Cref{etale_bound_space} can be applied to $G_3$.

Let $\kappa$ be the residue field of $x$. Then  $G_2(\kappa) = G_3(\kappa)$ and  we can lift $\varphi \circ \rho_x$ to a representation $\rho_3 : \Gamma_F \to G_3(\kappa)$.
    Using \Cref{fin_conn}, we get $\dim \Rbar^{\square}_{G_2, \varphi \circ \rho_x} = \dim \Rbar^{\square}_{G_3, \rho_3}$. Since $\kappa$ is finite $\rho_3$ defines a closed point in $\Xbar^{\gen}_{G_3, \rho_3^{\mathrm{ss}}}$ with a finite residue field. 
  Using \Cref{dim_closed_pts} for the inequality, we get
  \begin{equation}\label{transfer}
  \begin{split}
\dim \Rbar^{\square}_{G_1, \varphi_1 \circ \rho_x} &\overset{\eqref{torus_transfer}}{=} \dim \Rbar^{\square}_{G_3, \rho_3} + ([F:\Qp]+1) \dim Z(G^0)_k \\
        &\leq \dim \Xbar^{\gen}_{G_3, \rho_3^{\mathrm{ss}}} + ([F:\Qp]+1) \dim Z(G^0)_k.
        \end{split}
\end{equation}
\Cref{etale_bound_space} applied to  $G_3$ gives 
        $\dim \Xbar^{\gen}_{G_3, \rho_3^{\mathrm{ss}}} = \dim G_3 ([F:\Qp] +1)$. Since 
        $\dim G_k = \dim (G_2)_k + \dim Z(G^0)_k = \dim G_3 + \dim Z(G^0)_k$, we obtain 
        \begin{equation}
          \dim \Rbar^{\square}_{G, \rho_x}\le \dim G_k ([F:\Qp]+1), \quad \forall x\in Y_{\rhobarss}.
          \end{equation}
       \Cref{dim_closed_pts} implies that $\dim \Xbar^{\gen}_{G, \rhobarss} \le \dim G_k ([F:\Qp]+1)$. It follows 
from \eqref{bound_d} that this is an equality. The assertion follows from Proposition \ref{cpnhg}.
\end{proof}

\begin{thm}\label{thm_bound_space} 
The hypothesis \eqref{dim_pair} holds for all pairs 
$(H_1, \rhobar_1)$ with $\dim H_1< \dim G_k$. In particular, the following hold:
\begin{enumerate}
\item the complement of $\Vbar^{\nspcl}_{GG, \rhobarss}$  in $\Xbar^{\gen}_{G, \rhobarss}$ 
 has positive codimension;
\item $\dim \Xbar^{\gen}_{G, \rhobarss}= \dim G_k ([F:\Qp] +1)$;
\item $\dim \Xbar^{\gps}_{G, \rhobarss}= \dim G_k [F:\Qp] + \dim Z(G_k)$.
\end{enumerate} 
\end{thm}

\begin{proof}
Let $S$ be the set of pairs $(H_1, \rhobar_1)$ such that $\dim H_1 < \dim G_k$, and let $S'$ be the subset of $S$ consisting of those pairs for which \eqref{dim_pair} does not hold.
  If $S'$ is non-empty then let $(H_1, \rhobar_1) \in S'$ be such that $\dim H_1$ is minimal.
  Then \eqref{dim_pair} holds for all pairs of dimension less than $\dim H_1$, so \Cref{some_prop} applies to $(H_1, \rhobar_1)$, yielding a contradiction.
  So $S' = \emptyset$ and  \Cref{some_prop} implies that \eqref{dim_pair}
  holds for $(G_k, \rhobarss)$. Moreover, \Cref{cpnhg} implies that 
  $\dim \Xbar^{\gen}_{G, \rhobarss} \le \dim G_k ([F:\Qp]+1)$, and \eqref{bound_d}
  implies that part (2) of the Theorem holds. Parts (1) and (3) follow from 
  \Cref{cpnhg}.
\end{proof}

\subsection{Consequences}\label{sec_consequence}

We record some consequences of Theorem \ref{thm_bound_space}. Most of the results follow 
the proof of analogous results for $\GL_d$ in \cite{BIP_new} with the exception of Corollaries 
\ref{BJ1} and \ref{BJ2}, where we prove the analogues of the main results of \cite{BJ_new}
and the argument is new even when $G=\GL_d$.

\begin{cor}\label{complete_intersection}
Let $x$ be either a closed point of $X^{\gen}_{G, \rhobarss}\setminus Y_{\rhobarss}$ or  a closed point of $Y_{\rhobarss}$ and let $\Lambda$ be a coefficient ring for $\kappa(x)$ introduced in \Cref{sec_def_rings}.  Then the following hold:
\begin{enumerate}
\item $R^{\square}_{G, \rho_x}$ is a flat $\Lambda$-algebra of relative 
dimension $\dim G_k([F:\Qp]+1)$ and is a complete intersection;
\item if $\chara(\kappa(x))=p$ then $R^{\square}_{G, \rho_x}/\varpi$ is a complete 
intersection of dimension $\dim G_k([F:\Qp]+1)$.
\end{enumerate}
\end{cor}
\begin{proof} The proof is essentially the same as the proof of \cite[Corollary 3.38]{BIP_new}, so 
we only give a sketch. Corollary \ref{represent_kappa} and Lemma \ref{pre_lci_flat} applied 
with $A=\Lambda$ and $B=R^{\square}_{G, \rho_x}$ imply that to prove part (1) it is enough 
to show that 
\begin{equation}\label{need_to_prove}
 \dim \kappa(x)\otimes_{\Lambda} R^{\square}_{G, \rho_x}\le \dim G_k ( [F:\Qp]+1).
\end{equation} 
If $\chara(\kappa(x))=0$ then $\Lambda=\kappa(x)$. If $\chara(\kappa(x))=p$ then $\Lambda$
is a DVR with uniformiser $\varpi$ and the fibre ring in \eqref{need_to_prove} is $R^{\square}_{G, \rho_x}/\varpi$. Thus, if we can prove \eqref{need_to_prove} then Lemma \ref{pre_lci_flat} applied 
with $A=\kappa(x)$ and $B=R^{\square}_{\rho_x}/\varpi$ will imply part (2). 

Let $X:= X^{\gen}_{G, \rhobarss}$ and let $\Xbar:= \Xbar^{\gen}_{G, \rhobarss}$. 
If $x\in Y_{\rhobarss}$ then Lemma \ref{local_ring_dim} applied to $Z=\overline{X}$ gives us 
$\dim \OO_{\Xbar, x}\le \dim \Xbar= \dim G_k ([F:\Qp]+1)$, where the last equality follows from \Cref{thm_bound_space} (2). Lemma \ref{local_ring_def_ring} 
implies that $R^{\square}_{\rho_x}/\varpi\cong \hat{\OO}_{\Xbar,x}$ and this implies \eqref{need_to_prove}. 
If $x\in \Xbar\setminus Y_{\rhobarss}$ then $R^{\square}_{\rho_x}/\varpi\cong \hat{\OO}_{\Xbar,x}\br{T}$
by Lemma \ref{local_ring_def_ring} and Lemma \ref{local_ring_dim} gives
$\dim \OO_{\Xbar, x} +1 \le \dim \Xbar$. If $\chara(\kappa(x))=0$ then 
$x\in X[1/p]$, $R^{\square}_{G, \rho_x}\cong \hat{\OO}_{X, x}$
by Lemma \ref{local_ring_def_ring} and $\dim \OO_{X, x} \le \dim X -1 \le \dim \Xbar$ 
by Lemma \ref{local_ring_dim} applied to $Z=X$ and Proposition \ref{dim_sp_gen}. Hence, 
\eqref{need_to_prove} holds also in this case. 
\end{proof} 

\begin{cor}\label{dim_XgenG} Let $d= \dim G_k([F:\Qp]+1)$. Then the following hold: 
\begin{enumerate} 
\item $X^{\gen}_{G, \rhobarss}$ is flat over $\Spec \OO$, equidimensional of relative
dimension $d$ and is locally complete intersection;
\item $\Xbar^{\gen}_{G, \rhobarss}$ is equidimensional of dimension $d$ and is locally complete 
intersection.
\end{enumerate} 
\end{cor}
\begin{proof} The proof is essentially the same as the proof of \cite[Corollary 3.40]{BIP_new} 
so we only give a sketch. Lemma \ref{local_ring_def_ring} relates $R^{\square}_{G, \rho_x}$ and the local ring 
$\OO_{X, x}$, where $X:=X^{\gen}_{G, \rhobarss}$ and $x$ is a closed point of $X$. 
Using Corollary \ref{complete_intersection} we deduce that $\OO_{X, x}$ are complete intersection 
and $\OO$-flat for all closed points $x$. Thus $X$ is flat over $\Spec \OO$ and locally complete 
intersection. The last property implies that $X$ is equidimensional. Part (1) implies part (2).
\end{proof} 

\begin{cor}\label{dim_XgenH} Let $H$ be a closed generalised reductive subgroup of $G$ (resp. $G_k$) such that 
$X^{\gen}_{H,\rhobarss}$ (resp. $\Xbar^{\gen}_{H, \rhobarss}$) is non-empty. 
Let $d=\dim H_k ([F:\Qp]+1)$ then $X^{\gen}_{H, \rhobarss}$ is $\OO$-flat of
relative dimension $d$ and is locally complete intersection (resp. 
$\Xbar^{\gen}_{H, \rhobarss}$ is equidimensional of dimension $d$ and is locally complete intersection).
\end{cor}
\begin{proof} This follows from Corollary \ref{dim_XgenG}  applied to $X^{\gen}_{H_j, \rhobarss_j}$ 
(resp. $\Xbar^{\gen}_{H_j, \rhobarss_j}$) appearing in  Corollary \ref{cover_Hi} (1).
\end{proof} 

\begin{cor}\label{non-spcl_dense} The non-special absolutely irreducible locus is Zariski dense in $X^{\gen}_{G, \rhobarss}$ and 
in $\Xbar^{\gen}_{G,\rhobarss}$.
\end{cor} 

\begin{proof}  \Cref{thm_bound_space} (1) implies that
the complement $\overline{C}$ of $\Vbar^{\nspcl}_{GG, \rhobarss}$ in $\Xbar^{\gen}_{G, \rhobarss}$
is closed and has positive codimension. Since $\Xbar^{\gen}_{G, \rhobarss}$ is equidimensional
by Corollary \ref{dim_XgenG} we conclude that $\Vbar^{\nspcl}_{GG, \rhobarss}$ contains 
all the generic points of $\Xbar^{\gen}_{G, \rhobarss}$ and hence is dense. 
Let $C$ be the complement of $V^{\nspcl}_{GG, \rhobarss}$ in $X^{\gen}_{G, \rhobarss}$. 
It follows from Lemma \ref{abs_irr_nspcl} 
that the reduced special fibre of $C$ is equal to $\overline{C}$. 
Since $C$ is $G^0$-invariant we have $\dim C \le \dim \overline{C}+1$ by Proposition \ref{dim_sp_gen}. 
Since $\dim X^{\gen}_{G, \rhobarss}= \dim \Xbar^{\gen}_{G, \rhobarss}+1$ by Corollary \ref{dim_XgenG}, 
$C$ has positive codimension in $X^{\gen}_{G, \rhobarss}$ and the same argument 
shows that $V^{\nspcl}_{GG, \rhobarss}$ is dense in $X^{\gen}_{G, \rhobarss}$. 
\end{proof}

\begin{cor}\label{BJ1} Let $d=\dim G_k[F:\Qp]+\dim Z(G_k)$. Then the following hold:
\begin{enumerate}
\item $X^{\ps}_{G, \rhobarss}$ is equidimensional of dimension $d+1$; 
\item $\Xbar^{\ps}_{G, \rhobarss}$ is equidimensional of dimension $d$.
\end{enumerate}
\end{cor} 
\begin{proof} Since the claim is about the underlying topological spaces, 
it is enough to prove the Corollary after replacing $X^{\ps}_{G, \rhobarss}$ with $X^{\gps}_{G, \rhobarss}$ 
and $\Xbar^{\ps}_{G, \rhobarss}$ with $\Xbar^{\gps}_{G, \rhobarss}$, as these spaces are homeomorphic by 
Proposition \ref{nu_fin_u}. 

We will prove (1); the proof of (2) is identical. Let $\eta$ be a
generic point of $X^{\gps}_{G, \rhobarss}$. Let $\etabar$ be a geometric point above $\eta$. 
It follows from \cite[Theorem 3 (ii)]{seshadri} that there is a geometric point $\xi: \Spec \kappa(\etabar)\rightarrow
X^{\gen}_{G, \rhobarss}$ above $\etabar$. Let $X$ be an irreducible component of $X^{\gen}_{G, \rhobarss}$
containing $\xi$ and let $X^{\gps}$ be the image of $X$ in $X^{\gps}_{G, \rhobarss}$. Then $X^{\gps}$ contains 
$\eta$. It follows from \cite[Lemma 2.1]{BIP_new} that $X$ is $G^0$-invariant. 
Since $X$ is closed in $X^{\gen}_{G, \rhobarss}$, \cite[Theorem 3 (iii)]{seshadri} implies that 
$X^{\gps}$ is closed in $X^{\gps}_{G, \rhobarss}$. Since $X$ is irreducible, $X^{\gps}$ is also irreducible, \cite[\href{https://stacks.math.columbia.edu/tag/0379}{Tag 0379}]{stacks-project}. We conclude that 
$X^{\gps}$ is the irreducible component of  $X^{\gps}_{G, \rhobarss}$ containing $\eta$.

Corollary \ref{non-spcl_dense} implies that $V:=X\cap V_{GG, \rhobarss}$ is non-empty 
open dense subset of $X$, thus $U:= X^{\gps}\cap U_{GG, \rhobarss}$ is a non-empty 
open dense subset of $X^{\gps}$. Lemmas \ref{useful_U} and \ref{useful_V} give us
$\dim X = \dim V+1$, $\dim X^{\gps}=\dim U+1$.
Corollary \ref{UV_GG} implies that 
$ \dim V = \dim U + \dim G_k - \dim Z(G_k),$
which together with Corollary \ref{dim_XgenG} implies that $\dim X^{\gps} = \dim G_k [F:\Qp] + \dim Z(G_k)+1$. 
\end{proof} 

\begin{cor}\label{BJ2} The image of the non-special absolutely irreducible locus is Zariski dense in 
$X^{\ps}_{G, \rhobarss}$ and in $\Xbar^{\ps}_{G, \rhobarss}$. 
\end{cor}

\begin{proof} As explained in the proof of Corollary \ref{BJ1} it is enough to prove the 
statement for $X^{\gps}_{G, \rhobarss}$ and $\Xbar^{\gps}_{G, \rhobarss}$. 
Let $U^{\nspcl}_{GG, \rhobarss}$ be the image of $V^{\nspcl}_{G, \rhobarss}$ in $X^{\gps}_{G, \rhobarss}$ and let $U$ be a non-empty open subset of $X^{\gps}_{G, \rhobarss}$. Then its preimage 
$V$ in $X^{\gen}_{G, \rhobarss}$ is open and non-empty. Corollary \ref{non-spcl_dense} implies 
that $V\cap V^{\nspcl}_{GG, \rhobarss}$ is non-trivial, which implies that 
$U\cap U^{\nspcl}_{GG, \rhobarss}$ is non-trivial, and hence $U^{\nspcl}_{GG, \rhobarss}$
is dense in $X^{\gps}_{G, \rhobarss}$. The argument with $\Xbar^{\gps}_{G, \rhobarss}$ is the same.
\end{proof} 

\begin{remar} Corollaries \ref{BJ1} and \ref{BJ2} are the analogues of the main results 
of \cite{BJ_new}; compare \cite[Theorem 5.5.1]{BJ_new}. Even if $G=\GL_d$ the proof given here is new: although we use 
many techniques introduced by B\"ockle and Juschka in \cite{BJ_new}, we never appeal to 
\cite[Section 5]{BJ_new}.
\end{remar}

\begin{cor} Let $\rho: \Gamma_F \rightarrow G(\kappa)$ be a continuous representation 
with $\kappa$ a local field. Then the conclusions of Corollary \ref{complete_intersection}
apply to $R^{\square}_{G, \rho}$. 
\end{cor} 
\begin{proof} Using Lemma \ref{extend_scalars1} it is enough to prove the statement 
after replacing $\kappa$ by a finite extension. Using Proposition \ref{building} we may assume 
that there exist $g\in G^0(\kappa)$ such that $\rho'(\gamma):=g \rho(\gamma)g^{-1} \in G(\OO_{\kappa})$ 
for all $\gamma\in \Gamma_F$. Since $G$ is smooth over $\OO$ we may pick $h\in G(\Lambda)$ which maps 
to $g$ in $G(\kappa)$. Conjugation by $h$ induces an isomorphism of $\Lambda$-algebras between 
$R^{\square}_{G, \rho}$ and $R^{\square}_{G, \rho'}$, so we may assume that $\rho$ takes values in 
$G(\OO_{\kappa})$. Let $\rhobar$ be the reduction of $\rho$ modulo the uniformiser 
of $\OO_{\kappa}$. Let $\OO$ be the Witt ring of residue field of $\OO_{\kappa}$. 
Since $\rho$ is a deformation of $\rhobar$ to $\OO_{\kappa}$ we obtain 
a map of local $\OO$-algebras $R^{\square}_{G, \rhobar} \rightarrow \OO_{\kappa}$. 
Lemma \ref{extend_scalars2} says that $R^{\square}_{G, \rhobar}$ is naturally 
isomorphic to the completion of $\Lambda\otimes_{\OO} R^{\square}_{G, \rhobar}$ 
with respect to the kernel $\qq$ of the natural map to $\kappa(x)$.
Corollary \ref{complete_intersection} together with 
\cite[Lemmas 3.36, 3.37]{BIP_new} allow us to bound the 
dimension of the completion by $\dim \Lambda +\dim G_k([F:\Qp]+1)$, 
and if $\chara(\kappa(x))=p$ then the dimension of the
special fibre of the completion by $\dim G_k([F:\Qp]+1)$. We obtain 
$ \dim \kappa\otimes_{\Lambda} R^{\square}_{G, \rho}\le \dim G_k([F:\Qp]+1)$
and the proof of Corollary \ref{complete_intersection} goes through. 
\end{proof}

\begin{cor}\label{char_zero_lift} Every continuous representation $\rhobar: \Gamma_F \rightarrow G(k)$ can be lifted to characteristic 
zero. More precisely, there is a finite extension $L'$ of $L$ with a ring of integers $\OO'$ and uniformiser $\varpi'$ and 
a continuous representation
$\rho: \Gamma_F \rightarrow G(\OO')$ such that $\rho \equiv \rhobar \pmod{\varpi'}$.
\end{cor} 
\begin{proof} It follows from Corollary \ref{complete_intersection} that 
$R^{\square}_{G, \rhobar}[1/p]$
is non-zero. If $x$ is a closed point of $\Spec R^{\square}_{G, \rhobar}[1/p]$ then 
its residue field is a finite extension of $L$ and specialising the universal 
deformation along $x$ gives the required lift. 
\end{proof} 

\begin{cor}\label{dense_formal} Let $\rhobar: \Gamma_F \rightarrow G(k)$ be a continuous representation, 
$X^{\square}_{G, \rhobar}:= \Spec R^{\square}_{G, \rhobar}$ and let $\ast$
be the closed point in $X^{\square}_{G, \rhobar}$. Let $\Sigma^{\nspcl}$ be the
subset of closed points of $X^{\square}_{G, \rhobar}\setminus \{\ast\}$ such that 
$x\in \Sigma^{\nspcl}$ if and only if $\rho_x$ is absolutely $G$-irreducible and 
$H^2(\Gamma_F, (\Lie G'_{\sic})_x)=0$. Then $\Sigma^{\nspcl}$ is Zariski dense
in $X^{\square}_{G, \rhobar}$ and $\Sigma^{\nspcl}\cap \Xbar^{\square}_{G, \rhobar}$ 
is Zariski dense in the special fibre $\Xbar^{\square}_{G, \rhobar}$. 
\end{cor}
\begin{proof} The proof is essentially the same as the proof of \cite[Proposition 3.55]{BIP_new}. 
The map $X^{\square}_{G, \rhobar} \rightarrow X^{\gen}_{G, \rhobarss}$ is flat, since 
it is a localisation followed by completion. Let $V$ be a preimage of $V^{\nspcl}_{GG, \rhobarss}$
in $X^{\square}_{G, \rhobar}$. Then $V$ is contained in $X^{\square}_{G, \rhobar}\setminus \{\ast\}$
and Lemma 3.54 in \cite{BIP_new} together with Corollary \ref{non-spcl_dense} implies that $V$ is Zariski dense in $X^{\square}_{G, \rhobar}$.
The set of closed points in $V$ is equal to $\Sigma^{\nspcl}$. 
Since $V$ is open $X^{\square}_{G, \rhobar}\setminus \{\ast\}$, it is Jacobson, and hence $\Sigma^{\nspcl}$ is dense in $V$ and also in $X^{\square}_{G, \rhobar}$. The same argument works for the special fibre. 
\end{proof} 

\begin{remar} Corollary \ref{dense_formal} implies that $\Sigma^{\nspcl}\cap U$ is dense in every 
open subscheme $U$ of $X^{\square}_{G, \rhobar}$. In particular, this applies when $U$ is the generic fibre $X^{\square}_{G, \rhobar}[1/p]$. 
\end{remar}

\begin{cor} One may choose a characteristic zero lift $\rho$ of $\rhobar$ in Corollary 
\ref{char_zero_lift} to be absolutely irreducible with $H^2(\Gamma_F, \ad^0 \rho)=0$.
\end{cor} 
\begin{proof} Since $X^{\square}_{G, \rhobar}[1/p]$ is non-empty  and open 
in $X^{\square}_{G, \rhobar}$, and $\Sigma^{\nspcl}$ is dense by Corollary \ref{dense_formal}, 
we conclude that $X^{\square}_{G, \rhobar}[1/p]\cap \Sigma^{\nspcl}$ is non-empty. 
Lemma \ref{abs_irr_nspcl} implies that the specialisation of the universal deformation at $x\in X^{\square}_{G, \rhobar}[1/p]\cap \Sigma^{\nspcl}$
gives the required lift. 
\end{proof}

\section{Non-special locus}\label{sec:bound_special} 
Let $G$ be a generalised reductive $\OO$-group scheme. By \Cref{Gprimesc_action}
we have a $G$-equivariant map $\Lie G'_{\sic}\rightarrow \Lie G'$, where 
$G'_{\sic}\rightarrow G'$ is the simply connected central cover of $G'$. 
This map becomes an isomorphism after inverting $p$.  

Let $X^{\spcl}_{G, \rhobarss}$ and $\Xbar^{\spcl}_{G,\rhobarss}$ be the $(\Lie G'_{\sic})^{\ast}$-special locus in $X^{\gen}_{G, \rhobarss}$ and in $\Xbar^{\gen}_{G, \rhobarss}$, respectively. Then $\Xbar^{\spcl}_{G,\rhobarss}$ coincides with the reduced special fibre of $X^{\spcl}_{G, \rhobarss}$. In this subsection we bound the codimension of $X^{\spcl}_{G, \rhobarss}$ in 
$X^{\gen}_{G,\rhobarss}$ by $1+[F:\Qp]$ from below.  Recall via Definition 
\ref{defi_W_special} that $x\in X^{\gen}_{G,\rhobarss}$
lies in $X^{\spcl}_{G, \rhobarss}$ if and only if $h^0(\Gamma_F, (\Lie G'_{\sic})^{\ast}_x(1))\neq 0$.
If $x$ is a closed point of $Y_{\rhobarss}$ or a closed point of $X^{\gen}_{G,\rhobarss}\setminus 
Y_{\rhobarss}$ then $\kappa(x)$ is either a finite or a local field and local Tate duality implies that $x$ is $(\Lie G'_{\sic})^{\ast}$-non-special
if and only if $H^2(\Gamma_F, (\Lie G'_{\sic})_x)=0$. 

If $\pi_1(G')$ is \'etale  or 
$x\in X^{\gen}_{G,\rhobarss}[1/p]$  then $(\Lie G'_{\sic})_x = (\Lie G')_x$ and the above is equivalent to vanishing of  
$H^2(\Gamma_F, \ad^0 \rho_x)$. Proposition \ref{present_over_RH} implies that in this case
$R^{\square}_{G, \rho_x}$ is formally smooth over $R^{\square}_{G/G', \varphi\circ\rho_x}$, 
where $\varphi: G \rightarrow G/G'$ is the quotient map. In the next section we will use 
this to compute the set of irreducible components of $X^{\gen}_{G, \rhobarss}$. It will be important 
for us that $1+[F:\Qp]\ge 2$ as this will be used to verify Serre's criterion for normality.

If $G^0$ is a torus then $G'$ is trivial and there is nothing to verify. In particular, 
we assume that $\dim G_k >1$ in this section.  

Let $H$ be a closed generalised reductive subgroup of $G$ such that $X^{\gen}_{H, \rhobarss}$ is non-empty. 
Then $H^0$ acts on $G^0 \times X^{\gen}_{H, \rhobarss}$ via $h\cdot(g, x)= (gh^{-1}, h \cdot x)$.
We let $G^0\times^{H^0} X^{\gen}_{H, \rhobarss}$ be the GIT quotient
for this action. Since the action is free the fibre at a geometric point 
$x$ is isomorphic to $H^0_{\kappa(x)}$. This implies that 
\begin{equation}\label{dim_again_1}
 \dim (G^0\times^{H^0} X^{\gen}_{H, \rhobarss})= \dim G^0 + \dim X^{\gen}_{H, \rhobarss} -\dim H^0.
\end{equation}
It follows from Corollaries  \ref{dim_XgenG}, \ref{dim_XgenH} that if $X^{\gen}_{H, \rhobarss}$ is non-empty then 
\begin{equation}\label{dim_again_2}
 \dim X^{\gen}_{G, \rhobarss} - \dim X^{\gen}_{H, \rhobarss}= (\dim G -\dim H) ([F:\Qp]+1).
\end{equation}
A similar calculation can be made over $\Spec k$. 

\begin{defi} If $H$ is a closed generalised reductive subgroup of $G$ (resp. $G_k$) we denote 
by $G^0\cdot X^{\gen}_{H, \rhobarss}$ (resp. $G^0\cdot \Xbar^{\gen}_{H, \rhobarss}$) the scheme theoretic image 
of $G^0\times^{H^0} X^{\gen}_{H, \rhobarss}$ (resp. $G^0\times^{H^0} \Xbar^{\gen}_{H, \rhobarss}$) in $X^{\gen}_{G, \rhobarss}$. 
\end{defi}
The following Lemma formalises an argument made in \cite[Proposition 3.8]{BIP_new}.

\begin{lem}\label{comm_alg_bacher} Let $f:X\rightarrow Y$ be a morphism of affine schemes of finite type over 
$S=\Spec R$, where $R$ is a complete local noetherian $\OO$-algebra with residue field $k$. Let 
$s\in S$ be the closed point and let $X_s$ and $Y_s$ denote the fibres of $X$ and $Y$ 
at $s$ respectively. Suppose that the following hold:
\begin{enumerate} 
\item $f$ is dominant;
\item every irreducible component of $X$ intersects $X_s$ non-trivially;
\end{enumerate}
Then $\dim Y \le \dim X$. 
\end{lem}
\begin{proof} Let $U= X\setminus X_s$ and let $V=Y\setminus Y_s$. Since $Y_s$ is closed in $Y$
and $f(X_s)\subseteq Y_s$, the closure of $f(X_s)$ cannot contain the generic points 
of $V$. Since $f$ is dominant its restriction  $f:U\rightarrow V$ is still dominant. 
Now both $U$ and $V$ are Jacobson by \cite[Lemma 3.18 (1)]{BIP_new} and \cite[Lemma 3.14]{BIP_new} implies that 
$\dim V \le \dim U$. Let $Z$ be the closure of $U$ in $X$ and let $W$ be the closure of $V$ in 
$Y$. The assumption (2) and \cite[Lemma 3.18 (5)]{BIP_new} imply that $\dim Z = \dim U+1$. 
The proof of \cite[Lemma 3.18 (5)]{BIP_new}, namely the inequality (14) in \cite{BIP_new},  implies that $\dim W\le \dim V+1$. Thus $ \dim W\le \dim Z $. 

Let $W'$ be the closure of $f(X_s)$ in $Y_s$. Since both $X_s$ and $Y_s$ are of finite type
over $\Spec k$ they are Jacobson, and hence $W'$ is Jacobson. Lemma 3.14 in \cite{BIP_new}
implies that $\dim W'\le \dim X_s$. Since $Y_s$ is closed in $Y$, $W'$ is also closed in $Y$. 

Since $f$ is dominant, we have $Y= W\cup W'$. Thus
$\dim Y= \max(\dim W, \dim W')\le \max(\dim Z, \dim X_s)=\dim X.$
\end{proof}

\begin{remar}\label{QpZp} We note that  $\Qp= \Zp[x]/(px-1)$ is 
of finite type over $\Zp$ and $\Spec \Qp \rightarrow \Spec \Zp$ is dominant, and so that the assumption (2) in Lemma \ref{comm_alg_bacher} is necessary. 
\end{remar}

\begin{prop}\label{Benediktiner} If $H$ is a closed generalised reductive subgroup of $G$ such that $X^{\gen}_{H, \rhobarss}$ is non-empty then 
\begin{equation}\label{kille}
\dim X^{\gen}_{G, \rhobarss} - \dim G^0\cdot X^{\gen}_{H, \rhobarss} \ge (\dim G - \dim H) [F:\Qp].
\end{equation}
If $H$ is a closed generalised reductive subgroup of $G_k$ such that $\Xbar^{\gen}_{H, \rhobarss}$ is non-empty then 
\begin{equation}\label{pitsch}
 \dim \Xbar^{\gen}_{G, \rhobarss} - \dim G^0\cdot \Xbar^{\gen}_{H, \rhobarss} \ge (\dim G_k - \dim H) [F:\Qp].
 \end{equation}
\end{prop}
\begin{proof} Let $H$ be a closed generalised reductive subgroup of $G$. We claim that the assumptions 
of Lemma \ref{comm_alg_bacher} are satisfied with $X:= G^0\times^{H^0} X^{\gen}_{H, \rhobarss}$
and $Y:= G^0\cdot X^{\gen}_{H, \rhobarss}$. Since $\dim G=\dim G^0$ and $\dim H=\dim H^0$, the claim together with Lemma \ref{comm_alg_bacher}, \eqref{dim_again_1} and \eqref{dim_again_2} imply \eqref{kille}.

To prove the claim we observe that the map $X^{\gen}_{H, \rhobarss} \rightarrow G^0\times X^{\gen}_{H, \rhobarss}$, $x\mapsto (1,x)$
induces an isomorphism  
$ X^{\gen}_{H, \rhobarss}\sslash H^0 \overset{\cong}{\longrightarrow} (G^0\times^{H^0} X^{\gen}_{H, \rhobarss})\sslash G^0$
with the inverse map induced by $(g, x)\mapsto x$. Now $X^{\gen}_{H, \rhobarss}\sslash H^0= X^{\gps}_{H, \rhobarss}$ is a spectrum of a finite product of local noetherian $\OO$-algebras by \eqref{semi-local}. 
Since $G^0$ is connected every irreducible component $X'$ of $X$ is $G^0$-invariant by \cite[Lemma 2.1]{BIP_new}
and hence its image in $X\sslash G^0$ is closed by Corollary 2 (ii) to \cite[Proposition 9]{seshadri}. 
Thus the image of $X'$ in $X\sslash G^0=X^{\gps}_{H, \rhobarss}$ will contain a closed point. 
Since $X^{\gps}_{H, \rhobarss} \rightarrow X^{\gps}_{G, \rhobarss}$ is finite, it sends closed
points to the closed point of $X^{\gps}_{G, \rhobarss}$. This implies that $X$ satisfies the assumption 
(2) of Lemma \ref{comm_alg_bacher}. The assumption (1) holds as $Y$ is by definition 
the closure of $f(X)$. The argument in the special fibre is the same. 
\end{proof}

\begin{prop}\label{alt_yeah}
Let $L$ be a $\rhobarss$-compatible R-Levi of $G$ such that $\dim G -\dim L=2$.  There is a closed subscheme 
$\overline{D}_L$ of $\Xbar^{\gen}_{G, \rhobarss}$ such that the following hold: 
\begin{enumerate}
\item $\dim \Xbar^{\gen}_{G, \rhobarss} - \dim \overline{D}_L\ge 1+[F:\Qp]$;
\item if $x\in \Vbar_{LG}$ and $x\not\in \overline{D}_L$ then $h^0(\Gamma_F, (\Lie G'_{\sic})^*_x(1))=0$.
\end{enumerate}
\end{prop}
\begin{proof} We want to define $\overline{D}_L$, so that for every geometric point 
$x\in \Vbar_{LG, \rhobarss}$ such that $x\not \in \overline{D}_L$ the assumptions of
Proposition \ref{vanish_H2_please} hold for $\rho_x: \Gamma_F \rightarrow G(\kappa(x))$, 
since then Proposition \ref{vanish_H2_please} implies that part (2) holds. 

If $x\in \Vbar_{LG, \rhobarss}$ then $\rho_x(\Gamma_F)$ is contained in 
$P(\kappa(x))$ for some R-parabolic subgroup $P$ of $G_{\kappa(x)}$ such that some
R-Levi of $P$ is conjugate to $L_{\kappa(x)}$. Thus if $x\not\in G^0\cdot \Xbar^{\gen}_{L, \rhobarss}$
then $\rho_x(\Gamma_F)$ is not contained in any R-Levi subgroup of $P$, and so part (1) 
of Proposition \ref{vanish_H2_please} holds for $\rho_x$. Moreover, Proposition \ref{Benediktiner}
implies that 
\begin{equation}\label{codim_L}
\dim \Xbar^{\gen}_{G, \rhobarss} - \dim G^0\cdot \Xbar^{\gen}_{L, \rhobarss} \ge 2 [F:\Qp].
\end{equation}

Let $\varphi: G_k \rightarrow \Aut(G^0_k)$ be the map induced by the conjugation 
action of $G_k$ on $G^0_k$. Let $\overline{G}= G_k/\Ker \varphi$ and $\overline{L}= L_k/\ker \varphi$.  It follows from \cite[Theorem 7.1.9]{bcnrd} 
and Proposition \ref{tasho1} that $\overline{G}_{\kbar}\cong G_1 \times \PGL_2$ and 
$\overline{L}_{\kbar} \cong G_1\times T_2$, where $G_1$ is a 
generalised reductive group and $T_2$ is the subgroup of diagonal matrices in $\PGL_2$. The isomorphisms are defined over some finite extension of $k$, and after extending 
scalars we may assume that it is defined over $k$. Let $p_2: G\rightarrow \PGL_2$ be the 
quotient map composed with the projection onto the second factor. 

Let  $H_0$ be the preimage of $\mu_{p-1} \subset T_2$ in $G$ under $p_2$. It follows from 
Lemma \ref{bound_centre_Hzero} applied to the base change of $H_0$ over $\kbar$ that $H_0$ is reductive and 
$\dim Z(H_0)< \dim Z(L_k)$. Since the inequality is strict we can improve the codimension bound
in Corollary \ref{bound_1} by $1$: 
\begin{equation}\label{codim_Hzero}
\dim \Xbar^{\gen}_{G, \rhobarss} - \dim \Zbar_{H_0 G, \rhobarss}\ge 1+ (\dim L_k - \dim H_0)[F:\Qp]= 1+[F:\Qp],
\end{equation}
where $\Zbar_{H_0 G, \rhobarss}$ is the closure of the preimage of $\Xbar^{\gps}_{H_0 G}\cap U_{LG, \rhobarss}$ in $\Xbar^{\gen}_{G, \rhobarss}$, and the extra $1$ comes as a difference between the dimensions of the centres in \eqref{third_ofmany} and 
\eqref{fifth_ofmany}.

If $x$ is a geometric point of $\Vbar_{LG, \rhobarss}$ then the $G$-semisimplification of 
$\rho_x$ is contained in $H_0$ if and only if $x\in \Zbar_{H_0 G, \rhobarss}$.
Thus if $x\not \in \Zbar_{H_0 G, \rhobarss}$ then part (1) of Proposition \ref{vanish_H2_please}
holds for $\rho_x$. 

If $G_1^0$ is a torus then $G_1'$ is trivial and part (3) of Proposition \ref{vanish_H2_please} 
holds for $\rho_x$. Proposition \ref{vanish_H2_please} implies that 
if $x\not\in  \overline{D}_L:= G^0\cdot \Xbar^{\gen}_{L, \rhobarss} \cup \Zbar_{H_0 G, \rhobarss}$
then part (2) holds, and \eqref{codim_L} and \eqref{codim_Hzero} imply that part (1) holds.

If $G_1^0$ is not a torus then we have to further cut out some pieces to make sure that 
$h^0(\Gamma_F, (\Lie G'_1)_x^*(1))=0$, so that part (3) of Proposition \ref{vanish_H2_please} is satisfied. Let $p_1: G\rightarrow G_1$ be the composition of the quotient map with the projection 
onto the first component. Let $\rhobar_1=p_1\circ\rho: \Gamma_F \rightarrow G_1(k)$. 
As in the proof of Proposition \ref{bound_special}
let $J_1, \ldots, J_r$ 
be reductive subgroups of $G_1$, such that $\dim G_1 - \dim J_i\ge 2$ and the 
image of the special $G_1$-absolutely irreducible locus in $\Xbar^{\gps}_{G_1,\rhobar_1^{\mathrm{ss}}}$
is contained in the union of $\Xbar^{\gps}_{J_i G_1, \rhobar_1^{\mathrm{ss}}}$. 
Let $H_i$ be the preimage of $J_i \times T_2$ in $L_k$ for $1\le i \le r$. Then 
$\dim L_k -\dim H_i = \dim G_1 - \dim J_i \ge 2$. It follows from Corollary \ref{bound_1} that 
\begin{equation}\label{codim_Hi}
\dim \Xbar^{\gen}_{G, \rhobarss} - \dim \Zbar_{H_i G, \rhobarss} \ge (\dim L_k - \dim H_i)[F:\Qp]\ge 
2[F:\Qp].
\end{equation}
If $x\in \Vbar_{LG, \rhobarss}$ then the action of $\Gamma_F$ on $(\Lie (G'_1)_{\sic})^*_x$ factors through
$p_1\circ \rho_x: \Gamma_F \rightarrow G_1(\kappa(x))$. If $p_1\circ \rho_{x}$ is not $G_1$-irreducible, then its image would be contained in a proper
R-parabolic of $G_1$, and then the image of $\rho_x$ would be contained in a proper
R-parabolic of $L_k$, which would contradict $x\in \Vbar_{LG}$. Thus $p_1\circ \rho_{x}$ is $G_1$-irreducible
and $h^0(\Gamma_F, (\Lie (G'_1)_{\sic})_x^*(1))=0$ if and only if $p_1\circ \rho_{x}$ corresponds to a non-special 
point in $\Xbar^{\gen}_{G_1, \rhobarss_1}$. 
The map $p_1: L_k\rightarrow G_1$ induces a map 
$\Vbar_{LG, \rhobarss} \rightarrow \Ubar_{LG, \rhobarss} \rightarrow \Xbar^{\gps}_{G_1, \rhobar_1^{\mathrm{ss}}}$ and the preimage of $\Xbar^{\gps}_{J_i G_1, \rhobarss_1}$ under this map 
is equal to $\Vbar_{LG, \rhobarss}\cap \Zbar_{H_i G, \rhobarss}$. We conclude that if 
$x\not \in \Zbar_{H_i G, \rhobarss}$ for $1\le i\le r$ then $h^0(\Gamma_F, (\Lie (G'_1)_{\sic})_x^*(1))=0$. 
Thus part (3) of Proposition \ref{vanish_H2_please} holds for $\rho_x$. 
Hence, if $x\not\in \overline{D}_L:=G^0\cdot \Xbar^{\gen}_{L, \rhobarss} \cup \bigcup\nolimits_{i=0}^r\Zbar_{H_i G, \rhobarss}$ then Proposition \ref{vanish_H2_please} implies that part (2) holds. 
Moreover, \eqref{codim_L}, \eqref{codim_Hzero} and \eqref{codim_Hi} imply 
that part (1) holds.  
\end{proof}

\begin{thm}\label{codim_special_locus}  The following hold:
\begin{enumerate}
\item  $X^{\spcl}_{G,\rhobarss}$ has codimension at least $1+[F:\Qp]$ in $X^{\gen}_{G, \rhobarss}$;
\item  $\Xbar^{\spcl}_{G,\rhobarss}$ has codimension at least $1+[F:\Qp]$ in $\Xbar^{\gen}_{G, \rhobarss}$;
\item  $X^{\spcl}_{G,\rhobarss}[1/p]$ has codimension at least $1+[F:\Qp]$ in 
$X^{\gen}_{G, \rhobarss}[1/p]$.
\end{enumerate}
Moreover, if there is no $\rhobarss$-compatible R-Levi subgroup $L$ of $G$ such that $\dim G -\dim L=2$ 
then the bounds can be improved to $2[F:\Qp]$.
\end{thm}
\begin{proof} Since $\dim X^{\gen}_{G, \rhobarss} = \dim \Xbar^{\gen}_{G, \rhobarss} +1$ by 
Corollary \ref{dim_XgenG}, 
Corollary \ref{gen_sp_dim} applied with $W= (\Lie G'_{\sic})^*$ implies that 
it is enough to prove the assertion about the special fibre.

We claim that $\Xbar^{\spcl}_{G,\rhobarss}$ is contained in the union of $Y_{\rhobarss}$, $\Zbar^{\spcl}_{GG, \rhobarss}$, 
$\Zbar_{LG, \rhobarss}$, for $\rhobarss$-compatible R-Levi subgroups $L$ of $G$ containing a fixed maximal split torus $T$ and satisfying $\dim G -\dim L > 2$, and $\overline{D}_L$ 
 for $\rhobarss$-compatible R-Levi subgroups 
$L$ of $G$ containing $T$ and satisfying $\dim G - \dim L =2$. Since this is a finite union and 
each piece has codimension of at least $1+[F:\Qp]$ by Corollary \ref{bound_Y}, Propositions \ref{bound_special}, 
\ref{bound_Levi}, \ref{alt_yeah} we obtain the assertion. Moreover, if there are no $\rhobarss$-compatible 
R-Levi subgroups $L$ with $\dim G-\dim L=2$ then the same references imply that each piece
has codimension at least $2[F:\Qp]$.

We will now prove the claim. Let $x\in \Xbar^{\spcl}_{G,\rhobarss}$ be a geometric point. 
If $x$ is not contained in $Y_{\rhobarss}$ then it is contained in $\Vbar_{LG, \rhobarss}$ for some 
$\rhobarss$-compatible R-Levi $L$ of $G$ containing $T$. If $L=G$ then $x\in \Zbar^{\spcl}_{GG, \rhobarss}$ by definition of the absolutely irreducible special locus. 
If $L\neq G$ then $\dim G - \dim L\ge 2$ by part (3) of Lemma \ref{V_LG_maxdim}.  
If $\dim G-\dim L >2$ then $x\in \Zbar_{LG, \rhobarss}$, as this locus contains $\Vbar_{LG, \rhobarss}$. 
If $\dim G -\dim L=2$ then Proposition \ref{alt_yeah} implies that $x\in \overline{D}_L$.
\end{proof} 

\begin{examp} Let $G=\GL_1\times \GL_2$ and $\rhobar= \rhobar_1\times \rhobar_2$. If
$\rhobar_2: \Gamma_F \rightarrow \GL_2(k)$ is absolutely irreducible then the only 
$\rhobarss$-compatible Levi of $G$ is $G$ itself and so \Cref{codim_special_locus} implies that the 
codimension of $(\Lie G'_{\sic})^*$-special locus is at least $2[F:\Qp]$. On the other hand 
if $\rhobar_2$ is reducible then the subgroup of diagonal matrices in $G$ is also 
$\rhobarss$-compatible and the bound given by the theorem is $1+[F:\Qp]$.
\end{examp}

\section{Irreducible components and normality}\label{nightmare_cont}

We assume that
$G^0$ is split over $\OO$ and the map 
$\pi_G\circ \rhobarss: \Gamma_F \rightarrow (G/G^0)(\kbar)$ is surjective. It follows from Remark \ref{rem_extend_scalars} and
Lemma \ref{shrink} that we can always achieve this after shrinking the 
group of connected components of $G$ and after replacing $L$ by a 
finite unramified extension. Let $G'$ be the derived subgroup scheme 
of $G^0$ and let $\varphi: G\rightarrow G/G'$ be the quotient map.

Let $\Gamma_E$ be the kernel of $\pi_G \circ\rhobarss$. 
Our assumptions on $G$ imply that $G^0/G'$ is a torus over $\OO$. 
We thus may write $G^0/G' =\Spec \OO[M]$, where $M$ is the character lattice 
of $G^0/G'$. Let $\Delta:=(G/G^0)(\kbar)$. Our assumptions on 
$\pi_G \circ \rhobarss$ imply that 
$$\Delta= (G/G^0)(k)= \Gal(E/F).$$ 
The action of $G$ by conjugation on $\Hom_{\OO\hyphen\mathrm{GrpSch}}(G^0/G', \Gm)=M$ 
induces an action of $\Delta$ on $M$. We assume that we are given a $\Delta$-invariant decomposition
\begin{equation}\label{decomp_M}
M=M_1 \oplus M_2.
\end{equation}
The projections $M\rightarrow M_i$ identify $K_i=\Spec \OO[M_i]$ with a normal subgroup scheme of $G/G'$. We let 
$$ H_1 = (G/G')/K_2, \quad H_2:= (G/G')/K_1.$$

We fix $\rhobar: \Gamma_F \rightarrow G(k)$ corresponding to a point in $X^{\gen}_{G, \rhobarss}(k)$. 
Let $\varphi_i: G\rightarrow H_i$ denote the quotient map and let $\psibar_i= \varphi_i \circ \rhobar$ and $\psibar= \varphi\circ \rhobar$.
We fix 
a continuous representation $\psi_1: \Gamma_F\rightarrow H_1(\OO)$ lifting $\psibar_1: \Gamma_F\rightarrow H_1(k)$. The lift $\psi_1$ induces a morphism 
$\Spec \OO \rightarrow X^{\gen}_{H_1, \psibarss_1}$.
Whenever $X=\Spec A$ is a scheme over $X^{\gen}_{H_1, \psibarss_1}$, we denote 
by $X^{\psi_1}$ the fibre product along this morphism, and $A^{\psi_1}$ the ring of global sections of
$X^{\psi_1}$. 
By functoriality proved in Proposition \ref{functoriality_2} the quotient morphism $\varphi_1: G\rightarrow H_1$ induces a morphism
\begin{equation}\label{fantastic_man}
X^{\gen}_{G, \rhobarss} \rightarrow X^{\gen}_{H_1, \psibarss_1}.
\end{equation}
We obtain closed subschemes $X^{\gen, \psi_1}_{G, \rhobarss}$ of $X^{\gen}_{G, \rhobarss}$
and $\Xbar^{\gen, \psi_1}_{G, \rhobarss}$ of $\Xbar^{\gen}_{G, \rhobarss}$, respectively. 

In this section we study the irreducible components of
$X^{\gen, \psi_1}_{G, \rhobarss}$ and the corresponding deformation rings and show that they are normal in many cases. 
The main results are \Cref{chi_nor_int} and its Corollaries
\ref{reg_codim_x}, \ref{bij_comp_Xgen} and \ref{conj_BJ}, where we prove an analogue of a conjecture of B\"ockle--Juschka posed originally for $G=\GL_d$.

Before delving  into the details let us point out two interesting cases: if $M_1=0$
then $X^{\gen, \psi_1}_{G, \rhobarss}=X^{\gen}_{G, \rhobarss}$, if $M_1=M$ then $X^{\gen, \psi_1}_{G, \rhobarss}$ parameterises representations with 
`a fixed determinant'. If $G=\GL_d$ then $G/G'=\Gm$, $M=\ZZ$ and these are the only two cases that can occur. A further interesting case occurs, when $G$ is a $C$-group; we discuss it in Theorem 
\ref{C}.

Let $\kappa$ be either a finite or a local field, which is an $\OO$-algebra, and let $\Lambda$ be a coefficient ring for $\kappa$ in the sense of section \ref{def_probs}. Let $\rho: \Gamma_F \rightarrow G(\kappa)$ be a
continuous representation. Let $D^{\square, \psi_1}_{G, \rho}: \mathfrak A_{\Lambda} \rightarrow \Set$ be the functor $$D^{\square, \psi_1}_{G, \rho}(A):= \{\rho_A\in D^{\square}_{G, \rho}(A): \varphi_1\circ \rho_A = \psi_1\otimes_{\OO} A\}.$$ 
It is easy to see that $D^{\square, \psi_1}_{G, \rho}$ is pro-represented by a quotient of $R^{\square}_{G, \rho}$, which we denote by $R^{\square, \psi_1}_{G, \rho}$.

Let $N$ be the character lattice of $Z(G^0)^0$. The map $Z(G^0)\rightarrow G/G'$ is a central isogeny by \cite[Corollary 5.3.3]{bcnrd}.
We thus obtain a homomorphism $M\rightarrow N$ of $\Delta$-modules, which induces an isomorphism $M \otimes \QQ\cong N\otimes \QQ$.
Let $N_1$ be the image of $N$ in $(N\otimes \QQ)/ (M_2\otimes\QQ)$ and let $N_2$ be the image of $N$ in $(N\otimes \QQ)/ (M_1\otimes\QQ)$.
Then $Z_i:=\Spec \OO[N_i]$ is a $\Delta$-invariant subtorus of $Z(G^0)$, and hence
a normal subgroup of $G$. Moreover, the map $Z_i\rightarrow H_i$ is an isogeny as  $M_i\otimes \QQ\cong N_i\otimes\QQ$ by construction.

\subsection{Dimension} In this section we bound the dimension of $\Xbar^{\gen, \psi_1}_{G, \rhobarss}$. 
We  obtain the bound from the results of the previous section applied to the group $\Gbar:= G/Z_1$ 
via the map $X^{\gen, \psi_1}_{G, \rhobarss} \rightarrow X^{\gen}_{G, \rhobarss} \rightarrow X^{\gen}_{\Gbar, q\circ\rhobarss}$, where the second arrow is induced 
by the quotient map $q: G\rightarrow \Gbar$ and functoriality proved in Proposition \ref{functoriality_2}. 
For example, if $G=\GL_d$ then we deduce the bound in the fixed determinant case from our results on $\PGL_d$. This differs 
from the argument in \cite{BIP_new}, where the authors twist the universal deformation by characters to unfix the determinant and get back to the $\GL_d$-case. 
To ease the notation we will drop $q$ and write 
$X^{\gen}_{\Gbar, \rhobarss}$ instead of  $X^{\gen}_{\Gbar, q\circ\rhobarss}$. 

\begin{prop}\label{psi_Gbar_finite} The morphism $X^{\gen, \psi_1}_{G, \rhobarss}\rightarrow X^{\gen}_{\Gbar, \rhobarss}$ is finite. 
\end{prop}
\begin{proof} Let $H:= H_1 \times \Gbar$. Then the natural map $q: G\rightarrow H$ 
is a central isogeny as $Z_1\rightarrow H_1$ is an isogeny by construction.
 Thus the induced 
map $f: X^{\gen}_{G, \rhobarss}\rightarrow X^{\gen}_{H, q\circ \rhobarss}$ is finite by Proposition 
\ref{finite_maps}. The quotient maps $ G\rightarrow H_1$, $ H\rightarrow H_1$ induce 
maps $X^{\gen}_{G, \rhobarss}\rightarrow X^{\gen}_{H_1, \psibarss_1}$ and 
$X^{\gen}_{H, \rhobarss}\rightarrow X^{\gen}_{H_1, \psibarss_1}$ which make
$f$ into a morphism of $X^{\gen}_{H_1, \psibarss_1}$-schemes. We thus obtain a finite map 
$ X^{\gen, \psi_1}_{G, \rhobarss}\rightarrow X^{\gen, \psi_1}_{H, q\circ \rhobarss}$. 
It follows from Proposition \ref{products} that the projection onto the second factor 
$H\rightarrow \Gbar$ induces an isomorphism  $X^{\gen, \psi_1}_{H, q\circ \rhobarss} \cong X^{\gen}_{\Gbar, \rhobarss}$. 
\end{proof} 

\begin{cor}\label{dim_bound_psi} Let $Z$ be a closed subscheme of $X^{\gen}_{\Gbar, \rhobarss}$ and 
let $Z'$ be its preimage in $X^{\gen, \psi}_{G,\rhobarss}$. Then $\dim Z'\le \dim Z$. In particular, 
$$\dim \Xbar^{\gen, \psi_1}_{G, \rhobarss} \le \dim \Gbar_k([F:\Qp]+1).$$
\end{cor}
\begin{proof} Proposition \ref{psi_Gbar_finite} implies that $Z'\rightarrow Z$ is finite, which implies the assertion about the dimension.
The last assertion follows from \Cref{dim_XgenG}.
\end{proof} 
Since $\dim (H_1)_k+\dim (H_2)_k= \dim (Z_1)_k +\dim (Z_2)_k= \dim Z(G^0)_k$ we obtain
\begin{equation}\label{dim_Gbar}
\dim \Gbar_k = \dim G_k - \dim (H_1)_k = \dim (H_2)_k +\dim G'_k.
\end{equation}
\begin{cor}\label{psi_gps_finite} The morphism $X^{\gen, \psi_1}_{G, \rhobarss}\sslash Z_2 G' \rightarrow X^{\gps}_{\Gbar, \rhobarss}$ is finite.
\end{cor}
\begin{proof} Proposition \ref{psi_Gbar_finite} together with \cite[Theorem 2 (ii)]{seshadri} 
imply that $$X^{\gen, \psi_1}_{G, \rhobarss}\sslash Z_2G'\rightarrow X^{\gen}_{\Gbar, \rhobarss}\sslash Z_2G'$$
is finite. Since the  action of $G'$ on $X^{\gen}_{\Gbar, \rhobarss}$ factors through the 
surjective morphism $Z_2G'\rightarrow \Gbar{}^0$ we obtain $X^{\gen}_{\Gbar, \rhobarss}\sslash Z_2G'=X^{\gps}_{\Gbar, \rhobarss}$.
\end{proof}

\begin{cor}\label{semi-local_Gdash} Let $R$ be the ring of global sections of $X^{\gen, \psi_1}_{G, \rhobarss}\sslash Z_2G'$, 
then $R$ is a finite product of local noetherian $\OO$-algebras with residue fields finite extensions 
of $k$. Moreover, if the orders of $(G/G^0)(\kbar)$ and $(Z_1\cap Z_2G')(\kbar)$ are coprime then 
$R$ is a complete local noetherian $\OO$-algebra with residue field $k$.
\end{cor} 
\begin{proof}
Since $R^{\gps}_{\Gbar, \rhobarss}$ is a complete noetherian $\OO$-algebra with residue field $k$, 
and $R^{\gps}_{\Gbar, \rhobarss} \rightarrow R$ is a finite map by Corollary \ref{psi_gps_finite}, 
$R$ is a finite product of complete local noetherian $\OO$-algebras with residue field 
a finite extension of $k$. Moreover, $R$ is a local noetherian $\OO$-algebra with residue field $k$ if and only if $(X^{\gen, \psi_1}_{G, \rhobarss}\sslash Z_2G')(\kbar)$ consists of one point. 

Let $x\in X^{\gen, \psi_1}_{G, \rhobarss}(\kbar)$ be such that the orbit $Z_2G'\cdot x$ is closed. 
We claim that the orbit $G^0\cdot x$ is also closed in $X^{\gen}_{G, \rhobarss}$. Let $P$ be a minimal 
R-parabolic of $G_{\kbar}$ containing the image of $\rho_x$, and let  $U$ be its unipotent radical. 
Then $P\cap G^0= P^0$ is a parabolic of $G^0$ by Lemma \ref{para} and thus $P\cap Z_2G'$ 
is a parabolic of $Z_2G'$ with unipotent radical $U$. We may find a cocharacter $\lambda: \Gm\rightarrow
Z_2G'\cap P$ such that $\lim_{t\rightarrow 0} \lambda(t) U \lambda(t)^{-1}= {1}$. 
Let $x': = \lim_{t\rightarrow 0} \lambda(t) \cdot x$. Then $\rho_{x'} \equiv \rho_{x} \pmod{U}$ 
and the minimality of $P$ implies that $\rho_{x'}$ is the $G$-semisimplification of $\rho_x$.
Thus $G^0\cdot x'$ is the unique closed $G^0$-orbit contained in the closure of $G^0\cdot x$
by Proposition \ref{closed_orbit}.
On the other hand,  $x'$ lies in the closure 
of $Z_2G'\cdot x$.  Since 
$Z_2G'\cdot x$ is closed by assumption we deduce that $x\in Z_2G' \cdot x'$ and hence
$G^0 \cdot x = G^0\cdot x'$ is closed. 

Let $y, y'\in (X^{\gen, \psi_1}_{G, \rhobarss}\sslash Z_2G')(\kbar)$ and let 
$Z_2G'\cdot x$, $Z_2G'\cdot x'$ be the corresponding closed orbits in $X^{\gen, \psi_1}_{G, \rhobarss}(\kbar)$.
As discussed above $G^0\cdot x$ and $G^0\cdot x'$ are closed in $X^{\gen}_{G, \rhobarss}(\kbar)$. 
Since $R^{\gps}_{G, \rhobarss}$ is a local ring with residue field $k$ there is a unique such closed $G^0$-orbit. Thus there is $g\in G^0(\kbar)$ such that 
\begin{equation}\label{clean_1}
 g \rho_x(\gamma) g^{-1} = \rho_{x'}(\gamma), \quad  \forall \gamma\in \Gamma_F.
\end{equation}
We may write 
$g = z_1 z_2 h$, where $z_1\in Z_1(\kbar)$, $z_2\in Z_2(\kbar)$ and $h\in G'(\kbar)$. After replacing $x$ by a translate with 
$z_2h$ we may assume that $g\in Z_1(\kbar)$. By applying $\varphi_1$ to \eqref{clean_1} we get 
\begin{equation}\label{clean_2}
 \varphi_1(g) \psi_1(\gamma) \varphi_1(g)^{-1} \psi_1(\gamma)^{-1} = 1, \quad \forall \gamma\in \Gamma_F
\end{equation} 
in $H_1(\kbar)$. Since by assumption the map $\pi_G \circ \rhobar: \Gamma_F \rightarrow \Delta$
is surjective, for every $\delta\in \Delta$ there is $\gamma \in \Gamma_F$ such that $\delta=\psi_1(\gamma)$.  
We deduce from \eqref{clean_2} that $\varphi_1(g)\in H_1^{\Delta}(\kbar)$. Since $g\in G^0(\kbar)$ we have $\varphi_1(g) \in (H_1^0)^{\Delta}(\kbar)$.

The isogeny $Z_1\times Z_2\rightarrow G^0/G'$ induces an isomorphism $\varphi_1: G^0/Z_2 G'\cong H_1^0$ and an exact sequence of diagonalizable 
groups  
$$0\rightarrow Z_1\cap Z_2G'\rightarrow Z_1\rightarrow H_1^0\rightarrow 0.$$
This induces an exact sequence on $\kbar$-points, and after taking $\Delta$-invariants
we obtain an exact sequence of abelian groups: 
$$0\rightarrow (Z_1\cap Z_2G')(\kbar)^{\Delta} \rightarrow Z_1^{\Delta}(\kbar)\rightarrow (H_1^0)^{\Delta}(\kbar)\rightarrow H^1(\Delta, (Z_1\cap Z_2G')(\kbar)).$$
If the orders of $\Delta$ and $(Z_1\cap Z_2G')(\kbar)$ are coprime then the $H^1$-term vanishes and 
there is $g'\in Z_1^{\Delta}(\kbar)$ which maps to $g\in (H_1^0)^{\Delta}(\kbar)$. 
Since $Z_1^{\Delta}$ is contained in $Z(G)$, $g'$ acts trivially on $x$, and hence $x$ and $x'$ lie
in the same $Z_2G'$-orbit and thus $y=y'$.
\end{proof} 

\begin{cor}\label{bip_recyc} Let $Z$ be an irreducible component of a closed $Z_2G'$-invariant subscheme of $X^{\gen, \psi_1}_{G, \rhobarss}$.
Then $Z$ is $Z_2G'$-invariant and intersects $Y^{\psi_1}_{\rhobarss}$
non-trivially. Moreover, if $x$ is a closed point of $Z$ then following hold: 
\begin{enumerate}
\item if $x$ is a closed point of  $Z\cap Y^{\psi_1}_{\rhobarss}$ then $\dim \OO_{Z,x}= \dim Z$;
\item if $x$ is a closed point of $Z\setminus (Z\cap Y^{\psi_1}_{\rhobarss})$ then $\dim \OO_{Z, x}= \dim Z -1$.
\end{enumerate}
\end{cor}
\begin{proof} It follows from Corollary \ref{semi-local_Gdash} that  $Y^{\psi_1}_{\rhobarss}$ is 
the preimage of closed points in $X^{\gen, \psi_1}_{G, \rhobarss}\sslash Z_2G'$. Given this 
the proof is the same as the proof of \cite[Lemma 3.21]{BIP_new}.
\end{proof}

\subsection{Generalised tori}\label{sec_gen_tori} We summarise some of the results of \cite{defT}. The group 
$$\mu:=(\mu_{p^{\infty}}(E)\otimes M_2)^{\Delta}$$
is a finite abelian group 
of order $p^m$. We assume that $L$ contains all the $p^m$-th roots of unity. 
We further assume that $L$ is large enough so that $X^{\gen}_{H_2, \psibar_2}(\OO)$
is non-empty. This is possible by \cite[Theorem 9.3]{defT}. 
Under these assumptions  we prove in \cite[Theorem 9.3]{defT} that there is an isomorphism of complete local noetherian $\OO$-algebras:
\begin{equation}\label{w_think_1}
\alpha: \OO\br{(\Gamma_E^{\ab,p}\otimes M_2)^{\Delta}}
\overset{\cong}{\longrightarrow}
R^{\ps}_{H_2, \psibar_2}
\end{equation}
where $\Gamma_E^{\ab,p}$ is the maximal abelian pro-$p$ quotient of $\Gamma_E$, and an isomorphism of $\OO$-algebras
\begin{equation}\label{w_think_2}
A^{\gen}_{H_2, \psibarss_2} \cong R^{\ps}_{H_2, \psibarss_2}[t_1^{\pm{1}}, \ldots, t_s^{\pm{1}}],
\end{equation}
where $s= \rank_{\ZZ} M_2 - \rank_{\ZZ} (M_2)_{\Delta}$. The Artin map of local 
class field theory $\Art_E: E^{\times} \rightarrow \Gamma_E^{\ab}$ identifies
$\mu$ with the torsion subgroup of $(\Gamma_E^{\ab,p}\otimes M_2)^{\Delta}$. We 
show in \cite[Corollary 8.8]{defT}
that
\begin{equation}\label{w_think_3}
\OO\br{(\Gamma_E^{\ab,p}\otimes M_2)^{\Delta}}\cong \OO[\mu]\br{x_1, \ldots, x_r},
\end{equation}
where $r= \rank_{\ZZ} M_2 \cdot [F:\Qp]+ \rank_{\ZZ} (M_2)_{\Delta}$ is the rank 
of $(\Gamma_E^{\ab,p}\otimes M_2)^{\Delta}$ as a $\Zp$-module. 
It follows from \eqref{w_think_1}, \eqref{w_think_2} and \eqref{w_think_3} that 
\begin{equation}\label{w_think_4}
A^{\gen}_{H_2, \psibarss_2}\cong \OO[\mu]\br{x_1,\ldots, x_r}[t_1^{\pm 1}, \ldots, t_s^{\pm 1}],
\end{equation}
where $r+s = \rank_{\ZZ} M_2 ([F:\Qp] +1)=\dim (H_2)_k ([F:\Qp]+1)$.

Let $\Xc(\mu)$ be the group of characters $\chi:\mu\rightarrow \OO^{\times}$. 
We may interpret $\chi\in \Xc(\mu)$ as $\OO$-algebra homomorphisms 
$\chi: \OO[\mu]\rightarrow \OO$. If $X=\Spec A$ is a scheme over 
$X^{\ps}_{H, \psibarss_2}$ then \eqref{w_think_1} induces a morphism 
$\OO[\mu]\rightarrow A$ and we define
\begin{equation}\label{notation_chi}
A^{\chi}:= A\otimes_{ \OO[\mu], \chi} \OO, \quad X^{\chi}:= \Spec A^{\chi}.
\end{equation}

\begin{lem}\label{dim_XgenG_chi} Let $Z$ be an affine scheme over $\Xbar^{\ps}_{H, \psibarss_2}$  
Then the underlying reduced subschemes of $Z$  and $Z^{\chi}$
coincide. In particular, $\dim Z = \dim Z^{\chi}$.
\end{lem}
\begin{proof} Since $\mu$ is a finite $p$-group $k[\mu]$ is a local $k$-algebra with a nilpotent maximal 
ideal. This implies the first assertion. Since dimension is a topological invariant 
the second assertion follows. 
\end{proof}

\begin{lem}\label{form_sm_chi} Let $\kappa$ be either a finite or a local $\OO$-field and  
let $\psi_2: \Gamma_F\rightarrow H_2(\kappa)$ be a continuous representation corresponding to $x\in X^{\gen, \chi}_{H_2, \psibarss_2}(\kappa)$.  Then 
$$R^{\square,\chi}_{H_2, \psi_2}\cong \Lambda\br{z_1,\ldots, z_t},$$
where $t=\dim (H_2)_k([F:\Qp]+1)$ and $\Lambda$ is a coefficient ring for $\kappa$.
\end{lem}
\begin{proof} 
It follows from the proof of \Cref{333} that $R^{\square,\chi}_{H_2, \psi_2}$
is the completion of $ \Lambda\otimes_{\OO} A^{\gen, \chi}_{H_2,\psibarss_2}$ with respect to 
 the kernel of the natural map 
$ \Lambda\otimes_{\OO} A^{\gen, \chi}_{H_2,\psibarss_2}\twoheadrightarrow \kappa.$
It follows from \eqref{w_think_4} that 
\begin{equation}\label{Agen_chi}
A^{\gen, \chi}_{H_2,\psibarss_2}\cong \OO\br{x_1,\ldots, x_r}[t_1^{\pm 1}, \ldots, 
t_s^{\pm 1}],
\end{equation}
with $r+s =\dim (H_2)_k ([F:\Qp]+1)$. The assertion 
follows from Lemma \ref{local_ring_def_ring}.
\end{proof}

\subsection{Serre's criterion for normality} In this subsection we show that the deformation rings 
$R^{\square, \psi_1, \chi}_{G, \rho_x}$ (defined below) and their special fibres are integral domains
by verifying Serre's criterion for normality for $X^{\gen, \psi_1, \chi}_{G, \rhobarss}$ and its special fibre, when $\pi_1(G')$ is \'etale. We also 
show that $X^{\gen, \psi_1, \chi}_{G, \rhobarss}[1/p]$ is normal without making any assumptions
on $\pi_1(G')$.

\begin{lem}\label{DUS} The composition with $\varphi_2: G/G'\rightarrow H_2$ induces an isomorphism 
\begin{equation}\label{eq_iso_psi1}
X^{\gen, \psi_1}_{G/G', \psibarss} \overset{\cong}{\longrightarrow} X^{\gen}_{H_2,\psibarss_2}.
\end{equation}
\end{lem}
\begin{proof} The natural map $G/G'\rightarrow  H:=H_1\times H_2$ induces an isomorphism between neutral components. 
Lemma \ref{shrink} implies that $X^{\gen}_{G/G', \psibarss} \cong X^{\gen}_{H, \psibarss}$. 
The assertion then follows from Proposition \ref{products} and Lemma \ref{prod_ps}.
\end{proof} 
If $X=\Spec A$ is a scheme over $X^{\gen}_{G/G', \psibarss}$ then $X^{\psi_1}$ 
is a scheme over $X^{\gen, \psi_1}_{G/G', \psibarss}$.  The isomorphism 
\eqref{eq_iso_psi1} allows us to consider $X^{\psi_1,\chi}$ for every 
character $\chi\in \mathrm X(\mu)$ with the notation introduced in \eqref{notation_chi}. We will denote its ring of global sections by $A^{\psi_1, \chi}$. 

\begin{lem}\label{non-empty} $X^{\gen, \psi_1, \chi}_{G, \rhobarss}$ and $\Xbar^{\gen, \psi_1, \chi}_{G, \rhobarss}$ are non-empty 
for every $\chi\in \mathrm X(\mu)$.
\end{lem}
\begin{proof} The underlying reduced schemes of $\Xbar^{\gen, \psi_1, \chi}_{G, \rhobarss}$ and $\Xbar^{\gen, \psi_1}_{G, \rhobarss}$
coincide by Lemma \ref{dim_XgenG_chi}. Moreover,  $\Xbar^{\gen, \psi_1}_{G, \rhobarss}(k)$ is non-empty as it contains 
a point corresponding to the representation $\rhobar$ fixed at the beginning of the section. 
\end{proof}

\begin{lem}\label{form_sm} Let $\kappa$ be  a finite or a local $\OO$-field and  
let $\psi: \Gamma_F\rightarrow (G/G')(\kappa)$ be a representation corresponding to $x\in X^{\gen, \psi_1, \chi}_{G/G', \psibarss}(\kappa)$
and let $\psi_2=\varphi_2 \circ \psi$. Then 
$$R^{\square, \psi_1, \chi}_{G/G', \psi} \cong R^{\square,\chi}_{H_2, \psi_2}.$$ 
In particular, $R^{\square, \psi_1, \chi}_{G/G', \psibar}\cong \Lambda\br{z_1,\ldots, z_t}$, where $t=\dim (H_2)_k([F:\Qp]+1)$.
\end{lem}
\begin{proof} The isomorphism follows from Lemma \ref{DUS} and
\Cref{333}.  Formal smoothness follows from Lemma \ref{form_sm_chi}.
\end{proof}

\begin{lem}\label{psi_def_rings} Let $X= X^{\gen, \psi_1,\chi}_{G, \rhobarss}$ and let $\Xbar= \Xbar^{\gen, \psi_1,\chi}_{G, \rhobarss}$ and let
$x$ be a closed point of $Y_{\rhobarss}^{\psi_1,\chi}$ or a closed point $X\setminus Y^{\psi_1,\chi}_{\rhobarss}$. Then we have the following isomorphisms of local rings: 
\begin{enumerate} 
\item if $x\in Y^{\psi_1,\chi}_{\rhobarss}$ then   $R^{\square, \psi_1,\chi}_{G, \rho_x}\cong\hat{\OO}_{X, x}$, $R^{\square, \psi_1,\chi}_{G, \rho_x}/\varpi \cong \hat{\OO}_{\Xbar, x}$;
\item if $x\in \Xbar\setminus Y^{\psi_1,\chi}_{\rhobarss}$ then 
$R^{\square,\psi_1,\chi}_{G, \rho_x}\cong\hat{\OO}_{X, x}\br{T}$, $R^{\square, \psi_1,\chi}_{G, \rho_x}/\varpi \cong \hat{\OO}_{\Xbar, x}\br{T}$;
\item if $x\in X[1/p]$ then $R^{\square, \psi_1,\chi}_{G, \rho_x}\cong \hat{\OO}_{X, x}$.
\end{enumerate}
\end{lem}
\begin{proof} Since $X^{\gen, \psi_1,\chi}_{G, \rhobarss}$ is a closed subscheme of $X^{\gen, \tau}_G$ the assertion follows from  Proposition \ref{333} and Lemma \ref{local_ring_def_ring}.
\end{proof}

\begin{cor}\label{complete_intersection_chi}
Let $x$ be either a closed point of $X^{\gen,\psi_1, \chi}_{G, \rhobarss}\setminus Y_{\rhobarss}^{\psi_1, \chi}$ or  a closed point of $Y^{\psi_1, \chi}_{\rhobarss}$. Then the following hold:
\begin{enumerate}
\item $R^{\square,\psi_1, \chi}_{G, \rho_x}$ is a flat $\Lambda$-algebra of relative 
dimension $\dim \Gbar_k([F:\Qp]+1)$ and is complete intersection;
\item if $\chara(\kappa(x))=p$ then $R^{\square,\psi_1, \chi}_{G, \rho_x}/\varpi$ is complete 
intersection of dimension $\dim \Gbar_k([F:\Qp]+1)$.
\end{enumerate}
\end{cor}
\begin{proof} Proposition \ref{present_over_RH} and Lemma \ref{form_sm} give us a presentation 
\begin{equation}\label{represent_chi}
R^{\square,\psi_1, \chi}_{G, \rho_x}\cong \frac{R^{\square,\psi_1, \chi}_{G, \varphi\circ \rho_x}\br{x_1, \ldots, x_r}}{(f_1, \ldots, f_s)}\cong \frac{\Lambda\br{z_1, \ldots, z_t, x_1, \ldots, x_r}}{(f_1, \ldots, f_s)}
\end{equation}
    with  $r-s= \dim G'_k ([F:\Qp]+1)$ and $t=\dim (H_2)_k ([F:\Qp]+1)$. Lemma \ref{pre_lci_flat} implies that it is enough
to show that
\begin{equation}\label{eq_ETP}
\dim (\kappa(x)\otimes_{\Lambda}  R^{\square,\psi_1, \chi}_{G, \rho_x}) \le r-s +t \overset{\eqref{dim_Gbar}}{=} \dim \Gbar_k ([F:\Qp]+1).
\end{equation}
Lemma \ref{dim_XgenG_chi} and Corollary \ref{dim_bound_psi} give us 
\begin{equation}\label{bound_dim_psi_chi}
\dim \Xbar^{\gen, \psi_1, \chi}_{G, \rhobarss}\le \dim \Gbar_k ([F:\Qp]+1).
\end{equation}
To ease the notation we let $X:= X^{\gen,\psi_1,\chi}_{G, \rhobarss}$,  let $\Xbar:= \Xbar^{\gen,\psi_1,\chi}_{G, \rhobarss}$ and let 
$Y:=Y^{\psi_1, \chi}_{\rhobarss}$. 

If $x\in Y$ then Corollary \ref{bip_recyc} applied to $Z=\overline{X}$ gives us 
$\dim \OO_{\Xbar, x}\le \dim \Xbar$. Lemma \ref{psi_def_rings}
implies that $R^{\square,\psi_1,\chi}_{\rho_x}/\varpi\cong \hat{\OO}_{\Xbar,x}$ and this implies \eqref{eq_ETP}. 

If $x\in \Xbar\setminus Y$ then $R^{\square,\psi, \chi}_{\rho_x}/\varpi\cong \hat{\OO}_{\Xbar,x}\br{T}$
by Lemma \ref{psi_def_rings}. Corollary \ref{bip_recyc} gives
$\dim \OO_{\Xbar, x} +1 \le \dim \Xbar$ and this implies \eqref{eq_ETP}. 

If $x\in X[1/p]$ then $R^{\square,\psi_1, \chi}_{G, \rho_x}\cong \hat{\OO}_{X, x}$
by Lemma \ref{psi_def_rings} and $\dim \OO_{X, x} \le \dim X -1$ 
by Corollary \ref{bip_recyc} applied to $Z=X$. Since by Corollary \ref{bip_recyc}
every irreducible component of $X$ intersects the special fibre we have  
$\dim X -1 \le \dim \Xbar$ and \eqref{eq_ETP} follows.  
\end{proof} 

\begin{remar}\label{1313} The proof of \Cref{complete_intersection_chi} also applies verbatim to 
the rings $R^{\square,\psi_1}_{G, \rho_x}$ after  replacing $X^{\gen, \psi_1, \chi}_{G, \rhobarss}$ with $X^{\gen, \psi_1}_{G, \rhobarss}$ everywhere. 
\end{remar}

\begin{cor}\label{XgenG_chi_lci} Let $d= \dim \Gbar_k ([F:\Qp]+1)$. Then the following hold: 
\begin{enumerate} 
\item $X^{\gen,\psi_1, \chi}_{G, \rhobarss}$ is $\OO$-flat  of relative dimension $d$ and is locally complete intersection;
\item $\Xbar^{\gen,\psi_1, \chi}_{G, \rhobarss}$  is locally complete 
intersection of dimension $d$.
\end{enumerate} 
\end{cor}
\begin{proof} The proof is the same as the proof of Corollary \ref{dim_XgenG} using Corollary \ref{complete_intersection_chi}
and Lemma \ref{psi_def_rings}.
\end{proof}

\begin{cor}\label{special_codim_psi} Let $c$ be one of the following:
\begin{enumerate}
\item the codimension of the $(\Lie G'_{\sic})^*$-special locus  in $X^{\gen,\psi_1,\chi}_{G, \rhobarss}$;
\item the codimension of the $(\Lie G'_{\sic})^*_k$-special locus  in $\Xbar^{\gen,\psi_1,\chi}_{G, \rhobarss}$;
\item the codimension of the $(\Lie G'_{\sic})^*_L$-special locus  in $X^{\gen,\psi_1,\chi}_{G, \rhobarss}[1/p]$.
\end{enumerate}
Then $c\ge 1+[F:\Qp]$. 
Moreover, if there is no $\rhobarss$-compatible R-Levi subgroup $L$ of $G$ such that $\dim G -\dim L=2$ 
then $c\ge 2[F:\Qp]$.
\end{cor}
\begin{proof} Let $Z^{\spcl}_G$ be the $(\Lie G'_{\sic})^*$-special locus 
in $X^{\gen, \psi_1,\chi}_{G, \rhobarss}$ and let $Z^{\spcl}_{\Gbar}$ be 
the $(\Lie \Gbar'_{\sic})^*$-special 
locus in $X^{\gen}_{\Gbar, \rhobarss}$. 
If $x'\in X^{\gen, \psi_1,\chi}_{G, \rhobarss}$ maps to $x\in X^{\gen}_{\Gbar, \rhobarss}$ then 
the representations $(\Lie G'_{\sic})^*_{x'}$ and $(\Lie \Gbar'_{\sic})^*_{x}\otimes_{\kappa(x)} \kappa(x')$ 
are isomorphic as representations of $G$ as $G'\rightarrow \Gbar'$ is a central isogeny of semisimple groups and the adjoint action of $G$ on both factors through the action of $\Gbar$.
Thus $h^0(\Gamma_F, (\Lie G'_{\sic})^*_{x'}(1))= h^0(\Gamma_F, (\Lie \Gbar'_{\sic})^*_{x}(1))$, and so 
$x$ is special if and only if $x'$ is special. Thus $Z^{\spcl}_G$   is equal to 
the preimage  of $Z^{\spcl}_{\Gbar}$. 

Since $X^{\gen, \psi_1,\chi}_{G, \rhobarss} \rightarrow X^{\gen}_{\Gbar, \rhobarss}$ 
is finite by \Cref{psi_Gbar_finite}, the map $Z^{\spcl}_G\rightarrow Z^{\spcl}_{\Gbar}$ is finite, 
and hence $\dim Z^{\spcl}_G \le \dim Z^{\spcl}_{\Gbar}$. Since $\dim X^{\gen, \psi_1,\chi}_{G, \rhobarss}=\dim X^{\gen}_{\Gbar, \rhobarss}$
by  \Cref{XgenG_chi_lci} and \Cref{dim_XgenG} the 
assertion follows from Theorem \ref{codim_special_locus} applied to $\Gbar$. The argument for the special and generic fibres is the same. 
\end{proof}

\begin{lem}\label{non-spec_reg} The $(\Lie G'_{\sic})^*_L$-non-special locus in 
$X^{\gen,\psi_1,\chi}_{G, \rhobarss}[1/p]$ is regular. Moreover, if $\pi_1(G')$ is \'etale then 
the following hold:
\begin{enumerate}
\item $(\Lie G'_{\sic})^*$-non-special locus in $X^{\gen,\psi_1, \chi}_{G, \rhobarss}$ is regular;
\item $(\Lie G'_{\sic})^*_k$-non-special locus in $\Xbar^{\gen,\psi_1, \chi}_{G, \rhobarss}$ is regular.
\end{enumerate}
\end{lem} 

\begin{proof}  Let $Z$ be the $(\Lie G'_{\sic})^*$-special locus in $X^{\gen,\psi_1, \chi}_{G, \rhobarss}$. Let $x$ be a closed point 
in the complement of $Z$.  Then
$\kappa(x)$ is either a local or a finite field and local Tate duality induces 
an isomorphism 
$ H^2(\Gamma_F, \ad^0 \rho_x)\cong H^0(\Gamma_F, (\ad^0 \rho_x)^*(1)).$

If  $\pi_1(G')$ is \'etale or $x\in X^{\gen}_{G, \rhobarss}[1/p]$ then $(\Lie G'_{\sic})_x=(\Lie G')_x$. Since $x$ is not contained in the $(\Lie G'_{\sic})^*$-special locus $H^2(\Gamma_F, \ad^0 \rho_x)=0$ and thus $s=0$ in the presentation \eqref{represent_chi}, which implies 
that $R^{\square,\psi_1, \chi}_{G, \rho_x}$ is formally smooth over $\Lambda$. It follows from Lemma \ref{psi_def_rings} that the local 
ring of $X^{\gen,\psi_1, \chi}_{G, \rhobarss}$ at $x$ is regular. If $x$ is in the special fibre then we deduce
that $R^{\square,\psi_1, \chi}_{G, \rho_x}/\varpi$ is formally smooth over $\kappa(x)$ and the same argument applies. 
\end{proof} 
If  there is no $\rhobarss$-compatible R-Levi subgroup $L$ of $G$ such that $\dim G -\dim L=2$ then we let $c(\rhobarss)=2[F:\Qp]-1$ and let $c(\rhobarss)=[F:\Qp]$, otherwise. 
\begin{lem}\label{R1_chi} $X^{\gen,\psi_1,\chi}_{G,\rhobarss}[1/p]$ is regular in codimension $c(\rhobarss)$. Moreover,  if $\pi_1(G')$ is \'etale then $X^{\gen,\psi_1, \chi}_{G, \rhobarss}$ and $\Xbar^{\gen, \psi_1, \chi}_{G, \rhobarss}$ are  also regular in codimension $c(\rhobarss)$.
\end{lem}
\begin{proof} The assertion follows from \Cref{special_codim_psi} and \Cref{non-spec_reg}.
\end{proof}

\begin{thm}\label{chi_nor_int} If $\pi_1(G')$ is \'etale then $X^{\gen, \psi_1, \chi}_{G, \rhobarss}$ 
and $\Xbar^{\gen, \psi_1, \chi}_{G, \rhobarss}$ are normal.  
Moreover, if the orders of $(G/G^0)(\kbar)$ and $(Z_1\cap Z_2G')(\kbar)$ are coprime then they are integral.
\end{thm}
\begin{proof} Corollary  \ref{XgenG_chi_lci} and Lemma \ref{R1_chi} imply that Serre's criterion for normality holds and
hence $X^{\gen, \psi_1, \chi}_{G, \rhobarss}$ and $\Xbar^{\gen, \psi_1, \chi}_{G, \rhobarss}$ are normal. 

It follows from \cite[Lemma 2.1]{BIP_new} 
that their connected components are $Z_2G'$-invariant. If there was more than one then \cite[Theorem 3 (iii)]{seshadri} implies that their images would disconnect
$X^{\gen, \psi_1, \chi}_{G, \rhobarss}\sslash Z_2 G'$
(resp.\,$X^{\gen, \psi_1, \chi}_{G, \rhobarss}\sslash Z_2 G'$). 
If the orders of $(G/G^0)(\kbar)$ and $(Z_1\cap Z_2G')(\kbar)$ are coprime
then this is not possible as the GIT quotients are spectra of local rings in both cases by \Cref{semi-local_Gdash}.
\end{proof}

\begin{thm}\label{gen_fib_normal} $X^{\gen,\psi_1,\chi}_{G,\rhobarss}[1/p]$ is normal.
\end{thm}
\begin{proof} Corollary  \ref{XgenG_chi_lci}, Lemma \ref{R1_chi} and Serre's criterion for normality.
\end{proof}

\begin{cor}\label{reg_codim_x} Assume that $\pi_1(G')$ is \'etale. Let $x$ be either a closed point of $\Xbar^{\gen,\psi_1,\chi}_{G, \rhobarss}\setminus Y_{\rhobarss}^{\psi_1,\chi}$ or  a closed point of $Y^{\psi_1,\chi}_{\rhobarss}$. Then $R^{\square, \psi_1,\chi}_{G, \rho_x}$ and $R^{\square, \psi_1,\chi}_{G, \rho_x}/\varpi$ are normal domains, which are regular in codimension
$c(\rhobarss)$.
\end{cor}
\begin{proof} Let $A$ be either $A^{\gen, \psi_1,\chi}_{G, \rhobarss}$ or $A^{\gen,\psi_1, \chi}_{G, \rhobarss}/\varpi$. Since
$A$ is of finite type over a complete local noetherian ring, it is excellent by \cite[\href{https://stacks.math.columbia.edu/tag/07QW}{Tag 07QW}]{stacks-project}.
Let $\pp$ be a prime of $A$. If $A$ is regular in codimension $i$ then $A_{\pp}$ is regular in codimension $i$, and since $A_{\pp}$ is again excellent, 
its completion $\hat{A}_{\pp}$ is also regular in codimension $i$ by \cite[Theorem 23.9 (ii)]{matsumura}. (Since $A_{\pp}$ is excellent, it is a G-ring, 
and this implies that the assumptions made in the reference are satisfied. This is used implicitly in the proof of \cite[Theorem 32.2 (i)]{matsumura}.)
By repeating the same argument with $A[T]$ we deduce that $\hat{A}_{\pp}\br{T}$ is regular in codimension $i$. Lemma \ref{R1_chi} and 
Lemma \ref{local_ring_def_ring} imply that $R^{\square, \psi_1,\chi}_{G, \rho_x}$ and 
$R^{\square, \psi_1,\chi}_{G, \rho_x}$ are regular in codimension $c(\rhobarss)$.
Since $c(\rhobarss)\ge 1$,  Corollary \ref{complete_intersection_chi} imply that Serre's criterion for normality holds and we deduce that $R^{\square, \psi_1,\chi}_{G, \rho_x}$ and $R^{\square, \psi_1,\chi}_{G, \rho_x}/\varpi$ are normal domains.
\end{proof}

\begin{cor}\label{reg_codim_x_gen} If $x$ is a closed point of $X^{\gen,\psi_1,\chi}_{G, \rhobarss}[1/p]$ then $R^{\square, \psi_1,\chi}_{G, \rho_x}$ is a normal domain, which 
 is regular in codimension $c(\rhobarss)$.
\end{cor}
\begin{proof} The same proof as the proof of Corollary \ref{reg_codim_x}.
\end{proof} 

\begin{cor}\label{parafactorial} 
Assume that  either $[F:\Qp]\ge 3$ or 
 $[F:\Qp]=2$ and there is no $\rhobarss$-compatible R-Levi subgroup $L$ of $G$ such that $\dim G -\dim L=2$.
 
If $x$ is a closed point of $X^{\gen,\psi_1,\chi}_{G, \rhobarss}[1/p]$ then  $R^{\square, \psi_1,\chi}_{G, \rho_x}$ is factorial. If $x$ is a closed point of
$\Xbar^{\gen,\psi_1,\chi}_{G, \rhobarss}\setminus Y_{\rhobarss}^{\psi_1,\chi}$ 
or  a closed point of $Y^{\psi_1,\chi}_{\rhobarss}$ and we additionally assume that $\pi_1(G')$ is \'etale then  $R^{\square, \psi_1,\chi}_{G, \rho_x}$ and 
$R^{\square, \psi_1,\chi}_{G, \rho_x}/\varpi$ are also factorial. 
\end{cor} 
\begin{proof} The assumptions together with Corollary \ref{reg_codim_x} imply that the rings 
are regular in codimension $3$. Since they are complete intersection by Corollary \ref{complete_intersection_chi}, 
a theorem of Grothendieck implies the assertion; see \cite{R3_factorial} for a quick proof. 
\end{proof}

\subsection{Irreducible components} We now focus on the case, when $M_1=0$ 
and $M_2=M$, so that $X^{\gen,\psi_1}_{G, \rhobarss}=X^{\gen}_{G, \rhobarss}$.

\begin{cor} $X^{\gen}_{G, \rhobarss}$ is reduced and $X^{\gen}_{G, \rhobarss}[1/p]$ is normal.
\end{cor} 
\begin{proof} The normality of $X^{\gen}_{G, \rhobarss}[1/p]$ follows from Theorem \ref{gen_fib_normal}.
This implies that $A^{\gen}_{G, \rhobarss}[1/p]$ is reduced. Since $A^{\gen}_{G, \rhobarss}$ is $\OO$-flat by Corollary \ref{dim_XgenG}, 
it is a subring of $A^{\gen}_{G, \rhobarss}[1/p]$ and hence is reduced.
\end{proof} 

\begin{cor}\label{gps_normal} $X^{\gps}_{G, \rhobarss}$ is reduced, $\OO$-flat and $X^{\gps}_{G, \rhobarss}[1/p]$ is normal. The map $X^{\gen}_{G, \rhobarss}\rightarrow X^{\gps}_{G, \rhobarss}$
induces a bijection between  the sets of irreducible components. 
\end{cor}
\begin{proof} Since $R^{\gps}_{G, \rhobarss}$ is a subring of $A^{\gen}_{G, \rhobarss}$ it is $\OO$-torsion free (and hence
$\OO$-flat) and reduced. Hence, the irreducible components of $R^{\gps}_{G, \rhobarss}$ and $R^{\gps}_{G, \rhobarss}[1/p]$ coincide.  Since $A^{\gen}_{G, \rhobarss}[1/p]$ is 
normal by Theorem \ref{gen_fib_normal} and $R^{\gps}_{G, \rhobarss}[1/p]=(A^{\gen}_{G, \rhobarss}[1/p])^{G^0}$, \cite[Proposition 6.4.1]{BH} implies that $R^{\gps}_{G, \rhobarss}[1/p]$ is normal. 
Since $G^0$ is connected the irreducible components of $X^{\gen}_{G, \rhobarss}[1/p]$ are 
$G^0$-invariant and since $X^{\gen}_{G, \rhobarss}[1/p]$ is normal its irreducible 
and connected components coincide. Thus irreducible components of $X^{\git}_{G, \rhobarss}[1/p]$ correspond to 
the GIT quotients of irreducible components of  $X^{\gen}_{G, \rhobarss}[1/p]$.
\end{proof}

\begin{cor}\label{ps_normal} The reduced subscheme $(X^{\ps}_G[1/p])^{\red}$ is normal. 
\end{cor} 
\begin{proof} Corollary \ref{gps_normal} implies that $X^{\gps}_{G, \rhobarss}[1/p]$ is reduced. Corollary \ref{nu_iso_red} gives us an isomorphism 
of schemes between $(X^{\ps}_G[1/p])^{\red}$ and $X^{\gps}_{G, \rhobarss}[1/p]$. The assertion follows from Corollary \ref{gps_normal}.
\end{proof}

\begin{cor}\label{bij_comp_Xgen} If $\pi_1(G')$ is \'etale then the map $X^{\gen}_{G, \rhobarss}\rightarrow X^{\gen}_{G/G', \psibarss}$ induces a bijection between the sets of irreducible components. 
\end{cor}
\begin{proof} Since both schemes are $\OO$-flat it is enough to verify the assertion after inverting $p$. Since $\OO[\mu][1/p]\cong \prod_{\chi\in \mathrm X(\mu)} L$ we have 
\begin{equation}\label{decomp_gen_fibre}
A^{\gen}_{G, \rhobarss}[1/p] = \prod_{\chi\in \mathrm X(\mu)} A^{\gen, \chi}_{G, \rhobarss}[1/p], \quad A^{\gen}_{G/G', \psibarss}[1/p] = \prod_{\chi\in \mathrm X(\mu)} A^{\gen, \chi}_{G/G', \psibarss}[1/p].
\end{equation}
It follows from \eqref{Agen_chi} that $A^{\gen,\chi}_{G/G', \psibarss}$ are $\OO$-flat integral domains. 
As $Z_1$ is trivial in our case, Theorem \ref{chi_nor_int},  Lemma \ref{non-empty} and \Cref{XgenG_chi_lci} 
imply that  $A^{\gen, \chi}_{G, \rhobarss}$ are $\OO$-flat integral domains. The claim follows from \eqref{decomp_gen_fibre}.
\end{proof}

\begin{cor}\label{irr_comp_ps} If $\pi_1(G')$ is \'etale then the natural maps $X^{\gen}_{G,\rhobarss}\rightarrow X^{\git}_{G,\rhobarss}\rightarrow X^{\ps}_G \rightarrow X^{\ps}_{G/G'}$ induce  bijections between 
the sets of irreducible components. 
\end{cor} 
\begin{proof} Since $X^{\git}_{G, \rhobarss}\rightarrow X^{\ps}_{G}$ is 
 a homeomorphism by Proposition \ref{nu_fin_u}, the assertion follows from Corollaries \ref{gps_normal}, \ref{bij_comp_Xgen}.
\end{proof}

\subsection{Conjecture of B\"ockle--Juschka}
We fix $z\in X^{\gen,\psi_1}_{G, \rhobarss}(k)$ and to ease the notation 
we write $\rhobar:=\rho_z$, $R^{\square}_{\rhobar}:= R^{\square,\psi_1}_{G, \rho_z}$, $\psibar:= \varphi\circ \rhobar$ and 
$R^{\square}_{\psibar}:= R^{\square,\psi_1}_{G/G', \psibar}$. It follows from 
\eqref{w_think_4} and \Cref{333} that 
\begin{equation}\label{R_box_psi}
R^{\square}_{\psibar}\cong \OO[\mu]\br{z_1, \ldots, z_t},
\end{equation}
where $t=\dim (H_2)_k ([F:\Qp]+1)$. We let $X^{\square}_{\rhobar}:=\Spec R^{\square}_{\rhobar}$ and 
$X^{\square}_{\psibar}:= \Spec R^{\square}_{\psibar}$.
\begin{cor}\label{flat_det} The natural map $R^{\square}_{\psibar} \rightarrow R^{\square}_{\rhobar}$ 
is flat. 
\end{cor}
\begin{proof} It follows from \eqref{R_box_psi} that there is an $\OO$-algebra homomorphism  $x: R^{\square}_{\psibar}\rightarrow \OO$.
Let $\psi: \Gamma_F \rightarrow (G/G')(\OO)$ be the deformation of 
$\psibar$ corresponding to $x$. Then 
\begin{equation}\label{butz_1}
R^{\square, \psi}_{\rhobar}\cong R^{\square}_{\rhobar}\otimes_{R^{\square}_{\psibar}, x} \OO,
\end{equation}
where the left hand side is defined taking $M_1=M$ and $M_2=0$. Reducing this identity modulo $\varpi$ we obtain an isomorphism 
\begin{equation}\label{butz_2}
R^{\square, \psi}_{\rhobar}/\varpi \cong R^{\square}_{\rhobar}\otimes_{R^{\square}_{\psibar}} k.
\end{equation}
Proposition \ref{present_over_RH} gives us a presentation 
\begin{equation}\label{present_again_chi}
R^{\square}_{\rhobar}\cong R^{\square}_{\psibar}\br{x_1,\ldots, x_r}/(f_1,\ldots, f_s)
\end{equation} 
with $r-s= \dim G'_k ([F:\Qp]+1) = \dim (R^{\square, \psi}_{\rhobar}/\varpi)$, where 
the last equality is given by Corollary \ref{complete_intersection_chi} (2) and \eqref{dim_Gbar} noting that $\dim (H_2)_k=0$ in the `fixed determinant'  case. Lemma \ref{pre_lci_flat} 
applied with $A= R^{\square}_{\psibar}$ and $B=R^{\square}_{\rhobar}$
implies the assertion. 
\end{proof}

\begin{cor}\label{R_norm_red} $R^{\square}_{\rhobar}[1/p]$ is normal and $R^{\square}_{\rhobar}$ is reduced. 
\end{cor} 
\begin{proof} Since $R^{\square}_{\rhobar}$ is $\OO$-flat either by Lemma \ref{flat_det} or by Remark \ref{1313}, it is a subring of $R^{\square}_{\rhobar}[1/p]$ and hence it is enough to show that $R^{\square}_{\rhobar}[1/p]$ is normal. Further, 
it is enough to show that $R^{\square, \chi}_{\rhobar}[1/p]$ is normal 
as $R^{\square}_{\rhobar}[1/p]$ is a 
finite product of such rings. Since normality is a local property, 
it is enough to show that the localisations of 
$R^{\square,\chi}_{\rhobar}[1/p]$ at maximal ideals are normal. 
Since $R^{\square}_{\rhobar}$ is a complete noetherian local ring, it is excellent and hence it is enough to show that the 
completions of $R^{\square, \chi}_{\rhobar}[1/p]$ at maximal ideals are normal. These are isomorphic to $R^{\square, \psi_1, \chi}_{G, \rho_x}$ by \Cref{extend_scalars2}, which are normal by \Cref{reg_codim_x_gen}.
\end{proof}

The following result proves an analogue of the conjecture of B\"ockle--Juschka posed originally for $G=\GL_d$ (and proved in that case in \cite{BIP_new}). 
\begin{cor}\label{conj_BJ} If $\pi_1(G')$ is \'etale then the natural map $R^{\square}_{\psibar} \rightarrow R^{\square}_{\rhobar}$ induces a bijection between the 
sets of irreducible components. 
\end{cor}
\begin{proof} It follows from \eqref{R_box_psi} that 
$R^{\square}_{\psibar}$ is $\OO$-flat and its irreducible 
components are given by $R^{\square, \chi}_{\psibar}$ for $\chi\in \mathrm X(\mu)$, which are formally smooth and hence flat over $\OO$.

The natural map $R^{\square}_{\psibar} \rightarrow R^{\square}_{\rhobar}$
induces a map of local rings $R^{\square,\chi}_{\psibar} \rightarrow 
R^{\square,\chi}_{\rhobar}$. Since the reduced special fibres of 
of $X^{\gen,\psi_1,  \chi}_{G, \rhobarss}$ and of 
$X^{\gen, \psi_1}_{G, \rhobarss}$ coincide by Lemma \ref{dim_XgenG_chi}, 
the point $z$ lies in $X^{\gen,\psi_1,  \chi}_{G, \rhobarss}(k)$ for every 
$\chi\in \mathrm X(\mu)$. Hence, $R^{\square,\chi}_{\rhobar}$  is an $\OO$-flat integral domain by Corollaries \ref{complete_intersection_chi}, \ref{reg_codim_x}. 
Thus $R^{\square,\chi}_{\rhobar}[1/p]$ is an integral domain for every 
$\chi\in \mathrm X(\mu)$. Since $R^{\square}_{\psibar}$ is $\OO$-flat, 
Lemma \ref{flat_det} implies that $R^{\square}_{\rhobar}$ is $\OO$-flat 
(this also follows from Remark \ref{1313}), hence the 
irreducible components of $R^{\square}_{\rhobar}$ and $R^{\square}_{\rhobar}[1/p]$
coincide. Since the order of $\mu$ becomes invertible once we invert $p$ 
we have $R^{\square}_{\rhobar}[1/p]\cong \prod_{\chi} R^{\square, \chi}_{\rhobar}[1/p]$, which implies the assertion.
\end{proof}

\begin{remar} If $x$ is a closed point of $X^{\gen, \psi_1}_{G,\rhobarss}[1/p]$ 
then the isomorphism $\OO[\mu][1/p]\cong \prod_{\chi\in \mathrm X(\mu)} L$
implies that there is a unique $\chi\in \mathrm X(\mu)$ such that 
$x\in X^{\gen, \psi_1,\chi}_{G,\rhobarss}$. Then $R^{\square,\psi_1}_{G, \rho_x}= R^{\square, \psi_1, \chi}_{G, \rho_x}$ is an integral domain 
by \Cref{reg_codim_x_gen}. If $\chi'\neq \chi$ then $R^{\square, \psi_1, \chi'}_{G, \rho_x}$ is the zero ring. 
\end{remar}

\subsection{Labelling irreducible components}\label{label} 
We have seen in the proofs of Corollaries \ref{bij_comp_Xgen},
\ref{irr_comp_ps}, \ref{conj_BJ} that if $\pi_1(G')$ is \'etale then the set of irreducible components of corresponding schemes are in bijection with the group of characters 
$\mathrm{X}(\mu)$. This bijection is non-canonical in general as the isomorphism 
$\alpha$ in \eqref{w_think_1} is non-canonical.

We show in \cite[Section 7.4]{defT} that there is a canonical action of $\mathrm X(\mu)$ 
on the set of irreducible components of $X^{\ps}_{H_2}= \Spec R^{\ps}_{H_2, \psibarss_2}$, and this action is faithful and transitive. So to fix an $\mathrm X(\mu)$-equivariant bijection 
between $\mathrm X(\mu)$ and the set of irreducible components 
it is enough to specify to which component the identity in $\mathrm X(\mu)$ 
corresponds to. We don't know how to do this canonically in general, 
except in one case, which we will now discuss. 

Assume that $H_2\cong H_2^0 \rtimes \Delta$. We have already identified 
$\Delta=\Gal(E/F)$ and $\psibar_2: \Gamma_F \rightarrow H_2(k)$ is a representation, such that composition of $\psibar_2$ with the projection
to $\Gal(E/F)$ is the map $\gamma \mapsto \gamma|_E$. The surjection 
$$H_2^0(\OO)=\Hom(M_2, \OO^{\times})\twoheadrightarrow \Hom(M_2, k^{\times})=
H_2^0(k)$$
has a canonical section given by the Teichm\"uller lift. This section is 
$\Delta$-equivariant and induces a section $\sigma: H_2(k)\rightarrow H_2(\OO)$.
The composition 
$\sigma\circ \psibar_2: \Gamma_F \rightarrow H_2(\OO)$ is a minimally ramified 
lift of $\psibar_2$. It is canonical and defines an $\OO$-valued point on  a unique  irreducible 
component of $X^{\gen}_{H_2,\psibarss_2}$ and its $H_2$-pseudocharacter defines 
an $\OO$-valued point on the unique irreducible component of 
$X^{\ps}_{H_2}$. We thus obtain a canonical bijection between the set of irreducible components
of $X^{\ps}_{H_2}$ and $\mathrm X(\mu)$, where the identity in $\mathrm X(\mu)$ is mapped to the 
distinguished component constructed above.

We will now give a different description of the bijection in the 
semi-direct product case. Let $X$ be an irreducible component of $X^{\ps}_{H_2}$. 
It follows from \eqref{w_think_1} and \eqref{w_think_3} that every 
point $x\in X^{\ps}_{H_2}(\Qpbar)$ lies on a unique irreducible component 
and every irreducible component contains a $\Qpbar$-valued point. 
We choose $x\in X(\Qpbar)$ and let $\psi_2: \Gamma_F \rightarrow H_2(\Qpbar)$ 
be a continuous representation such that its $H_2$-pseudocharacter corresponds to $x$. The representation $\psi_2$ is uniquely determined by $x$ up to $H_2^0(\Qpbar)$-conjugation. Using that $H_2$ is a semi-direct product, we 
may write $\psi_2(\gamma)=(\Phi(\gamma), \gamma|_E)$ for all 
$\gamma\in \Gamma_F$. Since $\psi_2$ is a continuous representation 
one may check that $\Phi: \Gamma_F \rightarrow H_2^0(\Qpbar)$ defines 
a continuous $1$-cocycle and the $H_2^0(\Qpbar)$-conjugacy class 
of $\psi_2$ defines a cohomology class in $H^1(\Gamma_F, H^0(\Qpbar))$. 
We show in \cite{defT} using arguments that go back to the work of Langlands 
\cite{langlands_tori} on his correspondence for tori that 
$$\Hom^{\cont}_{\mathrm{Group}}((\Gamma_F^{\ab,p}\otimes M_2)^{\Delta}, \Qpbar^{\times})\cong H^1(\Gamma_F, H^0_2(\Qpbar)).$$
Thus to $x\in X(\Qpbar)$ we may canonically attach a continuous character 
$$\tilde{\chi}: (\Gamma_F^{\ab,p}\otimes M_2)^{\Delta} \rightarrow \Qpbar^{\times}.$$
The Artin map induces an isomorphism between $\mu:=(\mu_{p^{\infty}}(E)\otimes M_2)^{\Delta}$ and the $p$-power torsion subgroup of $(\Gamma_F^{\ab,p}\otimes M_2)^{\Delta}$. 
Let $\chi: \mu\rightarrow  \Qpbar^{\times}$ be the  restriction  of $\tilde{\chi}$ to $\mu$. Then $\chi$ takes values in $\OO^{\times}\subset \Qpbar^{\times}$
and we show in \cite[Section 7.4]{defT} that the irreducible component $X$ is uniquely determined by $\chi$.  

\section{\texorpdfstring{Deformations into $L$-groups and $C$-groups}{Deformations into L-groups and C-groups}}\label{LandC}
In this section we  spell out that the results proved in sections \ref{sec_consequence} and \ref{nightmare_cont} are applicable, when $G$ is 
an $L$-group or a $C$-group. Galois representations with values in 
such groups appear in the Langlands correspondence. 

Let $H$ be a connected reductive 
group defined over $F$ and let $E$ be a finite Galois extension of $F$ such that $H_E$ is split. Let $T$ be a 
maximal split torus in $H_E$ and let $B$ be a Borel subgroup of $H_E$ containing 
$T$. To the triple $(H_E, B, T)$ one may attach a \emph{based root datum} 
$(X^{\ast}(T), \Phi, \Delta, X_{\ast}(T), \Phi^{\vee}, \Delta^{\vee})$, 
where the symbols denote characters of $T$, the set of all roots, the 
set of positive roots corresponding to $B$, cocharacters of $T$, the set 
of coroots, and the set of positive coroots corresponding to $B$, respectively,
see \cite[Section 1.5]{bcnrd}. 
The dual group is a triple $(\widehat{H}, \widehat{B}, \widehat{T}$), where 
$\widehat{H}$ is a connected split reductive group scheme defined over $\ZZ$, $\widehat{B}$ 
is a Borel subgroup scheme of $\widehat{H}$ containing a maximal split 
torus $\widehat{T}$, such that the based root datum of $(\widehat{H}, \widehat{B}, \widehat{T})$ is isomorphic to $(X_{\ast}(T), \Phi^{\vee}, \Delta^{\vee}, X^{\ast}(T), \Phi, \Delta)$. Out of the action of $\Gal(E/F)$ 
on $H_E$ one may construct a group of homomorphism of $\Gal(E/F)$ 
to the automorphism group of the based root datum. After fixing additional 
datum, called a pinning, one rigidifies the situation to obtain a group homomorphism 
$\Gal(E/F)\rightarrow \Aut(\widehat{H})$. We refer the reader to 
\cite[Section I.1]{borel_corvallis}, or \cite[Section 2]{DPS} for further details. The $L$-group 
is defined as 
$${}^L H:= \widehat{H}\rtimes \Gal(E/F),$$
which we will consider as a group scheme over $\OO$. An important
consequence of the construction is that there are canonical $\Gal(E/F)$-equivariant identifications
$$X_{\ast}(T)=X^{\ast}(\widehat{T}), \quad X^{\ast}(T)=X_{\ast}(\widehat{T}).$$

\begin{lem}\label{char_cochar} Let $Z(H)^0$ be the neutral component of the centre of $H$ and let $M$ the character lattice 
of $\widehat{H}/\widehat{H}'$. Then there is a canonical identification 
$$Z(H)^0(F)=(E^{\times} \otimes M)^{\Gal(E/F)}.$$
In particular, $\mu:=(\mu_{p^{\infty}}(E) \otimes M)^{\Gal(E/F)}$ gets canonically identified with $p$-power torsion subgroup $Z(H)^0(F)_{p^{\infty}}$ of $Z(H)^0(F)$.
\end{lem}
\begin{proof} The base change of $Z(H)^0$ to $E$ is equal to $Z(H_E)^0$. 
It follows from \cite[Proposition 14.2]{borel} that 
$$ X_{\ast}(Z(H_E)^0)=\{ \beta^{\vee}\in X_{\ast}(T): \langle \beta^{\vee}, \alpha\rangle=0, \forall \alpha\in \Phi\}.$$  
Thus the canonical identification $X_{\ast}(T)=X^{\ast}(\widehat{T})$ identifies $X_{\ast}(Z(H_E)^0)$ with $M$. Since
this identification is $\Gal(E/F)$-equivariant we obtain
\begin{equation}
\begin{split}
 Z(H)^0(F)&=Z(H)^0(E)^{\Gal(E/F)}=\Hom(X^{\ast}(Z(H_E)^0), E^{\times})^{\Gal(E/F)}\\
 &= (X_{\ast}(Z(H_E)^0)\otimes E^{\times})^{\Gal(E/F)}= (M\otimes E^{\times})^{\Gal(E/F)}.
 \end{split}
 \end{equation}
\end{proof}

We fix a continuous representation $\rhobar: \Gamma_F \rightarrow {}^L H(k)$ 
such that the composition of $\rhobar$ with the projection onto $\Gal(E/F)$ 
is the map $\gamma\mapsto \gamma|_E$. Representations satisfying this condition
are called \emph{admissible}. 

Let $D_{\rhobar}: \Aa_{\OO}\rightarrow \Set$ be the functor which maps
$(A, \mm_A)$ to the set of continuous representations $\rho_A:\Gamma_F\rightarrow 
{}^L H(A)$ such that $\rho_A(\gamma)\equiv \rhobar(\gamma)\pmod{\mm_A}$ for all 
$\gamma \in \Gamma_F$. This functor is representable by a complete local 
noetherian $\OO$-algebra $R^{\square}_{\rhobar}$. 
The results of  \Cref{sec_consequence} apply to $R^{\square}_{\rhobar}$.
We have $R^{\square}_{\rhobar}= R^{\square}_{G, \rho_x}$, where 
$G={}^L H$ and $x\in X^{\gen}_{G, \rhobarss}(k)$ is the point corresponding 
to $\rhobar$. We highlight a few of the results below.

\begin{thm}\label{L1} The ring $R^{\square}_{\rhobar}$ is $\OO$-flat, complete 
intersection of relative dimension $\dim \widehat{H}_k$. Its special 
fibre $R^{\square}_{\rhobar}/\varpi$ is complete intersection.
\end{thm}
\begin{proof} This follows from Corollary \ref{complete_intersection}, observing that $G^0= \widehat{H}$. 
\end{proof}

\begin{thm}[\Cref{R_norm_red}]\label{L1bis} $R^{\square}_{\rhobar}$ is reduced and $R^{\square}_{\rhobar}[1/p]$ is normal.
\end{thm}

We assume that $\OO$ contains all $p^m$-th roots of unity, where $p^m$ is the 
order of $Z(H)^0(F)_{p^{\infty}}$. We let $\mathrm X$ be the group of characters
$\chi: Z(H)^0(F)_{p^{\infty}}\rightarrow \OO^{\times}$. 

\begin{thm}\label{L2} If $\pi_1(\widehat{H}')$ is \'etale then there is a canonical bijection 
$$ \chi\mapsto \Spec R^{\square, \chi}_{\rhobar}$$
between $\mathrm X$ and the set of irreducible components of $\Spec R^{\square}_{\rhobar}$. The rings $R^{\square, \chi}_{\rhobar}$ and
their special fibres $R^{\square, \chi}_{\rhobar}/\varpi$ are complete intersection normal domains.
\end{thm}
\begin{proof} We have 
$ G/G'\cong (\widehat{H}/\widehat{H}')\rtimes \Gal(E/F).$
The assertion about the bijection follows from 
Corollary \ref{conj_BJ}, Section \ref{label} and Lemma \ref{char_cochar}. 

Let us make the bijection more explicit. If $X$ is an irreducible 
component of $\Spec R^{\square}_{\rhobar}$ then we may pick 
a geometric point $y\in X(\Qpbar)$. Then $y$ corresponds 
to a continuous representation $\rho: \Gamma_F \rightarrow {}^L H(\Qpbar)$. 
By quotienting out $\widehat{H}'(\Qpbar)$  we obtain a continuous 
admissible representation 
$$\psi: \Gamma_F \rightarrow ({}^L H/\widehat{H}')(\Qpbar)=(\widehat{H}/\widehat{H}')(\Qpbar)\rtimes \Gal(E/F).$$
The group ${}^LH/\widehat{H}'$ is the $L$-group for $Z(H)^0(F)$. 
The  $\widehat{H}(\Qpbar)$-conjugacy of $\psi$ corresponds to a continuous character $\tilde{\chi}: Z(H)^0(F)\rightarrow \Qpbar$ under the local Langlands
correspondence for tori, \cite{langlands_tori}, \cite{birkbeck}.  By restricting 
$\tilde{\chi}$ to $Z(H)^0(F)_{p^{\infty}}$ we obtain the required character $\chi$. 

The rest of the theorem follows from the results of \Cref{nightmare_cont}. 
Namely, $R^{\square,\chi}_{\rhobar}= R^{\square, \psi_1, \chi}_{G, \rho_x}$, 
where $G={}^L H$, $x\in X^{\gen}_{G, \rhobarss}(k)$ corresponds to $\rhobar$, $M_1=0$, so that $H_1= \Gal(E/F)$, and $\psi_1:\Gamma_F \rightarrow H_1(\OO)$
is the character $\gamma\mapsto \gamma|_E$. The assertions about the rings 
$R^{\square,\chi}_{\rhobar}$ and 
$R^{\square,\chi}_{\rhobar}/\varpi$ follow from Corollaries 
\ref{complete_intersection_chi}, \ref{reg_codim_x}.
\end{proof}

The $C$-group ${}^C H$ of $H$ is a variant of an $L$-group. 
It has been defined in \cite[Definition 5.3.2]{BG}; a nice exposition is given in  \cite[Section 1.1]{zhu}. For us only the following properties are important: 
there is an exact 
sequence 
\begin{equation}\label{need_1}
0\rightarrow {}^LH \rightarrow {}^CH \overset{d}{\rightarrow} \Gm\rightarrow 0.
\end{equation}
The derived subgroup of the neutral component of 
${}^C H$ is equal to $\widehat{H}'$ and \eqref{need_1} induces an isomorphism: 
\begin{equation}\label{need_2}
{}^CH/\widehat{H}'\cong {}^{L}H/\widehat{H}' \times \Gm.
\end{equation}
These properties of the $C$-group easily follow from its construction, see 
for example \cite[Section 2.1]{DPS}.
It follows from \eqref{need_2} that the subgroup scheme $\widehat{H}$ is normal in ${}^C H$ and \eqref{need_2} induces an isomorphism
\begin{equation}\label{need_3}
{}^CH/ \widehat{H}\cong \Gal(E/F)\times \Gm.
\end{equation}

Let $\rhobar: \Gamma_F\rightarrow {}^C H(k)$ be a continuous representation 
such that $d\circ \rhobar \equiv \chi_{\cyc} \pmod{\varpi}$, where $\chi_{\cyc}$ 
is the $p$-adic cyclotomic character, and the composition 
$\Gamma_F\overset{\rhobar}{\longrightarrow} {}^C H(k)\rightarrow 
\bigl({}^CH/ \widehat{H}\bigr)(k)$ with the projection onto $\Gal(E/F)$ 
under \eqref{need_3} is the map $\gamma \mapsto \gamma|_E$. 

Let $D^C_{\rhobar}: \Aa_{\OO}\rightarrow \Set$ be the functor such that 
$D^{C}_{\rhobar}(A)$ is the set of continuous representations 
$\rho_A: \Gamma_F\rightarrow {}^C H(A)$, such that 
$\rho_A\equiv \rhobar \pmod{\mm_A}$ and $d\circ \rho_A = \chi_{\cyc}\otimes_{\OO} A$. The definition of this 
deformation problem is motivated by the first two bullet points in \cite[Conjecture 5.3.4]{BG}, which  describes
Galois representations attached to $C$-algebraic automorphic forms. 
The functor $D^C_{\rhobar}$ is representable by a complete local noetherian $\OO$-algebra
$R^{\square,C}_{\rhobar}$ with residue field $k$.

\begin{thm}\label{C} Theorems \ref{L1}, \ref{L1bis} and \ref{L2} hold with $R^{\square, C}_{\rhobar}$
instead of $R^{\square}_{\rhobar}$.
\end{thm}
\begin{proof} It follows from \eqref{need_2} that we are in the setting of \Cref{nightmare_cont} with $G={}^C H$, $M_1=\ZZ$ with the trivial $\Gal(E/F)$-action, 
$M_2=X^{\ast}(\widehat{H}/\widehat{H}')$ with the natural $\Gal(E/F)$-action. 
Moreover, 
$$H_1= \Gal(E/F)\times \Gm, \quad  H_2= (\widehat{H}/\widehat{H}')\rtimes \Gal(E/F)$$
and $\psi_1: \Gamma_F \rightarrow H_1(\OO)$ the character 
$\gamma\mapsto (\gamma|_E, \chi_{\cyc}(\gamma))$. 
The ring $R^{\square, C}_{\rhobar}$ is the ring $R^{\square, \psi_1}_{G, \rho_x}$ 
with $x\in X^{\gen, \psi_1}_{G, \rhobarss}(k)$ corresponding to $\rhobar$. 

The assertions of Theorem \ref{L1} follow from Remark \ref{1313}, which indicates
how to modify the proof of Corollary \ref{complete_intersection_chi} to get the 
assertion. 

The assertions of Theorem \ref{L1bis} follow from Corollary \ref{R_norm_red}.

Lemma \ref{char_cochar} implies
that we may canonically identify $(\mu_{p^{\infty}}(E)\otimes M_2)^{\Gal(E/F)}$
with $Z(H)^0(F)_{p^{\infty}}$. It follows from the proof of \Cref{conj_BJ} 
 that the set of irreducible components of $R^{\square, C}_{\rhobar}$ is in bijection of the character group of $Z(H)^0(F)_{p^{\infty}}$. 
The irreducible component corresponding to $\chi$ is the spectrum 
of $R^{\square, \psi_1, \chi}_{G, \rho_x}$. The assertion of Theorem \ref{L2}
follows from Corollaries \ref{XgenG_chi_lci}, \ref{reg_codim_x} and observing 
that \eqref{dim_Gbar} implies $\dim \Gbar_k= \dim \widehat{H}_k$. 

It follows from \eqref{need_2} that $G/G'\cong (G^0/G')\rtimes \Gal(E/F)$.
Hence,  as we explain in section \ref{label}, the map associating an irreducible component to a character is canonical. Concretely, if $X$ is a an irreducible 
component of $\Spec R^{\square, C}_{\rhobar}$ then we may pick 
a geometric point $y\in X(\Qpbar)$. Then $y$ corresponds 
to a continuous representation $\rho: \Gamma_F \rightarrow {}^C H(\Qpbar)$. 
By quotienting out $\widehat{H}'(\Qpbar)$ and projecting onto 
$({}^L H/ \widehat{H}')(\Qpbar)$ under \eqref{need_2} we obtain a continuous 
admissible representation 
$$\psi_2: \Gamma_F \rightarrow (\widehat{H}/\widehat{H}')(\Qpbar)\rtimes \Gal(E/F).$$
We then proceed as in the proof of Theorem \ref{L2}. The group ${}^LH/\widehat{H}'$ is the $L$-group for $Z(H)^0(F)$. 
The  $\widehat{H}(\Qpbar)$-conjugacy of $\psi_2$ corresponds to a continuous character $\tilde{\chi}: Z(H)^0(F)\rightarrow \Qpbar$ under the local Langlands
correspondence for tori, \cite{langlands_tori}, \cite{birkbeck}.  By restricting 
$\tilde{\chi}$ to $Z(H)^0(F)_{p^{\infty}}$ we obtain the required character $\chi$. 
\end{proof}


\appendix

\section{Condensed sets}\label{app_cond_set}

In \cite{ScholzeCond} Clausen and Scholze introduced the notion of a \emph{condensed set}, which is a more algebraic notion of topological space.
It turns out that condensed sets behave better in algebraic contexts and lead to clearer statements.

A condensed set can be defined as an accessible functor $X : \mathrm{ProFin}^{\mathrm{op}} \to \Set$ on the opposite category of profinite sets \cite[Footnote 5]{ScholzeCond}, which satisfies the following two conditions:
\begin{enumerate}
    \item $X$ maps finite disjoint unions of profinite sets to finite products.
    \item For any surjection $S' \twoheadrightarrow S$ of profinite sets with projections $p_1, p_2 : S' \times_S S' \to S'$, the canonical map
    $$ X(S) \to \{x \in X(S') \mid p_1^*(x) = p_2^*(x) \in X(S' \times_S S')\} $$
    is bijective.
\end{enumerate}
A map of condensed sets is a natural transformation.
The category of condensed sets can be identified with the category of functors on the full subcategory of extremally disconnected profinite sets, which only satisfy condition (1) \cite[Proposition 2.7]{ScholzeCond}. We will switch between the two perspectives without further comment.

It can be shown that the category of condensed sets is (up to a set-theoretic condition) a coherent topos.
In particular, we can define all kinds of algebraic structures internal to condensed sets and thereby obtain a notion of condensed group, condensed ring, etc.
In general we can speak about quasi-compact and quasi-separated objects of a coherent topos.
We will start by giving a concrete characterisation of quasi-compact condensed sets and quasi-separated condensed sets.

\begin{lem}\label{profiniteiscompact}
    Profinite sets are compact objects in the category of condensed sets, i.e. for every filtered colimit of condensed sets $X = \varinjlim_i X_i$ and every profinite set $S$, we have $X(S) = \varinjlim_i X_i(S)$.
\end{lem}

\begin{proof}
    We have to check, that the presheaf colimit $\varinjlim_i X_i$ is already a sheaf. The sheaf condition as formulated in \cite[Section 1]{ScholzeCond} is in terms of finite limits, and these commute with filtered colimits.
\end{proof}

\begin{lem}\label{all_effective}
    Every surjection of condensed sets is an effective epimorphism.
\end{lem}

\begin{proof}
    For a surjection $f : X \to Y$ of condensed sets, we need to see that $f$ is the coequaliser of its kernel pair
    $X \times_Y X \rightrightarrows X$. This is a reflexive coequaliser and therefore sifted. The condensed sets $X$ and $Y$ can be seen as finite product preserving functors on the opposite category of the category of extremally disconnected profinite sets. Since sifted coequalisers commute with finite products, we can reduce the claim to sets. Every surjection of sets is an effective epimorphism. This completes the proof.
\end{proof}

\begin{lem}\label{char_qc}
    Let $X$ be a condensed set. The following are equivalent.
    \begin{enumerate}
        \item $X$ is a quasi-compact object of the topos of condensed sets. That is, for every collection of condensed sets $(T_i)_{i \in I}$ and every effective epimorphism $\coprod_{i \in I} T_i \twoheadrightarrow X$, there is a finite subset $I_0 \subseteq I$, such that $\coprod_{i \in I_0} T_i \twoheadrightarrow X$ is an effective epimorphism.
        \item There is a profinite set $T$ and a surjection $\underline T \twoheadrightarrow X$.
    \end{enumerate}
\end{lem}

\begin{proof}
    Suppose $X$ is a quasi-compact condensed set. By definition it arises as a left Kan extension from a $\kappa$-condensed set.
    By the density theorem \cite[Theorem III.7.1]{MacLaneCat} we can write $X$ as a small colimit of $\kappa$-small profinite sets.
    The inclusion functor from $\kappa$-condensed sets into all condensed set preserves colimits \cite[Proposition 2.9]{ScholzeCond}, so $X$ is a small colimit of profinite sets. In particular, there is an epimorphism $\coprod_{i \in I} \underline{S_i} \twoheadrightarrow X$, where $S_i$ are profinite sets. Since this epimorphism is effective by \Cref{all_effective} and $X$ is assumed quasi-compact, there is a finite subset $I_0 \subseteq I$, such that $\coprod_{i \in I_0} \underline{S_i} \twoheadrightarrow X$ is a surjection.

    Suppose there is a profinite set and a surjection $\underline T \twoheadrightarrow X$ and let $\coprod_{i \in I} X_i \twoheadrightarrow X$ be an effective epimorphism. By pullback we obtain a surjection $\coprod_{i \in I} (X_i \times_X T) \twoheadrightarrow T$. Since by \cite[Proposition 1.2 (3)]{ScholzeAnalytic} $T$ is a quasi-compact object in the category of condensed sets, there is a finite subset $I_0 \subseteq I$, such that $\coprod_{i \in I_0} (X_i \times_X T) \twoheadrightarrow T$ is an effective epimorphism. The map $\coprod_{i \in I_0} (X_i \times_X T) \twoheadrightarrow X$ factors over $\coprod_{i \in I_0} X_i \twoheadrightarrow X$, which is thus also a surjection.
\end{proof}

\begin{lem}\label{char_qs}
    Let $X$ be a condensed set. The following are equivalent.
    \begin{enumerate}
        \item $X$ is a quasi-separated object of the topos of condensed sets. That is, for every pair of morphisms $Y \to X \leftarrow Z$, where $Y$ and $Z$ are quasi-compact the fibre-product $Y \times_X Z$ is quasi-compact.
        \item For every profinite set $S$ together with a morphism $\underline S \to X \times X$, the pullback $\underline S \times_{X \times X} X$ along the diagonal $\Delta_X : X \to X \times X$ is quasi-compact.
    \end{enumerate}
\end{lem}

\begin{proof}
    Assume, that $X$ is a quasi-separated condensed set.
    Let $S$ be a profinite set together with a map $\underline S \to X \times X$. Then $\underline S \times_X \underline S$ is quasi-compact where the maps $\underline S \to X$ are given by composing with the projections.    
    In the following diagram all squares are cartesian.
    \begin{center}
    \begin{tikzcd}
        \underline S_{X \times X} X \arrow[d] \arrow[r] & \underline S \times_X \underline S \arrow[d] \arrow[r] & X \arrow[d, "\Delta_X"] \\
        \underline S \arrow[r] & \underline S \times \underline S \arrow[r] & X \times X
    \end{tikzcd}
    \end{center}
    Since $\underline S \times \underline S = \underline{S \times S}$ comes from a profinite set it is quasi-separated by \cite[Proposition 1.2 (3)]{ScholzeAnalytic}.
    So $\underline S_{X \times X} X$ is quasi-compact.

    Now suppose, that $X$ satisfies (2). Let $Y \to X \leftarrow Z$ be a pair of morphisms, where $Y$ and $Z$ are quasi-compact.
    By \Cref{char_qc}, there exist profinite sets $S$ and $T$ and surjections $\underline S \twoheadrightarrow Y$ and $\underline T \twoheadrightarrow Z$.
    We get a diagram in which are squares are cartesian:
    \begin{center}
    \begin{tikzcd}
        \underline S \times_X \underline T \arrow[d] \arrow[r, twoheadrightarrow] & Y \times_X Z \arrow[d] \arrow[r] & X \arrow[d, "\Delta_X"] \\
        \underline S \times \underline T \arrow[r, twoheadrightarrow] & Y \times Z \arrow[r] & X \times X
    \end{tikzcd}
    \end{center}
    By (2) $\underline S \times_X \underline T$ is quasi-compact, hence by one further application of \Cref{char_qc} also $Y \times_X Z$.
\end{proof}

\begin{lem}\label{injlimisqs}
    A filtered colimit of quasi-compact quasi-separated condensed sets along quasi-compact injections is quasi-separated.
\end{lem}

\begin{proof}
    Let $X = \varinjlim_i X_i$ be a filtered colimit of quasi-compact quasi-separated condensed sets where all transition maps in the diagram are injections.
    Let $T$ be a profinite set together with a map $T \to X \times X$.
    The map $X_i \times X_i \to X \times X$ is still a quasi-compact injection, so the fibre product $T_i := T \times_{X \times X} (X_i \times X_i)$ is quasi-compact. We know that the diagonals $\Delta_{X_i} : X_i \to X_i \times X_i$ are quasi-compact.
    Since in the diagram
    \begin{center}
    \begin{tikzcd}
        T \times_{X_i \times X_i} X_i \arrow[r] \arrow[d, hookrightarrow] & X_i \arrow[r] \arrow[d, hookrightarrow, "\Delta_{X_i}"] & X \arrow[d, "\Delta_X"] \\
        T \arrow[r] & X_i \times X_i \arrow[r] & X \times X
    \end{tikzcd}
    \end{center}
    both squares are cartesian, we have $T \times_{X \times X} X = T \times_{X_i \times X_i} X_i$, which is quasi-compact.
    The claim follows from \Cref{char_qs}.
\end{proof}

\begin{lem}\label{subsp_qs}
    Let $X$ be a quasi-separated condensed set and let $Y \hookrightarrow X$ be an injection.
    Then $Y$ is quasi-separated.
\end{lem}

\begin{proof}
    Let $T$ be a profinite set. In the diagram
    \begin{center}
    \begin{tikzcd}
        T \times_{Y \times Y} Y \arrow[r] \arrow[d] & Y \arrow[r, hookrightarrow] \arrow[d, "\Delta_Y"] & X \arrow[d, "\Delta_X"] \\
        T \arrow[r] & Y \times Y \arrow[r, hookrightarrow] & X \times X
    \end{tikzcd}
    \end{center}
    the right square is cartesian, since $Y \to X$ is injective. So the outer square is cartesian.
    Hence $T \times_{Y \times Y} Y$ is quasi-compact, so $Y$ is quasi-separated by \Cref{char_qs}. 
\end{proof}

\begin{lem}\label{cond_set_unique}
    Let $S$ be a profinite set, let $X$ be a quasi-separated condensed set and let $f : S \to X(*)$ be a map.
    Then there is at most one morphism of condensed sets $\tilde f : \underline{S} \to X$, whose map on underlying sets is $f$.
\end{lem}

\begin{proof} 
    By \cite[Proposition 1.2 (4)]{ScholzeAnalytic} $X$ is a filtered colimit of quasi-compact quasi-separated condensed sets $X_i$ along injections. Recall, that the $X_i$ can be seen as compact Hausdorff spaces \cite[Proposition 1.2 (3)]{ScholzeAnalytic}. So for any profinite set $S$
    $$ X(S) = \bigcup_i \Cont(S, X_i) \subseteq \Map(S, X(*)) $$
    using \Cref{profiniteiscompact}.
\end{proof}

\subsection{Condensed structure on points}

Let $X : \mathrm{CRing} \to \Set$ be an accessible presheaf and let $A : \mathrm{ProFin}^{\mathrm{op}} \to \mathrm{CRing}$ be a condensed commutative ring.
We denote by $X(A)$ the condensed set obtained by composition $X \circ A$ with the functor of points of $X$.
By construction $X \mapsto X(A)$ preserves limits.

\begin{lem}\label{cl_subscheme_qc_inj}
    Let $X \hookrightarrow \mathbb A^n_R$ be a closed immersion of schemes over a commutative ring $R$ and let $A$ be a quasi-separated condensed $R$-algebra.
    Then $X(A) \hookrightarrow A^n$ is a quasi-compact injection.
\end{lem}

\begin{proof}
    We see $X$ as the zero set of a family of polynomials $g_1, \dots, g_r \in R[x_1, \dots, x_n]$.
    We can see them as a map of condensed sets $\underline g : A^n \to A$.
    Since $A$ is quasi-separated $\Delta_A : A \to A \times A$ is quasi-compact and the map $0 : * \to A$, is quasi-compact by pullback along $(0, \id) : A \to A \times A$.
    It follows from the cartesian diagram
    \begin{center}
        \begin{tikzcd}
            X(A) \arrow[r] \arrow[d] & \{0\} \arrow[d] \\
            A^n \arrow[r, "\underline{g}"] & A
        \end{tikzcd}
    \end{center}
    that $X(A) \to A^n$ is a quasi-compact injection.
\end{proof}



\printunsrtglossary[type=main, style=tree]

\clearpage

\bibliographystyle{plain}
\bibliography{Ref}

\end{document}